\documentclass[10pt, a4paper]{amsart}

\usepackage{amsmath}
\usepackage{tikz}
\usepackage{amssymb}
\usepackage{geometry}
\usepackage{amsthm}%%% it must be load
\usepackage{tikz-cd}
\usepackage[all]{xy}
\usepackage{amsfonts}
\usepackage{mathrsfs}%%% to use \mathscr
\usepackage{hyperref}%%% for hyperlinks
\usepackage{xcolor}
\usepackage{MnSymbol} % provides the ``end of def'' symbol
\usepackage{mathtools}
\usepackage{subfig, caption} % provides subfloats
\usepackage{wrapfig}
\newcommand{\CC}{\mathbb{C}}

\newcommand{\HH}{\mathbb{H}}

\newcommand{\NN}{\mathbb{N}}

\newcommand{\RR}{\mathbb{R}}

\newcommand{\ZZ}{\mathbb{Z}}

\newcommand{\cC}{\mathcal{C}}
\newcommand{\cD}{\mathcal{D}}

\newcommand{\cF}{\mathcal{F}}
\newcommand{\cG}{\mathcal{G}}

\newcommand{\cI}{\mathcal{I}}

\newcommand{\cL}{\mathcal{L}}
\newcommand{\cM}{\mathcal{M}}

\newcommand{\cO}{\mathcal{O}}

\newcommand{\cS}{\mathcal{S}}
\newcommand{\cT}{\mathcal{T}}

\newcommand{\cV}{\mathcal{V}}

\newcommand{\sC}{\mathscr{C}}

\newcommand{\sF}{\mathscr{F}}
\newcommand{\sG}{\mathscr{G}}
\newcommand{\sH}{\mathscr{H}}

\newcommand{\sL}{\mathscr{L}}
\newcommand{\sM}{\mathscr{M}}

\newcommand{\sO}{\mathscr{O}}

\newcommand{\sQ}{\mathscr{Q}}

\newcommand{\sV}{\mathscr{V}}

\newcommand{\fC}{\mathsf{C}}

\newcommand{\fE}{\mathsf{E}}
\newcommand{\fF}{\mathsf{F}}

\newcommand{\fL}{\mathsf{L}}
\newcommand{\fM}{\mathsf{M}}

\newcommand{\fP}{\mathsf{P}}
\newcommand{\fQ}{\mathsf{Q}}
\newcommand{\fR}{\mathsf{R}}
\newcommand{\fS}{\mathsf{S}}
\newcommand{\fT}{\mathsf{T}}

\newcommand{\fV}{\mathsf{V}}
\newcommand{\fW}{\mathsf{W}}

\newcommand{\Homeop}{\operatorname{Homeo^+}}

\newcommand{\E}[1]{\operatorname{E}(#1)}% the end point set of a laminations system.
\newcommand{\stem}[2]{S_{#1}^{#2}}
\newcommand{\closure}[1]{\overline{#1}}
\newcommand{\interior}[1]{\operatorname{Int}\left( #1 \right)}

\newcommand{\Fix}[1]{\operatorname{Fix}(#1)}
\newcommand{\Per}[1]{\operatorname{Per}(#1)}
\newcommand{\Stab}[2]{\operatorname{Stab}_{#1}(#2)}
\newcommand{\tipp}[1]{\curlyvee(#1)}
\newcommand{\tipg}[1]{\diamondsuit(#1)}

\newcommand{\PSL}[1]{\operatorname{PSL}_2(#1)}

\newcommand{\COL}{\operatorname{COL}}
\newcommand{\Aut}{\operatorname{Aut}}
\newcommand{\cphi}{\varphi_{\operatorname{co}}}
\newcommand{\id}{\operatorname{Id}}

\newcommand{\opi}[3][S^1]{(#2, #3)_{#1}}
\newcommand{\ropi}[3][S^1]{[#2, #3)_{#1}}
\newcommand{\lopi}[3][S^1]{(#2, #3]_{#1}}
\newcommand{\cldi}[3][S^1]{[#2, #3]_{#1}}

\newcommand{\seq}[2][n]{\{ #2\}_{#1=1}^\infty}

\newcommand{\seqc}[3]{\{#1_{#2}\}_{#2\in \ZZ_{#3}}}

\newcommand{\im}{\operatorname{Im}}
\newcommand{\re}{\operatorname{Re}}

\newcommand{\emap}[2]{\varphi_{#1\to#2}}

\newcommand{\mrk}{\mathbf{m}}
\newcommand{\Mrk}{\mathbf{M}}
\newcommand{\ord}{\operatorname{ord}}
\newcommand{\cusp}{\operatorname{cusp}}

\makeatletter
\let\c@equation\c@subsection
\makeatother
\numberwithin{equation}{section} 

\makeatletter
\let\c@figure\c@equation
\makeatother
\numberwithin{figure}{section}

\theoremstyle{plain}
\newtheorem{thm}[equation]{Theorem}
\newtheorem{cor}[equation]{Corollary}
\newtheorem{lem}[equation]{Lemma}

\newtheorem{prop}[equation]{Proposition}

\theoremstyle{definition}
\newtheorem{defn/}[equation]{Definition}
\newtheorem{ex}[equation]{Example}

\newtheorem{ques}[equation]{Question}

\newenvironment{defn}
  { 
  \pushQED{\qed} \begin{defn/}}
  {\popQED \end{defn/}}

\theoremstyle{remark}
\newtheorem{rmk/}[equation]{Remark}
\newtheorem*{rmk*}{Remark}

\newtheorem*{case*}{Case}

\newtheorem{claim}[equation]{Claim}
\newtheorem*{claim*}{Claim}

\newenvironment{rmk}
  {
   \pushQED{\qed}\begin{rmk/}}
  {\popQED\end{rmk/}}

\newcommand{\refsec}[1]{Section~\ref{Sec:#1}}

\newcommand{\refthm}[1]{Theorem~\ref{Thm:#1}}
\newcommand{\refcor}[1]{Corollary~\ref{Cor:#1}}
\newcommand{\reflem}[1]{Lemma~\ref{Lem:#1}}
\newcommand{\refprop}[1]{Proposition~\ref{Prop:#1}}
\newcommand{\refclm}[1]{Claim~\ref{Clm:#1}}
\newcommand{\refrmk}[1]{Remark~\ref{Rmk:#1}}

\newcommand{\reffig}[1]{Figure~\ref{Fig:#1}}

\newcommand{\refitm}[1]{(\ref{Itm:#1})}

\newcommand{\fakeenv}{} %%% prints the emptystring

\newenvironment{restate}[2]  %%% restate takes two arguments 
{ 
 \renewcommand{\fakeenv}{#2} %%% So now \fakeenv prints #2
 \theoremstyle{plain} 
 \newtheorem*{\fakeenv}{#1~\ref{#2}} %%% so now #2 is the name of a
 \begin{\fakeenv}
}
{
 \end{\fakeenv}
}

\title{Groups acting on veering pairs and Kleinian groups}

\author[Baik]{Hyungryul Baik}
\address{\hskip-\parindent
Department of Mathematical Sciences\\
Korea Advanced Institute of Science and Technology (KAIST)\\
291 Daehak-ro, Yuseong-gu,  Daejeon 34141, Republic of Korea}
\email{hrbaik@kaist.ac.kr }

\author[Jung]{Hongtaek Jung}
\address{\hskip-\parindent
Center for Geometry and Physics\\  Institute for Basic Science (IBS), Pohang 37673, Korea}
\email{htjung@ibs.re.kr }

\author[Kim]{KyeongRo Kim}
\address{\hskip-\parindent
Research institute of Mathematics\\
Seoul National University\\
GwanAkRo 1, Gwanak- Gu, Seoul 08826, Korea}
\email{cantor14@snu.ac.kr}

\date{\today}

\begin{document}

\begin{abstract}
We show that some subgroups of the orientation preserving circle homeomorphism group with invariant veering pairs of laminations are hyperbolic 3-orbifold groups. On the way, we show that from a veering pair of laminations, one can construct a loom space (in the sense of Schleimer-Segerman) as a quotient. Our approach does not assume the existence of any 3-manifold to begin with so this is a geometrization-type result, and supersedes some of the results regarding the relation among veering triangulations, pseudo-Anosov flows, taut foliations in the literature. 
\end{abstract}


\maketitle
%\tableofcontents

\section{Introduction}

\subsection{Background and Motivation}
The main theme of geometric group theory is to draw interesting algebraic properties of a group from the geometry of its actions on spaces. In that respect, two fundamental questions for a given group are: whether can this group act on some space nicely and what properties does this group have if it acts on some space nicely.

Let us focus on the fundamental groups of compact orientable 3-manifolds, and ask whether they can act nicely on some space. It turns out that these groups admit various interesting actions on the circle and the sphere. It is classic that the fundamental group of a closed hyperbolic surface always acts on the circle, and this action is known to be convergence, minimal and faithful. The fundamental group of a closed hyperbolic 3-manifold group also admits a convergence action on the sphere. Less obviously, the fundamental group of a closed 3-manifold with a pseudo-Anosov flow admits a convergence action on the sphere as well \cite{Fenley12}. On the other hand, the fundamental group of a closed orientable atoroidal 3-manifold with a co-orientable taut foliation, or more generally with a certain type of essential lamination, acts faithfully on the circle \cite{CalegariDunfield03, Thurston97}. An interesting feature of this action is that it leaves some structure, called a \emph{circle lamination}, invariant. In some circumstances, these two classes of actions are closely related by some maps, so-called Cannon-Thurston maps. See \cite{CannonThurston07} and \cite{Bowditch07}. 

Conversely, what can we say about group that admits a nice action on the circle or the sphere? For example, suppose that a group admits a convergence action on the circle and ask whether we can characterize this group. A surprisingly results proven by many authors, Tukia \cite{Tukia88},  Gabai \cite{Gabai92}, and Casson-Jungreis \cite{CassonJungreis94} (cf. Hinkkanen\cite{Hinkkanen90} for the indiscrete case), show that such a group must be a Fuchsian group (including a hyperbolic surface group). On the other hand, whether a group with a convergence action on the sphere is a (virtually) hyperbolic 3-manifold group, which is one part of Cannon's conjecture, is still open. And much less is known for groups acting on the circle with invariant circle laminations.

This is where the study of \emph{laminar groups}, subgroups of $\Homeop(S^1)$ with invariant circle laminations, is initiated \cite{CalegariDunfield03, cal06, Calegari07}. Perhaps, one of the earliest results on this object is \cite{Baik15}, showing that a laminar group preserving three special kind of circle laminations must be a surface group. After this work, many subsequent studies have been published by various authors including \cite{AlonsoBaikSamperton}, \cite{BaikKim20}, and \cite{BaikKim20}. One of the implicit goals of these works is to generalize \cite{Baik15} and find decent conditions for a laminar group to become a hyperbolic 3-manifold group.

Behavior of a laminar group is controlled by properties of circle laminations that the group preserves. Therefore, the study of laminar groups boils down to fine-tune the invariant circle laminations so that the laminar group preserving these circle laminations has the desired property. 

The concept of veering triangulations, namely, an ideal triangulation on a cusped 3-manifold with extra data that captures combinatorial features of pseudo-Anosov flow, was proposed by Agol \cite{Agol10}. Over the last decade, the theory of veering triagulations has been studied by many authors including \cite{FuterGueritaud13}, \cite{HodgsonRubinsteinSegerman},  \cite{HodgsonRubinsteinSegermanTillmann}, 
\cite{LandryMinskyTaylor20}, 
\cite{LandryMinskyTaylor}, \cite{SchleimerSegerman19}, 
\cite{SchleimerSegerman20},
\cite{SchleimerSegerman21}, \cite{Landry18}, \cite{Landry19}, \cite{Landry22}, and \cite{AgolTsang22}. Recently, the theory of veering triangulations is spotlighted to study the Thurston norm balls and pseudo-Anosov flow theory. Among them, Frankel-Schleimer-Segerman \cite{SchleimerSegerman19} (see also \cite{SchleimerSegerman20}, \cite{SchleimerSegerman21}) show that the fundamental group of a veering triangulated 3-manifold  can act faithfully on the circle (so-called the veering circle) and leave a pair of special circle laminations invariant. Moreover, Schleimer-Segerman \cite{SchleimerSegerman19, SchleimerSegerman21} also show that one can functorially build a veering triangulated 3-manifold  from these circle laminations data. Their results suggest promising candidates of circle laminations such that a laminar group preserving them becomes a hyperbolic 3-manifold group.

\subsection{Main Results}
In this paper, we will introduce \emph{veering pair} generalizing the pair of circle laminations constructed in \cite{SchleimerSegerman19} and study laminar groups that preserve a veering pair. Our final goal is to show that a veering pair preserving laminar group is the fundamental group of an irreducible 3-orbifold. If, in addition, the action of a laminar group is ``cocompact'' we can prove that this group is the fundamental group of a hyperbolic 3-orbifold. We also present how to construct a loom space from an abstract veering pair, generalizing the loom space construction of \cite{SchleimerSegerman19}.

\subsubsection{Circle Laminations and Loom Spaces}
In the first half of the paper, we show that one can construct a loom space out of a pair of circle laminations called \emph{veering pair}, generalizing the construction of \cite{SchleimerSegerman19}. This construction is also functorial in a sense that a laminar group that preserves the veering pair acts on the loom space so constructed as loom isomorphisms.

To present our result, we briefly recall and introduce some terminology. Regard the circle $S^1$ as the ideal boundary of the hyperbolic disk $\HH^2$. Given a set of unoriented geodesics $L$ in $\HH^2$, their endpoints $e(L)$ define a subset in $\cM:=(S^1\times S^1 \setminus \{(x,x)\,:\, x\in S^1\})/{\sim}$, where $(x,y)\sim (w,z)$ if and only if $x=z$ and $y=w$. A \emph{circle lamination} is a subset $\lambda$ of $\cM$ such that $\lambda=e(\overline{\lambda})$ for some (in fact, unique) geodesic lamination $\overline{\lambda}$ of $\HH^2$. Connected components of the complement $\HH^2\setminus \overline{\lambda}$  are called \emph{gaps}\footnote{For reader's convenience, we use the traditional definition for gaps at this moment. What we call gaps here will be referred to as \emph{non-leaf gaps} in the main text.} of $\lambda$. As mentioned above, a 3-manifold $M$ with a veering triangulation (e.g., pseudo-Anosov mapping torus with singular orbits removed) gives rise to a pair of circle laminations $(\lambda^+, \lambda^-)$ invariant under the $\pi_1(M)$-action on $S^1$. We recall key features of this pair of circle laminations $\lambda^\pm$. See \cite{SchleimerSegerman19} for details.
\begin{itemize}
    \item (Crown gaps) Each gap of $\lambda^\pm$ is a crown, namely, an infinite polygon with vertices accumulate to a unique point. See \reffig{crownexample}.
    \item (Loose) Each $\lambda^\pm$ is \emph{loose}, namely, two complementary crowns of $\lambda^\pm$ share no vertices.
    \item (Strongly transverse) $\lambda^+$ and $\lambda^-$ are \emph{strongly transverse}, i.e., end $\ell_1 \cap \ell_2=\emptyset$ as subsets in $S^1$ for every pair $(\ell_1,\ell_2)\in \overline{\lambda^+} \times \overline{\lambda^-}$.
    \item (Interleaving gaps) Given a crown gap $C^+$ of $\lambda^+$, there is a crown gap $C^-$ of $\lambda^-$ such that the set of vertices of $C^+$ and $C^-$ alternate in $S^1$.
\end{itemize}
Motivated from this, we define a \emph{veering pair} to be a pair of circle laminations $\sV=\{\sL_1, \sL_2\}$ such that
\begin{itemize}
    \item $\sL_1$ and $\sL_2$ are loose.
    \item $\sL_1$ and $\sL_2$ are \emph{quite full}, namely, each gap is either an ideal polygon or a crown.
    \item $\sL_1$ and $\sL_2$ are strongly transverse.
    \item For each gap $\sG_i$ of $\sL_i$, there is a gap $\sG_j$ of $\sL_j$, $j\ne i$, that interleaves with $\sG_i$. 
\end{itemize}
See \refsec{veeringsec} for the detailed definition. Note that in the main text, we use slightly different notion, called \emph{lamination system} instead of circle lamination which is more proper  to keep track of the circle actions. Although there are some differences in usage of terminology, the key ideas remain the same.

The major difference between veering pairs and pair of circle laminations in \cite{SchleimerSegerman19} is the potential existence of polygon gaps. Having polygon gaps is a more general and natural setting because it allows us to take torsion elements of laminar groups into account. 

Our construction of loom spaces is along the similar line with that of \cite{SchleimerSegerman19}. However, a veering pair is just an abstract object and is not necessarily induced from a veering triangulation. This lack of the background geometry of the veering pair makes extra difficulties. We have to exploit the abstract properties of veering pair. 

First, we define the \emph{stitch space}
$$\fS(\sV):=\{(\ell_1,\ell_2)\,:\,\ell_i\text{ is a leaf of } \overline{\sL_i} \text{ such that } \ell_1\cap \ell_2 \ne \emptyset\},$$ which corresponds to the link space in \cite{SchleimerSegerman19}. Each element of $\fS(\sV)$ is called a \emph{stitch}. We introduce the \emph{weaving relation} $\sim_\omega$ on $\fS(\sV)$ and consider the quotient space $\overline{\fW}(\sV):= \fS(\sV)/{\sim_\omega}$, called the \emph{cusped weaving}. Each equivalence class in $\overline{\fW}(\sV)$ consists of a single stitch, two stitches, four stitches,  or the set of stitches that forms interleaving polygons or crowns. The equivalence class of interleaving crowns is called a \emph{cusp class}, which corresponds to a boundary point of $\overline{\fW}(\sV)$. The \emph{weaving} $\fW(\sV):=\overline{\fW}(\sV)\setminus\{\text{cusp classes}\}$  is an intermediate space toward a loom space. 
\begin{restate}{Theorem}{Thm:weaving}
Let $\sV$ be a veering pair. The weaving $\fW(\sV)$ is homeomorphic to an open disk with transverse (singular) foliations induced from the circle laminations.
\end{restate}

Singular loci of the foliations, called \emph{singular  classes}, come from interleaving pairs of non-leaf ideal polygons. Later, these singular classes turn out to be potential fixed points of finite order elements of a veering pair preserving laminar group. On the other hand, if a laminar group contains 2-torsions, a pair of real leaves may happen to be a fixed point. To capture this information, we introduce a set $\Mrk$ of \emph{markings} on the stitches. The required properties of $\Mrk$ are
\begin{itemize}
\item $\Mrk$ is a closed and discrete subspace of $\fS(\sV)$.
\item Each element of $\Mrk$ consists of a pair of real leaves.
\item Each leaf of $\sV$ is a component of at most one element of $\Mrk$.
\end{itemize}
The corresponding equivalence classes in $\fW(\sV,\Mrk)$ are called \emph{marked classes}. The marked and singular classes are potential fixed points of finite order elements in a laminar group. 

Now we simply remove these singular and marked classes to obtain $\fW^\circ(\sV,\Mrk)$ called the \emph{regular weaving}. This space admits a pair of transverse foliations but not homeomorphic to $\RR^2$. Hence, we take the universal cover $\widetilde{\fW^\circ}(\sV,\Mrk)$ to obtain our potential candidates for a loom space. Then we have the following:
\begin{restate}{Theorem}{Thm:loomspaceconstruction}
Let $\sV$ be a veering pair and let $\Mrk$ be any marking on the stitch $\fS(\sV)$. Then $\widetilde{\fW^\circ}(\sV,\Mrk)$ is a loom space. 
\end{restate}

\subsubsection{Veering Pair Preserving Actions on the Circle}

The second half is devoted to studying groups preserving veering pairs. The main question is to understand the structure of these groups and, at the end, we will show that under a ``cocompact'' assumption they are the fundamental groups of hyperbolic 3-orbifolds. 

For a veering pair $\sV=\{\sL_1,\sL_2\}$, we denote by $\Aut(\sV)$ the group of elements in $\Homeop(S^1)$ that preserve $\sL_1$ and $\sL_2$ individually. 

We first need to understand dynamics of individual elements of $\Aut(\sV)$. Recall that a nontrivial element of $\operatorname{PSL}_2(\RR)$ is elliptic, parabolic or hyperbolic depending on its dynamics on $S^1$. Quite similar classification holds for elements of $\Aut(\sV)$. We rephrase our result as follow.
\begin{restate}{Theorem}{Thm:classOfLaminarAut}
Let $\sV=\{\sL_1, \sL_2\}$ be a veering pair. Let $g$ be a non-trivial element in $\Aut(\sV)$. Then, $g$ falls into one of the following cases:
\begin{enumerate}
    \item (Elliptic) $g$ has finite order and there is an interleaving pair of ideal polygons\footnote{We always think of  geodesics as ideal bi-gons in this paper.} of $\sV$ preserved by $g.$
    \item (Parabolic) $g$ has a unique fixed point and there is an interleaving pair of crowns of $\sV$ preserved by $g.$
    \item (Hyperbolic) $g$ has exactly two fixed points, one is  attracting and the other is repelling. 
    \item (Pseudo-Anosov like) $g$ preserves  an interleaving pair $(\sG_1, \sG_2)$ of ideal polygons.
    \item (Properly pseudo-Anosov) $g$ preserves an interleaving pair $(\sG_1, \sG_2)$ of gaps such that the vertices of one of $\sG_i$ are the attracting fixed points of $g$ and the vertices of the other polygon are the repelling fixed points of $g$.
\end{enumerate}
\end{restate}
In fact, we will prove a slightly general result. See \refthm{classification}. 

We then classify ``elementary''  veering pair preserving groups. It is known that elementary Kleinian groups are either a rank 2 abelian group or a virtually cyclic group. Similarly, we have the following classification: 

\begin{restate}{Theorem}{Thm:elementary}
Let $\sV=\{\sL_1, \sL_2\}$ be a veering pair and $G$ a subgroup of $\Aut(\sV)$. If $G$ has an infinite cyclic normal subgroup, $G$ is isomorphic to $\ZZ,$ the infinite dihedral group, $\ZZ\times \ZZ_n$ for some $n\in \NN$ with $n>1,$ or $\ZZ\times \ZZ.$
Furthermore, one of the following cases holds.
\begin{enumerate}
    \item When $G\cong \ZZ,$
    \begin{enumerate}
        \item $G$ is generated by a parabolic automorphism, 
        \item $G$ is generated by a hyperbolic automorphism, or
        \item $G$ is generated by a pA-like automorphism.
    \end{enumerate} 
    \item When $G$ is isomorphic to the infinite dihedral group, $G$ is generated by an hyperbolic automorphism $g$ and an elliptic automorphism $e$ of order two, and so $G=\langle  g, e \ | \ e^2=1,\ ege=g^{-1} \rangle.$  
    \item When $G\cong \ZZ \times \ZZ_n$ for some $n\in\NN$ with $n>1,$ $G$ preserves a unique interleaving pair of  ideal polygons of $\sC$ and  there is a pA-like automorphism $g$ and an elliptic automorphism $e$ of order $n$ such that $G=\langle g \rangle \times \langle e \rangle.$
    \item When $G\cong \ZZ \times \ZZ,$ $G$ preserves a unique asterisk of crowns of $\sC$ and there is a properly pseudo-Anosov $g$ and a parabolic automorphism $h$ such that $G=\langle g \rangle \times \langle h \rangle.$ 
\end{enumerate}
\end{restate}

These two results are primitive evidences that the group $\Aut(\sV)$ resembles a Kleinian group.

Recall that we defined a marking on the stitch space of a veering pair. A prior, this marking on the stitches is just a random subset of $\fS(\sV)$ satisfying the desired properties. However, there is a canonical choice of a marking that records the internal order two symmetry of the veering pair. 
\begin{restate}{Corollary}{Cor:canonicalmarking}
Given a veering pair $\sV=\{\sL_1,\sL_2\}$ and a subgroup $G\le \Aut(\sV)$,
\[
\Mrk(G):=\{s\in \fS(\sV)\,|\,g(s)=s\text{ for some order 2 elliptic element }g\in G\}
\]
is a marking. 
\end{restate}

Now we are ready to prove the main theorem. We first state the theorem and give a sketch of its proof. 
\begin{restate}{Theorem}{Thm:3orbifoldgroup}
Let $\sV$ be a veering pair. Let $G$ be a subgroup of $\Aut(\sV)$. Then $G$ is the fundamental group of an irreducible 3-orbifold. 
\end{restate}
\begin{proof}[Proof Sketch]
In \refthm{loomspaceconstruction}, we constructed a loom space $\widetilde{\fW^\circ}(\sV, \Mrk(G))$. By work of \cite{SchleimerSegerman21}, we have a veering triangulated $\RR^3$ where the deck group $D$ for the universal cover $\widetilde{\fW^\circ}(\sV, \Mrk(G)) \to \fW^\circ (\sV,\Mrk(G))$  acts as taut isomorphisms. Then we can show that $\Aut(\sV)$ acts properly discontinuously on $\RR^3/D$. Topologically,  $\RR^3/D$ is a 3-manifold with many cylindrical boundary components. We fill in these cylinder boundary components with solid cylinders to obtain a simply-connected 3-manifold $X$ and extend the $G$-action to $X$. Hence $X/G$ is the desired 3-orbifold with $G=\pi_1(X/G)$. To show that $X/G$ is irreducible, we compute the homology group of $X$ and prove that $H_2(X)=\pi_2(X)=0$. 
\end{proof}

This result is a variation of \cite{Baik15} answering to the question what happens if a laminar group preserves two circle laminations. 

Finally, suppose that $X/G$ is compact. This condition can be translated as follow
\begin{itemize}
    \item The stabilizer of each  interleaving non-leaf polygons and each marked class  is of the form $\ZZ\times \ZZ_n$ or $\ZZ$ (see \refthm{gapStab}).
    \item The stabilizer of each interleaving crowns is of the form $\ZZ\times \ZZ.$
    \item There are only finitely many orbit classes of gaps and marked classes.
\end{itemize}
See \reflem{cofinite} and \refprop{converseCofinite}. We call the $G$ action is \emph{cofinite} if the above three conditions are satisfied. We can also show that $X$ is homotopically atoroidal. If $X/G$ is a manifold, then Perelman-Thurston hyperbolization theorem applies and $X/G$ is hyperbolic. If $X/G$ has a non-empty singular locus, $X/G$ a geometric orbifold by the (orbifold) geometrization theorem. Since $X/G$ is homotopically atoroidal, $X/G$ supports a hyperbolic geometry. Eventually, we get the following result:
\begin{restate}{Corollary}{Cor:geometric}
Let $\sV$ be a veering pair. Let $G$ be a subgroup of $\Aut(\sV)$. Suppose that the $G$-action is cofinite. Then $G$ is the fundamental group of a hyperbolic 3-orbifold. 
\end{restate}

We believe that \refcor{geometric} holds without the cofinite assumption. Because the geometrization theorem does not apply to noncompact 3-orbifolds we need to find a more intrinsic approach.

We know the fundamental group of a closed hyperbolic 3-manifold with a taut foliation acts faithfully on $S^1$ and preserves a pair of circle laminations \cite{CalegariDunfield03,Thurston97,Calegari06}. We record a corollary which is a partial converse of this fact. 

\begin{restate}{Corollary}{Cor:tautFoli}
Let $\sV$ be a veering pair without polygonal gaps and $G$ be a torsion free subgroup of $\Aut(\sV)$. If the $G$-action is cofinite, then $G$ is the fundamental group of a tautly foliated hyperbolic $3$-manifold.  
\end{restate}

\subsection{Organization}
In Part~\ref{veeringpart}, we discuss basic definitions related to lamination systems. Most importantly, the main object of the paper, veering pair, is defined in \refsec{veeringsec}.

Part~\ref{loomspacepart} is about the construction of loom spaces. The first step, \refsec{weavingsec}, is to define stitch spaces, weavings and their relatives. To construct transverse foliations, the notion of threads is introduced in \refsec{threadsec}. In \refsec{framesec} we present how to obtain rectangles from a veering pair. Then, in \refsec{foliationsec}, we prove that the regular weaving is an open disk with singular transverse foliations (\refthm{weaving}). The notion of marking is defined in \refsec{markingsec}. Finally, in \refsec{loomsec} we show that the universal cover of the regular weaving is a loom space (\refthm{loomspaceconstruction}). 

Part~\ref{grouppart} deals with groups acting on the circle with an invariant veering pair. After a quick review (\refsec{laminarpropertysec}) on properties of these groups, we prove various classification results in \refsec{classificationsec}, including classification of single elements (\refthm{classOfLaminarAut}), elementary groups (\refthm{elementary}), and gap stabilizers (\refthm{gapStab}). In the same section, we also prove that the order two elliptic elements give rise to the canonical marking (\refcor{canonicalmarking}). In \refsec{frameactionsec}, we explain that an action on the circle naturally gives an action on frames and rectangles. Then we prove in \refsec{triangulationactionsec} that a veering pair preserving action on the circle induces a properly discontinuous action on an irreducible simply-connected 3-manifold proving \refthm{3orbifoldgroup}. Finally, in \refsec{kleiniansec}, we explain the cofinite property and show \refcor{geometric}.

\subsection{Acknowledgement} The first author was partially supported by Samsung Science and Technology Foundation under Project Number SSTF-BA1702-01. The second author was supported by the Institute for Basic Science (IBS-R003-D1). The third author was supported by the National Research  Foundation of Korea Grant funded by the Korean Government (NRF-2022R1C1C2009782). We would like to thank Juan Alonso, Wonyong Jang, Sang-hyun Kim, Saul Schleimer, Samuel Taylor and Maxime Wolff for helpful conversations.

\part{Lamination Systems}\label{veeringpart}

The first part is a preliminary in nature. We recall relevant notions including circular ordering, lamination systems, and gaps. We also define veering pair in \refsec{veeringsec}. 

We introduce lamination systems which replace the role of circle laminations. The fundamental objects of a lamination system are good intervals in $S^1$, rather than leaves. This change of a viewpoint makes it easy to study circle actions. 

As we promote circle laminations to lamination systems, some terms have changed in meaning. For instance, leaves are regarded as gaps in a lamination system, and what is traditionally known as a gap is called a non-leaf gap. 

\section{Circular Orders and Intervals in $S^1$}
Let $S^1$ be the unit circle in the complex plane $\CC.$ If not mentioned otherwise, $S^1$ is the unit circle. 

\subsection{Circular Orders}
Let $X$ be a set. For each $n\in \NN$ with $n\geq 2,$ we write  $\Delta_n(X)$ for the set 
$$\{(x_1, \cdots, x_n)\in X^n : x_i=x_j \ \text{for some } i\neq j \}.$$ Then, we say that a function $\cphi$ from $X^3$ to $\{-1,0,1\}$ is a \emph{circular order} if the function satisfies the following:
\begin{itemize}
\item For any element $x^\circ \in \Delta_3(X)$, $\cphi(x^\circ)=0.$
\item For any element $(x_1, x_2, x_3,x_4)$ in $X^4,$ 
$$\cphi(x_2, x_3, x_4)-\cphi(x_1, x_3, x_4)+\cphi(x_1, x_2, x_4)-\cphi(x_1, x_2, x_3)=0.$$
\end{itemize}

Now, we define a circular order $\cphi$ for $S^1.$ Note that  the Cayley transformation 
$$\psi(z)=i\frac{1+z}{1-z}$$ 
maps $S^1$ to $\hat{\RR}=\RR \cup \{\infty \}.$ Fix $n$ in $\NN$ with $n\geq 3.$  We say that a $n$-tuple $(x_1, \cdots, x_n)$ in $(S^1)^n$ is \emph{counter-clockwise} if $\psi(x_1^{-1}x_i)<\psi(x_1^{-1}x_{i+1})$ for all $i$ in $\{2, \cdots, n-1\}.$ Similarly, 
 $(x_1, \cdots, x_n)$ is said to be \emph{clockwise} if $\psi(x_1^{-1}x_i)<\psi(x_1^{-1}x_{i+1})$ for all $i$ in $\{2, \cdots, n-1\}.$ We define $\cphi$ by 
 
\begin{equation*} \cphi(x)=
\begin{cases}
-1 & \text{if } \  x \ \text{is clockwise,}\\
 0 & \text{if } \  x \in\Delta_3(S^1) ,\\ 
 1 & \text{if } \ x \ \text{is counter-clockwise.}
\end{cases} \end{equation*}
We can see easily that $\cphi$ is a circular order for $S^1.$ In this paper, the circular order of $S^1$ is $\cphi.$

\subsection{Good Open Intervals}
We refer to non-empty proper connected open subsets of $S^1$ as \emph{open intervals} in $S^1.$ Let $I$ be an open interval in $S^1.$ If $S^1-I=\{u\}$ for some $u$ in $S^1,$ then we denote $I$ by $\opi{u}{u}.$ In this case, $I$ is said to be \emph{bad.} If not, there are two distinct points $u$ and $v$ such that 
$$I=\{z\in S^1 : \cphi(u, z, v)=1\},$$ 
so we denote $I$ by $\opi{u}{v},$ and we say that $I$ is \emph{good.} 

Let $\opi{a}{b}$ be good. Similarly we define $\ropi{a}{b},$ $\lopi{a}{b},$ and $\cldi{a}{b}.$ For convenience, we write $\opi{a}{b}^*$ for $\opi{b}{a}.$

\section{Lamination Systems}

\subsection{Laminations on $\HH^2$}
Let $\fM$ be the set of all geodesics in the hyperbolic plane  $\HH^2.$ The set  $\fM$ has the topology induced by Hausdorff distance.  A \emph{geodesic lamination} is a closed subset of $\fM$ whose elements are pairwise unlinked. Each element of a geodesic lamination is called a \emph{leaf} of the geodesic lamination.

Now, we consider $\HH^2$ as the Poincar\'e disk and the boundary $\partial \HH^2$ of $\HH^2$ is the unit circle $S^1$. As every geodesic connects two boundary points and vice versa, $\fM$  is parameterized by the set $\cM$ of two points subset of $S^1,$ namely,
$$\cM:=(S^1\times S^1)-\Delta_2(S^1)/(x,y)\sim(y,x).$$ 
We denote the parametrization from $\cM$ to $\fM$ by $g$ so that $g(\{x,y\})$ is the geodesic having $x$ and $y$ as endpoints. Also, we can think of a geodesic lamination on $\HH^2$ as a subset of $\cM.$

For any distinct elements $\ell_1$ and $\ell_2$ in $\cM,$ we say that $\ell_1$ and $\ell_2$ are \emph{linked} if each component of $S^1-\ell_1$ contains a point of $\ell_2,$ and they are \emph{unlinked} otherwise. A \emph{circle lamination} is a closed subset of $\cM$ whose elements are pairwise unlinked. Like above, given a circle laminations, there is a corresponding geodesic lamination. Hence, each element of a circle lamination is called a \emph{leaf} 
of the circle lamination. 

Now, we consider the set
$$\sM:=\{ \ell(I): I \text{ is a good interval} \ \}$$ where $\ell(I)=\{I, I^*\}.$ We define the \emph{endpoint map} $\epsilon$ from $\sM$ to $\cM$ as 
\[
\epsilon(\ell(\opi{x}{y}))=\{x,y\}.
\]
Via the endpoint map, $\cM$ is parameterized by $\sM.$ Then, we give $\sM$  a topology so that the endpoint map is a  homeomorphism. 

We introduce the following definition to say the unlinkedness in $\sM$ in terms of good intervals. Given an element $\ell(I)$ in $\sM$, we say that $\ell(I)$ \emph{lies} on an open interval $J$ if $I\subseteq J$ or $I^*\subseteq J$.

\subsection{Lamination Systems}
In many circumstances, intervals in $S^1$ separated by a leaf of a circle lamination play key roles to understand the dynamics of laminar groups. This suggests that intervals are more fundamental objects than leaves in the study of circle laminations and motivates us to define lamination systems in terms of good open intervals. This lamination systems are equivalent to circle laminations in spirit, but allow us to prove many results more rigorously.

\begin{defn}
A \emph{lamination system} is a nonempty collection $\sL$ of good intervals that satisfies the following properties
\begin{enumerate}
    \item If $I\in \sL$, then $I^*\in \sL$.
    \item (unlinkedness) Given two elements $I,J \in \sL$,  $\ell(I)$ lies on $J$ or $J^*,$ that is, $I\subset J,$ $I^*\subset J,$ $I\subset J^*,$ or $I^*\subset J^*.$ 
    \item (closedness) Given an ascending family $\{I_k\}_{k\in \NN}$ of elements of $\sL$, we have $\bigcup_{k\in \NN} I_k\in \sL$ provided $\bigcup_{k\in \NN} I_k$ is a good interval.
\end{enumerate}
\end{defn}

Note that lamination systems are examples of pocsets. The unlinkedness condition in the definition of lamination systems is sometimes called the \emph{nestedness}. A discrete pocset of good intervals in $S^1$ with all pairs of distinct elements being nested is an example of lamination system.

\subsection{Leaf Spaces}
Let $\sL$ be a lamination system. The subspace $\ell(\sL)$ of $\sM$ is called the \emph{leaf space} of $\sL$ and each element of $\ell(\sL)$ is called a \emph{leaf} of $\sL.$ Then, $\epsilon(\ell(\sL))$ is a circle lamination and we denote this circle lamination by $\cC(\sL).$ 
The unlinkedness of $\cC(\sL)$ is obvious.  To see the closedness, we introduce the notion for the convergence of a sequence of leaves in terms of good intervals.

A sequence $\{\ell_k\}_{k\in \NN}$ of leaves of $\sL$ \emph{converges} to a good interval $J$ if there is a sequence $\{I_k\}_{k\in \NN}$ of elements of $\sL$ such that for each $k\in \NN,$ $\ell_k=\ell(I_k)$ and 
\[
J\subseteq \liminf I_k \subseteq \limsup I_k \subseteq \closure{J}.
\] 
We write $\ell_k\to J$ when $\ell_k$ converges to $J$. Then,  $J\in \sL$. Moreover, $\ell_k \to J^*$. Therefore, we can say that a sequence $\{\ell_k\}_{k\in \NN}$ of leaves of $\sL$ \emph{converges} to an element $\ell$ of $\sM$ if $\ell_k\to I$ for some $I\in \ell.$ This notion of convergence is equivalent to the convergence under the topology of the leaf space. This implies that the closedness of $\cC(\sL).$ For more details, we refer to \cite[Section~10.7]{BaikKim20}.

Given a good interval $I$, a sequence $\{\ell_k\}_{k\in \NN}$ of leaves in $\sL$ is said to be \emph{$I$-side} if it satisfies the following:
\begin{itemize}
    \item $\ell_k\to I$
    \item$\ell_k\neq \ell(I)$ and $\ell_k$ lies on $I$ for all $k\in \NN.$
\end{itemize}
When $I$ in $\sL$ has no $I$-side sequence, $I$ is said to be \emph{isolated.} A given leaf  is \emph{isolated} if every element of the leaf is isolated.

\subsection{Gaps}
Given a geodesic lamination on a closed hyperbolic surface, its complementary regions are subsurfaces with geodesic boundary. We can bring this notion into the context of lamination system, leading us to the following definition:
\begin{defn}
Let $\sL$ be a lamination system. A \emph{gap} $\sG$ of $\sL$ is a subset of $\sL$ such that
\begin{enumerate}
      \item any two distinct elements of $\sG$ are disjoint, and
    \item each leaf of $\sL$ lies on some element of $\sG$.
\end{enumerate}
\end{defn}
The closed set $S^1\setminus \bigcup \sG$ is called the \emph{vertex set} of $\sG$ and  is denoted by $v(\sG).$ Each element of $v(\sG)$ is called  a \emph{vertex} of $\sG.$ A vertex of $\sG$ is called a \emph{tip} of $\sG$ if it is isolated in $v(\sG).$ For each $I\in \sG,$ the leaf $\ell(I)$ of $\sL$ is called a \emph{boundary leaf} of $\sG.$ 

Note that a leaf itself is a gap in our definition. Hence, we say that a gap of $\sL$ is \emph{non-leaf} if the gap is not a leaf of $\sL.$  We denote the set $\bigcup_{I\in\sL} v(\ell(I))$ by $\E{\sL}$ and call $\E{\sL}$ the \emph{endpoint set} of $\sL.$ Also, each element of the endpoint set is called an \emph{endpoint} of $\sL.$

We say that a gap $\sG$ in a lamination system \emph{lies on} a good interval $I$ if $J^*\subset I$ for some $J\in \sG.$ Also, $\sG$ \emph{properly lies} on a good interval $I$ if $\closure{J^*}=J^c\subset I$ for some $J\in \sG.$ Note that the unlinkedness of leaves implies the following proposition. 
For a detailed proof, we refer to~\cite[Lemma~10.7.13]{BaikKim20}
\begin{prop}\label{Prop:twoGapUnlinked}
    Let $\sL$ be a lamination system. Suppose that $\sG$ and $\sH$ are gap of $\sL$ with $|\sG|$, $|\sH|\geq 2.$ Then $\sG=\sH$ or $\sG$ lies on some open interval in $\sH.$ 
\end{prop}

The following proposition gives a criterion for the existence of a non-leaf gap.
\begin{prop}[\cite{BaikKim20}]\label{Prop:gapExist}
Let $\sL$ be a lamination system and $I\in \sL.$ If $I$ is isolated, then there is a non-leaf gap $\sG$ containing $I^*.$
\end{prop}

Let $\sL$ be a lamination system. A gap $\sG$ is said to be \emph{real} if there is no isolated element in $\sG.$ Then, any real leaf is not a boundary leaf of a non-leaf gap.
We say that a pair $\{I,J\}$ of elements of $\sL$ is a \emph{distinct pair} in $\sL$ if  $I\subsetneq J^*$. Note that a distinct pair is not a leaf. A distinct pair $\{I, J\}$ is said to be \emph{separated} if there is a non-leaf gap $\sG$ of $\sL$ such that $I\subset M$ and $J\subset N$ for some $\{M,N\}\subset \sG.$ Here, $M$ and $N$ are not necessarily distinct. Nonetheless, whenever a distinct pair $\{I,J\}$ is separated, we can take a non-leaf gap $\sG$ so that $\sG$ contains a distinct pair $\{M,N\}$ with $I\subseteq M$ and $J\subseteq N$. Then, $\sL$ is said to be \emph{totally disconnected} if every distinct pair is separated. Equivalently, for any two distinct leaves $\ell_1$ and $\ell_2$ of $\sL,$ there is a non-leaf gap $\sG$ containing a distinct pair $\{J_1, J_2\}$ such that  $\ell_i$ lies on $J_i$ for all $i\in \{1,2\}$. Note that totally disconnectedness guarantees the existence of a non-leaf gap. 
 
There are two important types of gaps. An \emph{ideal polygon} (or simply a polygon) is 
a gap with finitely many vertices. When $\sG$ is a polygon, we can write $\sG=\{\opi{t_k}{t_{k+1}}\,:\, k=1,2,\cdots,n \}$ where $n=|v(\sG)|$ and $v(\sG)=\{t_k\,:\,k=1,2,\cdots,n\}$ (cyclically indexed). Note that vertices and tips of a polygon are identical.

The other type is a \emph{crown} which is defined as follows. A gap $\sG$ is called a \emph{crown} if there is a sequence $\{t_k\}_{k\in \ZZ}$ and a point $p$ in $S^1$ satisfying the following:
\begin{itemize}
    \item $p\neq t_k$ for all $k\in \ZZ,$
    \item $\sG=\{\opi{t_k}{t_{k+1}}:k\in \ZZ\},$
    \item $v(\sG)=\{t_k : k\in \ZZ\} \cup \{p\}.$
\end{itemize}
In this situation, the point $p$ is called the \emph{pivot} of $\sG$. See \reffig{crownexample}. 
\begin{rmk}\label{Rmk:aboutCrown}
Note that $t_k$ are tips of $\sG$ and that $p$ is the unique limit point of $v(\sG).$
\end{rmk}

\begin{figure}[htb]
    \centering
    \begin{tikzpicture}[scale=1.5]
    \draw[thick] (0,0) circle (2 cm);
\draw[thick, purple] (	1.732050808	,	1	) arc (	300	:	160	:	0.7279404685	);
\draw[thick, purple] (	0.6840402867	,	1.879385242	) arc (	340	:	200	:	0.7279404685	);
\draw[thick, purple] (	-0.6840402867	,	1.879385242	) arc (	20	:	-120	:	0.7279404685	);
\draw[thick, purple] (	-1.732050808	,	1	) arc (	60	:	-80	:	0.7279404685	);
\draw[thick, purple] (	-1.969615506	,	-0.3472963553	) arc (	100	:	-50	:	0.5358983849	);
\draw[thick, purple] (	-1.532088886	,	-1.285575219	) arc (	130	:	-25	:	0.4433893253	);
\draw[thick, purple] (	-0.8452365235	,	-1.812615574	) arc (	155	:	-10	:	0.2633049952	);
\draw[thick, purple] (	-0.3472963553	,	-1.969615506	) arc (	170	:	-5	:	0.08732188582	);
\draw[thick, purple] (	-0.1743114855	,	-1.992389396	) arc (	175	:	-2	:	0.05237184314	);
\draw[thick, purple] (	0.06979899341	,	-1.998781654	) arc (	182	:	5	:	0.05237184314	) ;
\draw[thick, purple] (	0.1743114855	,	-1.992389396	) arc (	185	:	10	:	0.08732188582	) ;
\draw[thick, purple] (	0.3472963553	,	-1.969615506	) arc (	190	:	25	:	0.2633049952	);
\draw[thick, purple] (	0.8452365235	,	-1.812615574	) arc (	205	:	50	:	0.4433893253	);
\draw[thick, purple] (	1.532088886	,	-1.285575219	) arc (	230	:	80	:	0.5358983849	);
\draw[thick, purple] (	1.732050808	,	1	) arc (	120	:	260	:	0.7279404685	);
\draw (0,-2) node{$\bullet$} node[below]{$p$};
\draw ({2*cos(30)},{ 2*sin(30)}) node[right]{$t_i$};
\draw ({2*cos(70)},{ 2*sin(70)}) node[right, above]{$t_{i+1}$};
\draw ({2*cos(350)},{ 2*sin(350)}) node[right]{$t_{i-1}$};
\end{tikzpicture}
    \caption{A crown. The pivot is marked with the dot.}
    \label{Fig:crownexample}
\end{figure}
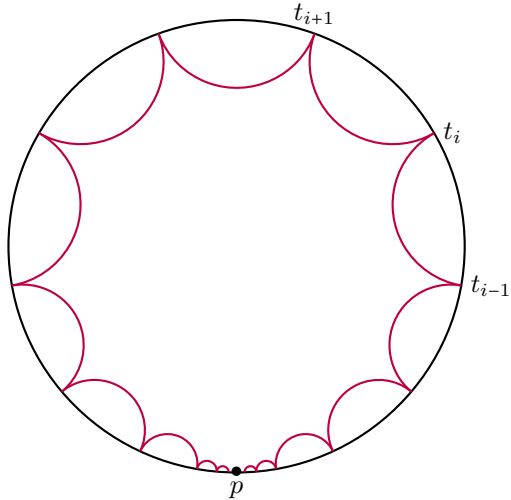

\subsection{Stems and Rainbows}\label{Sec:stem}

Let $\sL$ be a lamination system and $B$ a good interval. Also, let $E$ be a non-empty subset of $B.$ Then, the \emph{stem} from $B$ to $E$ in $\sL$ is a set 
\[
\stem{E}{B}:=\{J\in \sL: E \subseteq J \subseteq B \}.
\] which is totally ordered by the set inclusion $\subseteq.$

Suppose that a stem $\stem{E}{B}$ is not empty. Then, $\bigcup \stem{E}{B}$ is the maximal element of $\stem{E}{B}.$ We call the maximal element the \emph{base} of $\stem{E}{B}.$ We call $\interior{\bigcap \stem{E}{B}}$ the \emph{end} of the stem $\stem{E}{B}$ if $\interior{\bigcap \stem{E}{B}}\neq \emptyset.$ 

Observe that when $E$ has at least two points, $\widehat{\stem{E}{B}}:=\interior{\bigcap \stem{E}{B}}$ is never empty. To see this, we denote the smallest closed interval containing $E$ and contained in $B$ by $m.$ Then, $\widehat{\stem{E}{B}}$ is a good interval containing the interior $\interior{m}$ of $m$ and it is contained in $\sL.$ Hence, if $\widehat{\stem{E}{B}}= \emptyset$, then $E$ necessarily has only one point. 

A descending sequence $\{I_k\}_{k\in\NN}$ of elements of $\sL$ is called a \emph{rainbow} at $p$ if $\bigcap_{k\in \NN} I_k=\{p\}.$ We call $p$ a \emph{rainbow point} of $\sL$.  For convenience, when $E=\{p\}$ for some $p\in S^1,$ then we write $\stem{p}{B}$ for $\stem{E}{B}.$ A point $p \in S^1$ is a rainbow point of $\sL$ precisely when $\widehat{\stem{p}{B}}$ is empty. To see this, observe that $\widehat{\stem{p}{B}}$ is not empty, then  $\{p\} \subset v(\ell(\widehat{\stem{p}{B}} ))$ or $\widehat{\stem{p}{B}}$ is the minimal element of $\stem{p}{B}$. 
On the other hand, if  $\widehat{\stem{p}{B}} = \emptyset$, then one can take a rainbow $\{I_k\}_{k\in\NN}$ at $p$ in $\stem{p}{B}$, namely, $I_k\in \stem{p}{B}$ for all $k\in \NN.$ 

\section{Quite Full Loose Lamination Systems}

Not all lamination systems are interesting. Most lamination systems naturally occurring in the context of geometry and dynamics have a certain degree of complexity. In this, and the forthcoming sections, we introduce quite full, loose lamination systems and ultimately veering pairs by packaging some properties of the stable/unstable pair of laminations associated to a pseudo-Anosov mapping class. Along this, we also define pseudo-fibered pairs which is a weaker version of veering pairs. In fact, pseudo-fiberedness is enough to prove the classification theorems for laminar automorphisms, e.g.\refthm{classification}, \refthm{elementary}, and  \refthm{gapStab}. Also, a pseudo-fibered pair is a right  generalization of a pair of laminations contained the definition of  pseudo-fibered triples defined in \cite{AlonsoBaikSamperton}. In \refsec{promotiontocol2}, it turns out that every pseudo-fibered pair is a pants-like pair studied in \cite{Baik15} and \cite{BaikKim21}.

\subsection{Quite full}\label{Sec:quiteFull}
 A lamination system $\sL$ is said to be \emph{quite full} if every gap of $\sL$ is either an ideal polygon or a crown. In particular, we say that $\sL$ is \emph{very full} if  every gap is an ideal polygon. 

The following lemma shows that the local shape of a quite full lamination system is regulated. 

\begin{lem}[Rainbows are quite abundant]\label{Lem:trichotomy}
Let $\sL$ be a quite full lamination system. Each point in $S^1$ is either an endpoint of $\sL$, the pivot of a unique crown, or a rainbow point. These cases are mutually exclusive.  
\end{lem}
\begin{proof}
    Let $p$ be a point in $S^1$. As $\sL\ne \emptyset$, we may choose an element $I.$ If $p\in v(\ell(I))$, then $p \in \E{\sL}$. If not, there is an element $J$ in $\ell(I)$ containing $p.$ 
    
    Now, we consider the stem $\stem{p}{J}$. As $J\in \stem{p}{J}$, the stem is not empty. 
    If  $\widehat{\stem{p}{J}}=\emptyset$, then $p$ is a rainbow point. See \refsec{stem}. Now suppose that the end $E$ of the stem exists. There are two cases. One is the case where $p$ is an endpoint of $\ell(E).$ This implies the statement. In the other case, since $E$ is the minimal element in $\stem{p}{J},$ there is no $E$-side sequence of leaves and so $E$ is isolated. Therefore, by \refprop{gapExist}, there is a non-leaf gap $\sG$ of $\sL$ containing $E^*$. 
    
    If there is $K$ in $\sG$ containing $p,$ then 
    $$p\in K \subsetneq E\subseteq J$$ 
    and $K\in \stem{p}{J}$. This contradicts that $E$ is the end of $\stem{p}{J}$. Therefore, $p\in v(\sG)$. Then, since $\sL$ is quite full, $\sG$ is either an ideal polygon and a crown. Hence, $p$ is either the pivot of $\sG$ or a tip of $\sG$. By \refrmk{aboutCrown}, these cases are mutually exclusive. Thus, $p$ is the pivot of $\sG$ or $p\in \E{\sL}$. 
    \end{proof}
\begin{cor}\label{Cor:denseEnd}
        The endpoint set of a quite full lamination system is dense in $S^1$.
\end{cor}

\subsection{Loose}
A quite full lamination system $\sL$ is \emph{loose} if for any distinct non-leaf gaps $\sG$ and $\sH$ of $\sL,$ $v(\sG)$ and $v(\sH)$ are disjoint. In this section, we investigate the topology of loose lamination systems  and renew \reflem{trichotomy} in loose lamination systems.
\begin{prop}\label{Prop:threeLeaves}
Let $\sL$ be a quite full lamination system. Suppose that $\sL$ is loose and totally disconnected. Then, for three distinct leaves $\ell_1$, $\ell_2$, and $\ell_3$, we have $v(\ell_1) \cap v(\ell_2) \cap v(\ell_3)=\emptyset$.       
\end{prop}
\begin{proof}
Assume that $v(\ell_1) \cap v(\ell_2) \cap v(\ell_3)=\{p\}$ for some $p\in S^1$. For each $i\in \{1,2,3\}$, we write $\ell_i=\ell(I_i)$. We may assume that $I_1\subsetneq I_2 \subsetneq I_3$. By totally disconnectedness, there is a non-leaf gap $\sG_1$ containing a distinct pair $\{J_1, K_1\}$ such that $I_1\subset J_1$ and $I_2^* \subset K_1$ and there is also a non-leaf gap $\sG_2$ containing a distinct pair $\{J_2, K_2\}$ such that $I_2\subset J_2$ and $I_3^* \subset K_2$. Then, $p\in v(\sG_1) \cap v(\sG_2)$ and by looseness, $\sG_1=\sG_2$. On the other hand,  $\sG_1\neq \sG_2$ since $v(\sG_1)\subset \closure{I_2}$ and $v(\sG_2)\subset \closure{I_2^*}$. This is a contradiction. Thus, $v(\ell_1) \cap v(\ell_2) \cap v(\ell_3)=\emptyset$.
\end{proof}

\begin{prop}\label{Prop:parallelLeaves}
Let $\sL$ be a quite full lamination system. Suppose that $\sL$ is loose and totally disconnected. Given distinct leaves $\ell_1$ and $\ell_2$ sharing a vertex $p,$ there is a non-leaf gap $\sG$ having $\ell_1$ and $\ell_2$ as boundary leaves, and $p$ as a tip. 
\end{prop}
\begin{proof}
By totally disconnectedness, there is a non-leaf gap $\sG$ having a distinct pair $\{J_1, J_2\}$ such that $\ell_i$ lies on $J_i$ for all $i\in \{1,2\}$. Then, by \refprop{threeLeaves}, $\ell_i=\ell(J_i)$ for all $i\in \{1,2\}$.
\end{proof}

\begin{prop}\label{Prop:realGap}
Let $\sL$ be a quite full lamination system. Suppose that $\sL$ is loose and totally disconnected. Then, every leaf is a real leaf or a boundary leaf of a non-leaf gap. In particular, there is no isolated leaf of $\sL$. Moreover, every non-leaf gap is real. 
\end{prop}
\begin{proof}
By \refprop{gapExist}, we only need to show that there is no isolated leaf in $\sL.$ Assume that $\ell=\ell(I)$ is an isolated leaf. By \refprop{gapExist}, there are distinct non-leaf gaps $\sG_1$ and $\sG_2$ such that $I\in \sG_1$ and $I^* \in \sG_2$. Then, $v(\ell)= v(\sG_1)\cap v(\sG_2)$ and this contradicts the looseness of $\sL$. 
\end{proof}

Now, we can restate \reflem{trichotomy} with \refprop{realGap} as follows. 

\begin{lem}\label{Lem:quadrachotomy}
Let $\sL$ be a quite full lamination system. Suppose that $\sL$ is loose and totally disconnected. Then each point in $S^1$ is either a vertex of a unique real leaf, a tip of a unique non-leaf gap, the pivot of a unique crown, or a rainbow point. These cases are mutually exclusive.  
\end{lem}

Let $\sL$ be a quite full lamination system. Assume that $\sL$ is loose and totally disconnected. 
Let $t$ be an endpoint of $\sL.$ Then, by \reflem{quadrachotomy}, there is a unique real gap $\sG$ having $t$ as a tip. We call $\sG$ the \emph{tip gap} of $t$ and $\sG$ is denoted by $\tipg{t}.$ If $\sG$ is a non-leaf gap, then there is a distinct pair $\{I,J\}$ in $\sG$ such that $\closure{I}\cap \closure{J}=\{t\}$. We call $\{I, J\}$ the \emph{tip pair} at $t$ and denote it by $\tipp{t}$.

\section{Pseudo-Fibered Pairs and Veering Pairs}
As we promised, we define pseudo-fibered pairs and veering pairs in this section. The motivating examples of our definition are the pair of laminations  constructed in \cite{SchleimerSegerman19} and a pair of geodesic laminations preserved by a pseudo-Anosov surface mapping class.

\subsection{Pseudo-fibered pairs and  veering pairs}\label{Sec:veeringsec}
Let $\sC=\{\sL_1, \sL_2\}$ be a pair of quite full lamination systems. We say that a collection $\sG$ of good intervals is a \emph{gap} of $\sC$ if $\sG$ is a gap of $\sL_1$ or $\sL_2$. 

Let $\sG_1$ and $\sG_2$ be gaps of $\sC.$ When $\sG_1$ and $\sG_2$ are leaves, we say that $\sG_1$ and $\sG_2$ are \emph{linked} if $v(\sG_1)$ and $v(\sG_2)$ are linked, and \emph{unlinked} otherwise. Suppose that $\sG_1$ and $\sG_2$ are unlinked. We say that $\sG_1$ and $\sG_2$ are \emph{parallel} if $v(\sG_1)\cap v(\sG_2)$ is a singleton, and \emph{ultraparallel} if $v(\sG_1)$ and $v(\sG_2)$ are disjoint.

In general, $\sG_1$ and $\sG_2$ are \emph{linked} if  $\ell(I_1)$ and $\ell(I_2)$ are linked for some $I_1\in \sG_1$ and $I_2 \in \sG_2.$ Otherwise, $\sG_1$ and $\sG_2$ are \emph{unlinked}.

We say that $\sG_1$ and $\sG_2$ \emph{interleave} if for each $i\in \{1,2\},$ $I\cap v(\sG_{i+1})$ is a singleton for all $I$ in $\sG_i$ (indexed cyclically). In this case, we call $\sG_i$  an \emph{interleaving gap} of $\sG_{i+1}$, for each $i\in\{1,2\}$ (cyclically indexed). Consult Figure~\ref{interleavingexample} for generic shape of interleaving gaps.  Note that if two gaps interleave and one is a leaf then the other is also a leaf. Likewise, the interleaving gap of a polygon or a crown is also a polygon or a crown, respectively. In particular, if two crowns interleave, then the pivots of these crowns coincide.  

An ordered pair of interleaving gaps is called an \emph{interleaving pair}, i.e., the ordered pair $(\sG_1, \sG_2)$ is called an interleaving pair of $\sC$ if $\sG_i$ is a gap of $\sL_i$ for each $i\in \{1,2\}$ and $\sG_1$ and $\sG_2$ interleave. A \emph{stitch} of $\sC$ is an interleaving pair $(\sG_1, \sG_2)$ such that each $\sG_i$ is a leaf.  An \emph{asterisk} of $\sC$ is an interleaving pair which is not a stitch. 

\begin{defn}
Let $\sC=\{\sL_1, \sL_2\}$ be a pair of lamination system. The pair $\sC$ is said to be \emph{pseudo-fibered} if 
\begin{itemize}
    \item each $\sL_i$ is quite full and loose, and 
    \item $\sL_1$ and $\sL_2$ are \emph{strongly transverse}, namely, $\E{\sL_1}\cap \E{\sL_2}=\emptyset.$
\end{itemize}
In particular, a pseudo-fibered pair $\sC$ is called a \emph{veering pair} if each non-leaf gap of $\sC$ has an interleaving gap.
\end{defn}

Note that the transversality of laminations implies the totally disconnectedness as follows.
\begin{prop}[\cite{AlonsoBaikSamperton}, \cite{BaikKim20}]\label{Prop:totDis}
If lamination systems $\sL_1$ and $\sL_2$ are strongly transverse and both $\E{\sL_1}$ and $\E{\sL_2}$ are dense in $S^1,$ then each $\sL_i$ are totally disconnected.   
\end{prop}
\begin{prop}\label{Prop:totDis2}
    Let $\sC=\{\sL_1, \sL_2\}$ be a pseudo-fibered pair. Then, each $\sL_i$ is totally disconnected. 
\end{prop}
\begin{proof}
    Since $\sL_i$ are quite full, by \refcor{denseEnd}, the endpoint sets $\E{\sL_i}$ are dense in $S^1.$ As $\sL_1$ and $\sL_2$ are strongly transverse, by \refprop{totDis}, $\sL_i$ are totally disconnected.
\end{proof}
    
\begin{rmk}
By \refprop{totDis2}, each lamination system $\sL_i$ of a pseudo-fibered pair is quite full, loose and totally disconnected. Therefore, each $\sL_i$ enjoys the properties shown in \refsec{quiteFull}.
\end{rmk} 

\subsection{Properties of Veering Pairs}
Let $\sV=\{\sL_1,\sL_2\}$ be a veering pair.
A \emph{singular stitch} of  an asterisk $(\sG_1,\sG_2)$ of $\sV$ is a stitch $(\ell_1,\ell_2)$ such that each $\ell_i$ is a boundary leaf of $\sG_i$ for all $i\in \{1,2\}$. A stitch of $\sV$ is \emph{singular} if it is a singular stitch of some asterisk. Otherwise, $(\ell_1, \ell_2)$ is \emph{regular}. We say that a stitch $(\ell_1,\ell_2)$ is \emph{genuine} if $\ell_1$ and $\ell_2$ are real.

A tuple $(I_1,I_2)$ in $\sL_1 \times \sL_2$ is called a \emph{sector} in $\sV$ if $(\ell(I_1),\ell(I_2))$ is a stitch. In particular, $(I_1,I_2)$ is a \emph{sector} of an interleaving pair $(\sG_1,\sG_2)$ if $I_i\in \sG_i$.  Also, $(I_1,I_2)$ is said to be \emph{counter-clockwise} if $\opi{u_1}{v_1}\cap \opi{u_2}{v_2}=\opi{u_2}{v_1}$ where $I_i=\opi{u_i}{v_i}$ for all $i\in\{1,2\}$. Otherwise, the sector is \emph{clockwise}. We say that a stitch $(\ell_1,\ell_2)$ \emph{lies} on a sector $(I_1,I_2)$ if for each $i\in\{1,2\}$, $\ell_i$ properly lies on $I_i$ and $\ell_i$ is linked with $\ell(I_j)$, $j\neq i$. 

% Let $\sV=\{\sL_1, \sL_2\}$ be a veering pair and $(\sG_1,\sG_2)$ an interleaving pair of $\sV.$ Choose elements $I_1$ and $I_2$ in $\sG_1$ and $\sG_2,$ respectively. A tuple $(I_1, I_2)$ is called a \emph{sector} of $(\sG_1, \sG_2)$ if $I_1\cap I_2 \neq \emptyset$. A sector $(I_1, I_2)$ is said to be \emph{counter-clockwise} if $\opi{u_1}{v_1}\cap \opi{u_2}{v_2}=\opi{u_2}{v_1}$ where $I_i=\opi{u_i}{v_i}$ for each $i\in\{1,2\}$. Otherwise, the sector $(I_1, I_2)$ is \emph{clockwise}. 

% A sector $(I_1,I_2)$ of a singular stitch is \emph{allowed} if $(I_1,I_2)$ is also a sector of some asterisk. Note that a stitch has at most one allowed sector. Note that the allowed sector of a singular stitch is unique. See \refprop{realGap}.
% We say that $(\ell_1, \ell_2)$ \emph{lies} on a sector $(I_1, I_2)$ if for each 
% $i\in \{1,2\}$, $\ell_i$ properly lies on $I_i$ and $\ell_i$ is linked with $\ell(I_{j})$, $j\ne i$. 

\begin{lem}\label{Lem:trichotomy2}
Let $\sV=\{\sL_1, 
\sL_2\}$ be a veering pair. Then, each point $p$ in $S^1$ falls into one of the following cases:
\begin{enumerate}
    \item $p$ is a rainbow point of both $\sL_1$ and $\sL_2$.
    \item $p$ is a rainbow point of $\sL_i$ for some $i\in \{1,2\}$, and $p$ is a tip of a unique real gap of $\sL_{j}$, $j\ne i$.
    \item $p$ is the pivot of some asterisk of crowns.
\end{enumerate}
\end{lem}
\begin{proof}
It follows from \reflem{quadrachotomy}.
\end{proof}

Let $t$ be an endpoint of $\sL_1$ or $\sL_2$.
By \reflem{trichotomy2}, there is a unique real gap of $\sV$ having $t$ as a tip. Hence, the notions of tip pairs and tip gaps in $\sV$ are well-defined. 
Now, we abuse the notions of tip pairs and tip gaps for $\sV$. 

We say that  an element  $I$ in $\sL_1\cup \sL_2$ \emph{crosses over} $t$ if $I\cap v(\tipg{t})=\{t\}.$ Note that if $t\in \E{\sL_i}$ for some $i\in \{1,2\}$, then $I\in \sL_{i+1}$ (cyclically indexed). Also, a leaf $\ell$ of $\sV$ \emph{crosses over} $t$ if some element of $\ell$ crosses over $t$. Moreover, we say that a tip pair $\tipp{t'}$ \emph{crosses over} $t$ if $\ell(J)$ crosses over $t$ for all $J\in \tipp{t'}$. Note that if $\tipp{t'}=\{J_1, J_2\},$ then $J_i$ crosses over $t$ for some $i\in \{1,2\}$ and so $J_{i+1}^*$ crosses over $t$ (cyclically indexed). Therefore, $J_{i+1}\cap v(\tipg{t})= v(\tipg{t})-\{t\}$.

\subsection{Linkedness and Crossing}Let $\sV=\{\sL_1, \sL_2\}$ be a veering pair. We say that gaps $\sG_1$ and $\sG_2$ in $\sL_1$ and $\sL_2$, respectively, \emph{cross}  if there are distinct pairs $\{I_1, J_1\}$ and  $\{I_2, J_2\}$ in $\sG_1$ and $\sG_2,$ respectively, such that for each $i\in \{1,2\}$, $v(\sG_{i+1})\subset (I_i \cup J_i)$ and $|v(\sG_{i+1})\cap I_i|=1$ (cyclically indexed). One purpose of this section  is to show that, in veering pairs, the concept of crossing coincide with the concept of crossing over in some sense. See \refprop{crossingTwoGaps} and \refprop{tipPairCross}. Eventually, we show \refprop{twoGap} which describes the possible configurations of linked gaps of $\sV$. 

\begin{rmk}
If $\sG_1$ and $\sG_2$ are unlinked, we can take $K_1$ and $K_2$ in $\sG_1$ and $\sG_2,$ respectively, so that $K_1^c \subset K_2.$  Therefore, for each $i \in \{1,2\}$, $\sG_i$ lies on $K_{i+1}$ (cyclically indexed). Compare with \refprop{twoGapUnlinked}. 
\end{rmk}
\begin{prop}\label{Prop:leafGapCross}
    Let $\sV=\{\sL_1, \sL_2\}$ be a veering pair and $\sG$ be a gap in $\sL_1.$ Then for every leaf $\ell$ in $\sL_2,$ either $\sG$ and  $\ell$ are unlinked  or they cross.
    \end{prop}
    \begin{proof}
    First, we consider the case where $\sG$ is a leaf. Then $\sG$ and $\ell$ are either unlinked or linked. If they are linked, then they cross. Therefore, this case is obvious. 
    
    Now, we assume that $\sG$ is a non-leaf gap. Since $\sV$ is a veering pair, there is the interleaving gap $\sH$ of $\sG$ and $\ell$ lies on a unique good interval $I$ in $\sH$. Then, we can take a good interval $J$ in $\ell$ contained in $I$. 
    
     Note that  $I\cap v(\sG)= \{v\}$ for some tip $v$ of $\sG$. If $J$ does not contain $v,$ then $v(\sG) \subset J^c= \closure{J^*}$ and so $\sG$ lies on $J^*$ since $\E{\sL_1} \cap \E{\sL_2}=\emptyset$. Therefore, $\sG$ and $\ell$ are unlinked. Assume that $J$ contains $v$. As $I^*\cap v(\sG)=v(\sG)-\{v\}$ and $J\subseteq I$,
    $$v(\sG)-\{v\}=I^*\cap v(\sG) \subseteq I^* \subseteq J^*.$$
    Moreover, endpoints of $\ell$ belong to different elements of $\sG$ since $\sG$ and $\ell$ is linked. Therefore, $\sG$ and $\ell$ cross. 
    \end{proof}
\begin{prop}\label{Prop:crossingTwoGaps}
        Let $\sV=\{\sL_1, \sL_2\}$ be a veering pair. Suppose that there are non-leaf gaps $\sG_1$ and $\sG_2$ in $\sL_1$ and $\sL_2$, respectively, such that $\sG_1$ and $\sG_2$ cross. Then there are tips $t_1$ and $t_2$ of $\sG_1$ and $\sG_2$, respectively, such that for each $i,j\in \{1,2\}$, $i\ne j$, the tip pair $\tipp{t_i}$ crosses over $t_j$. 
        \end{prop}
        \begin{proof}
        Since $\sG_1$ and $\sG_2$ cross, there are distinct pairs $\{I_1, J_1\}$ and $\{I_2, J_2\}$ in $\sG_1$ and $\sG_2$, respectively, such that for each $i,j\in \{1,2\}$, $i\ne j$,  $v(\sG_j)\subset I_i \cup J_i$ and  $|I_i \cap v(\sG_j)|=1$. Now we say that for each $i,j\in \{1,2\}$, $i\ne j$, the point in $I_i \cap v(\sG_j)$ is $t_j$. 
        
        For each $i\in \{1,2\}$, $t_j\in I_i \subset J_i^*$ and
        $$v(\sG_j)-\{t_j\} \subset J_i \subset I_i^*.$$ 
        Therefore, for each $i,j\in\{1,2\}$, $i\ne j$, $I_i$ and $J_i^*$ cross over the tip $t_j$. 
        
        If $t_2 \notin v(\ell(I_2)),$ then $v(\ell(I_2)) \subset J_1$. This implies that $\ell(I_2)$ lies on $J_1$. This is a contradiction, because $I_2$ crosses over $t_1.$ Hence, $t_2 \in v(\ell(I_2))$. Similarly, we can see that $t_2 \in v(\ell(J_2))$. Therefore, $\{I_2, J_2\}=\tipp{t_2}$ and we can conclude that $\tipp{t_2}$ crosses over $t_1$. In a similar way, we can see that $\tipp{t_1}$ crosses over $t_2$.
        \end{proof}

        \begin{prop}\label{Prop:tipPairCross}
        Let $\sV=\{\sL_1, \sL_2\}$ be a veering pair. Suppose that there are non-leaf gaps $\sG_1$ and $\sG_2$ in $\sL_1$ and $\sL_2,$ respectively. If there is a tip pair $\tipp{t}$ in $\sG_1$ crossing over a tip $t'$ of $\sG_2,$ Then $\sG_1$ and $\sG_2$ cross. 
        \end{prop}
        \begin{proof}
        With $\tipp{t_1}$ and $\tipp{t_2},$ $\sG_1$ and $\sG_2$ cross.
        \end{proof}

        \begin{prop}\label{Prop:linkedGaps}
        Let $\sV=\{\sL_1, \sL_2\}$ be a veering pair. Suppose that for each $i\in \{1,2\}$, $\sG_i$ is a gap in $\sL_i.$ If $\sG_1$ and $\sG_2$ are linked, then $\sG_1$ and $\sG_2$ interleave or cross.
        \end{prop}
        \begin{proof}
        If both $\sG_1$ and $\sG_2$ are leaves, then they obviously interleave and cross. Also, the case where one of $\sG_1$ and $\sG_2$ is a leaf and the other is a non-leaf gap follows from \refprop{leafGapCross}.
        
        Now, we assume that $\sG_1$ and $\sG_2$ are non-leaf gaps and that they do not interleave. Then, the interleaving gap $\sH_1$ of $\sG_1$ is not $\sG_2.$ Therefore, by \refprop{twoGapUnlinked}, $\sG_2$ lies on a good open interval $K$ in $\sH_1.$ Then, $J_2^* \subset K$ for some $J_2\in \sG_2.$ 
        
        Now, we say that the tip of $\sG_1$ over which $K$ crosses is $t_1.$ If $t_1\in J_2,$ then $v(\sG_1) \subset J_2$ since $v(\sG_1)-\{t_1\}\subset K.$ This implies that $\sG_1$ and $\sG_2$ are unlinked and so it is a contradiction. Hence, $t_1 \in J_2^*.$ 
        
        If $t_1$ is a pivot with respect to $\sL_2,$ then $t_1$ is also a pivot in $\sL_1$, but, by \reflem{trichotomy2}, it is a contradiction since $t_1$ is a tip of $\sG_1.$ Therefore, there is an element $I_2$ in $\sG_2$ containing $t_1$ by \reflem{trichotomy2}.
        
        Now, we consider distinct pairs $\tipp{t_1}$ and $\{I_2, J_2\}$ to show that $\sG_1$ and $\sG_2$ cross.
        Note that 
        $$ v(\sG_2) \subset \closure{J_2^*}-\{t_1\} \subset \closure{K}-\{t_1\} \subset \bigcup \tipp{t_1}.$$ We write $\{M,N\}$ for $\tipp{t_1}.$ Since $I_2$ is crossing over $t_1,$ $\ell(I_2)$ and $\ell(M)$ are linked and so $\ell(M)$ and $\sG_2$ are linked. Therefore, by \refprop{leafGapCross}, $\ell(M)$ and $\sG_2$ cross. Therefore, one of $M$ and $M^*$ contains only one tip $t_2$ of $v(\sG_2).$ This implies that one of $M$ and $N$ contains only one tip $t_2$ of $v(\sG_2)$ since $$v(\sG_2)\subset M \cup N \subset M\cup M^*.$$ Therefore, $\{I_2, J_2\}=\tipp{t_2}$ and so $\tipp{t_2}$ crosses over $t_1.$ Thus, by \refprop{tipPairCross}, $\sG_1$ and $\sG_2$ cross. 
        \end{proof}

Now, we may summarize as follows.
\begin{prop}\label{Prop:twoGap}
    Let $\sV=\{\sL_1, \sL_2\}$ be a veering pair. Suppose that a non-leaf gap $\sG_1$ and a gap $\sG_2$ of $\sV$ are linked. Then, one of the following cases holds. 
    \begin{itemize}
        \item $\sG_1$ and $\sG_2$ interleave.
        \item $\sG_2$ is a leaf and $\sG_2$ crosses over a tip of $\sG_1.$
        \item $\sG_2$ is also a non-leaf gap and there are tips $t_1$ and $t_2$ of $\sG_1$ and $\sG_2,$ respectively, such that $\tipp{t_i}$ crosses over $t_{j}$ for all $i\ne j$.
    \end{itemize}
    \end{prop}

\part{From Veering Pairs to Loom Spaces}\label{loomspacepart}

The major theme of Part 2 is to construct a loom space from a given veering pair. This construction is functorial.

Let $\sV=\{\sL_1,\sL_2\}$ be a veering pair. The main idea is to consider the stitch space $\fS(\sV)$  analogous to the link space in \cite{SchleimerSegerman19}. Then, define the \emph{weaving relation} $\sim_\omega$ on $\fS(\sV).$ Let $\lambda_1$ and $\lambda_2$ be  geodesic laminations on $\HH^2$ associated with $\sL_1$ and $\sL_2,$ respectively. r
We do doubling $\overline{\HH^2}$ to obtain the Riemann sphere $\hat{\CC}$ where  each hemisphere has a copy of a pair of laminations $\lambda_1$ and $\lambda_2$. Then, we collapse each component of the complement of the four geodesic laminations in $\hat{\CC}$ to a point and use Moore's theorem (see \refthm{Moore}) to show that the resulting space is still the sphere. Next, by showing that $\partial \HH$ is still alive after collapsing, we cut the sphere along $\partial \HH$ to get a transversely (singular) foliated disk. Furthermore, we show that the foliated disk is exactly the manifold part of the quotient space  $\fS(\sV)/\sim_\omega.$ After removing all singular points from   $\fS(\sV)/\sim_\omega,$ we show that the universal cover of the resulting space is a loom space in the sense of \cite{SchleimerSegerman21}.

\section{The Compressing Relation on $\closure{\HH^2}$}

\subsection{Moore's Theorem}
Let $X$ be a topological space. Suppose that $P$ is a partition of $X.$ Then we call $P$ a \emph{decomposition} of $X$ if every element of $P$ is  closed in $X.$ Let $\sim_P$ be the equivalence relation on $X$ induced by $P.$ Then there is a quotient map 
$$\pi_P : X \to X/\sim_P$$
defined by $x \mapsto [x]$
where $[x]$ is the equivalence class of $x.$ Then, $X/{\sim_P}$ is a topological space with the quotient topology induced by $\pi_P.$ We call the topological space $X/{\sim_P}$ the \emph{decomposition space} of $P$ and denote it by $\cD(P).$

A decomposition $P=\{C_\alpha\}_{\alpha \in \Gamma}$ of $X$ is said to be \emph{upper semicontinuous} if  for each $\alpha$ in $\Gamma,$ $C_\alpha$ is compact in $X$ and for any open set $U$ of $X$ containing $C_\alpha,$ we can take an open set $V$ so that it satisfies the following:
\begin{itemize}
    \item $C_\alpha \subset V \subset U$ and 
    \item if $C_\beta \cap V\neq \emptyset$ for some $\beta\in \Gamma,$ then $C_\beta \subset U.$
\end{itemize}
See \cite[Section 3-6, Chapter 3 ]{HockingYoung} for more detailed introduction to upper semicontinuous decompositions. See also \cite{Calegari07}.

Now, we assume that $X$ is a $2$-manifold without boundary and $P$ is an upper semicontinuous  decomposition of $X.$ Then we say that $P$ is \emph{cellular} if for each $p\in P,$ there is an  embedding $i$ in $\RR^2$ such that $\RR^2 - i(p)$ is connected.

\begin{thm}[Moore \cite{Moore29}]\label{Thm:Moore}
Let $S$ be either $S^2$ or $\RR^2$. Let $P$ be a cellular decomposition of $S$. Then, the quotient map $\pi_P : S \to \cD(P)$ can be approximated by homeomorphisms. In particular, $S$ and $\cD(P)$ are homeomorphic.  
\end{thm}

\subsection{Cellular Decompositions on $\overline{\HH^2}$}\label{Sec:decompositions}
In this section, we think of $\HH^2$ as the upper half plane of $\CC$ and so our circle $S^1$ is the extended real number $\hat{\RR}=\RR\cup \{\infty\}.$ Let $\sV=\{\sL_1, \sL_2 \}$ be a veering pair. Fix $i\in\{1,2\}$. We define an equivalence relation $\sim_i$ on $S^1$ as follows: For any $u$ and $v$ in $S^1$, $u \sim_i v$ if and only if either $u=v$ or there is a gap $\sG$ in $\sL_i$ such that $\{u,v\}\subset v(\sG)$. We denote the partition induced by $\sim_i$ by $P_i$.

\begin{lem}
$P_i$ is upper semicontinuous.
\end{lem}
\begin{proof}
We first show that each equivalence class of $\sim_i$ is compact subset of $S^1$. Let $v$ be a point in $S^1$. If $v$ is a rainbow point in $\sL_i$, the equivalence class $[v]$ is the one point set $\{v\}$. If $v\in \E{\sL_i}$, then there is a real gap $\sG$ such that $[v]=v(\sG)$ by \reflem{trichotomy2}. By \reflem{trichotomy}, the remaining case is that $v$ is the pivot of some crown $\sG$. In this case, $[v]=v(\sG)$. Therefore, for each $v$ in $S^1,$ $[v]$ is either $\{v\}$ or $v(\sG)$ for some real gap $\sG$ by \refprop{realGap} and so $[v]$ is compact. 

Let $U$ be an open set in $S^1$ containing an equivalence class $[v]$. We split the cases according to \reflem{trichotomy2}.

Suppose that $v$ is a rainbow point in $\sL_i$. In this case, we know that $[v]=\{v\}$. Let $\seq{I_n}$ be a rainbow at $v$ in $\sL_i$. Choose a large enough $n_0\in \NN$ such that $\overline{I_{n_0}}\subset U$. Let $V=I_{n_0}\in \sL_i$. Then, we have $[v] \subset V \subset U$.  Assume that $[w]\cap V \neq \emptyset$ for some $w\in S^1$. We need to show that $[w]\subset U$.  If $[w]=\{w\},$ then $[w]\subset V$ and so $[w] \subset U$. If $[w]=v(\sG)$ for some real gap $\sG,$ then  $v(\sG)\subset \overline{V}\subset U$, so that $[w]\subset U$ as we wanted.

Suppose that $v$ is a vertex of a real gap $\sG$. Then we have $[v]=v(\sG)$. First, consider the case where $\sG=\{I_1,\cdots, I_n\}$ is an ideal polygon. Let $U$ be an open set of $S^1$ containing $[v]$. For each $x \in v(\sG)$, choose an open interval $U_x$ containing $x$ in such a way that $\overline{U_x} \cap \overline{U_y} = \emptyset$ for all $x\ne y$ and that $\overline{U_x} \subset U$ for all $x\in v(\sG)$. Let $U' = \bigcup _{x\in v(\sG)} U_x$. Then, $U'$ is again a neighborhood of $[v]$ and $\overline{U'}\subset U$. Since $\sG$ is a real gap (\refprop{realGap}), we know that there is $J_j \in \sL_i$ such that  $\overline{I_j \setminus  U'} \subset J_j\subset \overline{J_j} \subset I_j$ for each $j=1,2,\cdots,n$. We define $V = U' \setminus \bigcup_{j=1} ^n \overline{J_j}$. From the construction, we know that $\overline{V} \subset \overline{U'} \subset U$. Assume that $w$ is a point in $S^1$ such that $[w]\cap V \neq \emptyset$. If $[w]=\{w\}$, then $[w] \subset V \subset U$. Otherwise, there is a real gap $\sG'$ such that $[w]=v(\sG')$. Again we know that $v(\sG') \subset \overline {V}\subset U$. Hence, $[w]\subset U$ as desired.

Finally, suppose that $v$ is a vertex of a crown $\sG$. Let $U$ be an open set of $S^1$ such that $[v]=v(\sG)\subset U$. We write $\sG=\{\opi{t_n}{t_{n+1}} : n\in \ZZ\}$ as in the definition of crowns and $p$ for the pivot. Since $p$ is the accumulation point of $\{t_n\}_{n\in \ZZ}$ and $\bigcap_{n\in \NN}\opi{t_{n+1}}{t_{-n}}=\{p\},$ there is a number $N \in \NN$ such that $p\in U_{\infty}=\cldi{t_{N+1}}{t_{-N}}\subset U$. We choose disjoint open intervals $U_n$ at each $t_k$, $|k|\le N$, such that $\closure{U_k}\subset U$. Let $U'=\bigcup_{|k|\le N} U_k\cup U_{\infty}$. Then, $[v]\subset U' \subset \closure{U'}\subset U$. Since each $\opi{t_k}{t_{k+1}}\in \sL_i$ is not isolated, we can find $J_k\in \sL_i$ for $|k|\le N$ such that $\closure{\opi{t_k}{t_{k+1}}\setminus U'}\subset J_k \subset \closure{J_k}\subset \opi{t_k}{t_{k+1}}$. Define $V=U'\setminus \bigcup_{|k|\le N} \overline{J_k}$ so that $[v]\subset V\subset \closure{V}\subset \closure{U'} \subset U$. By the same argument as before, we show that any equivalence class $[w]$ with $[w]\cap V \ne \emptyset$ is contained in $\closure{V}$ and so is in $U$.

This shows that $P_i$ is upper semicontinuous. 
\end{proof}

Now, we extend the equivalence relation $\sim_i$ to $\closure{\HH^2}$ as follows. For $x$ and $y\in \closure{\HH^2},$  $x \approx_i y$ if and only if $\{x, y\} \subset H(v)$ for some $v\in P_i$ where $H(K)$ is the Euclidean convex hull in the Klein model of $\closure{\HH^2}$.   Let $Q_i$ be the partition on $\closure{\HH^2}$ induced by $\approx_i$. Observe that $Q_i=\{H(v):v\in P_i\}$. As $P_i$ is an upper semicontinuous decomposition on $S^1,$ $Q_i$ is an upper semicontinuous decomposition on $\closure{\HH^2}$ by elementary hyperbolic geometry.

Finally, we define a partition $P$ on $\overline{\HH^2}$ to be the set 
$$\{c_1\cap c_2 : c_i\in Q_i \text { and } c_1\cap c_2\neq \emptyset\}.$$ Also, we denote the equivalence relation induced from $P$ by ${\approx}$. Then for $x$ and $y$ in $\overline{\HH^2}$, $x\approx y$ if and only if $x \approx_i y$ for all $i\in \{1,2\}$. Since $Q_1$ and $Q_2$ are upper semicontinuous decompositions, $P$ is also an upper semicontinuous decompositions.

Let $v$ be a point in $S^1.$ We claim that the intersection of $S^1$ and the equivalence class $\langle v \rangle$ under $\approx$ is $\{v\}$.  By \reflem{trichotomy2}, if $v$ is not the pivot of some crown, then $v$ is a rainbow point of $\sL_i$ for some $i \in \{1,2\}$. Then the equivalence class $\langle v \rangle_i$ under $\approx_i$ is $\{v\}$. Therefore, $\langle v \rangle=\{v\}$. Otherwise, there is an asterisk $(\sG_1,\sG_2)$ such that $\sG_i$ is a crown having $v$ as the pivot for all $i\in \{1,2\}$. Then, 
$$\langle v \rangle=H(v(\sG_1))\cap H(v(\sG_2)).$$ Hence, $\langle v \rangle \cap S^1=\{v\}$.

Let $P^*$ be the upper semi-continuous decomposition
$$\{J(c):c\in P\}$$ on $\closure{\HH^*}$ where $J$ is the complex conjugation given by $z\mapsto \bar{z},$ and ${\approx^*}$ denotes the equivalence relation defined by $P^*$. Now, we construct an upper semi-continuous decomposition on the Riemann sphere $\hat \CC$. 

First, we define a relation 
${\approxeq}$ on $\hat \CC$ as follows. For any $x$ and $y$ in $\hat \CC,$ $x\approxeq y$ if and only if either $x\approx y$ or $x\approx^* y$. Then, we denote the equivalence relation on $\hat \CC$ generated by ${\approxeq}$ by ${\triplesim}$. Also, we denote the partition defined by ${\triplesim}$ by $Q$.

Let $v$ be a point in $\hat \CC$. 
If the equivalence class $\lsem v \rsem$ of $v$ under ${\triplesim}$ does not intersect $\hat \RR$, then $\lsem v \rsem$ is in $P$ or $P^*$. Otherwise, there is a unique point $w$ in $\lsem v \rsem \cap S^1$ and so 
$$\lsem v \rsem = \langle w \rangle \cup  J(\langle w \rangle).$$ More precisely, if $w$ is not the pivot of some gap, then $\lsem v \rsem=\{w\}$. If $w$ is the pivot point of an asterisk $(\sG_1, \sG_2)$, then  
$$\lsem v \rsem =\left( H(v(\sG_1))\cap H(v(\sG_2)) \right) \cup J(H(v(\sG_1))\cap H(v(\sG_2))).$$ 
Then, we can see that $Q$ is also upper semi-continuous as $P$ and $P^*$ are upper semi-continuous. Moreover, $Q$ is cellular since every element of $Q$ intersects $\hat \RR$ in at most one point. Thus, by \refthm{Moore}, the quotient map 
$$\pi_Q: \hat \CC \to \cD(Q)$$
can be approximated by homeomorphisms and $\cD(Q)$ is homeomorphic to the Riemann sphere $\hat \CC$.

Observe that 
$$\pi_Q|\closure{\HH^2}=\pi_P$$
and $\pi_Q|\hat \RR =\pi_P|\hat \RR$ is an injective continuous map from $\hat \RR$ to $\cD(Q)$. As $\cD(Q)$ is  is homeomorphic to $S^2$, by applying  the Jordan-Schoenflies theorem to the Jordan  curve $\pi_Q(\hat \RR)$ on $\cD(Q)$, we can conclude that $\pi_Q(\closure{\HH^2})$ is homeomorphic to the closed disk whose boundary is the Jordan curve $\pi_Q(\hat \RR)$. Therefore, since $$\pi_Q(\closure{\HH^2})=\pi_P(\closure{\HH^2})=\cD(P),$$ 
the decomposition space $\cD(P)$ of $P$ is homeomorphic to the closed disk. Thus, $\cD(P)\setminus \pi_P(\hat \RR)$ is homeomorphic to $\RR^2$. 
 
From now on, we denote the upper semi-continuous decomposition $P$ on $\closure{\HH^2}$ by $C(\sV)$ where $\sV=\{\sL_1, \sL_2\}$. We refer to the decomposition $C(\sV)$ as the \emph{compressing decomposition}. Also, we call the equivalence relation $\sim_{C(\sV)}$ the \emph{compressing relation} of $\sV.$ Then, we can summarize the result of this section as follows.

\begin{lem}\label{Lem:compressH}
Let $\sV$ be a veering pair. Then, 
the quotient space $\cD(C(\sV))$ is the closed disk. Moreover, the restriction map $\pi_{C(\sV)}| S^1 $ is a Jordan curve which is the boundary of $\cD(C(\sV))$. 
\end{lem}

\section{Stitch Spaces and Weavings}\label{Sec:weavingsec}
As the first step toward loom spaces, we consider stitch spaces and weavings. The stitch space plays the role of the link space in \cite{SchleimerSegerman19}.

\subsection{Stitch Spaces}
Let $\sV=\{\sL_1, \sL_2\}$ be a veering pair. We define the \emph{stitch spaces} $\fS(\sV)$  of $\sV$ be the set of all stitches of $\sV$. Then, $\fS(\sV)$ is a subspace of $\ell(\sL_1)\times \ell(\sL_2)$. For each $i\in \{1,2\}$, we define a projection map $\eta_i$ from $\fS(\sV)$ to $\ell(\sL_i)$ by $\eta_i(\ell_1, \ell_2)=\ell_i.$

We define the \emph{endpoint map} $\epsilon_2$ from $\fS(\sV)$ to $\cC(\sL_1)\times \cC(\sL_2)$ by 
$$\epsilon_2(s):=(\epsilon(\eta_1(s)),\epsilon(\eta_2(s))).$$ We call the image of $\epsilon_2$ the \emph{pair spaces} of $\sV$ and denote it by $\fP(\sV)$. Note that $\fP(\sV)$ is a subspace of $\cM \times \cM$. We define the \emph{link map} $g_2$ from the pair space $\fP(\sV)$ to $\closure{\HH^2}$ by sending $(\mu, \nu)$ to the intersection point of geodesics $g(\mu)$ and $g(\nu)$.  We can see that $g_2$ is an injective continuous map and the image is contained in $\HH^2$. Hence, we can get an injective continuous map 
$$g_2\circ \epsilon_2:\fS(\sV)\to \closure{\HH^2}.$$

\subsection{Warp and Weft threads}
In preparation for the construction of transverse foliations, we introduce warps and wefts, each will play the role of leaves of the forthcoming foliations.

Let $\sV=\{\sL_1, \sL_2\}$ be a veering pair. For each $i\in \{1,2\}$, let $\ell_i$ be a leaf of $\sL_i$. We call the set $\eta_1^{-1}(\ell_1)$ the \emph{warp thread} on $\ell_1$ and denote it by $\overt(\ell_1)$. Likewise, $\eta_2^{-1}(\ell_2)$ is called the \emph{weft thread} on $\ell_2$ and is denoted by $\ominus(\ell_2)$. We say that a subset $\cT$ of $\fS(\sV)$ is called a \emph{thread} on a leaf $\ell$ of $\sV$ if $\cT=\eta_i^{-1}(\ell)$ for some $i\in \{1,2\}$. Also, we denote $\cT$ by $\oslash(\ell)$. 

\begin{rmk}\label{Rmk:twoSingStitch}
    If the thread $\oslash(\ell)$ on a leaf $\ell$ has a singular stitch, then $\ell$ is a boundary leaf of a non-leaf  gap. Moreover, in this case, since every leaf in $\sL_1$ or $\sL_2$ is not isolated, there are exactly two singular stitches in $\oslash(\ell)$ which is of the same asterisk. 
\end{rmk}

Given a vertex $e_1$ of $\ell_1$ and a stitch $s_1$ in $\overt(\ell_1)$, the \emph{half warp thread} $\overt(\ell_1,s_1,e_1)$  on $\ell$ emanating from $s_1$ to $e_1$ is the set of all stitches $s$ in $\overt(\ell_1)$ such that $\eta_2(s)$ lie on the element of $\eta_2(s_1)$ containing $e_1$. Likewise, given a vertex $e_2$ of $\ell_2$ and a stitch $s_2$ in $\ominus(\ell_2)$, we define the \emph{half weft thread} on $\ell_2$ 
emanating from $s_2$ to $e_2$. For a leaf  $\ell$  of $\sV$, a vertex $e$ of $\ell$, and a stitch $s$ in $\oslash(\ell)$, we define the \emph{half thread }$\oslash(\ell,s,e)$ on $\ell$ 
emanating from $s$ to $e$ in the same fashion.

Let $\sL$ be a lamination system. A leaf $\ell$ of $\sL$ 
 \emph{lies between} distinct leaves $m_1$ and $m_2$ of $\sL$ if $m_i\neq \ell$ and  $m_i$ lies on $I_i$ for all $i\in\{1,2\}$ where $\ell=\{I_1, I_2\}$. We say that $\ell$ \emph{properly lies between} $m_1$ and $m_2$ if $m_i$ properly lies on $I_i$ for all $i\in\{1,2\}$ where $\ell=\{I_1, I_2\}$.

Let $s_1$ and $s_2$ be distinct stitches of $\sV.$ We say that a stitch $s$    is \emph{between} $s_1$ and $s_2$ if $\eta_i(s_1)=\eta_i(s_2)=\eta_i(s)$ and $\eta_{i+1}(s)$ lies between $\eta_{i+1}(s_1)$ and $\eta_{i+1}(s_2)$ for some $i\in \{1,2\}$. Also, we denote the set of all stitches lying between $s_1$ and $s_2$ by $\cI(s_1, s_2)$ and call $\cI(s_1, s_2)$ the \emph{interval} between $s_1$ and $s_2$. The \emph{closure} $\closure{\cI}(s_1, s_2)$ of $\cI(s_1, s_2)$ is $\cI(s_1, s_2)\cup\{s_1, s_2\}$ and it is called the \emph{closed interval} between $s_1$ and $s_2.$

\begin{prop}\label{Prop:gapBtwStitches}
    Let $\sV=\{\sL_1, \sL_2\}$ be a veering pair. We take a leaf $\ell$ of $\sL_1$. Suppose that there are distinct stitches $s_1$ and $s_2$ in the thread $\oslash(\ell)$ on $\ell$. If the interval $\cI(s_1,s_2)$ on $\oslash(\ell)$ is empty, then there is a tip $v$ such that $\{s_j\}= \oslash(\ell)\cap \oslash(\ell(I_j))$ for all $j\in \{1,2\}$ where $\tipp{v}=\{I_1, I_2\}$. Hence, $\ell$ crosses over the tip $v$. The same is true when $\ell$ is a leaf of $\sL_2$. 
    \end{prop}
    \begin{proof}
    When $\ell$ is a leaf of $\sL_1$, we have $\oslash(\ell)=\overt(\ell)$. We write $s_1=(\ell, m_1)$ and $s_2=(\ell, m_2)$. We can take $I_1$ and $I_2$ in $\sL_2$ so that $\ell(I_j)=m_j$ for all $j\in \{1,2\}$ and $I_1 \subset I_2$. By assumption, the stem $\stem{I_1}{I_2}$ in $\sL_2$ is exactly $\{I_1, I_2\}.$ This implies that $I_2$ is isolated. 
    
    Since $I_2$ is isolated, there is a non-leaf gap $\sG$ in $\sL_2$ containing $I_2^*.$ since $I_1 \subset I_2,$ there is an element $I_3$ in $\sG$ such that $I_2\subset I_3$ If $I_2\neq I_3,$ then $I_2 \subsetneq I_3 \subsetneq I_1$ and so $I_3\in \stem{I_1}{I_2}.$ This is a contradiction. Therefore, $I_2=I_3\in \sG.$ Hence $\sG$ and $\ell$ cross. Then, by \refprop{leafGapCross}, $\{I_1, I_2^*\}$ is the tip pair at an end point $v$ of $\sG,$ and so  $\ell$ crosses over the tip $v.$
    \end{proof}
    
    \begin{rmk}\label{Rmk:tipBtwStitches}
    Let $\sV=\{\sL_1, \sL_2\}$ be a veering pair and $\ell$ a leaf of $\sV.$ Suppose that there is a singular stitch $s_1$ in $\oslash(\ell).$ By \refrmk{twoSingStitch}, there is another singular stitch $s_2$ in $\oslash(\ell)$ and the interval $\cI(s_1,s_2)$ on $\oslash(\ell)$ is empty. Therefore, by \refprop{gapBtwStitches}, there is a tip $v$ such that $\{s_i\}=\oslash(\ell)\cap \oslash(\ell(I_j))$ for all $j\in \{1,2\}$ where $\tipp{v}=\{I_1, I_2\}.$
    \end{rmk}

\subsection{Weaving Stitches}
Let $\sV=\{\sL_1, \sL_2\}$ be a veering pair. Now, we define an relation $\frown$ on the stitch space $\fS(\sV)$ as follows.  For any distinct stitches $s_1$ and $s_2$ in $\fS(\sV)$, $s_1 \frown s_2$ if and only if  $\{s_1,s_2\}\subset \oslash(\ell)$ for some leaf $\ell$ of $\sV$ and $\cI(s_1,s_2)=\emptyset$. Then, we denote the equivalence relation generated by the relation $\frown$ by $\sim_\omega$. We refer to the equivalence relation $\sim_\omega$ as the \emph{weaving relation} of $\sV$. The map $\#$ denotes the quotient map from $\fS(\sV)$ to $\fS(\sV)/\sim_\omega$.   

\begin{rmk}\label{Rmk:eqDef}
Let  $v$ be a tip of $\sL_1$. Assume that $\ell$ is a leaf crossing over $v$. Then for each $i\in \{1,2\}$, $s_i=(\ell(I_i), \ell)$ is a stitch where $\tipp{v}=\{I_1, I_2\}$. Note that since $I_1^*$ and $I_2^*$ are isolated, there is no $I$ in $\sL_1$ such that $I_1 \subsetneq I \subsetneq I_2^*$. Therefore, $\cI(s_1,s_2)$ is empty. Thus, $s_1 \frown s_2$. The converse follows from \refprop{gapBtwStitches}. Compare with the equivalence relation defined on the pair space in \cite{SchleimerSegerman19}.
\end{rmk}

Now, we define $\omega$ from the stitch space $\fS(\sV)$ to the decomposition space  $\cD(C(\sV))$ of the compressing decomposition by 
$$\omega=\pi_{C(\sV)}\circ g_2 \circ \epsilon_2$$
where  $\epsilon_2$ is the end point map and $g_2$ is the link map. We call $\omega$ the \emph{weaving map} of $\sV$.

\begin{prop}\label{Prop:compressingAndWeaving}
Let $\sV=\{\sL_1, \sL_2\}$ be a veering pair. Then for stitches $s_1, s_2\in \fS(\sV)$,  $s_1\sim_\omega s_2$ if and only if $\omega(s_1)=\omega(s_2)$.  
\end{prop}
\begin{proof}
Let $s_1$ and $s_2$ be stitches such that $s_1 \frown s_2$. Then, there is a thread $\oslash(\ell)$ containing $s_1$ and $s_2$. Since $\cI(s_1,s_2)$ is empty, by \refprop{gapBtwStitches}, there is a tip $t$ such that $\{s_i\}=\oslash(\ell)\cap \oslash(\ell(I_i))$ for all $i\in \{1,2\}$ where $\tipp{t}=\{I_1, I_2\}$.  Then  $g_2 \circ \epsilon_2(s_i)\in H(v(\tipg{t}))$ for all $i\in\{1,2\}$ as $g(v(\ell(I_i))) \subset H(v(\tipg{t}))$ for all $i\in\{1,2\}$. Meanwhile, $g_2 \circ \epsilon_2(s_i)\in H(v(\ell))$ for all $i\in\{1,2\}$. Since there is an element $C$ in $C(\sV)$ containing  $H(v(\tipg{t})) \cap H(v(\ell))$, $g_2 \circ \epsilon_2(s_1) \sim_{C(\sV)} g_2 \circ \epsilon_2(s_2)$. See \refsec{decompositions}. Therefore, $\omega(s_1)=\omega(s_2)$.  

Conversely, suppose that $\omega(s_1)=\omega(s_2)$ for stitches $s_1$ and $s_2$. Then there is an element  $D\in C(\sV)$ such that  $g_2 \circ \epsilon_2(s_i)\in D$, $i\in \{1,2\}$. There are real gaps $\sG_1$ and $\sG_2$ such that each gap $\sG_i$ is in $\sL_i$ and $D=H(v(\sG_1)) \cap H(v(\sG_2))$.

\begin{claim}\label{Clm:stitchInCompression}  For any stitch $s'$ with $g_2 \circ \epsilon_2(s')\in D$, there are boundary leaves $\eta_1(s')$ and  $\eta_2(s')$ of $\sG_1$ and $\sG_2,$ respectively such that  $s'=(\eta_1(s'), \eta_2(s'))$.
\end{claim}
\begin{proof}
See \refsec{decompositions} and observe the following. Let $\sL$ be a quite full lamination system and $\sG$ be a gap in $\sL.$ Then the boundary of $H(v(\sG))$ consists of convex hulls of boundaries of elements of $\sG.$ Namely, 
    $$\partial H(v(\sG))=\bigcup_{I\in \sG} H(v(\ell(I))).$$
\end{proof}

Note that $\sG_1$ and  $\sG_2$ are linked. If $(\sG_1,\sG_2)$ is a stitch, then $s_1=s_2=(\sG_1,\sG_2)$. Therefore, $s_1 \sim_\omega s_2$.

Now we consider the case where $(\sG_1,\sG_2)$ is an asterisk of $\sV$. By \refclm{stitchInCompression}, each element of $\#(s_1)$ is a singular stitches of $(\sG_1,\sG_2)$.  For each boundary leaf $\ell$ of $\sG_1$ or $\sG_2,$ the thread $\oslash(\ell)$ has exactly two singular stitches and the interval between these stitches is empty. See \refrmk{tipBtwStitches} and \refrmk{eqDef}. This implies that $\#(s_1)$ contains the set of all singular stitches of $(\sG_1, \sG_2)$. Therefore, $(g_2 \circ \epsilon_2)^{-1}(D)=\#(s_1)$. The same argument shows that $(g_2 \circ \epsilon_2)^{-1}(D)=\#(s_2)$. Therefore, $\#(s_1)=\#(s_2)$ and $s_1\sim_\omega s_2$.

Now we assume that $\sG_1$ is a non-leaf gap and $\sG_1$ and $\sG_2$ are not interleaving. If $\sG_2$ is a leaf, then by \refprop{leafGapCross}, there is a tip $v$ of $\sG_1$ over which $\sG_2$ crosses. Then we can see that $(g_2 \circ \epsilon_2)^{-1}(D)$ has exactly two points, namely, $(g_2 \circ \epsilon_2)^{-1}(D)=\{(\ell(I_1), \sG_2), (\ell(I_2), \sG_2)\}$ where $\tipp{v}=\{I_1,I_2\}$. By \refrmk{eqDef},  $(\ell(I_1), \sG_2)\frown (\ell(I_2), \sG_2)$.  Thus, $(g_2 \circ \epsilon_2)^{-1}(D)= \#(s_1)=\#(s_2)$. 

Then we consider the case where $\sG_2$ is also a non-leaf gap. Since $\sG_1$  and $\sG_2$ are not interleaving, by \refprop{linkedGaps}, $\sG_1$ and $\sG_2$ cross. Then by \refprop{crossingTwoGaps}, there are tips $t_1$ and $t_2$ of $\sG_1$ and $\sG_2$, respectively, such that  for each $i\in \{1,2\}$, the tip pair $\tipp{t_i}$ crosses over $t_{i+1}$. Then 
$(g_2 \circ \epsilon_2)^{-1}(D)$ has exactly four elements, namely, 
$$(g_2 \circ \epsilon_2)^{-1}(D)=\{(\ell(I_1), \ell(I_2)):I_i\in \tipp{t_i} \text{ for all } i\in \{1,2\}\}.$$
Note that for each $i\in \{1,2\}$, and for any $I\in \tipp{t_i}$, $\ell(I)$ crosses over $t_{i+1}$. By \refrmk{eqDef}, $(g_2 \circ \epsilon_2)^{-1}(D)$ is an equivalence class under the compression relation. Therefore, $(g_2 \circ \epsilon_2)^{-1}(D)=\#(s_1)=\#(s_2)$. 

The case where $\sG_2$ is a non-leaf gap can be proven similarly. Thus, $$\omega^{-1}(\omega(s_1))=(g_2\circ \epsilon_2)^{-1}(D)=\#(s_2).$$ 
\end{proof}

\begin{rmk}\label{Rmk:subrelation}Let $\sV$ be a veering pair. For stitches $s_1$ and $s_2$ in $\fS(\sV)$, by \refprop{compressingAndWeaving},   $s_1\sim_\omega s_2$ if and only if $g_2 \circ \epsilon_2 (s_1) \sim_{C(\sV)} g_2 \circ \epsilon_2(s_2)$. Moreover, from the proof of \refprop{compressingAndWeaving}, we can see that if $s_1$ is a singular stitch of some asterisk $(\sG_1, \sG_2)$ and $s_1\sim_\omega s_2$, then $s_2$ is also a singular stitch of $(\sG_1,\sG_2)$. Also, $\#(s_1)=\#(s_2)$ is the set of all singular stitches of $(\sG_1,\sG_2)$.    
\end{rmk}

\subsection{Weavings}
Let $\sV=\{\sL_1, \sL_2\}$ be a veering pair.
we denote the quotient space of $\fS(\sV)$ under $\sim_\omega$ by $\closure{\fW}(\sV)$ and call it the \emph{cusped weaving} of $\sV.$  Now, we may summarize the results in the proof of \refprop{compressingAndWeaving} as the following propositions. Compare with \cite[Lemma~10.5]{SchleimerSegerman19}. 

\begin{prop}\label{Prop:eqClass} 
Let $\sV=\{\sL_1, \sL_2\}$ be a veering pair. For each $w$ in $\closure{\fW}(\sV)$, there are real gaps $\sG_1$ and $\sG_2$ of $\sL_1$ and $\sL_2$, respectively, such that $$\#^{-1}(w)=\{s\in \fS(\sV):  \text{ $\eta_i(s)$ is a boundary leaf of $\sG_i$ for all $i\in \{1,2\}$}\}.$$
\end{prop}

More precisely, we can restate \refprop{eqClass} with \refprop{twoGap} as follows.
\begin{prop}\label{Prop:eqClass2}
Let $\sV=\{\sL_1,\sL_2\}$ be a veering pair and $w$ an element in the cusped weaving. 
Then, one of the following cases holds. 
\begin{itemize}
    \item $\#^{-1}(w)$ is a singleton whose element is a genuine stitch.
    \item $\#^{-1}(w)$ consists of two regular stitches $s_1$ and $s_2$ such that 
    \begin{itemize}
        \item $\ell_1=\eta_i(s_1)=\eta_i(s_2)$ for some $i \in\{1,2\}$, 
        \item $\ell_1$ is real, and
        \item $\eta_{i+1}(s_1)$ and $\eta_{i+1}(s_2)$ are parallel. 
    \end{itemize}  
    \item There are points $t_1$ and $t_2$ in $\E{\sL_1}$ and $\E{\sL_2},$ respectively, such that 
    \begin{itemize}
        \item  $t_i$ is a tip of a non-leaf gap $\sG_i$ of $\sL_i$ for all $i\in {1,2},$
        \item $\tipp{t_i}$ crosses over $t_{j}$ for all $i,j\in \{1,2\}$, $i\ne j$, and 
        \item $\#^{-1}(w)=\{(\ell(I), \ell(J)): I\in \tipp{t_1} \text{ and  } J\in \tipp{t_2}\}.$
    \end{itemize}
    \item $\#^{-1}(w)$ is the set of all singular stitches of an asterisk.
\end{itemize}
\end{prop}

By \refprop{eqClass2}, each element of the cusped weaving has only regular stitches or singular stitches. We say that $w\in \closure{\fW}(\sV)$ is a \emph{singular class (cusp class, respectively)} associated with an asterisk $(\sG_1,\sG_2)$ if $\sG_1$ and $\sG_2$ are ideal polygons (crowns, respectively) and $\#^{-1}(w)$ is the set of singular stitches of $(\sG_1,\sG_2)$. We say $w$ is a \emph{regular class} if $w$ is neither a singular class nor a cusp class. 

For convenience, we denote the class $w$ in $\closure{\fW}(\sV)$ \emph{associated with an interleaving pair} $(\sG_1,\sG_2)$ of real gaps by $w(\sG_1,\sG_2),$ that is, 
$$w=w(\sG_1,\sG_2)=\{s\in \fS(\sV): \text{$\eta_i(s)$ is a boundary leaf of $\sG_i$ for all $i\in \{1,2\}$}\}.$$ 

We define the \emph{order} $\ord(w)$ of $w=w(\sG_1, \sG_2)$  by
\[
\ord(w):= \begin{cases}
2 & \text{if }(\sG_1,\sG_2) \text{ is a genuine stitch}\\
|v(\sG_1)| & \text{if $w$ is a singular class} \\
\infty & \text{if  $w$ is a cusp class} 
\end{cases}.
\]

Now we define the \emph{weaving} $\fW(\sV)$ of $\sV$ to be 
$$\fW(\sV):=\closure{\fW}(\sV)\setminus \{ \text{cusp classes} \}$$
and the \emph{regular weaving} $\fW^\circ(\sV)$ of $\sV$ to be 
$$\fW^\circ(\sV):=\closure{\fW}(\sV)\setminus \{ \text{cusp and singular classes} \}.$$

\begin{lem}\label{Lem:weavingIsDisk}
    Let $\sV=\{\sL_1, \sL_2\}$ be a veering pair. Then the weaving $\fW(\sV)$ of $\sV$ is homeomorphic to the Euclidean plane $\RR^2.$ More precisely, there is a unique homeomorphism $\hat \omega$ such that the following diagram commutes:
    \begin{center}
    \begin{tikzcd}
    \fS(\sV)\setminus \{ \text{stitches associated with cusp classes}\}  \arrow[r, "\omega"] \arrow[ddr, "\#"]& \interior{\cD(C(\sV))}\\
    &\\
    & \fW(\sV) \arrow[uu,dashed,"\hat{\omega}"]
    \end{tikzcd}
    \end{center}
    where  $\interior{\cD(C(\sV))}=\cD(C(\sV))\setminus \partial \cD(C(\sV))$ and  $\#$ is the quotient map defined by the weaving relation $\sim_\omega$. We call the map $\hat \omega $ the \emph{trivialization} of the weaving. 
    \end{lem}
    \begin{proof}
    By \reflem{compressH}, the decomposition space $\cD(C(\sV))$ of the compressing decomposition is homeomorphic to the closed disk. We know that $\pi_{C(\sV)}| \partial \HH^2$ is an injective continuous map and $\partial \cD(C(\sV))=\pi_{C(\sV)}(\partial \HH^2)$.
    
    Now, for each point $p$ in $\closure{\HH^2}$, we denote the equivalence class of $p$ under the compressing relation $\sim_{C(\sV)}$ by $\lsem p \rsem$ as in \refsec{decompositions}. Then we can see that $$\pi_{C(\sV)}^{-1}(\interior{\cD(C(\sV))})=\closure{\HH^2}- \bigcup_{p\in \partial \HH^2} \lsem p \rsem
    .$$ Recall that if $p$ in $\partial \HH^2$ is not a pivot, then $\lsem p \rsem=\{p\}$ and  if $p$ is the pivot of an asterisk $(\sG_1,\sG_2),$ then $\lsem p \rsem = H(v(\sG_1))\cap H(v(\sG_2)).$ See \refsec{decompositions}. Therefore, if $s$ is  a stitch in $\fS(\sV)$ such that  $g_2\circ \epsilon_2(s)\in \bigcup_{p\in \partial \HH} \lsem p \rsem,$ then $s$ is a singular stitch of an asterisk of crowns and so $\#(s)$ is a cusp class by \refclm{stitchInCompression} and \refrmk{subrelation}. This implies that $$\omega^{-1}(\interior{\cD(C(\sV))})=\fS(\sV)\setminus \{ \text{stitches associated with cusp classes}\} .$$ Then, by \refprop{compressingAndWeaving} and the universal property of quotient maps, there is a unique homeomorphism $\hat{\omega}$ such that the following diagram commutes:
    \begin{center}
    \begin{tikzcd}
    \fS(\sV)\setminus \{ \text{stitches associated with cusp classes}\}  \arrow[r, "\omega"] \arrow[ddr, "\#"]& \interior{\cD(C(\sV))}\\
    &\\
    & \fW(\sV) \arrow[uu,dashed,"\exists! \hat{\omega}"]
    \end{tikzcd}
    \end{center}
    Thus, $\fW(\sV)$ is homeomorphic to $\RR^2$ as $\interior{\cD(C(\sV))}$ is homeomorphic to $\RR^2.$
    \end{proof}
Similarly, we can get the following lemma for the cusped weavings.

\begin{lem}\label{Lem:trivialization}
    Let $\sV$ be a veering pair. Then, there is a unique map $\hat \omega$ from $\closure{\fW}(\sV)$ to $\cD(C(\sV))$ such that it  makes the following diagram  commute.
    \begin{center}
    \begin{tikzcd}
    \fS(\sV)\arrow[r,"\#"] \arrow[d,swap , "g_2 \circ \epsilon_2"] \arrow[rd, "\omega"]& \closure{\fW}(\sV)\arrow[d, dashed,"\exists ! \hat{\omega}"]\\
    \closure{\HH^2} \arrow[r,swap,"\pi_{C(\sV)}"]& \cD(C(\sV)) 
    \end{tikzcd}
    \end{center}
\end{lem}

For clarity, we make the following remarks.

\begin{rmk}\label{Rmk:compressingClass}
Let $\sV=\{\sL_1,\sL_2\}$ be a veering pair. Assume that $\sG_i$ is a real gap of $\sL_i$ for all $i\in \{1,2\}$ and that  $\sG_1$ and $\sG_2$ are linked. Now, we set $$w=\{s\in \fS(\sV): \text{$\eta_i(s)$ is a boundary leaf of $\sG_i$ for all $i\in \{1,2\}$} \}$$ and
$$c=H(v(\sG_1))\cap H(v(\sG_2)).$$
By \refprop{eqClass2}, $w \in \closure{\fW}(\sV)$ and so $c=\hat{\omega}(w)\in \cD(C(\sV)).$ Furthermore, $c$ is the convex hull $H(g_2\circ \epsilon_2(w))$ except the case where $\sG_i$ are crowns. When $(\sG_1,\sG_2)$ is an asterisk of crowns, $c$ is the union of  $H(g_2\circ \epsilon_2(w))$  and the pivot point of $\sG_i$.
\end{rmk}

\begin{rmk}\label{Rmk:boundary}
Let $\sV=\{ \sL_1, \sL_2\}$ be a veering pair. By \reflem{trivialization}, we can identify $\partial \HH^2$ with $\partial \cD(C(\sV))$ via $\pi_{C(\sV)}|\partial \HH^2.$ Moreover, by \refrmk{compressingClass}, each point in $\im(\hat{\omega})\cap S^1$ is the image of a cusp class $w$ under $\hat \omega$ and so it is the pivot point of the asterisk associated with $w.$ See \reflem{trichotomy2}. 
\end{rmk}

\section{Threads on Weavings}\label{Sec:threadsec}
Even though we have shown that a weaving is a disk, this does  not tell us  that each leaf of a veering pair is realized as a foliation line in the weaving.

In this section, we show that each thread gives a foliation line in the weaving.

\subsection{Leaves on Weavings}
Let $S$ be a topological space. Then $S$ is an \emph{arc} if there is a homeomorphism $\phi_S$  from the closed interval $[0,1]$ on the real line $\RR$ to $S$. We say that $S$ is an \emph{arc from $a$ to $b$}  if $\phi_S(0)=a$ and $\phi_S(1)=b$. In this case, $a$ and $b$ are the \emph{end points} of $S$.

Let $x$ be a point in $X$. Assume that $X$ is connected. Then we call $x$ a \emph{cut point} of $X$ if $X\setminus \{x\}$ is disconnected. Otherwise, $x$ is called a \emph{non-cut point} of $S$. 

The following theorem provides a topological characterization of arcs. 

\begin{thm}[See \cite{Wilder49}] \label{Thm:characterizationOfArc}
Let $S$ be a compact, connected, separable and metrizable space. If there is a two elements subset $\{a,b\}$ of $S$ such that every element in $S$ is a cut point of $S$, then $S$ is an arc from $a$ to $b$. \end{thm}

\begin{prop}\label{Prop:arcImage}
Let $\sV=\{\sL_1, \sL_2\}$ be a veering pair and $\ell$ be a leaf of $\sL_1$ or $\sL_2$. Then $\pi_{C(\sV)}(H(v(\ell)))$ is an arc in $\cD(C(\sV))$ whose end point set is $v(\ell).$
\end{prop}
\begin{proof}
We assume that $\ell$ is a leaf of $\sL_1.$ The convex hull $h=H(v(\ell))$ is an arc which is properly embedded in $\closure{\HH^2}$. Since $h$ is compact, connected, and  separable, the continuous image $\pi_{C(\sV)}(h)$ is also compact, connected, and separable. Moreover, $\cD(C(\sV))$ is the closed disk and so $\pi_{C(\sV)}(h)$ is metrizable. 

Now, we write $\{a,b\}$ for $v(\ell)$. By \reflem{trichotomy2}, $a$ and $b$ are not pivots.  Therefore, the equivalence classes $\lsem a\rsem$ and $\lsem b \rsem$ of $a$ and $b$, respectively, under the compressing relation, are $\{a\}$ and $\{b\}$. See \refsec{decompositions}. 

Choose a point $x$ in $\pi_{C(\sV)}(h)\setminus \{\lsem a \rsem,\lsem b \rsem \}$.  Then there is a gap $\sG$ in $\sL_2$ such that 
$$ H(v(\ell))\cap x = H(v(\ell))\cap H(v(\sG)).$$
If $\sG$ is a leaf of $\sL_2$, then $H(v(\ell))\cap x=\{s\}$ where $s=g_2\circ \epsilon_2((\ell, \sG))$.
Then $h\setminus \{s\}$ has two connected components. We denote the component containing $a$ by $h_a$ and the component containing $b$ by $h_b$. 

Now, we write $\sG=\ell(I)$ for some $I\in \sL_2$ with $a\in I$.  Then observe that $$\{H(I^c)^c,H(v(\ell)), H((I^*)^c)^c \}$$ is a partition
on $\closure{\HH^2}$ and that $\pi_{C(\sV)}(H(I^c)^c)$ and $\pi_{C(\sV)}(H((I^*)^c)^c)$ are open in $\cD(C(\sV))$ since $\sG$ is a real leaf by \refrmk{compressingClass}. See \refsec{decompositions}. 

Note that $h_a\subset H(I^c)^c$ and $h_b\subset H((I^*)^c)^c$. Since 
$$\pi_{C(\sV)}(h\setminus \{s\})=\pi_{C(\sV)}(h)\setminus \{x\}\subset 
\pi_{C(\sV)}(H(I^c)^c)\cup \pi_{C(\sV)}(H((I^*)^c)^c),$$
$\pi_{C(\sV)}(h)\setminus \{x\}$ is disconnected. Hence,  $x$ is a cut point of $\pi_{C(\sV)}(h)$.

If $\sG$ is a non-leaf gap, then by \refprop{leafGapCross}, there is a tip $t$ of $\sG$ over  which $\ell$ crosses. Now we write $\{I_1, I_2\}$ for $\tipp{t}$ and assume that $a\in I_1$ and $b\in I_2$. Then $H(v(\ell))\cap x$ is the geodesic arc $g$ from $g_2\circ \epsilon_2((\ell, \ell(I_1)))$ to $g_2\circ \epsilon_2((\ell, \ell(I_2)))$. See \refrmk{compressingClass}. Then $h\setminus g$ has two connected components. Again, we denote the  component containing $a$ by $h_a$ and the component containing $b$ by $h_b$.  

We consider the partition
$$\{H(I_1^c)^c, H((I_1\cup I_2)^c) ,H(I_2^c)^c\}$$
on $\closure{\HH^2}$. Then, $\pi_{C(\sV)}(H(I_1^c)^c)$ and $\pi_{C(\sV)}(H(I_2^c)^c)$ are open in $\cD(C(\sV))$. See \refsec{decompositions}.

Note that $h_a\subset H(I_1^c)^c$ and $h_b\subset H(I_2^c)^c$. Since 
$$\pi_{C(\sV)}(h\setminus g)=\pi_{C(\sV)}(h)\setminus \{x\}\subset \pi_{C(\sV)}(H(I_1^c))\cup \pi_{C(\sV)}(H(I_2^c)),$$
$x$ is a cut point of $\pi_{C(\sV)}(h)$. By \refthm{characterizationOfArc}, $\pi_{C(\sV)}(h)$ is an arc from $\lsem a\rsem$ to $\lsem b \rsem$. Similarly, we can show the case where $\ell$ is a leaf of $\sL_2$.
\end{proof}

Now we show that the image of a thread in cusped weaving is homeomorphic to a real line. 

\begin{lem}\label{Lem:leafToLeaf}
Let $\sV=\{\sL_1, \sL_2\}$ be a veering pair and $\ell$ be a leaf of $\sV.$ Then $\#(\oslash(\ell))$ is homeomorphic to the real line $\RR$ in the cusped weaving $\closure{\fW}(\sV)$. Moreover, $\closure{ \omega(\oslash(\ell))}$ is an arc whose end point set is $v(\ell)$.
\end{lem}
\begin{proof}
We assume that $\ell$ is a leaf of $\sL_1$. Choose a point $x$ in $\omega(\overt(\ell))$. Then there are real gaps $\sG_1$ and $\sG_2$ in $\sL_1$ and $\sL_2$ such that 
$$\pi_{C(\sV)}^{-1}(x)=H(v(\sG_1))\cap H(v(\sG_2)).$$ See \refrmk{compressingClass}. Note that $g_2\circ \epsilon_2(\overt(\ell))\subset H(v(\ell))$ and $\ell$ is a boundary leaf of $\sG_1$.
Then 
$$g_2\circ \epsilon_2(\overt(\ell)) \cap \pi_{C(\sV)}^{-1}(x) \subset H(v(\ell))\cap H(v(\sG_2))$$ by \refrmk{compressingClass}. Therefore, $$g_2\circ \epsilon_2(\overt(\ell))\subset H(v(\ell))\setminus v(\ell)\subset \pi_{C(\sV)}^{-1}(\omega(\overt(\ell))),$$ and 
we also have that 
$$\hat{\omega}(\#(\overt(\ell)))=\omega(\overt(\ell))=\pi_{C(\sV)}(H(v(\ell))\setminus(\ell)).$$

Meanwhile, by \refprop{arcImage}, $\pi_{C(\sV)}(H(v(\ell)))$ is an arc whose end point set is $v(\ell)$ and so $\pi_{C(\sV)}(H(v(\ell))\setminus v(\ell))$ is the interior of the arc. Therefore, $\hat{\omega}(\#(\oslash(\ell)))$ is the interior of the arc. Thus, $\#(\oslash(\ell))$ is homeomorphic to the real line $\RR.$ 
\end{proof}

By the similar argument of \reflem{leafToLeaf} and by \refprop{arcImage}, we can also get the following lemma.

\begin{lem}\label{Lem:ray}
Let $\sV=\{\sL_1, \sL_2\}$ be a veering pair. Choose $i\in \{1,2\}$. Suppose that $\ell$ is a leaf of $\sL_i$ and $s$ is a stitch in $\oslash(\ell)$. For any $e$ in $v(\ell)$, $\closure{\omega(\oslash(\ell, s, e))}$ is the subarc of $\closure{\omega(\oslash(\ell))}$ whose end points are $\omega(s)$ and $e$. 
\end{lem}
Then, we can show the following lemma.
\begin{lem}\label{Lem:subarc}
    Let $\ell$ be a leaf of $\sL_1$ or $\sL_2$. Assume that $s_1$ and $s_2$ are stitches in $\oslash(\ell)$ and that $\cI(s_1,s_2)\neq \emptyset$. Then $\#(\closure{\cI}(s_1,s_2))$ is the subarc of $\#(\oslash(\ell))$ whose end point set is $\{\#(s_1),\#(s_2)\}$.
\end{lem}
\begin{proof}
Observe that we can write $v(\ell)=\{e_1,e_2\}$ so that $$\closure{\cI}(s_1,s_2)=\oslash(\ell,s_1,e_1)\cap \oslash(\ell,s_2,e_2).$$
Then, by \reflem{leafToLeaf}, $\closure{\omega(\oslash(\ell))}$ is an arc whose end point set is $v(\ell),$ and, 
for each $i\in \{1,2\}$, by \reflem{ray},  $\closure{\omega(\oslash(\ell,s_i,e_i))}$ is the subarc of $\closure{\omega(\oslash(\ell))}$ whose end point set is $\{e_i,\omega(s_i)\}$. Therefore, $\omega(\closure{\cI}(s_1,s_2))$ is the subarc of $\closure{\omega(\oslash(\ell))}$ whose end point set is $\{\omega(s_1), \omega(s_2)\}$. Thus, 
$\#(\closure{\cI}(s_1,s_2))$ is the subarc of $\#(\oslash(\ell))$ whose end point set is $\{\#(s_1),\#(s_2)\}.$
\end{proof}

From now on, we discuss the possible configurations of the images of threads.

\begin{lem}\label{Lem:ultraparallelThread}
    Let $\sV=\{\sL_1, \sL_2\}$ be a veering pair. Choose $i\in \{1,2\}$. Suppose that leaves $\mu$ and $\nu$ of $\sL_i$ are ultraparallel. Then the intersection $\#(\oslash(\mu))\cap \#(\oslash(\nu))$ in the cusped weaving $\closure{\fW}(\sV)$ is empty or a singleton whose element is a cusp class or a singular class $w(\sG_1, \sG_2)$ associated with the asterisk $(\sG_1,\sG_2) $ such that $\sG_i$ has $\mu$ and $\nu$ as boundary leaves. 
    \end{lem}
    \begin{proof}
        We assume that $\mu$ and $\nu$ are leaves of $\sL_1$ and that $\#(\overt(\mu))\cap \#(\overt(\nu))$ is not empty. Then, $(\mu,l)\sim_\omega (\nu,m)$ for some leaves $l$ and $m$ of $\sL_2$. Because $\mu$ and $\nu$ are ultraparallel, the last case of \refprop{eqClass2} is the only possible case. Now, the result follows from \refprop{eqClass2} and \refrmk{compressingClass}. 
    \end{proof}

\begin{lem}\label{Lem:intersectLeaves}
Let $\sV=\{\sL_1, \sL_2\}$ be a veering pair. Suppose that $\ell_1$ is a leaf of $\sL_1$ and $\ell_2$ is a leaf of $\sL_2$. Then, the intersection of $\#(\overt(\ell_1))$ and $\#(\ominus(\ell_2))$ has at most one point. 
\end{lem}
\begin{proof}
Assume that there are distinct elements $c^1$ and $c^2$ in $\omega(\overt(\ell_1))\cap \omega(\ominus(\ell_2)).$ For each $i\in\{1,2\}$,  $$\pi_{C(\sV)}^{-1}(c^i)=H(v(\sG_1^i))\cap H(v(\sG_2^i))$$
for some real gaps $\sG_1^i$ and $\sG_2^i$ of $\sL_1$ and $\sL_2$, respectively, by \refrmk{compressingClass}. Since $\sG_i^1$ and $\sG_i^2$ have  $\ell_i$ as a boundary leaf for all $i\in \{1,2\}$, by \reflem{quadrachotomy}, $\sG_i^1=\sG_i^2$ for all $i\in\{1,2\}$. This implies that $c^1= c^2$ and it is a contradiction. Thus, $\omega(\overt(\ell_1))\cap \omega(\ominus(\ell_2))$ has at most one element. 
\end{proof}

\begin{lem}\label{Lem:crossingThread}
    Let $\sV=\{\sL_1, \sL_2\}$ be a veering pair. Choose $i\in \{1,2\}$. Suppose that $\mu_1$ and $\mu_2$ are leaves of $\sL_i$ that are parallel and $\nu$ is a leaf of $\sL_{i+1}$ (cyclically indexed). If $\#(\oslash(\nu))\cap \#(\oslash(\mu_1))=\{c_1\}$ and $\#(\oslash(\nu))\cap \#(\oslash(\mu_2))=\{c_2\},$ then $c_1=c_2.$
    \end{lem}
    \begin{proof}
    Without loss of generality, we may assume that $i=1$. Since $\mu_1$ and $\mu_2$ are parallel, there is a point $t$ in $S^1$ such that $\{t\}=v(\mu_1)\cap v(\mu_2).$ Note that $t$ is a tip and $\mu_j$ is a boundary leaf of $\tipg{t}$ for all $j\in\{1,2\}$. By \refrmk{compressingClass}, for each $j\in \{1,2\}$,
    $$\hat{\omega}(c_j)=H(v(\tipg{t}))\cap H(v(\sG_j))$$
    for some real gap $\sG_j$ of $\sL_j$. Then, $\mu$ is a boundary leaf of $\sG_j$ for all $j\in\{1,2\}$ and so $\sG_1=\sG_2.$  Thus, $\hat{\omega}(c_1)=\hat{\omega}(c_2)$ and $c_1=c_2.$ 
    \end{proof}

\begin{lem}\label{Lem:parallelThread}
Let $\sV=\{\sL_1, \sL_2\}$ be a veering pair. Choose $i\in \{1,2\}$. Suppose that leaves $\mu$ and $\nu$ of $\sL_i$ are parallel, namely, $v(\mu)\cap v(\nu)=\{p\}$ for some $p$ in $S^1.$ Then, there are singular stitches $s_\mu$ and $s_\nu$ in $\oslash(\mu)$ and $\oslash(\nu)$, respectively, such that 
$$\#(\oslash(\mu))\cap \#(\oslash(\nu))=\#(\oslash(\mu, s_\mu, p))=\#(\oslash(\nu,s_\nu,p))$$ and $\eta_{i+1}(s_\mu)=\eta_{i+1}(s_\nu)$ (cyclically indexed). 
\end{lem}
\begin{proof}
Without loss of generality, we assume that $\mu$ and $\nu$ are leaves of $\sL_1$. Then, $\tipg{p}$ has $\mu$ and $\nu$ as boundary leaves. Now, consider the interleaving gap $\sG$ of $\tipg{p}$. There is a unique boundary leaf $\ell$ of $\sG$ crossing over $p$. Let $s_\mu=(\mu,\ell)$ and $s_\nu=(\nu,\ell)$. Since $\mu$ and $\nu$  are parallel,  each class $c$ in $\#(\overt(\mu))\cap \#(\overt(\nu))$, except the class $w(\tipg{p},\sG)$, falls into either the second or third case of \refprop{eqClass2}. Thus, $c$ is given by $(\tipg{p},\sG')$ for some real gap $\sG'$ crossing over $p$. Moreover, $\sG'$ lies on the element of $\ell$ containing $p$. Thus, $\#(\oslash(\mu))\cap \#(\oslash(\nu))$ is a subset of both $\#(\oslash(\mu, s_\mu, p))$ and $\#(\oslash(\nu,s_\nu,p))$. Because $\#(\oslash(\mu, s_\mu, p))$ and $\#(\oslash(\nu,s_\nu,p))$ are clearly contained in $\#(\oslash(\mu))\cap \#(\oslash(\nu))$, the result follows.
\end{proof}
From the proofs of the above, we can see the following fact.
\begin{rmk}
Let $\sV$ be a veering pair and $\ell$ is a leaf of $\sV$. Then, $\alpha=\pi_{C(\sV)}(H(v(\ell)))$ is equal to $\omega(\oslash(\ell))\sqcup v(\ell)$ and one of the following cases holds.
\begin{itemize}
    \item If $\ell$ is a boundary leaf of a crown, then $\alpha \cap S^1=v(\ell)\cup \{p\}$ where $p$ is the pivot of the crown.
    \item Otherwise, $\alpha\cap S^1=v(\ell)$.
\end{itemize}
\end{rmk}

\section{Frames and Scraps}\label{Sec:framesec}

In this section, we introduce and study frames and scraps in veering pairs. Eventually, the scraps give rise to foliated charts for the weavings and rectangles on weavings. See \refsec{rectangle} for the definition of rectangles.

\subsection{Frames and Scraps} 
Let $\sV=\{\sL_1, \sL_2\}$ be a veering pair. A \emph{frame} in $\sV$ is a quadruple $\fF=(I_1, J_1, I_2, J_2)$  in $\sL_1\times \sL_2 \times \sL_1 \times \sL_2$ that satisfies the following:
\begin{itemize}
    \item $s_i=(\ell(I_i), \ell(J_i))$ is a stitch for all $i\in\{1,2\}$.
    \item The sector $(I_i,J_i)$ is counter-clockwise for all $i\in \{1,2\}$.
    \item $s_i$ lies on the sector $(I_{i+1}^*,J_{i+1}^*)$ for all $i\in \{1,2\}$.
\end{itemize}
A stitch $s$ of $\sV$ \emph{lies} on the frame $\fF$ if $\eta_1(s)$ properly lies between $\ell(I_1)$ and $\ell(I_2)$ and $\eta_2(s)$ properly lies between $\ell(J_1)$ and $\ell(J_2)$. We define the \emph{scrap} $\cS(\fF)$ framed by $\fF$ to be the set of all stitches lying on $\fF$. The \emph{closure} $\closure{\cS}(\fF)$ of $\cS(\fF)$ is defined as 
$$\closure{\cS}(\fF):=\{(\ell(N),\ell(M))\in \fS(\sV): N\in \stem{I_1}{I_2^*} \text{ and } M\in \stem{J_1}{J_2^*}\}.$$  
For convenience, we write  
\begin{align*}
  c_1(\fF)&=(\ell(I_1),\ell(J_1)),\\
c_2(\fF)&=(\ell(I_2),\ell(J_1)),\\
c_3(\fF)&=(\ell(I_2),\ell(J_2)), \text{ and }\\
c_4(\fF)&=(\ell(I_1),\ell(J_2)).  
\end{align*}
Also, for each $i\in \{1,2,3,4\}$, we call the stitch  $c_i(\fF)$ the \emph{i-th corner} of the frame $\fF.$ 

\subsection{Maximal Extension}
Let $\fF_I=(I_1, I_2, I_3, I_4)$ and $\fF_J=(J_1, J_2, J_3, J_4)$ be frames in $\sV$. Each $I_i$ is called the \emph{$i^{th}$-side} of $\fF_I$.
For each $i\in \{1,2,3,4\}$, the frame $\fF_J$ is an \emph{$I_i$-side extension} of $\fF_I$ if $I_j=J_j$ for all $j\neq i,$ and $J_i\subsetneq I_i.$ For each $i\in \{1,2,3,4\}$, the frame $\fF_I$ is \emph{$I_i$-side maximal} if there is no $I_i$-side extension. 
A frame $\fF_I$ is a \emph{tetrahedron frame} if $\fF_I$ is $I_i$-side. We say that  $\fF_J$ \emph{covers} $\fF_I$ if  $J_i\subseteq I_i$ for all $i\in \{1,2,3,4\}$.

Now, for each $i \in \{1,2,3,4\}$, we write $\opi{u_i}{v_i}$ for $I_i$. Then $I_i$ contains $\{v_{i-1}, u_{i+1}\}$ for all $i\in \{1,2,3,4\}$ (cyclically indexed). We define the \emph{$I_i$-side rail} of $\fF_I$ to be the stem from  $I_i$ to  $\{v_{i-1}, u_{i+1}\}$ in $\sL$ where $\sL$ is the lamination system in the veering pair containing $I_i$. 

% Also, we call $v_{i-1}$  the \emph{left terminal}, $u_{i+1}$ the \emph{right terminal}, and $\{v_{i-1}, u_{i+1}\}$ the \emph{terminal} of the rail. Since the left and right terminals are not in $\sL,$ there is the minimal element in the $I_i$-side rail which is isolated.

\begin{lem}\label{Lem:maximalSide}
    Let $\sV=\{\sL_1, \sL_2\}$ be a veering pair. Suppose that there is a frame $\fF=(I_1, I_2, I_3, I_4)$ and that $\fF$ is $I_j$-side maximal for some $j \in\{1,2,3,4\}$. Then there is a non-leaf  gap $\sG$ containing $I_j^*,$ and  $I_{j-1}$ and  $I_{j+1}$  cross over $u_j$ and $v_j$ respectively, where  $I_i=\opi{u_i}{v_i}$ (cyclically indexed).
    \end{lem}
    \begin{proof}
    First, we consider the case where $j=1.$ Let $\fR_1$ be the $I_1$-side rail of $\fF.$ If there is an element $J$ in $\fR_1-\{I_1\},$ then $(J, I_2, I_3, I_4)$ is a $I_1$-side extension of $\fF.$ This is a contradiction by assumption. Therefore, $\fR_1=\{I_1\}.$ This implies that $I_1$ is isolated. Hence, there is a non-leaf gap  $\sG$ in $\sL_1$ containing $I_1^*.$
    
    Now we want to show that $I_2$ crosses over $v_1$. Since $\sG$ and $\ell(I_2)$ are linked, by \refprop{leafGapCross},$\ell(I_2)$ cross over $u_1$ or $v_1.$ Assume that $\ell(I_2)$ crosses over $u_1.$ Then $I_2^*$ crosses over $u_1.$ Since $\ell(I_4)$ lies properly on $I_2^*$ and $\sG$ and $\ell(I_4)$ are linked, $\ell(I_4)$ also crosses over $u_1.$ Now, we write $\{I_1^*, K\}$ for the tip pair $\tipp{u_1}.$ Then, $\{v_4, u_1\}\subset K \subset I_1$ and so $K\in \fR_1.$ This is in contradiction with that $\fR_1=\{I_1\}.$ Therefore, $\ell(I_2)$ crosses over $v_1$ and so $I_2$ crosses over $v_1.$ Likewise, we can show that $I_4$ crosses over $u_1.$  Similarly, we can show the cases where $j\in\{2,3,4\}$. 
    \end{proof} 

\begin{rmk}\label{Rmk:maximality}
If the $I$-side rail of a frame $F$ has at least two elements, then there is an $I$-side extension of $F$ that is $I$-side maximal. Otherwise, $F$ is $I$-side maximal.  
\end{rmk}

The following lemma follows from \reflem{maximalSide}.

\begin{lem}\label{Lem:maxStitch}
Let $\sV=\{\sL_1, \sL_2\}$ be a veering pair and $\fF=(I_1, I_2, I_3, I_4)$ be  a frame. Assume that $\fF$ is $I_i$-maximal for some $i\in\{1,2,3,4\}$. Then there are exactly two singular stitches in $\closure{\cI}(c_{i-1}(\fF),c_i(\fF))$ (cyclically indexed). In particular, $\cI(c_{i-1}(\fF),c_i(\fF))$ contains at least one singular stitch.
\end{lem}

\reflem{maxStitch} says that if a frame $\fF$ is $i^{th}$-side maximal, then the arc $\#(\closure{\cI}(c_{i-1}(\fF),c_i(\fF)))$ contains a singular or cusp class.  

\begin{lem}\label{Lem:tetrahedron}
    Let $\sV=\{\sL_1, \sL_2\}$ be a veering pair. Then every frame is covered by  a tetrahedron frame. 
    \end{lem}
    \begin{proof}
    Assume that $\fF=(I_1, I_2,I_3, I_4)$ is a frame in $\{\sL_1,\sL_2\}.$ First, if $\fF$ is $I_1$-side maximal, then we set $\fF_1=\fF.$ Otherwise, by \refrmk{maximality}, there is a $I_1$-side extension $\fF'$ of $\fF$ that is $I_1$-side maximal. Then, we set $\fF_1=\fF'.$ Now, We write $\fF_1=(J_1,I_2, I_3, I_4).$ Then, we apply the same argument to $\fF_1$ to take the second side maximal frame $\fF_2=(J_1,J_2,I_3,I_4)$ that is the second side extension of $\fF_1$. By repeating the same argument in consecutive order, we can get a tetrahedron frame $\fF_4=(J_1,J_2,J_3,J_4)$ covering $\fF$.
    \end{proof}

Now, given a frame $\fF$, we show the natural one-to-one correspondence between opposite sides $\closure{\cI}(c_i(\fF),c_{i+1}(\fF))$ and $\closure{\cI}(c_{i+3}(\fF),c_{i+2}(\fF))$. 

Let $\fF=(I_1,I_2,I_3, I_4)$. For $i=1$, we define a map $\varphi_{2\to 4}$ from $\closure{\cI}(c_1(\fF),c_2(\fF))$ to $\closure{\cI}(c_4(\fF),c_3(\fF))$ by
$$\varphi_{2\to 4}(s)=(\eta_1(s),\ell(I_4)).$$ 

Then, we can observe the following lemma. 
       
\begin{lem}\label{Lem:edgeMap}
Let $\sV=\{\sL_1, \sL_2\}$ be a veering pair. Suppose that there is a frame $\fF=(I_1, I_2, I_3, I_4).$ Then  $\emap{2}{4}$ is a bijection from $\closure{\cI}(c_1(\fF),c_2(\fF))$ to $\closure{\cI}(c_4(\fF),c_3(\fF))$ with the following properties:
\begin{itemize}
\item $\varphi_{2 \to 4}(c_1(\fF))=c_4(\fF)$ and $\varphi_{2 \to 4}(c_2(\fF))=c_3(\fF).$
\item For any $s\in \cI(c_1(\fF), c_2(\fF)),$ $s$ and $\emap{2}{4}(s)$ are in a same thread.
\item For any distinct stitches $r_1$ and $r_2$ in $\closure{\cI}(c_1(\fF),c_2(\fF)),$ if $s\in \cI(r_1, r_2),$ then $\varphi_{2\to4}(s) \in \cI(\emap{2}{4}(r_1),\emap{2}{4}(r_2)).$ 
\end{itemize}
\end{lem}
We define $\emap{4}{2}$ by $\emap{4}{2}=\emap{2}{4}^{-1}$. Similarly, we can find $\emap{1}{3},$ and $\emap{3}{1}$ between $\closure{\cI}(c_4(\fF),c_1(\fF))$ and $\closure{\cI}(c_2(\fF),c_3(\fF))$ and they enjoy the same properties in \reflem{edgeMap} with an  appropriate re-indexing.

\begin{rmk}\label{Rmk:distinctWithCorner}
Assume that a stitch $(\ell_1,\ell_2)$ lies on a frame $\fF=(I_1,I_2,I_3,I_4)$. Then, the stitch $(\ell_1,\ell(I_2))$ is in $\cI(c_1(\fF),c_2(\fF))$.  Now, observe from  \refprop{eqClass2} that no three distinct stitches in a thread are equivalent under the weaving relation. Hence, $c_1(\fF)$ and $c_2(\fF)$ are not equivalent as $\ell_1$ properly lies between $\ell(I_1)$ and $\ell(I_3)$. Likewise, we can see that $c_i(\fF)\nsim_\omega c_{i+1}(\fF)$ for all $i\in\{1,2,3,4\}$ (cyclically indexed).    
\end{rmk}
% Assume that a stitch $s=(\ell_1, \ell_2)$  lies on the frame $\fF=(I_1, I_2, I_3, I_4).$ The leaf $\ell_1$ is linked with $\ell(I_2)$ and $\ell(I_4).$ Then the stitch $s_1=(\ell_1,\ell(I_2))$ is between the first corner $c_1(\fF)$ and the second corner $c_2(\fF),$ and  $s_2=(\ell_1, \ell(I_4))\in \cI(c_4(\fF),c_3(\fF)).$
    
% If $\ell_1$ is not a boundary leaf of a non-leaf gap, then every element in $\ell_1$ is not isolated. This implies that $s_1$ and $s_2$ are not equivalent to the corners under the weaving relation. Otherwise, there is a non-leaf gap $\sG_1$ that has  $\ell_1$ as a boundary leaf. By \refprop{twoGap}, there is a tip $t_1$ of $\sG_1$ over which  $\ell(I_2)$ crosses. 
    
% Now, we write $\tipp{t_1}=\{J_1, J_2\}$ so that $\ell(J_1)=\ell_1.$ Then the stitch $r_1=(\ell(J_2), \ell(I_2))$ is in $\cI(c_1(\fF), c_2(\fF))$
%     since by \refrmk{eqDef}, $\cI(s_1,r_1)$ is empty and $s_1\in \cI(c_1(\fF), c_2(\fF)).$ We write $r_2$ for the stitch $\emap{2}{4}(r_1).$ Note that by \reflem{edgeMap}, $\emap{2}{4}(r_1)=(\ell(J_2), \ell(I_4)),$ and $\ell(I_4)$ crosses over $t_1.$ Then we can conclude that  $\cI(s_2,r_2)$ is also empty and so $r_2\in \cI(c_4(\fF),c_3(\fF)).$ Hence, since $J_1$ and $J_2$ are not isolated, $s_1,$ $s_2,$ $r_1,$ and $r_2$ are not equivalent to the corners under the weaving relation. 

\subsection{Stitches Lying on Frames}

In this section, we study the topology of scraps. First of all, we show that every scarp is not empty.
    \begin{prop}\label{Prop:scrapIsNonempty}
    Let $\sV=\{\sL_1, \sL_2\}$ be a veering pair. Let $\fF=(I_1, I_2, I_3, I_4)$ be a frame in $\sV$. Then, the scarp $\cS(\fF)$ framed by $\fF$ is not empty.
    \end{prop}
    \begin{proof}
    Assume that $\cS(\fF)$ is empty. If $\cI(c_1(\fF),c_2(\fF))$ is empty, then, by \refprop{gapBtwStitches}, $\ell(I_1)$ and $\ell(I_3)$ are parallel. This is a contradiction. Therefore, $\cI(c_1(\fF),c_2(\fF))$ is not empty. Similarly, we can show that $\cI(c_i(\fF),c_{i+1}(\fF))$ is not empty for all $i \in \{1,2,3,4\}$ (cyclically indexed). Then  by \reflem{edgeMap}, there are $J_1$ and $J_2$ in $\sL_1$ and $\sL_2,$ respectively, such that $I_1 \subsetneq J_1 \subsetneq I_3^*$ and $I_2 \subsetneq J_2 \subsetneq I_4^*.$
    Note that $(\ell(J_1), \ell(J_2))$ is a stitch.

If $\ell(J_1)$ is parallel with both $\ell(I_1)$ and $\ell(I_3),$ then $\ell(J_1)$ is isolated as $I_1\subsetneq J_1\subsetneq I_3^*.$ This is a contradiction by \refprop{realGap}. Hence, $\ell(J_1)$ is ultraparallel with $\ell(I_1)$ or $\ell(I_3).$

First, if $\ell(J_1)$ is ultraparallel with both $\ell(I_1)$ and $\ell(I_3),$ then we set $L_1$ to be $J_1.$ Then, $\closure{I_1}\subset L_1 \subset \closure{L_1}\subset I_3^*.$ If not, $\ell(J_1)$ is parallel with one of $\ell(I_1)$ and $\ell(I_3)$ and is ultraparallel with the other. 

If $\ell(J_1)$ and  $\ell(I_1)$ are ultraparallel and so $\ell(J_1)$ and $\ell(I_3)$ are parallel, then we set $K_1=J_1.$ If $\ell(J_1)$ and  $\ell(I_3)$ are ultraparallel and so $\ell(J_1)$ and $\ell(I_1)$ are parallel, then we set $K_1=J_1^*.$  By the choice of $K_1,$ \refprop{parallelLeaves}, and \refprop{realGap}, $K_1$ is not isolated. Then there is a $K_1$-side sequence in $\sL_1$ and we can take an element $L_1$ in $\sL_1$ so that $\closure{I_1}\subset L_1\subset \closure{L_1}\subset I_3^*.$ Similarly, we can take an element $L_2$ in $\sL_2$ so that $\closure{I_2}\subset L_2\subset \closure{L_2}\subset I_4^*.$ Then, the stitch $(\ell(L_1), \ell(L_2))$ lies on the scrap $\cS(\fF).$ This is a contradiction. Thus,  $\cS(\fF)$ is not empty.  
\end{proof}

Then, in the following proposition, we can observe that the image of each scrap in the weaving consists of regular classes. 
    
    \begin{prop}\label{Prop:scrapIsRegular}
    Let $\sV=\{\sL_1, \sL_2\}$ be a veering pair. Let $\fF=(I_1, I_2, I_3,I_4)$ be a frame. If a stitch $s$ lies on $\fF$, then $s$ is a regular stitch. Thus, the scrap $\cS(\fF)$ framed by $\fF$ consists of  regular stitches   
    \end{prop}
    
    \begin{proof}
    We write $s=(\ell_1, \ell_2).$ Assume that $s$ is a singular stitch of an asterisk $(\sG_1, \sG_2).$ Since $\ell_1$ is linked with $\ell(I_2)$ and $\ell(I_4),$ there is a unique  tip $t$ of $\sG_1$ over which $\ell(I_2)$ and $\ell(I_4)$ cross. Then, $\ell_2$ crosses over $t$. Since an element $I$ of $\sG_2$ crossing over $t$ is unique, $\ell(I)=\ell_2$. Also, as $\ell(I_2)$ and $\ell(I_4)$ cross over $t$, they lie on $I$. It contradicts to the fact that $\ell_2$ properly lies between $\ell(I_2)$ and $\ell(I_4)$.     \end{proof}

    % See \refrmk{distinctWithCorner}. Also, there is a unique element $I$ in $\sG_2$ crossing over $t.$ Hence, $s$ is $(\ell(J_1), \ell(I))$ or $(\ell(J_2), \ell(I))$ where $\tipp{t}=\{J_1, J_2\}.$
    
    % On the other hands, $\ell(I_2)$ and $\ell(I_4)$ lie on $I$ since $\ell(I_2)$ and $\ell(I_4)$ cross over $t.$  This implies that $\ell(I)$ lies in $I_2$ or $I_4.$ It is a contradiction since $s$ lies on the sectors $(I_1^*,I_2^*)$ and   $(I_3^*,I_4^*)$. 

    % $\ell(I)=\ell(I_2)$ or $\ell(I)=\ell(I_4).$  Therefore, $s$ is $(\ell(J_1), \ell(I_2)),$ $(\ell(J_2), \ell(I_2)),$ $(\ell(J_1), \ell(I_4)),$ or $(\ell(J_2), \ell(I_4)).$ However, in any case, $s$ does not lie on $\fF.$ Thus, $s$ is a regular stitch.

Conversely, we can also get the following proposition.
    
    \begin{prop}\label{Prop:frameCovering}
    Let $\sV=\{\sL_1, \sL_2\}$ be a veering pair. Every regular stitch lies on a frame. 
    \end{prop}
    \begin{proof}
    Let $s=(\ell(I), \ell(J))$ be a regular stitch of $\sV$. Without loss of generality, we may assume that the sector $(I,J)$ is counter-clockwise. As $s$ is regular, $\#(s)$ is a regular class by \refprop{eqClass2}. Also, there are two pairs $\{I_1,I_2\}$ and $\{J_1,J_2\}$ in $\sL_1$ and $\sL_2$, respectively, so that  $I_1 \subset I \subset I_2^*$, $J_1 \subset J \subset J_2^*$, and $\#(s)=\{(\ell(I_i),\ell(J_j))\,:\,i,j\in\{1,2\}\}$.
   Then, $I_i$ and $J_j$ are not isolated. Therefore, for each $i\in \{1,2\}$, we can take $M_i$ in $\sL_1$ so that $\ell(M_i)$ are linked with both $\ell(J_1)$ and $\ell(J_2)$, and $\closure{M_i}\subset I_i$. Again, for each $i\in \{1,2\}$,  we can take $N_i$ in $\sL_2$ so that $\ell(N_i)$ are linked with both $\ell(M_1)$ and $\ell(M_2)$, and $\closure{N_i} \subset J_i$. As the sector $(I,J)$ is counter-clockwise, so are  the sectors $(M_i,N_i)$. Therefore, $(M_1,N_1,M_2,N_2)$ is a frame. Thus, by the choices of $M_i$ and $N_i$,  the stitch $s$ lies on the frame.     \end{proof}

So far, we have shown the following.
\begin{lem}\label{Lem:unmarkedFrame}
Let $\sV=\{\sL_1, \sL_2\}$ be a veering pair.
\begin{enumerate}
\item Every regular stitch lies on a frame. 
\item Any stitch lying in a frame is regular.
\item The scrap framed by a frame is not empty.
\item Every frame is covered by a tetrahedron frame.
\end{enumerate}
\end{lem}

\section{Markings}\label{Sec:markingsec}

In \reflem{elliptic}, we will show that if an orientation preserving homeomorphism of the weaving is of finite order, i.e. elliptic, then it fixes a point in the weaving.  These fixed points of elliptic elements of order $\ge 3$ are represented by singular stitches. On the other hand, an elliptic element of order two may fixes a point in the weaving represented by a regular stitch. In this section, we introduce the markings to deal with such fixed points in the weaving of order two elliptic elements. We also discuss the extensions of frames with marking.

\subsection{Markings on Stitch Spaces}

\begin{defn}\label{Defn:marking}
Let $\sV=\{\sL_1,\sL_2\}$ be a veering pair. A \emph{marking} on $\sV$ is a subset $\Mrk$ of stitches subject to the following properties:
\begin{enumerate}
    \item Each stitch in $\Mrk$ is genuine;
    \item $\Mrk$ is discrete and closed in $\fS(\sV)$;
    \item For each $i=1,2$, $\eta_i|\Mrk$ is injective. 
\end{enumerate}
A stitch $s$ in $\fS(\sV)$ is said to be \emph{marked} if $s\in \Mrk$ and \emph{unmarked} otherwise. Also, the pair $(\sV,\Mrk)$ is called a \emph{marked veering pair}. 
\end{defn}

Let $\sV=\{\sL_1, \sL_2\}$ be a veering pair and $\Mrk$ a marking on $\sV$. A leaf $\ell_i$ of $\sL_i$ is said to be \emph{marked} if $\ell_i\in \eta_i(\Mrk)$ and \emph{unmarked} otherwise. Note that by the definition, every marked leaf of $\sV$ is real. A frame $\fF$ in $\sV$ is said to be \emph{marked} if there is a marked stitch lying $\fF$ and \emph{unmarked} otherwise. 

Let $\fF_I=(I_1, I_2, I_3, I_4)$ be frames in $\sV$. Now, we assume that $\fF_I$ is unmarked.  We say that $\fF_I$ is \emph{$I_i$-side full} if there is no unmarked $I_i$-side extension. The frame $\fF_I$ is called a \emph{tetrahedron frame} under $\Mrk$ if $\fF_I$ is $I_i$-side full for all $i \in \{1,2,3,4\}$.

\begin{lem}\label{Lem:finiteness}
Let $\sV=\{\sL_1, \sL_2\}$ be a veering pair with marking $\Mrk$.  Then, for each frame $\fF$ in $\sV,$ there are only finitely many marked stitches lying on $\fF$. 
\end{lem}
\begin{proof}
It follows from the fact that a marking is discrete and closed in the stitch space.\end{proof}
% We write $\fF=(I_1, I_2, I_3, I_4)$ and $I_i=\opi{u_i}{v_i}$ for all $i\in \ZZ_4.$ 
% We consider the sets $$U_1=\{\{a,b\}\in \cM :a\in \opi{v_1}{u_3} \text{ and } b\in \opi{v_3}{u_1}\}$$
% and 
% $$U_2=\{\{a,b\}\in \cM: a\in \opi{v_4}{u_2}, \text{ and }b\in \opi{v_2}{u_4}\}.$$
% The sets $U_1$ and $U_2$  are open in $\cM.$
% The closures of $U_1$ and $U_2$ are 
% $$\closure{U}_1=\{\{a,b\}\in \cM :a\in \cldi{v_1}{u_3} \text{ and } b\in \cldi{v_3}{u_1}\}$$
% and 
% $$\closure{U}_2=\{\{a,b\}\in \cM: a\in \cldi{v_4}{u_2}, \text{ and }b\in \cldi{v_2}{u_4}\},$$ respectively.
% Since $\closure{U}_1$ and $\closure{U}_2$ are compact in $\cM,$ $\closure{U_1}\times \closure{U}_2$ is compact in $\cM \times \cM.$ As the pair spaces $\fP(\sV)$ is closed in $\cM \times \cM$ by the strongly transversality of $\sV,$  $\closure{U}=\fP(\sV)\cap (\closure{U_1}\times \closure{U}_2)$ is compact in $\fP(\sV)$ and so $\epsilon_2^{-1}(\closure{U})$ is compact in $\fS(\sV).$ Then, $\epsilon_2^{-1}(\closure{U})$ has an only finite number of marked stitches since $\Mrk$ is discrete and closed in $\fS(\sV).$ As  the scrap $\cS(\fF)$  of $\fF$ is contained in $\epsilon_2^{-1}(U_1\times U_2),$ $\cS(\fF)$ has an only finite number of marked stitches.

\subsection{Full Extensions}

Now, we study the full extensions of unmarked frame.

\begin{lem}\label{Lem:extension}
Let $\sV=\{\sL_1, \sL_2\}$ be a veering pair with marking $\Mrk$. Assume that $\fF=(I_1, I_2, I_3, I_4)$ is an unmarked frame. If $\cI(c_{i-1}(\fF),c_i(\fF))$ has no marked stitch and no singular stitch, then there is an unmarked $I_i$-side extension of $\fF.$  Thus, if $\fF$ is $I_i$-side full for some $i \in \{1,2,3,4\}$, then $\cI(c_{i-1}(\fF),c_i(\fF))$ contains a unique marked stitch or at least one singular stitch (cyclically indexed).
\end{lem}
\begin{proof}
Assume that $i=1$. We write $I_i=\opi{u_i}{v_i}$ for all $i\in \{1,2,3,4\}$ By \reflem{maxStitch}, $\fF$ is not a $I_1$-side maximal. Therefore, there is a $I_1$-side extension of $\fF$. This implies that the stem $S=\stem{\opi{v_4}{u_2}}{I_1}$ in $\sL_1$ has at least one element that is not $I_1$. Let $E$ be the end of the stem $S$. Since $\sL_1$ and $\sL_2$ are strongly transverse, $E \in S$ and $\cldi{v_4}{u_2} \subset E \subsetneq I_1.$ 

If $I_1$ is isolated, then there is a non-leaf gap $\sG$ containing $I_1^*.$ Then, since $E\subsetneq I_1$ there is an element $I_1'$ in $\sG$ containing $E.$ Then, $I_1'\in S$ and so $\fF'=(I_1', I_2, I_3, I_4)$ is an $I_1$-side extension of $\fF.$ Since $\ell(I_1)$ and $\ell(I_1')$  are linked with $\ell(I_2)$ and $\ell(I_4),$ by \refprop{twoGap}, $\{I_1', I_1^*\}$ is a tip pair $\tipp{t}$ for some $t \in v(\ell(I_1))$ and so both $\ell(I_2)$ and $\ell(I_4)$ cross over $t.$ For simplicity, we may assume that $t=v_1.$ Then, $I_4^*$ and $I_2$ cross over $v_1.$ 

Suppose that there is a marked stitch $s=(\ell_1, \ell_2)$ lying on $\fF'$. Then, $\ell_1$ properly lies between $\ell(I_1')$ and $\ell(I_3).$ This implies that  $\ell_1$ properly lies between $\ell(I_1)$ and $\ell(I_3)$ since $\ell_1$ is real. Hence, $s$ lies on $\fF$. It is a contradiction. Therefore, there is no marked stitch in $\fF'.$ Thus, $\fF'$ is a unmarked $I_1$-side extension of $\fF.$

 Now, we consider the case where $I_1$ is not isolated. The quadruple $\fF_E=(E,I_2, I_3,I_4)$ is an $I_1$-side extension of $\fF$ which is $E$-side maximal. If $\fF_E$ is unmarked, then we are done. Assume that $\fF_E$ is marked. By \reflem{finiteness}, $\cS(\fF_E)$ has an only finite number of marked stitches. Therefore, since $I_1$ is not isolated, we can take $I_1'$ in $S$ so that $I_1'\subsetneq I_1$ and the frame $\fF'=(I_1', I_2, I_3, I_4)$ is unmarked. Then, $\fF'$ is an unmarked $I_1$-side extension of $\fF.$ 
\end{proof}

The combination of the two following lemmas is the converse of \reflem{extension} in nature.

\begin{lem}\label{Lem:markedSide}
Let $\sV=\{\sL_1, \sL_2\}$ be a veering pair and $\Mrk$ a marking on $\sV.$ Let $\fF=(I_1, I_2, I_3, I_4)$ be an unmarked frame. If $\cI(c_{i-1}(\fF),c_i(\fF))$ has a marked stitch $s=(\ell_1, \ell_2),$ then $\fF$ is not $I_i$-side maximal but it is $I_i$-side full.
\end{lem}
\begin{proof}
Without loss, we may assume that $\cI(c_4(\fF), c_1(\fF))$ has a marked stitch $s=(\ell_1, \ell_2).$ Then, $\ell(I_1)=\ell_1.$ If $\fF$ is $I_1$-side maximal, then $I_1$ is isolated and so $\ell(I_1)$ is a boundary leaf of a non-leaf gap.  However, it is a contradiction as $\ell_1$ is real. See \refprop{realGap}. Therefore, $\fF$ is not $I_1$-side maximal.

Suppose that there is an $I_1$-side extension $\fF'=(I_1', I_2, I_3, I_4)$ of $\fF.$ Then, $I_1'\subsetneq I_1$ and by \reflem{quadrachotomy}, $\ell(I_1')$ and $\ell(I_1)$ are ultraparallel. Therefore, $\ell(I_1)$ properly lies between $\ell(I_1')$ and $\ell(I_3).$ As $s\in \cI(c_4(\fF), c_1(\fF)),$ by \reflem{quadrachotomy}, $\ell_2$ properly lies between $\ell(I_2)$ and $\ell(I_4).$ Hence, $s$ lies on $\fF'$ and so $\fF'$ is marked. Thus, there is no unmarked $I_1$-side extension of $\fF$ and $\fF$ is $I_1$-side full.   
\end{proof}

\begin{lem}\label{Lem:singularSide}
Let $\sV=\{\sL_1,\sL_2\}$ be a veering pair. Let $\fF=(I_1, I_2, I_3, I_4)$ be a frame. If $\cI(c_{i-1}(\fF),c_i(\fF))$ has a singular stitch $s=(\ell_1, \ell_2)$, then $\fF$ is $I_i$-side maximal. 
\end{lem}
\begin{proof}
Without loss of generality, we may assume that $i=1$. Assume that $\fF$ is not $I_1$-side maximal. Then, there is an $I_1$-side extension $\fF'=(J_1, I_2, I_3, I_4)$ of $\fF.$ Note that $J_1\subsetneq I_1$. Then, if $\ell(J_1)$ and $\ell(I_1)$ are ultraparallel, then $s$ properly lies between $\ell(J_1)$ and $\ell(I_3)$. Therefore, $s$ lies on $\fF'$. However, by \reflem{unmarkedFrame}, there is no singular stitch lying on $\fF'$. Therefore, $\ell(J_1)$ and $\ell(I_1)$ are parallel, namely, $v(\ell(J_1))\cap v(\ell(I_1))=\{t\}$ for some $t\in S^1$. Then, by \reflem{quadrachotomy}, $t$ is a tip of a non-leaf gap $\sG_1$ of $\sL_1$. Hence, $\tipp{t}=\{J_1, I_1^*\}$ and $I_1^*\in \sG_1$. 

Since $\ell_2$ properly lies between $\ell(I_2)$ and $\ell(I_4)$, $\ell_2$ is linked with $\ell(J_1)$. Therefore, by \refprop{twoGap}, $\ell_2$ crosses over $t$. Also note that by \refprop{twoGap}, $\ell(I_2)$ and $\ell(I_4)$ cross over $t.$ 

Let $\sG_2$ be the non-leaf gap of $\sL_2$ interleaving with $\sG_1$. There is a unique element $K$ in $\sG_2$ crossing over $t$. Since $s$  is a singular stitch of the asterisk $(\sG_1, \sG_2)$ and $\ell_2$ crosses over $t$, $\ell(K)=\ell_2$. On the other hands, $\ell(I_2)$ and $\ell(I_4)$ properly lie on $K$ since $\ell(I_2)$ and $\ell(I_4)$ cross over $t$ and are ultraparallel with $\ell_2$. This implies that $\ell_2$ does not lie between $\ell(I_2)$ and $\ell(I_4)$,  a contradiction. Thus, $\fF$ is $I_1$-side maximal. 
\end{proof}

Now, we show the existence of the full extensions.

\begin{lem}\label{Lem:fullExtension}
Let $\sV=\{\sL_1, \sL_2\}$ be a veering pair and $\Mrk$ a marking on $\sV.$ Assume that $\fF=(I_1, I_2, I_3, I_4)$ is an unmarked frame. If $\fF$ is not $I_i$-side full for some $i\in \{1,2,3,4\}$, then there is an unmarked $I_i$-side extension of $\fF$ that is $i^{th}$-side full. 
\end{lem}
\begin{proof}
Without loss of generality, we may assume that $i=1.$ We write $I_i=\opi{u_i}{v_i}$ for all $i\in\{1,2,3,4\}$. Now, we consider the stem $S=\stem{\opi{v_4}{u_2}}{I_1}$ in $\sL_1$. The stem $S$ has the end $E$ which contains $\opi{v_4}{u_2}$ and so $E\in S$. Note that $E$ is the minimal element of $S$.

Let $\fF'$ be the quadruple $(E, I_2, I_3, I_4)$. As $\sL_1$ and $\sL_2$ are strongly transverse, $\cldi{v_4}{u_2}\subset E$ and so $\fF'$ is a frame in $\sV$. By construction, $\fF'$ covers $\fF$ and since $\fF$ is not first-side full, $\fF'$ is a first-side extension of $\fF.$ By the minimality of $E,$ $\fF'$ is first-side maximal. Hence,  if $\fF'$ is unmarked, we are done. 

Assume that $\fF'$ is marked. By \reflem{finiteness}, there are only finitely many marked stitches in $\cS(\fF')$, that is, $\mrk:=\Mrk \cap \cS(\fF')$ is finite. We define the set $$R:=\{I\in S\,:\, \ell(I)\in \eta_1(\mrk)\}.$$
Then, $R$ is a finite subset of $S$ and there is a maximal element $M$ in $R$. 

Let $\fF''$ be the frame $(M, I_2, I_3, I_4)$ which is an $I_1$-side extension of $\fF$. By the construction, $\fF''$ is unmarked and $\cI(c_4(\fF''),c_1(\fF''))$ has a marked stitch. By \reflem{markedSide}, $\fF''$ is first-side full. Thus, $\fF''$ is a frame that we want. 
\end{proof}

Finally, we extend \reflem{unmarkedFrame} to the marked case.

\begin{lem}\label{Lem:rectangleisintetra}
Let $\sV=\{\sL_1, \sL_2\}$ be a veering pair and $\Mrk$ a marking on $\sV.$ Assume that $\fF=(I_1, I_2, I_3, I_4)$ is an unmarked frame. Then there is an tetrahedron frame under $\Mrk$ covering $\fF.$
\end{lem}
\begin{proof}
If $\fF$ is first-side full, then we set $\fF_1=\fF$. If not, then by \reflem{fullExtension}, there is an unmarked first-side extension $\fF'$ of $\fF$ that is first-side full. Then, we set $\fF_1=\fF'$. By \reflem{extension}, $\cI(c_4(\fF_1), c_1(\fF_1))$ has a marked or singular stitch.

Then, if $\fF_1$ is second-side full, then we set $\fF_2=\fF_1$. Otherwise, by \reflem{fullExtension}, there is an unmarked second-side extension $\fF_1'$ of $\fF_1$ that is second-side full. Then, we set $\fF_2=\fF_1'$. By \reflem{extension}, $\cI(c_1(\fF_2), c_2(\fF_2))$ has a marked or singular stitch. Also, as  $$\cI(c_4(\fF_1), c_1(\fF_1))\subset \cI(c_4(\fF_2), c_1(\fF_2)),$$ $\cI(c_4(\fF_2), c_1(\fF_2))$ has a marked or singular stitch.

Next, if $\fF_2$ is third-side full, then we set $\fF_3=\fF_2$. Otherwise, by \reflem{fullExtension},  there is an unmarked third-side extension $\fF_2'$ of $\fF_2$ that is third-side full. Then, we set $\fF_3=\fF_2'$. By \reflem{extension}, $\cI(c_2(\fF_3), c_3(\fF_3))$ has a marked or singular stitch. Also, both $\cI(c_4(\fF_3), c_1(\fF_3))$ and  $\cI(c_1(\fF_3), c_2(\fF_3))$ have marked or singular stitches. Therefore, by \reflem{markedSide} and \reflem{singularSide}, $\fF_4$ is $i^{th}$-side full for all $i\in\{1,2,3,4\}$. Thus, $\fF_4$ is an tetrahedron frame under $\Mrk$ covering $\fF.$

Finally, if $\fF_3$ is fourth-side full, then we set $\fF_4=\fF_3$. Otherwise, by \reflem{fullExtension},  there is an unmarked fourth-side extension $\fF_3'$ of $\fF_3$ that is fourth-side full. Then, we set $\fF_4=\fF_3'$. By \reflem{extension}, $\cI(c_3(\fF_4), c_4(\fF_4))$ has a marked or singular stitch. Moreover, $\cI(c_{i-1}(\fF_4), c_i(\fF_4))$ has a marked or singular stitch for all $i\in \{1,2,3,4\}$ (cyclically indexed).
\end{proof}

\begin{lem}\label{Lem:unmakredFrameCover}
Let $\sV=\{\sL_1, \sL_2\}$ be a veering pair and $\Mrk$ a marking on $\sV.$
Every unmarked regular stitch lies on an unmarked frame. 
\end{lem}
\begin{proof}
Let $s=(\ell_1, \ell_2)$ be an unmarked regular stitch. By \reflem{unmarkedFrame}, there is a frame $\fF=(I_1, I_2, I_3, I_4)$ on which $s$ lies. If $\fF$ is unmarked, then we are done. Assume that $\fF$ is marked. By \reflem{finiteness}, the scrap $\cS(\fF)$ has only finitely many marked stitches, that is, $\mrk:=\cS(\fF)\cap \Mrk$ is a finite subset of $\cS(\fF).$

Now, we write $I_i=\opi{u_i}{v_i}$ for all $i\in \{1,2,3,4\}$, and $\ell_1=\{J_1, J_3\}$ and $\ell_2=\{J_2, J_4\}$ so that $I_i\subset J_i$ for all $i\in \{1,2,3,4\}$. Then, we define
$$R_1:=\{K\in \sL_1: K \subsetneq J_1 \text{ and } \ell(K)\in \eta_1(\mrk) \}$$
$$R_2:=\{K\in \sL_2: K \subsetneq J_2 \text{ and } \ell(K)\in \eta_2(\mrk) \}$$
$$R_3:=\{K\in \sL_1: K \subsetneq J_3 \text{ and } \ell(K)\in \eta_1(\mrk) \},$$ and 
$$R_4:=\{K\in \sL_2: K \subsetneq J_4 \text{ and } \ell(K)\in \eta_2(\mrk) \}.$$
Then, $R_i$ are finite as $\mrk$ is finite.
Note that $R_1$ is totally ordered as $R_1$ is a finite subset of the stem $\stem{I_1}{J_1}$ in $\sL_1.$ Likewise, $R_i$ are finite and totally ordered. Now, for each $i\in \{1,2,3,4\}$, we define $M_i$ to be  the maximal element of $R_i$ if $R_i$ is not empty and to be $I_i$ otherwise. For each $i\in\{1,2,3,4\}$, by \reflem{quadrachotomy}, $$I_i\subset M_i\subset \closure{M_i}\subset J_i.$$ This implies that the quadruple $\fF'=(M_1, M_2, M_3, M_4)$ is a frame on which $s$ lies. Moreover, by the choice of $M_i,$ $\fF'$ is unmarked. 
\end{proof}

\subsection{Scraps on Weavings}

The following lemma is the key lemma  in the next section where we prove that weavings are transversely foliated. The following says that the scarps give rise to the transversely foliated charts for weavings. 

\begin{lem}\label{Lem:weavingRect}
    Let $\sV=\{\sL_1, \sL_2\}$ be a veering pair and $\fF$ a frame in $\sV$. Then, we can take a homeomorphism $\rho_\fF$ from $[0,1]^2$ to $\#(\closure{\cS}(\fF))$ so that the following hold. 
\begin{itemize}
    \item For each $s\in [0,1],$ $\rho_\fF(\{s\}\times [0,1])\subset \#(\overt(\ell))$ for some leaf $\ell$ of $\sL_1.$
    \item For each $t\in [0,1],$ $\rho_\fF( [0,1]\times \{t\})\subset \#(\ominus(\ell))$ for some leaf $\ell$ of $\sL_2.$
    \item $$\rho_\fF(\{0\}\times [0,1])=\#(\closure{\cI}(c_1(\fF),c_4(\fF))),$$
    $$\rho_\fF(\{1\}\times [0,1])=\#(\closure{\cI}(c_2(\fF),c_3(\fF))),$$ 
    $$\rho_\fF([0,1]\times\{0\})=\#(\closure{\cI}(c_1(\fF),c_2(\fF))),$$
    and 
    $$\rho_\fF ([0,1]\times \{1\})=\#(\closure{\cI}(c_4(\fF),c_3(\fF))).$$
\end{itemize}
Therefore, $\#^{-1}(\rho_\fF((0,1)^2))=\cS(\fF).$  
\end{lem}
\begin{proof}
We define a map $\sigma_\fF$ from $\fC:=\closure{\cI}(c_1(\fF), c_2(\fF))\times \closure{\cI}(c_1(\fF),c_4(\fF)))$ to $\closure{\cS}(\fF)$ by $\sigma_\fF(s_1,s_2)=(\eta_1(s_1),\eta_2(s_2)).$ By \reflem{edgeMap} and the definition of the closure of the scrap, $\sigma_\fF$ is a homeomorphism.

Let $\fF=(I_1,I_2,I_3,I_4)$. We claim that for any elements $(s_1,t_1)$ and $(s_2,t_2)$ in $\fC$,
$s_1 \sim_\omega s_2$ and $t_1 \sim_\omega t_2$ if and only if $\sigma_\fF(s_1,t_1) \sim_\omega \sigma_\fF(s_2,t_2)$. The case where $(s_1,t_1)=(s_2,t_2)$ is obvious. 

First, we consider the case where $s_1\neq s_2$ and  $t_1=t_2$.
Assume that $s_1 \sim_\omega s_2$ and $t_1 \sim_\omega t_2$. As $s_1$ and $s_2$ are in the same thread  and $s_1 \sim_\omega s_2$, by \refprop{eqClass2}, $\cI(s_1,s_2)=\emptyset$. By \refprop{gapBtwStitches}, $\ell(I_2)$ crosses over a tip $t$ in $\E{\sL_1}$ such that  $\{s_1,s_2\}=\{(\ell(J),\ell(I_2))\,:\,J\in \tipp{t}\}$. By \reflem{edgeMap}, $\ell(I_4)$ also crosses over $t$. Hence, $\eta_2(t_1)$ crosses over $t$ as $\eta_2(t_1)$ properly lies between $\ell(I_2)$ and $\ell(I_4)$. Therefore, by \refrmk{eqDef}, $\sigma_\fF(s_1,t_1)=\sigma_\fF(s_2,t_2)$.

Conversely, if $\sigma_\fF(s_1,t_1)\sim_\omega \sigma_\fF(s_2,t_2)$, then  
$\mu_1=\sigma_\fF(s_1,t_1)$ and $ \mu_2=\sigma_\fF(s_2,t_2)$ are stitches in the same thread $\ominus(\eta_2(t_1))$ with $\cI(\mu_1,\mu_2)=\emptyset$. By \refprop{gapBtwStitches}, $\eta_2(t_1)$ crosses over a tip $t$ in $\E{\sL_1}$ such that $\{\mu_1,\mu_2\}=\{(\ell(J),\eta_2(t_1)): J\in\tipp{t}\}$. Then, $\ell(I_2)$ also crosses over $t$. Hence, as $\mu_i=\eta_1(s_i)$,  by \refrmk{eqDef}, $s_1\sim_\omega s_2$.

Likewise, we can show the claim in the cases where $s_1=s_2$ and $t_1\neq t_2$, or $s_1\neq s_2$ and $t_1\neq t_2$. Thus, the claim follows.

Now, we consider the continuous map $\# \circ \sigma_\fF$ from $\fC$ to $\#(\closure{\cS}(\fF)).$ We define a continuous map $\#^2$ from $\closure{\cI}(c_1(\fF),c_2(\fF))\times \closure{\cI}(c_1(\fF),c_4(\fF))$ to $\#(\closure{\cI}(c_1(\fF),c_2(\fF)))\times \#(\closure{\cI}(c_1(\fF),c_4(\fF)))$ by $$\#^2(s_1,s_2)=(\#(s_1),\#(s_2)).$$ Then, by the previous claim, there is a unique homeomorphism $$\delta_\fF: \#(\closure{\cI}(c_1(\fF),c_2(\fF)))\times \#(\closure{\cI}(c_1(\fF),c_4(\fF)))\to \#(\closure{\cS}(\fF))$$ that makes the following diagram commute:
\begin{center}
\begin{tikzcd}
\closure{\cI}(c_1(\fF),c_2(\fF))\times \closure{\cI}(c_1(\fF),c_4(\fF)) \arrow[dd,"\#^2"] \arrow[rr,"\sigma_\fF"]
\arrow[rrdd,"\#\circ \sigma_\fF"]&&\closure{\cS}(\fF) \arrow[dd,"\#"]\\
&&\\
\#(\closure{\cI}(c_1(\fF),c_2(\fF)))\times \#(\closure{\cI}(c_1(\fF),c_4(\fF)))\arrow[rr,dashed,"\exists !\delta_\fF"]&&\#(\closure{\cS}(\fF)) 
\end{tikzcd}
\end{center}
Note that by \refprop{scrapIsNonempty} and \refrmk{distinctWithCorner}, $\cI(c_i(\fF),c_{i+1}(\fF))\neq \emptyset$ for all $i\in \{1,2,3,4\}$ (cyclically indexed). By \reflem{subarc}, we can take a homeomorphism $\alpha$ from $[0,1]^2$ to $\#(\closure{\cI}(c_1(\fF),c_2(\fF)))\times \#(\closure{\cI}(c_1(\fF),c_4(\fF)))$ so that $\rho_\fF=\delta_\fF \circ \alpha$ satisfies the properties that we wanted. Then, it follows from construction that  $\#^{-1}(\rho_\fF((0,1)^2))=\cS(\fF)$. \end{proof}

\begin{rmk}\label{Rmk:regularNbhd}
Let $\fF$ be a frame. Let $s$ be a stitch lying on $\fF$. Then, by \refrmk{distinctWithCorner}, we can see that $\#(s)$ is contained in $\rho_\fF((0,1)^2)$ where $\rho_\fF$ is from \reflem{weavingRect}. Thus, $\cS(\fF)=\#^{-1}(\rho_\fF((0,1)^2))$ and so $\#(\cS(\fF))$ is homeomorphic to $(0,1)^2.$ 
\end{rmk}

\section{Transverse Foliations on  Weavings}\label{Sec:foliationsec}

In this section, we show that weaving are transversely foliated. First of all, we review the basic terminology for the singular foliations.

\subsection{Singular Foliations}\label{Sec:defnOfFoliation} 

Let $P_2$ be the open set $$\{z\in \CC\,:\, \max\{|\re z|,|\im z |\}<1 \}$$  of the complex plane $\CC$. We define transverse foliations $\sF^2$ and $\sF_2$ on $P_2$ as the sets of vertical lines and horizontal lines, respectively. 

For each integer $k$ with $k>1$, the map on $\CC$ defined by $z\mapsto z^k$ is denoted by $\varphi_k$. Then, we set $P_1=\varphi_2(P_2)$. Observe that $P_1$ is also an open neighborhood of the origin. Now, we define
decompositions $\sF^1$ and $\sF_1$ for $P_1$ as $\sF^1=\{\varphi_2(l):l\in \sF^2\}$ and
$\sF_1=\{\varphi_2(l):l\in \sF_2\}$. 

% Observe that $\varphi_2(\{z:\re z=0\})=(-1,0]$ and $\varphi_2(\{z: \im z =0\})=[0,1)$ and that for any $s$ and $t$ in $(-1,1)-\{0\},$   $\varphi_2(\{z:\re z= s\})$ and $\varphi_2(\{z: \im z = t\})$  are homeomorphic to $\RR.$ Therefore, $\sF^1$ and $\sF_1$ are not foliations but $P_1-\{0\}$ has a pair of transverse foliations induced from $\sF^1$ and $\sF_1.$ 
For each element $p$ of  $\sF^1,$ we say that $p$ is a \emph{plaque}. In particular, we say $p$ to be \emph{regular} if $p$ is homeomorphic to  $\RR.$ Otherwise, $p$ is \emph{singular}. Similarly, for each element $t$ of  $\sF_1,$ we say that $t$ is a \emph{transversal}. In particular, $t$ is \emph{regular} if $t$ is homeomorphic to  $\RR.$ Otherwise, $t$ is \emph{singular}.

Fix an integer $k$ with $k>1.$ We define $P_k$ as $P_k=\varphi_k^{-1}(P_1).$
%Also, we define decompositions  $\sF^k$ and $\sF_k$ on $P_k$ as follows. 
% Observe that if $l$ is a regular plaque or transversal in $\sF^1$  or $\sF_1,$ respectively, then $\varphi_k^{-1}(l)$ has exactly $k$ connected components which are homeomorphic to $\RR$ and If $l$ is a singular plaque or transversal in $\sF^1$ or $\sF_1,$ respectively, then $\varphi_k^{-1}(l)$ is connected and is homeomorphic to the wedge sum of $k$ copies of $[0,1)$ at $0.$
For each $l$ in $\sF^1,$ each connected component $c$ of $\varphi_k^{-1}(l)$ is called a \emph{plaque} in $P_k.$ In particular, $c$ is said to be \emph{regular} if $k=2$ or $l$ is regular. Otherwise, $c$ is singular. Then, we define a foliation $\sF^k$ on $P_k$ to be the set of all plaques in $P_k$. Likewise, we can define a singular/regular transversal in $P_k$ and $\cF_k$.

% Likewise, for each $l$ in $\sF_1,$ we call each connected component $c$ of $\varphi_k^{-1}(l)$ a \emph{transversal} in $P_k.$ Also, we say $c$ to be \emph{regular} if $l$ is regular. Otherwise, $c$ is \emph{singular}. Then, we define $\sF_k$ to be the set of all transversal in $P_k.$

Let $S$ be a surface. We define a \emph{singular foliation} $\sF$ on the surface $S$ to be a partition of $S,$ whose elements are called the \emph{leaves} of $\sF,$ satisfying the following:
\begin{itemize}
    \item For each point $x$ in $S,$ there is an open neighborhood $U$ of $x$ and an embedding $\phi_U$ from $U$ to $\CC$ such that $\phi_U(U)=P_k$ for some integer $k$ with $k>1,$ $\phi_U(x)=0$ and for any plaque $p$ in $\sF^k,$ $\phi_U^{-1}(p)$ is a connected component of  $U\cap l$ for some leaf  $l$ in $\sF.$ 
\end{itemize}
We call $(U,\phi_U)$ a \emph{foliated chart} at $x$. A foliated chart $(U,\phi_U)$ is  \emph{regular} if $k=2$. Otherwise, $(U,\phi_U)$ is  \emph{singular}.  A point $x$ is a \emph{regular point} of $\sF$ if there is a regular foliated chart at $x$. Non regular points are called \emph{singular}. Note that if $\sF$ has no singular point, then $\sF$ is just a foliation. 

Assume  $\sF_1$ and $\sF_2$ are singular foliations on $S.$ Then, we say that $\sF_1$ and $\sF_2$ are \emph{transverse} if for each $x$ in $S,$ there is an open neighborhood $U$ of $x$ and an embedding $\phi_U$ from $U$ to $\CC$ satisfying the following:
\begin{itemize}
    \item $\phi_U(U)=P_k$ and $\phi_U(x)=0$ for some integer $k$ with $k>1.$
    \item For any plaque $p$ in $\sF^k,$ $\phi_U^{-1}(p)$ is a connected component of $U\cap l$ for some leaf $l$ in $\sF_1.$
    \item For any transversal $t$ in $\sF_k,$ $\phi_U^{-1}(t)$ is a connected component of $U\cap l$ for some leaf $l$ in $\sF_2.$
\end{itemize}
We call $(U,\phi_U)$ a \emph{transversely foliated chart} at $x.$ Moreover, $(U,\phi_U)$ is \emph{regular} if $k=2.$ Otherwise, $(U,\phi_U)$ is  \emph{singular.} 

\subsection{Polygonal Frames}
In this section, we generalize the notion of frames to deal with the singular points in the weavings which are exactly singular classes. 

From now on, we sometimes index items over $\ZZ_n$ to emphasize that they are cyclically indexed. 

Let $\sV=\{\sL_1, \sL_2\}$ be a veering pair.  Let $(\sG_1,\sG_2)$ be an asterisk of ideal polygons in $\sV$. We write $\sG_1=\{\opi{s_k}{s_{k+1}}:k\in\ZZ_n\}$
and $\sG_2=\{\opi{t_k}{t_{k+1}}:k\in\ZZ_n\}$
so that $t_k\in \opi{s_k}{s_{k+1}}$ for all $k\in\ZZ_n$.
A \emph{polygonal frame} $\fP$ at the asterisk $(\sG_1,\sG_2)$ is a pair $(\seqc{I}{k}{n},\seqc{J}{k}{n})$ of finite sequences of good intervals satisfying the following:
\begin{itemize}
    \item $I_k\in \sL_1$ and $J_k\in \sL_2$ for all $k\in \ZZ_n.$
    \item For each $k\in \ZZ_n,$ $t_k\in I_k\subsetneq \opi{s_k}{s_{k+1}}$ and $s_{k+1}\in J_k\subsetneq \opi{t_k}{t_{k+1}}$.
    \item For each $k\in \ZZ_n,$ $\ell(I_k)$ is linked with both $\ell(J_{k-1})$ and $\ell(J_k).$
\end{itemize}

% \begin{rmk}
% Note that for each $k\in \ZZ_n,$ $(\ell(I_k),\ell(J_k))$ and $(\ell(I_k),\ell(J_{k-1}))$ are stitches which lie on the allowed sectors $(\opi{s_k}{s_{k+1}},\opi{t_k}{t_{k+1}})$ and $(\opi{s_k}{s_{k+1}},\opi{t_{k-1}}{t_k}),$ respectively. Hence, for each $k\in \ZZ_n,$ $(\opi{s_{k+1}}{s_k},J_{k-1},I_k,J_k)$ and $(I_k,J_k,I_{k+1},\opi{t_{k+1}}{t_k})$ are  frames. Also, for each $k\in \ZZ_n,$ $$(I_k, J_k, \opi{s_{k+1}}{s_k},\opi{t_{k+1}}{t_k})$$ and $$(I_k, \opi{t_k}{t_{k-1}},\opi{s_{k+1}}{s_k},J_{k-1})$$
% are frames. 
% \end{rmk}

For each $\opi{s_k}{s_{k+1}}\in \sG_1,$ observe that  $(\opi{s_{k+1}}{s_k},J_{k-1},I_k,J_k)$ is a frame. Hence, it is called the \emph{$\opi{s_k}{s_{k+1}}$-side frame} of $\fP$ and we denote it by $\fF(\fP, \opi{s_k}{s_{k+1}}).$ Similarly, for each $\opi{t_k}{t_{k+1}}\in \sG_2,$  $(I_k,J_k,I_{k+1},\opi{t_{k+1}}{t_k})$ is the \emph{$\opi{t_k}{t_{k+1}}$-side frame} of $\fP$ and we denote it by $\fF(\fP, \opi{t_k}{t_{k+1}}).$

We define the \emph{polygonal scrap} $\cS(\fP)$ framed by $\fP$ to be
$$\cS(\fP)=\bigcup_{I\in \sG_1}\cS(\fF(\fP, I)) \cup \bigcup_{J\in \sG_2}\cS(\fF(\fP, J))\cup w(\sG_1,\sG_2).$$ Also, the \emph{closure} $\closure{\cS}(\fP)$ of $\cS(\fP)$ is defined as the set $\bigcup_{I\in \sG_1}\closure{\cS}(\fF(\fP, I))$ which is equal to $\bigcup_{J\in \sG_2}\closure{\cS}(\fF(\fP, J))$.

The following lemma is a preliminary result to construct singular foliated charts. 

\begin{lem}\label{Lem:polyFrameCovering}
Let $\sV=\{\sL_1, \sL_2\}$ be a veering pair and let $(\sG_1,\sG_2)$ be an asterisk of ideal polygons of $\sV$. Then, there is a polygonal frame $\fP$ at the asterisk $(\sG_1,\sG_2).$
\end{lem} 
\begin{proof}
We write $\sG_1=\{\opi{s_i}{s_{i+1}}: i\in \ZZ_n\}$ and $\sG_2=\{\opi{t_i}{t_{i+1}}:i\in \ZZ_n\}$ so that $t_i\in \opi{s_i}{s_{i+1}}$ for all $i\in \ZZ_n.$ By \refprop{realGap}, every element of $\sG_1$ or $\sG_2$ is not isolated in $\sL_1$ or $\sL_2,$ respectively. Again, for each $i\in \ZZ_n,$ we take an element $I_i=\opi{u_i}{v_i}$ in $\sL_1$ so that $$t_i\in I_i \subset \closure{I_i}\subset \opi{s_i}{s_{i+1}}.$$ Likewise, for each $i\in \ZZ_n,$ we can take an element $J_i=\opi{x_i}{y_i}$ in $\sL_2$ so that $$\cldi{v_i}{u_{i+1}}\subset J_i\subset \closure{J_i}\subset \opi{t_i}{t_{i+1}}.$$
Observe that $\ell(I_i)$ is linked with $\ell(J_{i-1})$ and $\ell(J_i)$ for all $i\in \ZZ_n.$ Thus, $(\seqc{I}{i}{n},\seqc{J}{i}{n})$ is a polygonal frame at $(\sG_1, \sG_2).$
\end{proof}

\subsection{Weavings with Marking}
Let $\Mrk$ be a marking on a veering pair $\sV$. By \refprop{eqClass2}, we can think of each marked stitch as an element in $\closure{\fW}(\sV)$. From now on, we also consider $\Mrk$ as a subset of $\closure{\fW}(\sV)$. We refer to each element of $\Mrk$ in $\closure{\fW}(\sV)$ as a \emph{marked class} and we say that an element $w$ of $\closure{\fW}(\sV)$ is a \emph{cone class} if $w$ is a marked class or singular class. 

The \emph{regular weaving} of a marked veering pair $(\sV,\Mrk)$ is defined as 
\[
\fW^\circ(\sV,\Mrk):=\closure{\fW}(\sV) \setminus \{\text{cusp and cone classes}\}.
\]
Note that  $\Mrk \subset \fW^\circ(\sV)$ and $$\fW^\circ(\sV, \Mrk)=\fW^\circ(\sV)\setminus \Mrk.$$

\subsection{Foliations on Weavings}
Let $\sV=\{\sL_1, \sL_2\}$ be a veering pair. We define partitions $\closure{\cF}_1(\sV)$ and $\closure{\cF}_2(\sV)$ on $\closure{
\fW}(\sV)$ induced from $\sL_1$ and $\sL_2$, respectively as follows. Choose $i\in \{1,2\}$. %For each element $l$ in $\closure{\cF}_i(\sV)$, \jung{??? $\closure{\cF}_i(\sV)$ is not defined yet?} there is a real gap $\sG$ in $\sL_i$ such that 
The partition $\closure{\cF}_i(\sV)$ is the collection of subsets of $\closure{\fW}(\sV)$ that are of the form
\[
\#\{s\in \fS(\sV)\,:\, \text{$\eta_i(s)$ is a boundary leaf of some real gap $\sG$ of $\sL_i$}\}=\#\bigcup_{I\in \sG} \oslash(\ell(I)).    
\] Note that $\closure{\cF}
_i(\sV)$ is a well defined partition by  \refprop{eqClass2}.

Then, we define partitions $\cF_1(\sV)$ and $\cF_2(\sV)$ on $\fW(\sV)$ as follows. Choose $i\in \{1,2\}$. 
The partition $\cF_i(\sV)$ is the collection of the subsets
of $\fW(\sV)$ each of which is  a connected component of  $l\cap \fW(\sV)$ for some $l\in \closure{\cF}_i(\sV)$. Note that if $l$ has a cusp class $w$, then each associated element in $\cF_i(\sV)$ is a connected component of $l\setminus\{w\}$.

Now, we consider a marked veering pair $(\sV,\Mrk)$. Similarly, we define partitions $\cF_1^\circ(\sV, \Mrk)$ and $\cF_2^\circ(\sV,\Mrk)$ on $\fW^\circ(\sV, \Mrk)$ induced by $\sL_1$ and $\sL_2$, respectively, as follows. For each $i\in \{1,2\}$, $\cF_i^\circ(\sV,\Mrk)$ is the collection of connected components of subsets $l\cap \fW^\circ(\sV,\Mrk)$, $l\in \closure{\cF_i}(\sV)$. Note that if $l$ has a cusp or cone class $w,$ then each associated element in $\cF_i^\circ(\sV,\Mrk)$ is a connected component of $l\setminus\{w\}$. When $\Mrk$ is empty, we  briefly denote $\cF_i^\circ(\sV)$ by $\cF_i^\circ(\sV,\Mrk)$.

\subsection{$\fW(\sV)$ Is Singularly Foliated by $\cF_1(\sV)$ and $\cF_2(\sV)$}

Now, we show that every weaving $\fW(\sV)$ is transversely foliated by $\cF_1(\sV)$ and $\cF_2(\sV)$.

    \begin{lem}\label{Lem:openness}
    Let $\sV=\{\sL_1,\sL_2\}$ be a veering pair. Then, the regular weaving $\fW^\circ(\sV)$ is open in the weaving $\fW(\sV).$
    \end{lem}
    \begin{proof}
    Let $c$ be an element of $\fW^\circ(\sV).$ Note that $c$ is a regular class. Choose $s$ in $c$ which is a regular stitch. By \refprop{frameCovering}, there is a frame $\fF$ on which $s$ lies. Then, by \refrmk{regularNbhd}, $\#(\cS(\fF))$ is an open neighborhood of $\#(s)=c$ in $\closure{\fW}(\sV).$ Moreover, by \refprop{scrapIsRegular}, $\#(\cS(\fF))\subset \fW^\circ(\sV).$ Thus, $\fW^\circ(\sV)$ is open in $\fW(\sV).$ 
    \end{proof}

    \begin{lem}\label{Lem:coveringPolyScrap}
    Let $\sV=\{\sL_1, \sL_2\}$ be a veering pair. Let $\fP=(\seqc{I}{i}{n}, \seqc{J}{i}{n})$ be a polygonal frame at some asterisk $(\sG_1,\sG_2)$. Then, there is a homeomorphism $\rho_\fP$ from  $P_n$ to $\#(\cS(\fP))$ satisfying the following: 
    \begin{itemize}
        \item For any plaque $p$ in $\sF^n,$ $\rho_\fP(p)= l\cap \#(\cS(\fP)) $ for some $l\in \cF_1(\sV).$
        \item For any transversal $t$ in $\sF_n,$ $\rho_\fP(t) 
        =l\cap \#(\cS(\fP)) $ for some $l\in \cF_2(\sV).$
        \item $\rho_\fP(0)$ is the singular class  $w(\sG_1, \sG_2)$.
    \end{itemize}
    \end{lem}
    \begin{proof}
    Let $c$ be the singular class in $\fW(\sV)$ corresponding the asterisk $(\sG_1, \sG_2).$
    We write 
    $\sG_1=\{M_1,\cdots, M_n\}$ and $\sG_2=\{N_1,\cdots,N_n\}$ so  that $I_i\subset M_i\in \sL_1$ and $J_i \subset N_i \in \sL_2$ for all $i\in \{1,2,\cdots,n\}$. 
    Also, we write $M_i=\opi{s_i}{s_{i+1}}$ and $N_i=\opi{t_i}{t_{i+1}}.$ For each $i\in \ZZ_n$, we denote the $M_i$-side frame of $\fP$ by $\fF_i,$ namely, $\fF_i=(M_i^*, J_{i-1}, I_i, J_i)$. 
    
    For each $i\in \{1,2,\cdots,n\}$, by \reflem{weavingRect}, there is a homeomorphism $\rho_{\fF_i}$ from $[0,1]^2$ to $\#(\closure{\cS}(\fF_i))$  satisfying the conditions in the lemma. We  also may assume that $\rho_{\fF_i}(0,\frac{1}{2})=w(\sG_1,\sG_2)$ for all $i\in \{1,2,\cdots,n\}$ since $w(\sG_1,\sG_2)\in \#(\cI(c_1(\fF_i),c_4(\fF_i))))$.
    
    % Now, fix $i\in \ZZ_n.$ Note that for each $j\in \ZZ_n,$ and for any stitch $s$ in $\fF_j,$ $\eta_1(s)$ properly lies on $M_j.$ By \reflem{leafToLeaf}, $\closure{\omega(\overt(\ell(M_i)))}$ is an arc whose end point set is $\{s_i, s_{i+1}\}$ and $\closure{\omega(\overt(\ell(M_i)))}\cap \partial \cD(C(\sV))=\{s_i, s_{i+1}\}$. Hence,   $\closure{\omega(\overt(\ell(M_i)))}$ separates $\cD(C(\sV))$ into two connected components. 
    Observe that  $\#(\cS(\fF_i))\cap\#(\cS(\fF_{i+1}))=\emptyset$ and  
    $$\#(\closure{\cS}(\fF_i))\cap\#(\closure{\cS}(\fF_{i+1}))=\#(\closure{\cI}(c, c_4(\fF_1)))=\#(\closure{\cI}(c_1(\fF_{i+1}), c))$$
    by \reflem{parallelThread} and \reflem{crossingThread} as $J_i$ and $N_i$ cross over $s_{i+1}.$
    Thus, $$\rho_{\fF_i}([0,1]^2)\cap\rho_{\fF_{i+1}}([0,1]^2)=\rho_{\fF_i}(\{0\}\times [\frac{1}{2},1])=\rho_{\fF_{i+1}}(\{0\}\times [0,\frac{1}{2}])$$ and the restriction of  $\rho_{\fF_{i+1}}^{-1}\circ \rho_{\fF_i}$ to $\{0\}\times [
    \frac{1}{2},1]$ is strictly decreasing with respect to the second variable.

    Now, after reparameterization, we may assume that for any $i\in \ZZ_n$ $$\rho_{\fF_{i+1}}^{-1}\circ \rho_{\fF_i}(0,t)=(0,1-t)$$ for all $t\in [\frac{1}{2},1]$. Then, we consider $n$-copies of $[0,1]^2,$ namely, $[0,1]^2\times \ZZ_n.$ We identify a point   $(x,y,i)$ in $[0,1]^2\times \ZZ_n$ with $(0,1-y,i+1)$ and denote the resulting space by $Q_n.$ Now, we define a map $\delta_{\fP}$ from $Q_n$ to $\closure{\cS}(\fP)$ as 
    $\delta_{\fP}(x,y,i)=\rho_{\fF_i}(x,y)$. By the gluing lemma, $\delta_\fP$ is well defined and is a homeomorphism. Observe that $Q_n$ is homeomorphic to the closed disk and the boundary $\partial Q_n$, which is homeomorphic to $S^1$,  is $$\bigcup_{i\in \ZZ_n} \left( \partial [0,1]^2-\{0\}\times (0,1) \right)\times \{i\}.$$ Hence,  by \refrmk{regularNbhd},
    the restriction $\delta_\fP|\interior{Q_n}$ of $\delta_\fP$ to $\interior{Q_n}$ is a homeomorphism from $\interior{Q_n}$ to $\cS(\fP)$, where $\interior{Q_n}=Q_n-\partial Q_n.$
    
    For each $i\in \ZZ_n,$ $[0,1]^2\times \{i\}$ is foliated by transverse foliations $\sQ^i$ and $\sQ_i$ whose leaves are vertical and horizontal lines, respectively. Hence, there are transverse foliations $\sQ^+$ and $\sQ_-$ on $\interior{Q_n}$ induced from $\{\sQ^i\}_{i\in \ZZ_n}$ and $\{\sQ_i\}_{i\in \ZZ_n},$ respectively. Obviously, there is a foliated chart $\varphi_{\interior{Q_n}}$ from $\interior{Q_n}$ to $P_n$ mapping  each leaf of $\sQ_+$ to a plaque in $\sF^n$ and each leaf of $\sQ_-$ to a transversal in $\sF_n.$ Then, $$\varphi_{\interior{Q_n}}(0,\frac{1}{2},i)=0$$ for all $i\in \ZZ_n.$ Thus, we define $\rho_\fP$ to be 
     $\rho_\fP=\delta_\fP\circ \varphi_{\interior{Q_n}}^{-1}$ and  the property of $\rho_\fP$ follows from \reflem{weavingRect}. \end{proof}

\begin{thm}\label{Thm:weaving} 
Let $\sV=\{\sL_1,\sL_2\}$ be a veering pair. The weaving $\fW(\sV)$ is an open disk foliated by transverse $1$-dimensional singular foliations $\cF_1(\sV)$ and $\cF_2(\sV).$ Moreover, the singularities of $\cF_1(\sV)$ and $\cF_2(\sV)$ are precisely the singular classes. Finally,  each $\cF_i(\sV)$ has exactly $\ord(s)$ prongs at each singular class $s$.
\end{thm}

\begin{proof}
Let $c$ be a point in $\fW(\sV).$ First, we consider the case where $c$ is a regular class. Choose a stitch $s$ in $c$. Note that $s$ is a regular stitch. Then, by \refprop{frameCovering},there is a frame $\fF$ on which $s$ lies. By \refrmk{regularNbhd} and \refprop{scrapIsRegular}, $\#(\cS(\fF))$ is an open neighborhood of $c$ in $\fW(\sV).$ Now, we consider the homeomorphism $\rho_\fF$
    given by \reflem{weavingRect}. Then, by \refrmk{regularNbhd}, 
    $$\psi \circ \left( \rho_\fF^{-1}|\#(\cS(\fF))\right)$$ is a transversely foliated foliated chart at $c$ with respect to $\cF_1(\sV)$ and $\cF_2(\sV),$ where  $\rho_\fF^{-1}|\#(\cS(\fF))$ is the restriction  of $\rho_\fF^{-1}$ to $\#(\cS(\fF))$ and $\psi$ is the map from $[0,1]^2$ to $P_2$ defined by $$\psi(x,y)=(2x-1, 2y-1).$$
    
    Now, we assume that $c$ is a singular class. Then, by \reflem{polyFrameCovering}, there is a polygonal frame $\fP$ at the asterisk $(\sG_1, \sG_2)$. We consider the homeomorphism $\rho_\fP$ given by \reflem{coveringPolyScrap}. Then, as $c\in \#(\cS(\fP))$ and \refprop{scrapIsRegular}, $\#(\cS(\fP))$ is an open neighborhood of $c$ in $\fW(\sV)$. Hence,  $\rho_\fP^{-1}$ is a transversely foliated chart at $c$ with respect to $\cF_1(\sV)$ and $\cF_2(\sV)$. Thus,  $\cF_1(\sV)$ and $\cF_2(\sV)$ are transverse singular foliations on  $\fW(\sV)$.

    The last assertion follows from the construction of $\cF_i(\sV)$ together with \reflem{parallelThread}, \reflem{ultraparallelThread}, and  \reflem{ray}.
    \end{proof}

\subsection{$\fW^\circ(\sV, \Mrk)$ Is Foliated by $\cF_1^\circ (\sV,\Mrk)$ and $\cF_2^\circ (\sV,\Mrk)$.}

In this section, we show that given a marked veering pair, the regular weaving associated with the marking is open in the weaving. Therefore, we may conclude that the regular weaving is foliated without singularity. 

\begin{lem}\label{Lem:planar}
Let $(\sV,\Mrk)$ be a marked veering pair. Then, $\fW^\circ(\sV,\Mrk)$ is open in $\closure{\fW}(\sV).$ Therefore, $\fW^\circ(\sV,\Mrk)$ is a planar surface as $\fW^\circ(\sV)$ is an open disk.  
\end{lem}
\begin{proof}
Let $w$ be an element in $\closure{\fW}(\sV).$ Suppose that $w$ is not a cone nor cusp class. Then, $w=\#(s)$ for some unmarked and regular stitch $s$. By \reflem{unmakredFrameCover}, $s$ lies on an unmarked frame $\fF$. By applying \reflem{weavingRect}, $U=\rho_\fF((0,1)^2)$ is an open neighborhood of $w$ in $\closure{\fW}(\sV)$ and since $\#^{-1}(U) =\cS(\fF),$ each class in $U$ is not a cone and cusp class and so $U\subset \fW^\circ(\sV,\Mrk).$ Thus, $\fW^\circ(\sV,\Mrk)$ is open in $\closure{\fW}(\sV).$
\end{proof}

\begin{thm}\label{Thm:foliatedRegWeaving}
Let $(\sV,\Mrk)$ be a marked veering pair. Then, $\cF_1^{\circ}(\sV, \Mrk)$ and $\cF_2^{\circ}(\sV, \Mrk)$ are transverse $1$-dimensional foliations on $\fW^\circ(\sV, \Mrk)$.
\end{thm}
\begin{proof}
By \reflem{planar}, $\fW^\circ(\sV, \Mrk)$ is a surface. Let $w$ be an element in $\fW^\circ(\sV, \Mrk).$ Then, $w=\#(s)$ for some stitch $s$ which is unmarked and regular. By \reflem{unmakredFrameCover}, $s$ lies on an unmarked frame $\fF$. By applying \reflem{weavingRect}, $U=\rho_\fF((0,1)^2)$ is an open neighborhood of $w$ in $\closure{\fW}(\sV)$ and since $\#^{-1}(U) =\cS(\fF),$ each class in $U$ is not a cone and cusp class and so $U\subset \fW^\circ(\sV,\Mrk).$ Therefore, $(U,(\rho_\fF|(0,1)^2)^{-1})$ is a transversely foliated chart at $w$. This implies that $\cF_1^{\circ}(\sV, \Mrk)$ and $\cF_2^{\circ}(\sV, \Mrk)$ are transverse $1$-dimensional foliations on $\fW^\circ(\sV, \Mrk)$.
\end{proof}

\begin{lem}\label{Lem:keyProp}
Let $\sV=\{\sL_1,\sL_2\}$ be a veering pair and $\Mrk$ a marking on $\sV$. Then, every leaf of $\cF_i^\circ(\sV,\Mrk)$ is homeomorphic to $\RR$. Moreover, Any leaves $l_1$ and $l_2$ of $\cF_1^\circ(\sV,\Mrk)$  and $\cF_2^\circ(\sV,\Mrk),$ respectively, intersect at most one point.   
\end{lem}
\begin{proof}
Let $l$ be an element of $\closure{\cF_i}(\sV)$. Then, then there is an unique real gap $\sG$ in $\sL_i$ associated with  $l$. If $\sG$ is a leaf, then by \reflem{parallelThread}, \reflem{ultraparallelThread}, and \reflem{ray}, $l$ is a line contained in $\fW^\circ(\sV)$. Then, as $l$ and $\Mrk$ intersect in at most one point, the leaves in $\cF_i^\circ (\sV,\Mrk)$ associated with $l$ is homeomorphic to $\RR$. 

Now, we consider the case where $\sG$ is a non-leaf gap. There is a unique cusp or singular class $c$ in $l$ and, by \reflem{parallelThread}, \reflem{ultraparallelThread}, and \reflem{ray}, there are $\ord(c)$-many connected components of $l\setminus c$ which is homeomorphic to $\RR$. Note that these components do not intersect with $\Mrk$ and they are contained in $\fW^\circ(\sV,\Mrk)$. Therefore, the leaves in $\cF_i^\circ (\sV,\Mrk)$ associated with $l$ is homeomorphic to $\RR.$ 

Let $l_1$ and $l_2$ be leaves in $\cF_1^\circ (\sV,\Mrk)$ and $\cF_2^\circ (\sV,\Mrk)$, respectively. Assume that $l_1$ and $l_2$ intersect in $w.$ Let $\sG_1$ and $\sG_2$ be the real gaps associated with $l_1$ and $l_2,$ respectively. Since $w$ is a unmarked regular class, $\sG_1$ and $\sG_2$ are linked but they do not interleave. By \refprop{twoGap}, \reflem{parallelThread}, and \reflem{intersectLeaves},  $w= l_1'\cap l_2'$ where $l_i'$ are elements of $\closure{\cF}_i(\sV)$ containing $l_i$. This implies the second statement. 
\end{proof}

\begin{rmk}\label{Rmk:prong}
A leaf $l$ of $\cF_i^\circ(\sV, \Mrk)$ is called a \emph{prong} at a cusp or cone class $c$ if the element of $\closure{\cF}_i(\sV)$ containing $l$ contains  $c.$ From the proof of \reflem{keyProp}, if $l$ is a prong at $c,$ then there is a unique tip $t$ of the real gap $\sG$ associated with $l$ such that  $\closure{\hat{\omega}(l)}$ is an arc whose end point set is $\{\hat{\omega}(c),t\}.$ More precisely, when $\sG$ is a non-leaf gap, if 
$l\subset \#(\oslash(\ell))$ for some leaf $\ell$ of $\sL_i,$ then $\ell=\ell(I)$ for some $I\in \tipp{t}.$ 
In the case where $\sG$ is a leaf, if 
$l\subset \#(\oslash(\ell))$ for some leaf $\ell$ of $\sL_i,$ then $\ell=\sG.$ 
If $l$ is not a prong, then there is a unique unmarked leaf $\ell$ of $\sV$ such that $l=\#(\oslash(\ell)).$
\end{rmk}

\section{Construction of Loom Spaces}\label{Sec:loomsec}

In this section, we recall loom spaces in the sense of \cite{SchleimerSegerman19,SchleimerSegerman21} and construct a loom space out of data given by a veering pair. 

Our construction is not very different from that of \cite{SchleimerSegerman19}. We have built a regular weaving $\fW^\circ$, which is analogous to the link space $\fL(\cV)$ in \cite{SchleimerSegerman19}. Unlike to the link space, the regular weaving itself is not a loom space -- it is not even homeomorphic to $\RR^2$. However, we will show that its universal cover $\widetilde{\fW^\circ}$ is indeed a loom space. 

One element that distinguishes our theorem from previous work is that veering pairs in this paper may have polygonal gaps. Recall that veering laminations in \cite{SchleimerSegerman19} only have crown gaps, mainly because they are induced from cusped 3-manifolds. In fact, the existence of polygonal gaps leads us to a cusped 3-orbifold rather than 3-manifold. 

We begin with a bit more general situation; a laminar group with an invariant veering pair. Because veering pairs are just abstract circle laminations, we cannot use arguments that rely on the veering triangulation. Instead, we carefully utilize the abstract properties of a veering pair to build a loom space. 

\subsection{Orientations on Foliated Planar Surfaces} \label{Sec:orientation}

Let $X$ be a planar surface 
with $C^0$ transverse foliations $\cF_1$ and $\cF_2.$ We say that two transversely foliated charts $(U, \phi_U)$ and $(V,\phi_V)$ of $X$ are \emph{positively compatible}  if  either $U\cap V=\emptyset$ or the \emph{transition map} $$\phi_V\circ \phi_U^{-1}:\phi_U(U\cap V)\to \phi_V(U\cap V)$$
is of the form $(\phi,\psi)$ such that 
$$\phi_V\circ \phi_U^{-1}(x,y)=(\phi(x),\psi(y)),$$ and
either the interval maps $\phi$ and $\psi$ are increasing or $\phi$ and $\psi$ are decreasing.  

An atlas $\cO$ of $X$ is called an \emph{orientation} of $X$ with respect to $\cF_1$ and $\cF_2$ if $\cO$ consists of transversely foliated charts of $S$ and any two elements of $\cO$ are positively compatible. We say that an transversely foliated chart $(W,\phi_W)$ of $S$ is \emph{positively compatible} with the orientation $\cO$ if $(W,\phi_W)$ is positively compatible with any element of $\cO.$

For each $i\in \{1,2\}$, let $X^i$ be  a planar surface with transverse foliations $\cF_1^i$ and $\cF_2^i.$ Suppose that each $X^i$ has an orientation $\cO_i.$ We say that a continuous map $f$ from $X^1$ to $X^2$ is \emph{orientation preserving} if for every $p\in S^1,$ there exist  transversely foliated charts $(U_1,\phi_1)$ and $(U_2,\phi_2)$ of $X^1$ and $X^2,$ respectively, satisfying the following.
\begin{itemize}
    \item $p\in U_1$ and $f(U_1)\subset U_2,$
    \item each $(U_i,\phi_i)$ is positively compatible with $\cO_i,$ and 
    \item $(U_1,\phi_2\circ f)$ is positively compatible with $(U_1, \phi_1).$
  
\end{itemize}
Note that, by definition, for each $i \in \{1,2\}$, an orientation preserving map $f$ maps a leaf $l$ of $\cF_i^1$ into a leaf $l'$ of $\cF_i^2,$ namely, $f(l)\subset l'$.

From now on, the foliations $F_1$ and $F_2$ of  $(0,1)^2$  are the sets of vertical lines and of horizontal lines, respectively. Also, we fix the orientation of $(0,1)^2$ with respect to $F_1$ and $F_2$ as the atlas $\{((0,1)^2,\id_{(0,1)^2})\}.$

Let $\sV=\{\sL_1, \sL_2\}$ be a veering pair with marking $\Mrk$. We can see that the transversely foliated charts given by \reflem{weavingRect} compose an atlas $\cO$ for the planar surface $\fW^\circ(\sV,\Mrk)$ from the proof of \refthm{foliatedRegWeaving}. Moreover, any two elements of $\cO$ are positively compatible since every frame is oriented in the counter-clockwise direction.
Hence, $\cO$ is an orientation of $\fW^\circ(\sV,\Mrk)$. From now on, we fix the orientation of $\fW^\circ(\sV,\Mrk)$ as $\cO$.  Under these orientations, the transversely foliated charts given by \reflem{weavingRect} are orientation preserving. For convenience, we say that a homeomorphism $f$ on a weaving is \emph{orientation preserving} if $f$ is orientation preserving on the regular weaving.

\subsection{Loom Spaces}\label{Sec:loom}
We recall the definition of loom spaces. A \emph{loom space} $\cL$ is a copy of $\RR^2$ together with $C^0$ transverse foliations $\cF_1$ and $\cF_2$. We also assume that $\cL$ has an orientation $\cO$. 

Since $\cL$ is simply-connected, we may take a subatlas $\cO_+$ of $\cO$ so that each transition map of $\cO_+$ is of the form $(\phi, \psi)$ such that both $\phi$ and $\psi$ are increasing.  Hence, we can say that the positive direction of $\cF_1$ is \emph{north} and the negative direction of $\cF_1$ is \emph{south}. Similarly,  we refer to  the positive and negative direction of $\cF_2$ as \emph{east} and \emph{west}, respectively. 

A \emph{rectangle} $R$ is an open subset of $\RR^2$ such that there is a homeomorphism $f_R:(0,1)\times (0,1) \to R$ that maps vertical leaves $x\times (0,1)$ to a leaf of $\cF_1$ and horizontal leaves $ (0,1)\times y$ to $\cF_2$. We request that $f_R$ preserves the orientations of the foliations. Here, the orientation of $(0,1)^2$ is $\{((0,1)^2,\id _{(0,1)^2})\}$ and the orientation of $\cL$ is $\cO_+$. 

If $f_R$ admits a homeomorphic extension $\overline{f_R}:[0,1]\times [0,1] \setminus \{a\times b\}\to \overline{R}$ for some $a,b\in \{0,1\}$, we say that $R$ is a \emph{cusp rectangle}. In this case,  $\overline{f_R}(a\times (0,1))$ and $\overline{f_R}((0,1)\times b)$ are called \emph{cusp sides}.

If $f_R$ has a homeomorphic extension 
\[
\overline{f_R}: [0,1]\times[0,1] \setminus \{a\times 0, 1\times b, c\times 1, 0\times d\}\to \overline{R}
\]
for some $0<a,b,c,d<1$, we call $R$ a \emph{tetrahedron rectangle} with parameters $a,b,c$ and $d$. 

\begin{defn}\label{loomdef}
A \emph{loom space} $\cL$ is a topological space homeomorphic to $\RR^2$ with two oriented transverse foliations subject to the following properties:
\begin{enumerate}
    \item Every cusp side of a cusp rectangle is contained in a rectangle.
    \item Every rectangle is contained in a tetrahedron rectangle.
    \item If a tetrahedron rectangle has parameters $a,b,c$ and $d$, then we have $a\ne c$ and $b\ne d$.
\end{enumerate}
\end{defn}

Let $\cL$ be a loom space with oriented transverse foliations $\cF_1$ and $\cF_2.$ A \emph{edge rectangle} $R$ in $\cL$ is one such that an associated homeomorphism $f_R$ can be extended to the homeomorphism either of the forms
\begin{align*}
\closure{f_R}&:[0,1]^2-\{0\times 0, 1\times 1\} \to \closure{R} \text{ or }\\
\closure{f_R}&:[0,1]^2-\{0\times 1, 1\times 0\} \to \closure{R}. 
\end{align*}
 In particular, $R$ is called a \emph{red edge rectangle} if the first extension happens. $R$ is called a \emph{blue edge rectangle} otherwise. 

A \emph{south-west face rectangle} in $\cL$ is one whose associated homeomorphism $f_R$ has a continuous extension 
$$\closure{f_R}:[0,1]\times [0,1]\setminus \{0\times 0, 1\times a, b\times 1\}$$
for some $a$ and $b$ in $(0,1).$ Similarly, we define \emph{south-east}, \emph{north-east}, and \emph{north-west} \emph{face rectangles}.

We need a careful definition for cusps in $\cL$. Let us say that two cusp rectangles $P$ and $Q$ in $\cL$ are \emph{equivalent} if there is a finite sequence of cusp rectangles $P=R_1, R_2, \cdots, R_n=Q$ such that
for each pair $(R_i,R_{i+1})$, some cusp side of one is contained in a cusp side of the other.  A \emph{cusp} of the loom space is an equivalence class of this equivalent relation. A tetrahedron rectangle  \emph{contains} a cusp $c$ if it contains some cusp rectangle in $c$.

\subsection{Rectangles in $\fW^{\circ}(\sV, \Mrk)$}\label{Sec:rectangle}
Let $P$ be a planar surface with transverse $1$-dimensional foliations $\cF_1$ and $\cF_2.$ We fix an orientation of $P$. By abusing language, we call an open subset $R\subset P$ a \emph{rectangle} if there is an orientation preserving homeomorphism $f_R:(0,1)\times (0,1)\to R$ such that $f_R$ maps each vertical leaf $x\times (0,1)$, $x\in (0,1)$ to a leaf of $\cF_1$ and each horizontal leaf $(0,1)\times y$, $y\in (0,1)$ to a leaf of $\cF_2$.

We can also define skeletal rectangles similarly. A \emph{cusp rectangle} $R$ is a rectangle such that an associated homeomorphism $f_R$ can be extended to the homeomorphism $\overline{f_R}:[0,1]\times [0,1] \setminus \{a\times b\}\to \overline{R}$ for some $a,b\in \{0,1\}.$ In this case,  $\overline{f_R}(a\times (0,1))$ and $\overline{f_R}((0,1)\times b)$ are again called \emph{cusp sides}.

An \emph{edge rectangle} $R$ is one that $f_R$ admits a homeomorphic extension either of the form
\begin{align*}
\overline{f_R}& : [0,1]\times [0,1] \setminus \{0\times 0, 1\times 1\} \to \overline{R} \text{, or}\\
\overline{f_R}& : [0,1]\times [0,1] \setminus \{0\times 1, 1\times 0\} \to \overline{R}.
\end{align*}
If the first extension happens, we call $R$ a \emph{red edge rectangle}; otherwise $R$ is called a \emph{blue edge rectangle}. 

Similarly, a \emph{face rectangle} $R$  is one whose associated homeomorphism $f_R$ admits a homeomorphic extension of the form
\begin{align*}
    \closure{f_R}&:[0,1]\times [0,1]\setminus \{a\times a, (1-a)\times b, c\times (1-a)\}\text{ or }\\
    \closure{f_R}&:[0,1]\times [0,1]\setminus \{a\times (1-a),  b\times a, (1-a)\times c \}
\end{align*}
for some $a\in\{0,1\}$ and for some $b$ and $c$ in $(0,1).$

A \emph{tetrahedron rectangle} $R$ with parameters $a,b,c$ and $d$  is one that an associated homeomorphism $f_R$ can be extended to the homeomorphism 
\[
\overline{f_R}: [0,1]\times[0,1] \setminus \{a\times 0, 1\times b, c\times 1, 0\times d\}\to \overline{R}
\]
for some $0<a,b,c,d<1$.

\begin{rmk}\label{Rmk:rectWellDefine}
Suppose that a planar surface $P$ with transverse foliations $\cF_1$ and $\cF_2$ has an orientation $\cO$. When $P$ is simply-connected, we need to check the compatibility of the definition of rectangles in $P$ defined in this section and the definition of rectangles in loom spaces. In fact, the definition of rectangles in $P$ is weaker than the definition of rectangles in loom spaces since in loom spaces, we refined the orientation for loom spaces as $\cO_+$ to say north, east, west, and south. See \refsec{loom}.

Nonetheless, we can abuse these definitions because the only problem is defining the red and blue rectangles. To see this, we take a subatlas $\cO_+$ of $\cO$ so that each coordinate function of  transition maps of $\cO_+$ is increasing as in \refsec{loom}.

Then, choose a rectangle $R$ in $P$ and let $f_R:(0,1)^2 \to R$ be the associated map that is the orientation preserving map respecting $\cO.$ Then, we can take another orientation preserving map $g_R:(0,1)^2 \to R$  so that $g_R$ respects $\cO_+$, if necessary, by precomposing  the orientation preserving homeomorphism $\pi: (x,y)\mapsto (1-x,1-y)$ on $(0,1)^2$ which flips the direction of leaves of $F_1$ and $F_2$. Therefore, the definition of rectangles in transversely foliated planar surfaces is well-defined. 

Moreover, $\pi$ is essentially the only homeomorphism on $(0,1)^2$ (other than the trivial map) up to isotopy passing along the homeomorphisms that preserve the orientation and the leaves of foliations.  This means that an edge rectangle being blue or red is well-defined. 
\end{rmk}

\begin{lem}\label{Lem:maximalRep}
Let $\sV=\{\sL_1, \sL_2\}$ be a veering pair and $\Mrk$ a marking on $\sV$. Let $R$ be a rectangle on $\fW^\circ(\sV,\Mrk)$ such that an associate homeomorphism $f_R$ can be extended to the homeomorphism $\closure{f_R}$ from $[0,1]^2$ to the closure $\closure{R}$
in $\fW^\circ(\sV,\Mrk).$ Then, there is an unmarked frame $\fF=(I_1, I_2, I_3, I_4)$ of $\sV$ such that $\#^{-1}(\closure{R})=\closure{\cS}(\fF).$ Moreover, if there is a frame $\fF'$ in $\sV$ such that $\#^{-1}(\closure{R})=\closure{\cS}(\fF'),$ then $\fF'$ is either $(I_1, I_2, I_3, I_4)$ or $(I_3, I_4, I_1, I_2).$
\end{lem}
\begin{proof}
Let $f_R$ be an homeomorphism associated to $R$. By the assumption, $f_R$ is extended to a homeomorphism $\closure{f_R}:[0,1]\times [0,1]\to \closure{R}$. We write $l_1$ for the leaf of $\cF_1^\circ(\sV,\Mrk)$ containing $\closure{f_R}(0\times [0,1])$. Similarly, we also write $l_2$, $l_3$, and $l_4$ for the leaves containing $\closure{f_R}([0,1]\times 0),$ $\closure{f_R}(1\times [0,1]),$ and $\closure{f_R}([0,1]\times 1),$ respectively.
Also, for each $i\in \{1,2,3,4\}$ we denote the intersection point of $l_i$ and $l_{i+1}$ by $d_i$ (throughout the proof, assume everything is cyclically indexed). By \reflem{keyProp}, $d_i$ are uniquely determined and so 
$$d_1=\closure{f_R}(0,0), \ d_2=\closure{f_R}(1,0), \ d_3=\closure{f_R}(1,1), \text{ and } d_4=\closure{f_R}(0,1).$$

Observe that, for each $i\in \{1,2,3,4\}$, if $\ell_i$ and $\ell_{i+1}$  are leaves of $\sV$ such that $l_i \subset \#(\oslash(\ell_i))$ and $l_{i+1} \subset \#(\oslash(\ell_{i+1}))$, then $\ell_i$ and $\ell_{i+1}$ are linked since $d_i$ is an unmarked regular class. See \refprop{eqClass2}.

Now, for each $i\in \{1,2,3,4\}$, we choose a leaf $\ell_i$ such that $l_i \subset \#(\oslash(\ell_i))$. Note that by the definition of the weaving relation, $\ell_1$ and $\ell_3$ are ultraparallel and $\ell_2$ and $\ell_4$ are ultraparallel. Then, there is the unique sector $(J_1,J_2)$ of the stitch $(\ell_1, \ell_2)$ such that the stitch $(\ell_3, \ell_4)$ lies on $(J_1^*, J_2^*).$ Likewise, there is the unique sector $(J_3,J_4)$ of the stitch $(\ell_3, \ell_4)$ such that the stitch $(\ell_1, \ell_2)$ lies on $(J_3^*, J_4^*).$ Note that the sectors $(J_1, J_2)$ and $(J_3, J_4)$ are counter-clockwise as $f_R$ preserves the orientation. Therefore, the quadruple $\fQ=(J_1, J_2, J_3, J_4)$ is a frame of $\sV$. Then, we can take an orientation preserving homeomorphism $\rho_\fQ$ from $[0,1]\times [0,1]$ to $\#(\closure{\cS}(\fQ))$ as in \reflem{weavingRect}.
Since each $l_i$ is homeomorphic to $\RR$ by \reflem{keyProp}, $\rho_\fQ(\partial [0,1]^2)=\partial R$ and as $\fW(\sV,\Mrk)$ is the open disk, by the Jordan–Schoenflies theorem,  $\#(\closure{\cS}(\fQ))= \closure{R}.$ However, in general, $\closure{\cS}(\fQ)\neq \#^{-1}(\closure{R}).$ Hence, we need to modify the frame $\fQ.$ 

If $J_1$ is not isolated, then there is no leaf $\ell$ of $\sL_1$ lying on $J_1$ such that $l_1 \subset \#(\overt(\ell)).$ See \refrmk{prong}. In this case, we set $I_1=J_1.$ Otherwise,  by \refrmk{prong}, there is a unique element $I_1$ in $\sL_1$ such that $I_1\subsetneq J_1$ and $l_1 \subset \#(\overt(\ell(I_1))).$ Likewise, for each $i\in \{1,2,3,4\}$, we can take $I_i$. Then, by the previous observation, $\ell(I_i)$ and $\ell(I_{i+1})$ are linked for all $i\in \{1,2,3,4\}$. Therefore, the quadruple $\fF=(I_1, I_2, I_3, I_4)$ is  a frame covering $\fQ$. By the construction, $l_i\subset \#(\oslash(\ell(I_i)))$ for all $i\in \{1,2,3,4\}$. Then, by \reflem{keyProp}, $\#(\closure{\cS}(\fF))=\closure{R}$

Now, we want to show that $\#^{-1}(\closure{R})=\closure{\cS}(\fF).$ Note that $\#^{-1}(R)=\cS(\fF)$ by \reflem{weavingRect}. 
Let $s=(m_1,m_2)$ be a stitch with $\#(s)\in \closure{R}.$ Then, there is a unique leaf $n_1$ of $\cF_1^\circ(\sV,\Mrk)$ such that  $\#(s)\in n_1 \subset \#(\overt(m_1)).$ See \refrmk{prong}. Then, there is a unique class $w_1$ in the intersection of $n_1$ and $l_2.$ Then, $w_1=\#(s_1)$ where $s_1$ is the stitch $(m_1, \ell(I_2)).$

\begin{claim}\label{Clm:sameInterval}
Let $\ell$ be a leaf of $\sL_i.$ Assume that $\cI(t_1,t)$ and $\cI(t_2,t)$ are non-empty intervals in $\oslash(\ell).$ If $\closure{\cI}(t_1, t)\subsetneq \closure{\cI}(t_2,t)$ and $\#(\closure{\cI}(t_1, t))=\#(\closure{\cI}(t_2,t)),$ then $t_1\sim_\omega t_2$ and so $\eta_{i+1}(t_1)$ and $\eta_{i+1}(t_2)$ are parallel. 
\end{claim}
\begin{proof}
It follows from \reflem{subarc} and \refrmk{eqDef}.
\end{proof}

The interval $\cI(c_1(\fF),c_2(\fF))$ is the maximal interval on $\ominus(\ell( J_2))$ in the following sense. If there is an interval $\cI(t_1,t_2)$ in $\ominus(\ell( J_2))$ such that $\#(\closure{\cI}(t_1,t_2))=\closure{f_R}([0,1]\times 0),$ then $\closure{\cI}(t_1,t_2)\subset \closure{\cI}(c_1(\fF),c_2(\fF))$ by the choice of $I_1$ and $I_2$  and by \refclm{sameInterval}. Therefore, $s_1\in \closure{\cI}(c_1(\fF),c_2(\fF)).$ Likewise, we can show that the stitch $(\ell(I_1),m_2)$ belongs to $\closure{\cI}(c_1(\fF),c_4(\fF))$. Therefore, $s\in \closure{\cS}(\fF)$ as $\sigma_\fF(s_1,s_2)=s.$ Thus, $\#^{-1}(\closure{R})=\closure{\cS}(\fF)$. Also, since $R\subset \fW^\circ(\sV,\Mrk),$ $\fF$ is unmarked. 

Note that if a map $f_R'$ from $(0,1)^2$ to $R$ is defined by $f_R'(x,y)=f_R(1-x,1-y)$, then $f_R'$ is also an orientation preserving homeomorphism associated with $R$. This possibility of the choice of an associated homeomorphism implies the second statement. 
\end{proof}

For a rectangle $R$, we call frames $(I_1, I_2, I_3, I_4)$ and $(I_3, I_4, I_1, I_2)$ as in \reflem{maximalRep} the \emph{maximal frame representatives} of $R$.

\begin{prop}\label{Prop:coverRep}
Let $\sV=\{\sL_1, \sL_2\}$ be a veering pair and $\Mrk$ a marking on $\sV.$ Let $R_1$ and $R_2$ be rectangles in $\fW^\circ(\sV,\Mrk)$  such that  $R_1 \subset R_2$ and 
a homeomorphism $f_{R_2}$ associated with $R_2$ can be extended to the homeomorphism $\closure{f_{R_2}}$ from $[0,1]^2$ to
the closure $\closure{R_2}$ in $\fW^\circ(\sV,\Mrk).$ Then, there are maximal frame representatives $\fF_1$ and $\fF_2$ of $R_1$ and $R_2,$ respectively, such that $\fF_2$ covers $\fF_1.$ 
\end{prop}
\begin{proof}
It follows from \reflem{maximalRep}.
\end{proof}

\begin{lem}\label{Lem:minimalRep}
Let $\sV=\{\sL_1, \sL_2\}$ be a veering pair and $\Mrk$ a marking on $\sV.$ Let $R$ be a rectangle on $\fW^\circ(\sV,\Mrk).$ Then, there is an unmarked frame $\fF=(I_1, I_2, I_3, I_4)$ such that $\cS(\fF)=\#^{-1}(R)$ and $I_i^*$ is not isolated for all $i \in \{1,2,3,4\}$. Moreover, if $\fF'=(J_1,J_2, J_3, J_4)$ is a frame such that $\cS(\fF')=\#^{-1}(R)$ and $J_i^*$ is not isolated for all $i \in \{1,2,3,4\}$, then $\fF'$ is either $(I_1, I_2, I_3, I_4)$ or $(I_3, I_4, I_1, I_2)$.
\end{lem}

\begin{proof} Let $f_R$ be a homeomorphism associated with $R$.
Given a real numbers $\alpha$ with $0\leq \alpha<1/2,$ we define  a map $p^\alpha$ from $[0,1]^2$ to $[\alpha, 1-\alpha]^2$ by 
$p^\alpha(x,y)=((1-x)\alpha+x(1-\alpha),(1-y)\alpha+y(1-\alpha))$. Then, we write $f_R^\alpha:=f_R\circ p^\alpha$ and the rectangles $f_R^\alpha((0,1)^2)$ are denoted by $R^\alpha.$ Note that $f_R=f_R^0$ and that if $0\leq \alpha_1 < \alpha_2<1/2$, then $ \closure{R^{\alpha_2}}\subset R^{\alpha_1}.$ Since $\closure{R^\alpha} \subset R^0$ for all $\alpha$ with $0<\alpha<1/2,$ by \reflem{maximalRep}, there is a maximal frame representative $\fF^\alpha$ of $R^\alpha$ for all real number $\alpha$ with $0<\alpha<1/2$. Moreover, by \refprop{coverRep}, we may assume that if $0<\alpha_1 < \alpha_2<1/2$, then $\fF^{\alpha_1}$ covers  $\fF^{\alpha_2}$.

For each $\alpha$ with $0<\alpha<1/2$, we write $\fF^{\alpha}=(I_1^\alpha,I_2^\alpha,I_3^\alpha,I_4^\alpha)$ and $I_i^\alpha=\opi{u_i^\alpha}{v_i^\alpha}$ for all $i\in \{1,2,3,4\}$. Now, we consider the stem $S_1$ from ${I_1^{1/3}}$ to $\opi{v_4^{1/3}}{u_2^{1/3}}$ in $\sL_1$.
Observe that $I_1^\alpha\in S_1$ for all $\alpha$ with $0<\alpha \leq 1/3$. As $\{I_1^{1/(n+2)}\}_{n\in\NN}$ is a descending sequence in $S_1$, then there is an element $I_1$ in $S_1$ such that $I_1^*=\bigcup_{n\in \NN}(I_1^{1/(n+2)})^*,$ that is, $\ell(I_1^{1/(n+2)})\to \ell(I_1)$ as $n\to \infty$. Similarly, we can find $I_2$, $I_3,$ and $I_4$ such that $I_i^*=\bigcup_{n\in \NN}(I_i^{1/(n+2)})^*$ for all $i\in \{2,3,4\}$.

Note that as $\sL_1$ and $\sL_2$ are strongly transverse, $\ell(I_i)$ and $\ell(I_{i+1})$ are either linked or ultraparallel for all $i\in \{1,2,3,4\}$ (everything in the proof is cyclically indexed).  If $\ell(I_i)$ and $\ell(I_{i+1})$ are ultraparallel for some $i\in \{1,2,3,4\}$, then there is a number $N$ in  $\NN$ such that $\ell(I_i^{1/(n+2)})$ and $\ell(I_{i+1}^{1/(n+2)})$ are unlinked for all $n\in \NN$ with $N<n$ since $\ell(I_j^{1/(n+2)})\to \ell(I_j)$ for all $j\in \{1,2,3,4\}$. This is a contradiction. Therefore, $\ell(I_i)$ and $\ell(I_{i+1})$ are  linked for all $i\in \{1,2,3,4\}$. This implies that the quadruple $\fF=(I_1, I_2, I_3, I_4)$ is a frame. Note that  for each $n\in \NN,$ $\closure{\cS}(\fF^{1/(n+2)})\subset \cS(\fF)$ since $\ell(I_i^{1/(n+2)})$ and $\ell(I_i)$ are ultraparallel for all $i\in \{1,2,3,4\}$.  

Given a stitch $s$ lying on $\fF,$ there is a frame $\fF^{1/(m+2)}$ for some $m\in \NN$ such that $s\in \closure{\cS}(\fF^{1/(m+2)})$ since $\ell(I_j^{1/(n+2)})\to \ell(I_j)$ for all $j\in \{1,2,3,4\}$. Thus, $$\cS(\fF)=\bigcup_{n\in \NN}\closure{\cS}(\fF^{1/(n+2)})$$
Observe that $\#^{-1}(R)=\bigcup_{n\in \NN}\closure{\cS}(\fF^{1/(n+2)})$ since $\fF^{1/(n+2)}$ are maximal frame representatives. Thus, $\cS(\fF)=\#^{-1}(R)$.
The second statement comes from the possibility of the choice of $f_R$.
\end{proof}

Given a rectangle $R$, we call any frame that satisfies the conclusion of \reflem{minimalRep} is called a \emph{minimal frame representatives} of $R.$ \reflem{minimalRep} shows that a minimal frame is unique up to reordering of components.

\begin{rmk}
Let $(\sV,\Mrk)$ be a marked veering pair. Assume that $R$ is a rectangle in $\fW^\circ(\sV,\Mrk).$ If $\fF$ is a frame with $\#(\cS(\fF))=R,$ then $\fF$ covers a minimal frame representative of $R.$
\end{rmk}

\begin{rmk}\label{Rmk:cuspCorner}
Let $\sV=\{\sL_1, \sL_2\}$ be a veering pair and $\Mrk$ a marking on $\sV$. Assume that $R$ is a cusp rectangle in $\fW^\circ(\sV,\Mrk)$. By \reflem{minimalRep}, there is a minimal frame representative $\fF$ of $R$. Then, we can take a homeomorphism $\rho_\fF$ as in  \reflem{weavingRect} of which the restriction $\rho_\fF^\circ:=\rho_\fF|(0,1)^2$ is an associated homeomorphism of $R.$ 
Since $R$ is a cusp rectangle, we may assume that $\rho_\fF^\circ$ is extended to $\closure{\rho_\fF^\circ}$ from $[0,1]^2\setminus 0\times 0$ to the closure $\closure{R}$ in $\fW^\circ(\sV,\Mrk)$. Observe that $\closure{\rho_\fF^\circ}=\rho_\fF$ on $[0,1]^2\setminus 0\times 0.$ Therefore, $\rho_\fF(0,0)$ is a cusp or cone class in $\closure{\fW}(\sV)$ since $\rho_\fF([0,1]^2)\subset \closure{\fW}(\sV)$. Hence, we denote $\rho_\fF(0,0)$ by $\cusp(R)$.
\end{rmk}

\begin{rmk}\label{Rmk:tetFrameToRect}
Let $\sV=\{\sL_1, \sL_2\}$ be a veering pair with a marking $\Mrk$. Assume that $R$ is a rectangle in $\fW^\circ(\sV,\Mrk)$ and that there is a tetrahedron frame $\fF$ under $\Mrk$ such that $\#(\cS(\fF))=R$. Then, we can take a homeomorphism  $\rho_{\fF}$ as in \reflem{weavingRect}. Then, since $\fF$ is $i^{th}$-side full for all $i\in \{1,2,3,4\}$, by \reflem{extension} and \refrmk{cuspCorner},  $\rho_{\fF}(0\times (0,1))$, $\rho_{\fF}( (0,1)\times 0)$, $\rho_{\fF}(1\times (0,1))$, and $\rho_{\fF}((0,1)\times 1)$ contain  cusp or cone classes $w_1$, $w_2$, $w_3$,  and $w_4$, respectively. Moreover, by \refrmk{prong}, connected components of $\rho_{\fF}(0\times [0,1])-\{w_1\}$, are contained in different prongs at $w_1$. This implies that $w_1$ is a unique cusp or cone class in $\rho_{\fF}(0\times [0,1])$. Similar statement also holds for $w_2$, $w_3$, and $w_4$. This implies that $R$ is a tetrahedron rectangle in $\fW^\circ(\sV,\Mrk)$. 

Conversely, given a  tetrahedron rectangle $R$ in $\fW^{\circ}(\sV,\Mrk)$, we can take a minimal frame representative $\fF$ for $R$. Then, by \reflem{fullExtension} and \reflem{weavingRect}, it follows that $\fF$ is a tetrahedron frame under $\Mrk.$ Therefore, by \reflem{minimalRep}, there is a one-to-one correspondence between the set of all tetrahedron rectangles in $\fW^{\circ}(\sV, \Mrk)$ and the set of all tetrahedron frames under $\Mrk$ up to reordering of components.
\end{rmk}

\subsection{Construction of a Loom Space}

We now present our main construction of a loom space. 
Then our main theorem of this section is the following:
\begin{thm}\label{Thm:loomspaceconstruction}
Let $\sV=\{\sL_1, \sL_2\}$ be a veering pair with a marking $\Mrk$. Then the universal cover $\widetilde{\fW^\circ}(\sV,\Mrk)$ of the regular weaving $\fW^\circ (\sV,\Mrk)$ is a loom space with the foliations $\widetilde{\cF_1^\circ}(\sV,\Mrk)$ and $\widetilde{\cF_2^\circ}(\sV,\Mrk)$ induced  by lifting $\cF_1^\circ(\sV,\Mrk)$ and $\cF_2^\circ(\sV,\Mrk),$ respectively.
\end{thm}

In light of \refthm{weaving} and \refthm{foliatedRegWeaving},  $\fW^\circ(\sV,\Mrk)$ is an open disk with a discrete closed subset of countably many points removed. We then know that that the universal cover $\widetilde{\fW^\circ}(\sV,\Mrk)$ is homeomorphic to the plane $\RR^2$ equipped with (non-singular) transverse foliations $\widetilde{\cF_1^\circ}(\sV,\Mrk)$ and $\widetilde{\cF_2^\circ}(\sV,\Mrk)$ induced by lifting $\cF_1^\circ(\sV, \Mrk)$ and $\cF_2^\circ(\sV, \Mrk),$ respectively.  Therefore, it remains to show that  $\widetilde{\fW^\circ}(\sV,\Mrk)$ fulfils the properties listed in Definition \ref{loomdef}. 

\begin{rmk}\label{Rmk:liftingLeaf}
By \reflem{keyProp}, every leaf of $\widetilde{\cF_1^\circ}(\sV,\Mrk)$ and $\widetilde{\cF_2^\circ}(\sV,\Mrk)$ is homeomorphic to $\RR.$ 
\end{rmk}

Although $\fW^\circ(\sV,\Mrk)$ itself is not a loom space, it shares the same properties that rectangles in a loom space have. 

\begin{lem}\label{Lem:propertiesofweaving}
Let $\sV=\{\sL_1,\sL_2\}$ be a veering pair with a marking $\Mrk$. 
\begin{enumerate}
    \item Any cusp side of a cusp rectangle in $\fW^\circ(\sV,\Mrk)$ is contained in a rectangle in $\fW^\circ(\sV,\Mrk)$.
    \item Any rectangle in $\fW^\circ(\sV,\Mrk)$ is contained in a tetrahedron rectangle. 
    \item For parameters $a,b,c$ and $d$ of a tetrahedron rectangle in $\fW^\circ(\sV,\Mrk)$, we have $a\ne c$ and $b\ne d$. 
\end{enumerate}
\end{lem}
\begin{proof} Due to \reflem{minimalRep}, given a rectangle $R$ in $\fW^\circ(\sV,\Mrk),$ we can take a minimal frame represenative $\fF$ of $R.$

(1) Let $R$ be a cusp rectangle in $\fW^\circ(\sV,\Mrk)$ and $\fF$ a minimal frame representative of $R.$ Then,  we can take a homeomorphism $\rho_\fF$ as in \reflem{weavingRect}. For simplicity, we say that $\rho_\fF(0,0)$ is $\cusp(R).$ See \refrmk{cuspCorner}. Then, $c_1(\fF)$ is a marked or singular stitch and as $\fF$ is a minimal representative of $R,$ $\cI(c_1(\fF), c_4(\fF))$ and $\cI(c_1(\fF),c_2(\fF))$ have no marked or singular stitch. By \reflem{extension}, $\fF$ is not $1^{st}$-side and $2^{nd}$-side full. By \reflem{fullExtension}, there is unmarked $1^{st}$-side extension $\fF^1$ of $\fF$ that is $1^{st}$-side full. Also, there is an unmarked $2^{nd}$-side extension $\fF^2$ of $\fF$ that is $2^{nd}$-side full. Then, by \reflem{weavingRect}, $\#(\cS(\fF^1))$ and $\#(\cS(\fF^2))$ are rectangles containing 
the cusp sides $\rho_\fF(0\times (0,1))$ and $\rho_\fF([0,1]\times 0)$ of $R,$ respectively.

(2) Let $R$ a rectangle in $\fW^\circ(\sV,\Mrk)$ and $\fF$ a minimal frame representative of $R.$ By \reflem{rectangleisintetra}, $\fF$ is covered by a tetrahedron frame $\fF'$ under $\Mrk.$ By \refrmk{tetFrameToRect}, $\cS(\fF')$ is a tetrahedron rectangle in $\fW^\circ(\sV,\Mrk)$ containing $R.$

(3) It follows from \refrmk{cuspCorner} and \refrmk{prong}.
\end{proof}

\begin{lem}\label{Lem:correspondence}
Let $D$ be the group of deck transformations of the covering $p:\widetilde{\fW^\circ} (\sV,\Mrk)\to \fW^\circ (\sV,\Mrk)$. 
\begin{enumerate}
    \item The covering map $p$ sends a rectangle to a rectangle.
    \item For any $g\in D\setminus\{1\}$ and any rectangle $R$ in $\widetilde{\fW^\circ} (\sV,\Mrk)$, $g(R)$ is also a rectangle of $\widetilde{\fW^\circ} (\sV,\Mrk)$ and $g(R)\cap R=\emptyset$.
    \item For a rectangle $R$ in $\fW^\circ(\sV,\Mrk)$, each component of $p^{-1}(R)$ is a rectangle. 
\end{enumerate}
\end{lem}
\begin{proof}
(1) Let $R$ be a rectangle and let $f_R:(0,1)\times(0,1) \to R$ be the associated homeomorphism.   We claim that $p\circ f_R$ is a homeomorphism onto its image. We first show that $p\circ f_R$ is injective. 

On the contrary, suppose that there are two distinct points $a\times b$ and $c\times d$ in $(0,1)\times(0,1)$ that are mapped by $p\circ f_R$ to a same point in $\fW^\circ (\sV,\Mrk)$. Then, there are leaves $l_1$ and $l_2$ of $\widetilde{\cF_1^\circ} (\sV,\Mrk)$ containing $f_R(a\times [0,1])$ and $f_R(c\times [0,1]),$ respectively. By definition, $p(l_1)$ and $p(l_2)$ are also leaves of $\cF_1^\circ (\sV,\Mrk)$ and by the assumption that $p\circ f_R(a,b)=p\circ f_R(c,d),$ $p(l_1)$ and $p(l_2)$ are equal. We say that $l=p(l_1)=p(l_2).$ 

Observe that if  $a=c,$ then $b\neq d$ and $l_1=l_2.$ This implies that $l$ is homeomorphic to the circle. See \refrmk{liftingLeaf}. This is a contradiction by \reflem{keyProp}. Therefore, $a\neq c.$

On the other hand, there is a leaf $m$ of $\widetilde{\cF_2^\circ} (\sV,\Mrk)$ containing $f_R([0,1]\times 1/2).$ Also, by definition, $p(m)$ is  a leaf of $\cF_2^\circ(\sV, \Mrk).$ Then, if $p\circ f_R(a,1/2)\neq p\circ f_R(c,1/2),$ then $l$ and $p(m)$ intersect in two distinct points and it is a contradiction by \reflem{keyProp}. Hence, $p\circ f_R(a,1/2)= p\circ f_R(c,1/2).$ This implies that    $p(m)$ is homeomorphic to the circle as $a\neq c$. See \refrmk{liftingLeaf}. This is also a contradiction by \reflem{keyProp}. Therefore, $p \circ f_R$ is injective. 

Moreover, $p\circ f_R$ is an open map being the composition of a  homeomorphism and a covering map. Therefore, $p\circ f_R((0,1)\times (0,1))$ is a rectangle in $\fW^\circ(\sV,\Mrk)$.

(2) Let $R$ be a rectangle in $\widetilde{\fW^\circ} (\sV,\Mrk)$ with the associated homeomorphism $f_R:(0,1)\times (0,1)\to R$. Since the induced transverse foliations in $\widetilde{\fW^\circ} (\sV, \Mrk)$ are $D$-invariant, $g\circ f_R$ is a homeomorphism onto $g(R)$ that preserves the transverse foliations. This shows that $g(R)$ is a rectangle. The assertion that $g(R)\cap R = \emptyset$ follows from (1). 

(3) Clear from the lifting property together with the fact that $R$ is simply-connected.
\end{proof}

\begin{lem}\label{Lem:containedintetra}
Any rectangle in $\widetilde{\fW^\circ} (\sV,\Mrk)$ is contained in a tetrahedron rectangle in $\widetilde{\fW^\circ} (\sV,\Mrk)$. 
\end{lem}
\begin{proof}
Let $R$ be a rectangle in $\widetilde{\fW^\circ} (\sV,\Mrk)$.
In light of \reflem{correspondence}, we know that $p(R)$ is a rectangle in $\fW^\circ(\sV,\Mrk)$ and by \reflem{propertiesofweaving}, $p(R)$ is contained in a tetrahedron rectangle $R'$.

Given a tetrahedron rectangle $R'$ in $\fW^\circ (\sV,\Mrk)$ with a homeomorphism $f_{R'}:(0,1)\times (0,1) \to \fW^\circ (\sV,\Mrk),$ we have the extended homeomorphism $\overline{f_{R'}}:[0,1]\times [0,1]\setminus \{a \times 0,1\times b, c\times 1, 0\times d\}\to \overline{R}$ for some $0<a,b,c,d<1$. Because the domain of $\overline{f_{R'}}$ is simply-connected and $\overline{f_{R'}}$ is a homeomorphism onto its image, we know that there is a lift $\widetilde{\overline{f_{R'}}}:[0,1]\times [0,1]\setminus \{a \times 0,1\times b, c\times 1, 0\times d\} \to \widetilde{\fW^\circ}(\sV,\Mrk)$. 

Let $Z$ be the image of $\widetilde{\overline{f_{R'}}}$. Since the image $\overline{R'}$ of $\overline{f_{R'}}$ is simply-connected, 
we know that the connected component of $p^{-1}(\overline{R'})$ that contains $Z$ is $Z$ itself. Since $\overline{R'}$ is closed, one of its lift $Z$ is also closed in $\widetilde{\fW^\circ}(\sV,\Mrk)$. Hence, the interior $\interior{Z}$ of $Z$ is, by definition, a tetrahedron rectangle. Finally, by \reflem{correspondence}, $g(\interior{Z})$ is a tetrahedron rectangle containing $R $ for some $g\in D.$ 
\end{proof}

We summarize the above results in the following form. 

\begin{prop}\label{Prop:fullcorrespondence}
Let $\sV=\{\sL_1,\sL_2\}$ be a veering pair with a marking $\Mrk$. Let $p:\widetilde{\fW^\circ}(\sV,\Mrk)\to \fW^\circ(\sV,\Mrk)$ be the universal covering map. Then, $p$ induces a one-to-one correspondence:
\[
\{\text{Rectangles in }\widetilde{\fW^\circ} (\sV,\Mrk)\}/{\sim} \overset{\mathfrak{p}}{\longrightarrow} \{\text{Rectangles in }\fW^\circ (\sV,\Mrk)\}
\]
where two rectangles in $\widetilde{\fW^\circ} (\sV,\Mrk)$ are equivalent if there is a deck transformation that maps one to the other. Moreover, this correspondence respects being cusp, red/blue edge, face or tetrahedron.
\end{prop}
\begin{proof}
The correspondence $\mathfrak{p}$ is simply to map an equivalence class of rectangles $R$ in $\widetilde{\fW^\circ}(\fV,\Mrk)$ to $p(R)$. $\mathfrak{p}$ is well-defined by \reflem{correspondence}. It suffices to show that this correspondence respects the types of rectangles. 

We first show that $\mathfrak{p}$ sends tetrahedron rectangles to tetrahedron rectangles. According to \reflem{containedintetra} and its proof, a tetrahedron rectangle in $\fW^\circ(\sV,\Mrk)$ can be lifted to a tetrahedron rectangle in $\widetilde{\fW^\circ}(\sV,\Mrk)$. Conversely, a tetrahedron rectangle in $\widetilde{\fW^\circ}(\sV,\Mrk)$ is mapped to a tetrahedron rectangle by the same argument of the proof of \reflem{correspondence}. 

For the remaining cases, we use \reflem{containedintetra} and the one-to-one correspondence between tetrahedron rectangles. 
\end{proof}

Now we can prove the main theorem of this section.
\begin{proof}[Proof of \refthm{loomspaceconstruction}]
As mentioned in the beginning of this subsection, it is enough to show that rectangles in $\widetilde{\fW^\circ}(\sV,\Mrk)$ satisfy the properties in Definition \ref{loomdef}.

Let $R$ be a cusp rectangle in $\widetilde{\fW^\circ}(\sV,\Mrk)$ with a cusp side $e$. According to \refprop{fullcorrespondence}, $p(R)$ is the corresponding cusp rectangle in $\fW^\circ(\sV,\Mrk)$ with the cusp edge $p(e)$. By \reflem{propertiesofweaving}, we know that there is a rectangle $R'$ in $\fW^\circ(\sV,\Mrk)$ that contains $p(e)$. Then the corresponding equivalence class of rectangles in $\widetilde{\fW^\circ}(\sV,\Mrk)$ has a representative $\widetilde{R'}$ that contains $e$. Hence we have the first property. 

The second property is already proven in \reflem{containedintetra}.

For the last property, let $a,b,c$ and $d$ be parameters for a tetrahedron rectangle $R$ in $\widetilde{\fW^\circ}(\sV,\Mrk)$. Then $p(R)$ is a tetrahedron rectangle in $\fW^\circ(\sV,\Mrk)$ with the same parameters $a,b,c$ and $d$. By \reflem{propertiesofweaving}, we know that $a\ne c$ and $b\ne d$. This completes the proof of the theorem.
\end{proof}

\part{Veering Pair Preserving Laminar Groups}\label{grouppart}

In this part, we introduce laminar automorphisms and show that any groups of laminar automorphisms of a given veering pair is an irreducible 3-orbifold group. On the way of the proof, we prove the classification theorems for the actions of laminar automorphisms of pseudo-fibered pairs, e.g. \refthm{classification}, \refthm{elementary}, and \refthm{gapStab}.

\section{Laminar Automorphisms}
We begin with defining terms related to laminar groups.

Let $g$ be a homeomorphism on $S^1.$ Then we say that $g$ is \emph{orientation preserving} if for any clockwise triple $(x_1,x_2,x_3)$ in $S^1,$ $(g(x_1),g(x_2),g(x_3))$ is clockwise. We denote the group of orientation preserving homeomorphisms on $S^1$ by $\Homeop(S^1).$

Let $g$ be an element in $\Homeop(S^1)$. We denote the \emph{fixed point set} of $g$ on $S^1$ by $\Fix{g}$ and the \emph{periodic point set} of $g$ by $\Per{g}$. When $p$ is a fixed point of $g$, $p$ is an \emph{attracting fixed point} of $g$ if there is a good interval $I$ such that $I\cap \Fix{g}=\{p\}$ and $g(\closure{I}) \subset I$. Likewise, $p$ is a \emph{repelling fixed point} of $g$ if there is a good interval $I$ such that $I\cap \Fix{g}=\{p\}$ and $ I\subset g(\closure{I})$.

Assume that $g$ is a non-trivial element in $\Homeop(S^1).$ Then, we say that
\begin{itemize}
    \item $g$ is \emph{elliptic} if $g$ is of finite order.
    \item $g$ is \emph{parabolic} if $g$ has a unique fixed point.
    \item $g$ is \emph{hyperbolic} if $g$ has exactly two fixed points such that one is attracting and the other is repelling. 
\end{itemize}
We also say that $g$ is \emph{M\"obius-like} if $g$ is either elliptic, parabolic, or hyperbolic.

Let $\sL$ be a lamination system. We call an element $g$ in $\Homeop(S^1)$ a \emph{laminar automorphism} of $\sL$ if $g$ \emph{preserves}  $\sL,$ namely, $g(\sL)=\sL.$ Also, we define the \emph{laminar automorphism group} of $\sL$ to be the set of all laminar automorphisms of $\sL$ and denote it by $\Aut(\sL).$ 

A subgroup $G$ of $\Homeop(S^1)$ is said to be \emph{laminar} if there is a lamination system $\sL$ such that $G\leq \Aut(\sL).$ 

Let $\sC=\{\sL_\alpha\}_{\alpha\in \Gamma}$ be a collection of lamination systems. Then the \emph{laminar automorphism group} $\Aut(\sC)$ of $\sC$ is
$$\bigcap_{\alpha \in \Gamma} \Aut(\sL_\alpha),$$
and each element of $\Aut(\sC)$ is called an \emph{laminar automorphism} of $\sC.$ 

When $g$ is a non-trivial laminar automorphism of a pseudo-fibered pair $\sC,$ $g$ is said to be \emph{pseudo-Anosov-like} or \emph{pA-like} if $g$ is neither elliptic nor parabolic and $g$ preserves an interleaving pair of $\sC.$  In particular, a pA-like homeomorphism with non-empty fixed point set is called a \emph{properly pseudo-Anosov}.

\section{Laminar Automorphisms of Quite Full Lamination Systems}\label{Sec:laminarpropertysec}

In the subsequent  sections, we analyze the circle actions of laminar automorphisms of pseudo-fibered pairs. Before that, we study dynamics of laminar automorphisms preserving a single quite full lamination system. 

\subsection{Obstruction from the Rotation number}

In this subsection, we show that any element $g$ in $\Homeop(S^1)$ preserving a quite full lamination system has a rational rotation number. In other words, there is a periodic point. See~\cite[Chapter~11]{Katok95} and~\cite[Section~5]{Ghys01} for the rotation number theory.

First, we observe that any irrational rotation can not preserve a lamination system.

\begin{prop}\label{Prop:irrationalRot}
Let $g\in \Homeop(S^1).$ Suppose that $g$ is topologically transitive. Then, there is no lamination system $\sL$ such that $g\in \Aut(\sL)$.  
\end{prop}
\begin{proof}
Let $\tau(g)$ be the rotation number of $g.$ Since $g$ is topologically transitive, $\tau(g)$ is irrational and there is a $\phi$ in $\Homeop(S^1)$ such that $\phi^{-1} g \phi= R_{\tau(g)}$ where for each $\alpha\in \RR,$ $R_\alpha$ is the map defined as 
$$R_\alpha(z)=e^{2\pi i\alpha}z.$$

Now assume that there is a lamination system $\sL$ such that $g\in \Aut(\sL)$. Then $\phi^{-1}(\sL)$ is also a lamination system and  $R_{\tau(g)}\in \Aut(\phi^{-1}(\sL))$. Let $\opi{a}{b}$ be a good open interval in $\phi^{-1}(\sL)$. Since $R_{\tau(g)}$ is minimal, there is a sequence $\seq{r_n}$ in $\ZZ\setminus 0$ such that $\seq{(R_{\tau(g)})^{r_n}(a)}$ converges to $a.$ Then $\{a,b\}$ and $\{(R_{\tau(g)})^{r_k}(a), (R_{\tau(g)})^{r_{k}}(b)\}$ are linked for some $k\in \NN.$ It is a contradiction by the unlinkedness. Thus, there is no lamination system $\sL$ such that $g\in \Aut(\sL).$ 
\end{proof}

Then, we show the main lemma in this section.

\begin{lem}\label{Lem:rationalRot}
Let $g\in \Homeop(S^1)$ and $\sL$ a quite full lamination system. Suppose that $g\in \Aut(\sL)$. Then there is a periodic point, that is, $\Per{g}\neq \emptyset.$
\end{lem}
\begin{proof}
Suppose that $\Per{g} = \emptyset$. Then the rotation number $\tau$ of $g$ is irrational. Due to \refprop{irrationalRot}, $g$ is not minimal. Hence, there is the $g$-invariant exceptional minimal set $K$ which is a cantor set. Thus, by collapsing the components of $S^1\setminus K$,  we obtain a semi-conjugacy $\phi:S^1\to S^1$ such that $\phi g = R_\tau \phi $, where $R_\tau$ is the rotation by $\tau$.

We claim that for any leaf $\ell$ of $\cC(\sL)$, $\phi(\ell)$ is a singleton. For, if $\phi(\ell)$ is not a singleton, we can find $n \in \ZZ$ such that $R_{\tau}^n\phi(\ell)$ and $\phi(\ell)$ are linked. This implies that $\ell$ and $g^n(\ell)$ are linked. It is a contradiction.

From the above claim, every leaf of $\sL$ lies on a component of $S^1\setminus K$.  This contradicts the fact that $\sL$ is quite full. Thus, $\Per{g}\neq \emptyset$.

% We claim that for any $I\in \sL$, $\phi(I)$ is a single point. For if $\phi(I)$ is not a singleton, we find $n\in \ZZ$ such that $R_{\tau} ^n\phi(I)$ and $\phi(I)$ are linked. Since $\phi g^n (I) =  R_\tau ^n \phi(I)$, $I$ and $g^n (I)$ are also linked, a contradiction. Moreover, $\bigcup_{I\in \sL} \phi(I)$ is not the whole $S^1$. This shows that $\sL$ lies in a disjoint countable union of open good intervals, violating the fact that $\sL$ is quite full.
\end{proof}

\subsection{Laminar Automorphisms That Preserve Gaps}

If an element of a laminar group preserves a gap, its dynamics on $S^1$ becomes very restrictive.

\begin{prop}\label{Prop:fixingCrown}
Let $\sL$ be a quite full lamination system and let $\sG$ be a crown. Let $g\in \Aut(\sL)$. If $g$ preserves $\sG$, then the pivot of $\sG$ is a fixed point of $g$. Conversely, if $g$ fixes the pivot of $\sG$, then $g$ preserves $\sG$.
\end{prop}
\begin{proof}
$g$ must fix the pivot of $\sG$ as the pivot is the only accumulation point of $v(\sG)$. The converse statement follows from the fact that there are no distinct crowns sharing the pivot point.
\end{proof}

\begin{lem}\label{Lem:elliptic}
Let $\sL$ be a quite full lamination system and $g$ a laminar automorphism of $\sL.$ Suppose that there is a rainbow point $p.$ If $g$ is elliptic, then there is an ideal polygon $\sG$ preserved by $g.$ In particular, the order of $g$ divides the number of elements of $\sG.$
\end{lem}
\begin{proof}
Let $\sigma$ be the order of $g.$ Since $g$ is non-trivial, $\sigma >1.$ Let $\sO$ be the orbit of $p$ under the iteration of $g.$ Then $|\sO|=\sigma. $

First we consider the case where $\sigma>2.$ Then $S^1\setminus \sO$ is a disjoint union of good intervals. Then there are two good intervals $\opi{u}{p}$ and $\opi{p}{v}$ which are components of $S^1\setminus \sO.$ Since $u$ and $v$ are in $\sO$ and $\sigma >2,$ $u$ and $v$ are distinct. Hence, $\opi{u}{v}$ is good. 

Now, we consider the stem $\stem{p}{\opi{u}{v}}.$ As $p$ is a rainbow point, $\stem{p}{\opi{u}{v}}$ is not empty. Therefore, there is the base $B$ of $\stem{p}{\opi{u}{v}}.$ Moreover, since $u$ and $v$ are also rainbow points, $\closure{B} \subset \opi{u}{v},$ so   $B^*$ is isolated. Therefore, by \refprop{gapExist}, there is a non-leaf gap $\sG$ containing $B.$ 

\begin{claim}\label{Clm:noJump}
Let $\ell$ be a leaf of $\sL.$ Then there is a $I$ in $\ell$ such that $|I\cap \sO|\leq 1.$ 
\end{claim}
\begin{proof}
Note that  $v(\ell) \cap \sO =\emptyset$ since every point in $\sO$ is a rainbow point. The case where $\sigma=3$ is obvious so we assume $\sigma>3.$ 
Suppose that $1<|I\cap \sO|$ for all $I\in \ell.$ Now we write $v(\ell)=\{e,f\}.$ Then for each $w\in v(\ell),$ there is a good interval $\opi{x_w}{y_w}$ containing $w$ which is a component of $S^1\setminus \sO.$ Observe that $\cldi{x_e}{y_e}\cap \cldi{x_f}{y_f}=\emptyset,$  $\{x_e, y_f\} \subset \opi{f}{e},$ and $\{y_e, x_f\}\subset \opi{e}{f}.$
There is a $k$ in $\ZZ$ such that $y_e=g^{k}(x_e).$ Obviously, $y_f=g^{k}(x_f).$ Then $\{e,f\}$ and $\{g^k(e),g^k(f)\}$ are linked. This is a contradiction. Thus, there is a $I$ in $\ell$ such that  $|I\cap \sO|\leq 1.$ 
\end{proof}

Since $u$ is also a rainbow point, by \reflem{trichotomy}, $u\nin v(\sG)$ and there is an element $I_u$ in $\sG$ containing $u.$ Observe that $\{p, v\}\in I_u^*$ since $B$ is the base of $\stem{p}{\opi{u}{v}}.$ By \refclm{noJump}, $\closure{I_u} \cap \sO=\{u\}.$ Note that $u=g^k(p)$ for some $k\in \ZZ.$ Therefore, we have that $I_u\in \stem{u}{g^k(\opi{u}{v})}.$ Since $g^k(B)$ is  the base of $\stem{u}{g^k(\opi{u}{v})},$  $I_u \subset g^k(B).$ If $I_u \neq g^k(B),$ then $\sG$ lies on $g^k(B)$ by \refprop{twoGap}. This implies that the boundary leaf $\ell(B)$ lies on $g^k(\opi{u}{v}).$ However, it is a contradiction by \refclm{noJump} and the fact that $p\notin g^k(\opi{u}{v})$. Therefore, $I_u=g^k(B).$ By \refprop{twoGap}, this implies that $\sG=g^k(\sG).$ Note that by the choice of $k,$ $g^k$ is a generator of $\langle g \rangle.$ Thus, $\sG$ is preserved by $g.$ By \refprop{fixingCrown}, if $\sG$ is a crown, $g$ has a fixed point, but this contradict the assumption that $g$ is elliptic. Hence, $\sG$ is an ideal polygon. 

Now, we consider the case where $\sigma=2.$ Since $p$ is a rainbow point, there is a $I$ in $\sL$ containing $p$ such that $I\cap g(I)=\emptyset.$ We define $S_p$ to be the set 
$$\{J \in \sL : p\in J \text{ and }J \cap g(J)=\emptyset\}.$$
As $I\in S_p,$ $S_p$ is not empty. Observe that for each $J\in S_p,$ $p\in J$ and $g(p) \in J^*,$ and so $S_p$ is totally ordered by inclusion $\subseteq.$ Then $B=\bigcup S_p \in \sL$ and $B\in S_p.$ If $B^*=g(B),$ the leaf $\ell(B)$ is the gap preserved by $g.$ 

Assume that $B^* \neq g(B).$ Now we consider the stem $\stem{B}{g(B)^*}.$ If there is a $K$ in $\stem{B}{g(B)^*}\setminus \{B, g(B)^*\},$ then $K \cap g(K)\neq \emptyset$  as $K\notin S_p.$ Hence,  $g(K)^* \subset K$ and so $g(K)^*\cap K^*=\emptyset.$ Therefore, $g(K)^* \in S_p$ and $g(K)^* \subseteq B.$ This implies that $B^*\subseteq g(K)$ and $g(B)^* \subseteq K$ since $\sigma=2.$ It is a contradiction by the choice of $K.$ Hence, $\stem{B}{g(B)^*}=\{B, g(B)^*\}.$ Thus, $B^*$ is isolated. 

By \refprop{gapExist}, there is a non-leaf gap $\sG$ containing $B.$ Then, there is a $J$ in $\sG$ containing $g(p).$ Observe that $J=g(B).$ Thus $\sG$ is the gap preserved by $g$ and $\sG$ is also an ideal polygon. See \refprop{fixingCrown}.
\end{proof}

\section{Laminar Automorphisms of Pseudo-Fibered Pairs}\label{Sec:classificationsec}
In this section, we begin to study groups acting on the circle preserving pseudo-fibered or veering pairs. 

\subsection{Classification of Automorphisms for a Pseudo-fibered Pair}
In the first subsection, we present various classification results. We first begin with the study of dynamical properties of each individual element. 

\begin{lem}\label{Lem:preserveCrown}
Let $\sC=\{\sL_1, \sL_2\}$ be a pseudo-fibered pair and $g$ a laminar automorphism in $\Aut(\sC)$. If $g$ is parabolic, then for each $i\in\{1,2\}$, there is a crown in $\sL_i$ preserved by $g$ whose pivot is the fixed point of $g$. 
\end{lem}
\begin{proof}
Let $p$ be the fixed point of $g$. Without loss of generality, we assume that $i=1$. First, we consider the case where $p$ is a rainbow point in $\sL_1$. Choose a point $q$ in $S^1\setminus \{p\}$. Since $p$ is a rainbow point, there is a $I$ in $\sL_1$ containing $\{q, g(q)\}$ with $p\in I^*$. Then $v(\ell(I))$ and $v(\ell(g(I)))$ are linked. This is a contradiction. Therefore, $p$ is an endpoint of $\sL$ or the pivot of a crown by \reflem{trichotomy}. 

Assume that $p$ is a vertex of a leaf $\ell$ of $\sL_1$. Then, the distinct leaves $\{g^n(\ell)\}_{n\in \NN}$ share $p$ as a common vertex. This is a contradiction by \refprop{threeLeaves}. Therefore, $p$ is the pivot of a crown $\sG$. By \refprop{fixingCrown}, $g$ preserves $\sG$.
\end{proof}

Now we would like to further analyze the case when we have a pair of lamination systems invariant under a homeomorphism with two or more fixed point. First note that if $g \in \Homeop(S^1)$ has more than two fixed points and $I$ is a connected component of the complement, then $I$ is a good open interval.

\begin{lem}\label{Lem:hyperbolic}
Let $\sC=\{\sL_1, \sL_2\}$ be a pseudo-fibered pair and let $g\in\Aut(\sC)$. If $|\Fix{g}|=2$, then $g$ is hyperbolic and each fixed point of $g$ is a rainbow point in both $\sL_1$ and $\sL_2$. 
\end{lem}
\begin{proof}
Otherwise, it is easy to see that $\Fix{g}$ is in both $\cC(\sL_1)$ and $\cC(\sL_2).$ This contradicts the strongly transversality.
\end{proof}

\begin{lem}\label{Lem:noStickyLeaf}
Let $\sC=\{\sL_1, \sL_2\}$ be a pseudo-fibered pair and let $g\in \Aut(\sC)$. Then, for each leaf $\ell$ of $\sC$, either $v(\ell)\subset \Fix{g}$ or $v(\ell)\cap \Fix{g}=\emptyset.$ 
\end{lem}
\begin{proof}
Without loss of generality, we may assume that $\ell$ is a leaf of $\sL_1.$ Assume that $g$ fixes only one vertex $p$ of $\ell$. Then, the images of $\ell$ under iterations of $g$ forms an infinite collection of leaves which all have $p$ as a common vertex. This is a contradiction by \refprop{threeLeaves}. 
\end{proof}

\begin{lem} \label{Lem:noPivotEnd}
Let $\sL$ be a quite full lamination system and let $g\in \Aut(\sL)$. Assume that $g$ is a non-trivial automorphism with $|\Fix{g}|\geq2$. Let $I=\opi{u}{v}$ be a connected component of the complement of $\Fix{g}$. Then neither $u$ nor $v$ is the pivot of a crown of $\sL$. 
\end{lem}
\begin{proof}
Suppose that $\sG$ is a crown in $\sL$ and $u$ is the pivot of $\sG.$ By \refprop{fixingCrown}, $g$ preserves $\sG.$ Then either $g$ is parabolic or $v(\sG)\subset \Fix{g}$. Since $g$ has at least two fixed points, $g$ can not be parabolic. Hence, $v(\sG)\subset \Fix{g}.$ However, as $v(\sG)\cap I\neq \emptyset$, $I$ must contain a fixed point of $g.$ This is a contradiction since $I\subset S^1 \setminus \Fix{g}$. Therefore, $u$ is not a pivot. Similarly, $v$ cannot be a pivot either. 
\end{proof}

\begin{lem}\label{Lem:twoColored}
Let $\sC=\{\sL_1, \sL_2\}$ be a pseudo-fibered pair and let $g\in \Aut(\sC)$. Suppose that $g$ is a non-trivial automorphism with $|\Fix{g}|>2$. Then, for each component $I$ of $S^1\setminus \Fix{g}$, $|\partial I\cap \E{\sL_i}|=1$ for all $i\in \{1,2\}$.
\end{lem}
\begin{proof}
Let $I = (u, v)_{S^1}$ be a connected component of the complement of $\Fix{g}$.
Our claim is that $|\partial I \cap \E{\sL_i}|= 1$ for each $i = 1, 2$. 

By the fact that $\E{\sL_1} \cap \E{\sL_2} = \emptyset$ and \reflem{noPivotEnd}, $u$ is a rainbow point of one of $\sL_1$ and $\sL_2$. Without loss of generality, let us assume that $u$ is a rainbow point of $\sL_1$. Then there exists $J \in \sL_1$ such that $u \in J$, and one boundary point of $J$, say $x$, is in $I$ and the other boundary point of $J$, say $y$, is in $I^\ast$. We may take such $J$ so that $(v, y)_{S^1} \cap \Fix{g} \neq \emptyset$ since $|\Fix{g}|\geq 3$.    

Replacing $g$ by $g^{-1}$ if necessary, we may assume that $\{g^n(J)\}_{n\in \NN}$ is an increasing sequence of intervals. Since $J^*$ contains at least two fixed point of $g$ including $v,$  $\{g^n(\ell(J))\}_{n\in \NN}$ converges to a leaf $\ell'$ of $\sL_1.$ Since $g$ acts on $I$ as a translation without fixed points and $\{g^n(x)\}_{n\in \NN}$ converges to $v$, $\ell'$ has $v$ as a vertex. Therefore, $v\in \E{\sL_1}.$ 

Now in $\sL_2$, $v$ is not an endpoint and not a pivot, hence by \reflem{trichotomy} it is a rainbow point. By the symmetric argument, then $u$ must be an endpoint of $\sL_2$, which proves our claim. 
\end{proof}

\begin{lem} \label{Lem:ginvariantpairofgaps}
Let $\sC=\{\sL_1, \sL_2\}$ be a pseudo-fibered pair and let $g\in \Aut(\sC)$. Suppose that $g$ is a non-trivial automorphism with $|\Fix{g}|>2.$ Then, there is a unique interleaving pair $(\sG_1,\sG_2)$ of real gaps preserved by $g.$ Furthermore, $v(\sG_1)\cup v(\sG_2)=\Fix{g}.$   
\end{lem}

\begin{proof}
By definition, we need to find a real gap $\sG_1$ in $\sL_1$ and a real gap $\sG_2$ in $\sL_2$ such that $\sG_1$ and $\sG_2$ are preserved by $g$ and interleave. 

Let $I=\opi{u}{v}$ be a connected component of $S^1\setminus \Fix{g}.$ Note that $I$ is a good open interval as $|\Fix{g}|>2.$ Our first claim is that each boundary point of $I$ is isolated in $\Fix{g}.$

By \reflem{twoColored}, we may assume that $u$ is a rainbow point of $\sL_1$ and $v$ is an end point of $\sL_1$. Then there exists $J \in \sL_1$ such that $u \in J$, and one boundary point of $J$, say $x$, is in $I$ and the other boundary point of $J$, say $y$, is in $I^\ast$. We may take such $J$ so that $(v, y)_{S^1} \cap \Fix{g} \neq \emptyset$ since $|\Fix{g}|\geq 3$. By \reflem{noStickyLeaf}, there is another connected component $I_y=\opi{s}{t}$ of $S^1\setminus \Fix{g}$ containing $y.$

Replacing $g$ by $g^{-1}$ if necessary, we may assume that $\{g^n(J)\}_{n\in \NN}$ is a decreasing sequence of intervals. Note that $\closure{g^{n+1}(J)}\subset g^n(J)$ for all $n\in\NN.$ Suppose there exists a fixed point of $g$, say $p$, in the open interval $(y, u)_{S^1}$. Then $\bigcap_{n \in \NN} g^n(J) \in \sL_1$ and $u$ is a boundary point of it, which is impossible since $u$ is a rainbow point in $\sL_1$. This proves that $\Fix{g} \cap (y, u)_{S^1} = \emptyset$, hence $t=u$ and  $u$ is an isolated fixed point. By a symmetrical argument using $\sL_2$, we know that $v$ is an isolated fixed point too. 

Now let $I' = \bigcup_{i \in \NN} g^{-i}(J)$. From the above, we know that $I'=\opi{s}{v}$ and $u$ is the only fixed point of $g$ in $I'$. Furthermore, since $v \in \Fix{g}$, $s$ is also fixed by $g$ by \reflem{noStickyLeaf}. Let us denote $s$ by $u_1$. By applying the above argument to $(u_1, u)_{S^1}$ instead of $I$, we show that $u_1$ is isolated, and there exists $u_2 \in \Fix{g}$, and so on by induction. Similarly, we have fixed points $v_1, v_2, \ldots$ on the other side of $I.$ 

In $\sL_1$, $u$ is a rainbow point and since $v(\ell(I'))=\{u_1,v\},$  there is a leaf $\sL_1$ connecting $v$ to $u_1$. Similarly in $\sL_2$, $v$ is a rainbow point and there is a leaf in $\sL_2$ connecting $u$ to $v_1$. It is possible that $u_2 = v_1$ and $u_3 = v$. In that case, these two leaves form a genuine stitch preserved by $g.$ 

Suppose not. In any case, $u_1$ is a rainbow point of $\sL_2$. Hence there is a leaf in $\sL_2$ connecting $u$ to $u_2$. Similarly, there is also a leaf in $\sL_2$ connecting $v_1$ to $v_3$ (we are not excluding the possibility of $v_3 = u_2$ here). Observe that $\opi{u}{v_1}^*$ is isolated in $\sL_2$ as $\opi{u_2}{u},$ $\opi{u}{v_1},$ and $\opi{v_1}{v_3}$ are in $\sL_2.$ Then, by \refprop{gapExist}, there is a non-leaf gap $\sG_2$ of $\sL_2$ containing $\opi{u}{v_1}.$ Inductively, by \refprop{threeLeaves} and \reflem{quadrachotomy}, $\opi{u_{2n}}{u_{2(n-1)}}$ and $\opi{v_{2n-1}}{v_{2(n+1)-1}}$ are in $\sG_2$  for all $n\in \NN$(here one uses the convention that $u = u_0$ and $v= v_0$). Similarly, there exists a gap $\sG_1$ in $\sL_1$ containing $\opi{u_1}{v}$ and $\sG_1$ contains $\opi{u_{2(n+1)-1}}{u_{2n-1}}$ and $\opi{v_{2(n-1)}}{v_{2n}}$ for all $n\in \NN.$ 

Now there are two cases. First, it is possible that $u_n = v_m$ for some $n, m$. In that case, both $\sG_1$ and  $\sG_2$ are ideal polygons. Otherwise, both $\sG_1$ and $\sG_2$ have infinitely many tips, therefore they are crowns. Since each crown has only one pivot, the only possibility is that both sequences $\{u_i\}$ and $\{v_i\}$ converge to the same point, say $p$, and $p$ is the pivot of $\sG_1$ and $\sG_2$. In any case, by construction, it is obvious that $\sG_1$ and $\sG_2$ are real gaps all of whose tips are fixed by $g$, and $\sG_1$ and $\sG_2$ interleave. 

The claim that the set of vertices of $\sG_1$ and  $\sG_2$ is precisely the set $\Fix{g}$ is obviously by the construction. This ends the proof. 
\end{proof}

To classify a pseudo-fibered pair preserving homeomorphism on the circle, we need a slightly weaker notion of interleaving gaps. Let $\sC=\{\sL_1,\sL_2\}$ be a pair of quite full lamination systems. For convenience, we say that gaps $\sG_1$ and $\sG_2$ of $\sL_1$ and $\sL_2,$ respectively, \emph{weakly interleave} if it is one of the following cases:
\begin{enumerate}
    \item For each $i\in \{1,2\}$ $\sG_i=\{\opi{u_n^i}{u_{n+1}^i}:n\in \ZZ_{n_i}\}$ for some integer $n_i\geq2$ and there is a natural number $d$ dividing $\gcd(n_1,n_2)$ such that for each $i\in \{1,2\}$, whenever $u_n^i\in \opi{u_m^{i+1}}{u_{m+1}^{i+1}}$, $u_{n+(n_i/d)}^i\in \opi{u_{m+(n_{i+1}/d)}^{i+1}}{u_{m+(n_{i+1}/d)+1}^{i+1}}$ (cyclically indexed). 
    \item For each $i\in \{1,2\}$, $\sG_i=\{\opi{u_n^i}{u_{n+1}^i}:n\in \ZZ\}$ and there are natural numbers $d_1$ and $d_2$ such that for each $i\in \{1,2\}$, whenever $u_n^i\in \opi{u_m^{i+1}}{u_{m+1}^{i+1}}$, $u_{n+d_i}^i\in \opi{u_{m+d_{i+1}}^{i+1}}{u_{m+d_{i+1}+1}^{i+1}}$ (cyclically indexed). 
\end{enumerate}
Also, we call the pair $(\sG_1,\sG_2)$ a \emph{weakly interleaving pair} of $\sC.$

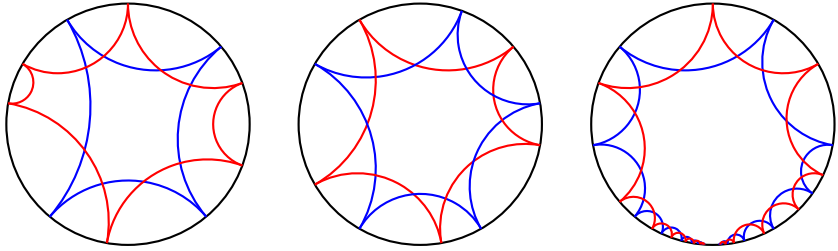
\begin{figure}[htb]
    \centering
    \begin{tikzpicture}[scale=0.8]
    \draw[thick] (0,0) circle (2 cm);
\draw[thick, blue] (	1.532088886	,	1.285575219	) arc (	310	:	210	:	1.678199262	);
\draw[thick, blue] (	-1	,	1.732050808	) arc (	30	:	-40	:	2.856296013	);
\draw[thick, blue] (	-1.285575219	,	-1.532088886	) arc (	140	:	40	:	1.678199262	);
\draw[thick, blue] (	1.532088886	,	1.285575219	) arc (	130	:	220	:	2	);
\draw[thick, red] (	0	,	2	) arc (	0	:	-120	:	1.154700538	);
\draw[thick, red] (	-1.732050808	,	1	) arc (	60	:	-100	:	0.3526539614	);
\draw[thick, red] (	-1.969615506	,	0.3472963553	) arc (	80	:	-10	:	2	);
\draw[thick, red] (	-0.3472963553	,	-1.969615506	) arc (	170	:	70	:	1.678199262	);
\draw[thick, red] (	1.879385242	,	0.6840402867	) arc (	110	:	250	:	0.7279404685	);\draw[thick, red] (	1.879385242	,	0.6840402867	) arc (	290	:	180	:	1.400415076	);
    \end{tikzpicture}\hspace{0.5cm}
    \begin{tikzpicture}[scale=0.8]
    \draw[thick] (0,0) circle (2 cm);
    \draw[thick, red] (	1.532088886	,	1.285575219	) arc (	310	:	210	:	1.678199262	);
\draw[thick, blue] (	0.6840402867	,	1.879385242	) arc (	340	:	240	:	1.678199262	);
\draw[thick,red] (	-1	,	1.732050808	) arc (	30	:	-60	:	2	);
\draw[thick,blue] (	-1.732050808	,	1	) arc (	60	:	-30	:	2	);
\draw[thick,red] (	-1.732050808	,	-1	) arc (	120	:	10	:	1.400415076	);
\draw[thick,blue] (	-1	,	-1.732050808	) arc (	150	:	30	:	1.154700538	);
\draw[thick,red] (	0.3472963553	,	-1.969615506	) arc (	190	:	80	:	1.400415076	);
\draw[thick,blue] (	1.969615506	,	0.3472963553	) arc (	100	:	210	:	1.400415076	);
\draw[thick,red] (	1.532088886	,	1.285575219	) arc (	130	:	260	:	0.9326153163	);
\draw[thick,blue] (	1.969615506	,	0.3472963553	) arc (	280	:	160	:	1.154700538	);
    \end{tikzpicture}\hspace{0.5cm}
    \begin{tikzpicture}[scale=0.8]
    \draw[thick] (0,0) circle (2 cm);
\draw[thick, blue] (	-1.969615506	,	-0.3472963553	) arc (	100	:	-40	:	0.7279404685	);
\draw[thick, blue] (	-1.285575219	,	-1.532088886	) arc (	140	:	-25	:	0.2633049952	);
\draw[thick, blue] (	-0.8452365235	,	-1.812615574	) arc (	155	:	-15	:	0.1749773271	);
\draw[thick, blue] (	-0.5176380902	,	-1.931851653	) arc (	165	:	-10	:	0.08732188582	);
\draw[thick, blue] (	-0.3472963553	,	-1.969615506	) arc (	170	:	-7	:	0.05237184314	);
\draw[thick, blue] (	-0.2437386868	,	-1.985092303	) arc (	173	:	-5	:	0.03491012986	);
\draw[thick, blue] (	-0.1743114855	,	-1.992389396	) arc (	175	:	-4	:	0.01745373558	);
\draw[thick, blue] (	0.1046719125	,	-1.99725907	    ) arc (	183	:	4	:	0.01745373558	);
\draw[thick, blue] (	0.1395129475	,	-1.995128101	) arc (	184	:	7	:	0.05237184314	);
\draw[thick, blue] (	0.2437386868	,	-1.985092303	) arc (	187	:	15	:	0.1398536239	);
\draw[thick, blue] (	0.5176380902	,	-1.931851653	) arc (	195	:	30	:	0.2633049952	);
\draw[thick, blue] (	1	,	-1.732050808	) arc (	210	:	55	:	0.4433893253	);
\draw[thick, blue] (	1.638304089	,	-1.147152873	) arc (	235	:	80	:	0.4433893253	);
\draw[thick, blue] (	-1.532088886	,	1.285575219	) arc (	50	:	-80	:	0.9326153163	);
\draw[thick, blue] (	1	,	1.732050808	) arc (	330	:	230	:	1.678199262	);
\draw[thick, blue] (	1	,	1.732050808	) arc (	150	:	260	:	1.400415076	);
\draw[thick, red] (	-1.532088886	,	-1.285575219	) arc (	130	:	-30	:	0.3526539614	);
\draw[thick, red] (	-1	,	-1.732050808	) arc (	150	:	-20	:	0.1749773271	);
\draw[thick, red] (	-0.6840402867	,	-1.879385242	) arc (	160	:	-13	:	0.1223252403	);
\draw[thick, red] (	-0.4499021087	,	-1.94874013	) arc (	167	:	-9	:	0.06984153898	);
\draw[thick, red] (	-0.3128689301	,	-1.975376681	) arc (	171	:	-6	:	0.05237184314	);
\draw[thick, red] (	-0.2090569265	,	-1.989043791	) arc (	174	:	-4.5	:	0.02618143417	);
\draw[thick, red] (	-0.1743114855	,	-1.992389396	) arc (	175	:	-4	:	0.01745373558	);
\draw[thick, red] (	0.1220970791	,	-1.996269597	) arc (	183.5	:	6	:	0.04364015524	);
\draw[thick, red] (	0.2090569265	,	-1.989043791	) arc (	186	:	10	:	0.06984153898	);
\draw[thick, red] (	0.3472963553	,	-1.969615506	) arc (	190	:	25	:	0.2633049952	);
\draw[thick, red] (	0.8452365235	,	-1.812615574	) arc (	205	:	45	:	0.3526539614	);
\draw[thick, red] (	1.414213562	,	-1.414213562	) arc (	225	:	65	:	0.3526539614	);
\draw[thick, red] (	-1.879385242	,	0.6840402867	) arc (	70	:	-50	:	1.154700538	);
\draw[thick, red] (	0	,	2	) arc (	0	:	-110	:	1.400415076	);
\draw[thick, red] (	1.732050808	,	1	) arc (	300	:	180	:	1.154700538	);
\draw[thick, red] (	1.732050808	,	1	) arc (	120	:	245	:	1.041134101	);
\end{tikzpicture}
    \caption{A weakly interleaving pair (left) and interleaving pairs (middle and right)}
    \label{interleavingexample}
\end{figure}

Now, summarizing what we have shown so far, we state the first main result of Part~\ref{grouppart} as follows:
\begin{thm} \label{Thm:classification}
Let $\sC=\{\sL_1, \sL_2\}$ be a pseudo-fibered pair. If $g$ is a non-trivial laminar automorphism in $\Aut(\sC)$, then $g$ falls into one of the following cases:

\begin{enumerate}
    \item \label{Itm:elliptic}$g$ is elliptic and there is a weakly interleaving pair of ideal polygons of $\sC$ preserved by $g.$
    \item \label{Itm:parabolic} $g$ is parabolic and there is a weakly interleaving pair of crowns of $\sC$ preserved by $g.$ Furthermore, the pivot of the crowns is the fixed point of $g.$
    \item \label{Itm:hyperbolic} $g$ is hyperbolic and each fixed point of $g$ is a rainbow point in both $\sL_1$ and $\sL_2.$ 
    \item \label{Itm:pAlike} $g$ is a pA-like automorphism without fixed point and there is an interleaving pair $(\sG_1, \sG_2)$ of real ideal polygons preserved by $g$ such that $v(\sG_1)\cup v(\sG_2)=\Per{g}.$
    \item \label{Itm:pA} $g$ is a properly pseudo-Anosov and there is an interleaving pair $(\sG_1, \sG_2)$ of real gaps preserved by $g$ such that $v(\sG_1)\cup v(\sG_2)=\Fix{g}.$ Moreover, the tips of $\sG_i$ are attracting fixed points of $g$ and the tips of $\sG_{i+1}$ are repelling fixed points of $g$ for some $i\in \{1,2\}$ (cyclically indexed). Therefore, we say that $\sG_i$ is the \emph{attracting gap} of $g$ and $\sG_{i+1}$ is the \emph{repelling gap} of $g.$ 
\end{enumerate}
\end{thm}
\begin{proof}
First, we consider the case where $\Fix{g}\neq \emptyset.$ If $g$ has only one fixed point $p$, then $g$ is parabolic. By \reflem{preserveCrown}, there are crowns $\sG_1$ and $\sG_2$ in $\sL_1$ and $\sL_2,$ respectively, whose pivots are $p$ and that are preserved by $g.$ 

To show that $\sG_1$ and $\sG_2$ weakly interleave, for each $i\in \{1,2\}$, we write $\sG_i=\{\opi{u_n^i}{u_{n+1}^i}:n\in \ZZ\}$.  Observe that there are non-zero integers $d_1$ and $d_2$ such that $d_1$ and $d_2$ have same sign and for each $i\in\{1,2\}$, $g(u_n^i)=u_{n+d_i}^i$ since $g$ acts on $\opi{p}{p}$ as a translation without fixed point. This implies that $\sG_1$ and $\sG_2$ weakly interleave. This gives the case \refitm{parabolic}.

By \reflem{hyperbolic}, the case where $g$ has exactly two fixed points is the case \refitm{hyperbolic}. Then, we consider the case where $|\Fix{g}|\geq 3$. By \reflem{ginvariantpairofgaps}, there is a unique interleaving pair $(\sG_1,\sG_2)$ of real gaps preserved by $g$ such that $v(\sG_1)\cup v(\sG_2)=\Fix{g}.$

Note that for each $i\in\{1,2\}$, each tip of $\sG_i$ is a rainbow point in $\sL_{i+1}$ by \reflem{noPivotEnd} and \reflem{trichotomy}. Fix a tip $t$ of $\sG_1$. There is a unique element $\opi{s_1}{s_2}$ in $\sG_2$ crossing over $t$ and we can take an element $J_t$ in $\sL_2$ crossing over $t$ so that $\closure{J_t}\subset \opi{s_1}{s_2}$. Then, either $\closure{g(J_t)}\subset J_t$ or $\closure{J_t}\subset g(J_t)$ since $\partial J_t\subset S^1\setminus \Fix{g}$. If $\closure{g(J_t)}\subset J_t,$ then $t$ is an attracting fixed point of $g$ and, otherwise, $t$ is a repelling fixed point of $g$. Without loss of generality, we may assume that $t$ is an attracting fixed point. 

As $s_1$ and $s_2$ are rainbow points in $\sL_1,$ we can take two elements $\opi{u_1}{v_1}$ and $\opi{u_2}{v_2}$ in $\sL_1$ crossing over $s_1$ and $s_2,$ respectively such that $v_1\in \opi{s_1}{t}$ and $u_2\in \opi{t}{s_2}.$ Since $t$ is an attracting fixed point, $g(\opi{v_1}{t})=\opi{g(v_1)}{t}\subsetneq \opi{v_1}{t}$ and $g(\opi{t}{u_2})=\opi{t}{g(u_2)}\subsetneq \opi{t}{u_2}$. Hence, we get that $\cldi{u_1}{v_1}\subset g(\opi{u_1}{v_1})$ and $\cldi{u_2}{v_2}\subset g(\opi{u_2}{v_2}).$ Therefore, $s_1$ and $s_2$ are repelling fixed points of $g.$ Inductively, applying similar arguments to tips of $\sG_1$ and $\sG_2$, we can conclude that the tips of $\sG_1$ are attracting fixed points and the tips of $\sG_2$ are repelling fixed points. This implies the case \refitm{pA}.

Now, we consider the case where $g$ has no fixed point. By \reflem{rationalRot}, $\Per{g}$ is not empty. Then, we can take a minimal natural number $k>1$ such that $\Per{g} \subset \Fix{g^k}$. Note that $\Per{g}=\Fix{g^k}$ as $\Fix{g^k} \subset \Per{g}$ by definition. Then, if $\Fix{g^k}=S^1,$ then $g$ is an elliptic automorphism of order $k.$ This is the case \refitm{elliptic} by \reflem{elliptic}. 

Assume that $\Fix{g^k}\neq S^1.$ If $|\Fix{g^k}|=1,$ then $|\Per{g}|=1$ and so $g$ has a fixed point. This is a contraction to the assumption. If $|\Fix{g^k}|=2,$ then $g$ has exactly two periodic points $p_1$ and $p_2$ with $g(p_1)=p_2$ and so $k=2.$ By \reflem{hyperbolic}, $g^2$ is a hyperbolic automorphism with $\Fix{g^2}=\{p_1, p_2\}.$ We say that $p_1$ is the attracting fixed point of $g^2$ and $p_2$ is the repelling fixed point of $g^2.$ Then, the hyperbolic automorphism $g\circ g^2\circ g^{-1}$ also has $g(p_1)=p_2$ as the attracting fixed point. On the other hand, $g\circ g^2\circ g^{-1}=g^2$ and so $p_2$ is also the repelling fixed points of $g\circ g^2\circ g^{-1}$. This is a contradiction. Therefore, $|\Fix{g^k}|>2.$ 

Now, we can apply \reflem{ginvariantpairofgaps} to $g^k.$ Then, there is a unique interleaving pair $(\sG_1, \sG_2)$ of real gaps preserved by $g^k$ with $v(\sG_1)\cup v(\sG_2)=\Fix{g^k}$. Also, $(g(\sG_1),g(\sG_2))$ is a unique interleaving pair preserved by $g\circ g^2 \circ g^{-1}$. By the uniqueness of the interleaving pair and the fact that $g\circ g^2 \circ g^{-1}=g^2$, $g(\sG_i)=\sG_i$ for all $i\in \{1,2\}$. When $\sG_i$ are crowns, $g$ fixed the pivot of $(\sG_1,\sG_2)$. This is a contradiction by assumption. Therefore, $\sG_i$ are ideal polygons preserved by $g$. Also, $\Per{g}=v(\sG_1)\cup v(\sG_2)$ since $\Per{g}=\Fix{g^k}$. This shows the case \refitm{pAlike}. 
\end{proof}

We may simplify \refthm{classification} in the following form. 

\begin{cor}\label{Cor:simpleClass}
Let $\sC=\{\sL_1, \sL_2\}$ be a pseudo-fibered pair. 
Each non-trivial automorphism in $\Aut(\sC)$ is either a hyperbolic automorphism or an automorphism preserving a weakly interleaving pair. 
\end{cor}

\begin{cor}\label{Cor:constantPer}
Let $\sC$ be a pseudo-fibered pair. Then, for each $g\in \Aut(\sC),$ $\Per{g}=\Per{g^n}$ for all non-zero $n\in \ZZ.$
\end{cor}
\begin{proof}
It follows from \refthm{classification}.
\end{proof}

For clarity, we restate \refthm{classification} for automorphism of veering pairs.

\begin{thm}\label{Thm:classOfLaminarAut}
Let $\sV=\{\sL_1, \sL_2\}$ be a veering pair. If $g$ is a non-trivial laminar automorphism in $\Aut(\sV),$ then $g$ falls into one of the following cases:
\begin{enumerate}
    \item $g$ is elliptic and there is an interleaving pair of ideal polygons of $\sV$ preserved by $g.$
    \item $g$ is parabolic and there is an interleaving pair of crowns of $\sV$ preserved by $g.$ Furthermore, the pivot of the crowns is the fixed point of $g.$
    \item $g$ is hyperbolic and each fixed point of $g$ is a rainbow point in both $\sL_1$ and $\sL_2.$ 
    \item $g$ is a pA-like automorphism without fixed point and there is an interleaving pair $(\sG_1, \sG_2)$ of real ideal polygons preserved by $g$ such that $v(\sG_1)\cup v(\sG_2)=\Per{g}.$
    \item  $g$ is a properly pseudo-Anosov and there is an interleaving pair $(\sG_1, \sG_2)$ of real gaps preserved by $g$ such that $v(\sG_1)\cup v(\sG_2)=\Fix{g}.$ Moreover,  $\sG_i$ is the attracting gap of $g$ and $\sG_{i+1}$ is the repelling gap of $g$ for some $i\in \{1,2\}$ (cyclically indexed). 
\end{enumerate}
\end{thm}

\subsection{Relationship with Pants-like $\COL_2$ Pairs}\label{Sec:promotiontocol2}

In this section, we explain how to extend a pseudo-fibered pair to a pants-like pair. As a result, every group which preserves a pseudo-fibered pair is a pants-like $\COL_2$ group. Recall from \cite{Baik15} that for a subgroup $G\leq \Homeop(S^1),$ we say that $G$ is \emph{pants-like $\COL_2$}  if there is a pair $\sC=\{\sL_1, \sL_2\}$ of very full lamination systems satisfying the followings:
\begin{itemize}
    \item $G\leq \Aut(\sC),$
    \item $\sL_1$ and $\sL_2$ are \emph{transverse}, that is, $
    \sL_1 \cap \sL_2 =\emptyset,$ and 
    \item the set $\E{\sL_1}\cap \E{\sL_2}$ is precisely the set of the cusp points of $G$ (i.e., the fixed points of parabolic elements of $G$). 
\end{itemize}
We also call such a pair $\sC$ a \emph{pants-like pair} for $G.$   

For future reference, we start with a pair of laminations which satisfies weaker conditions than those of a pseudo-fibered pair. Let $\sC=\{\sL_1,\sL_2\}$ be a pair of quite full lamination systems that are strongly transverse, namely, $\E{\sL_1}\cap \E{\sL_2}=\emptyset.$ In this section, given a subgroup $G$ of $\Aut(\sC),$ we are going to construct a pair of very full lamination systems $\closure{\sL_1}$ and $\closure{\sL_2}$ which contains $\sL_1$ and $\sL_2,$ respectively,  so that $\{\closure{\sL_1},\closure{\sL_2}\}$ is a pants-like pair for $G.$

In this subsection, we will abuse the language a bit and treat $\sL_1$ and $\sL_2$ as circle laminations instead of lamination systems in the following sense. Given a lamination system $\sL,$ we call a subset $\cG$ of $\cC(\sL)$ a \emph{gap} in $\cC(\sL)$ if there is a gap $\sG$ in $\sL$ such that $\cG=\{\epsilon(\ell(I)):I\in \sG\}.$ We denote $\cG$ by $\cG(\sG).$ 

Fix a subgroup $G$ of $\Aut(\sC).$ We write $\Lambda_1=\cC(\sL_1)$ and $\Lambda_2=\cC(\sL_2)$ for the circle laminations associated with $\sL_1$ and $\sL_2.$ Let $\sG$ be a crown in $\sL_1$ and let $p$ be its pivot. We write $\sG=\{\opi{t_n}{t_{n+1}}:n\in \ZZ\}$ and $\cG=\cG(\sG).$ First, we note that if $g$ is a nontrivial element of $G,$ then $\Stab{G}{g(\sG)}=g\Stab{G}{\sG}g^{-1}$ where for a group $G$ acting on a set $X$ and for any subset $Y$ of $X,$ $\Stab{G}{Y}$ denotes the stabilizer of $Y$ under the $G$-action. Hence, if we add leaves in $\Lambda_1$ dividing $\cG$ into ideal polygons so that the result is $\Stab{G}{\sG}$-invariant, then one can just add the orbit of those leaves under the $G$-action to get a new $G$-invariant circle lamination which contains $\Lambda_1$. 

Only thing we need to be cautious here is that we would like to have $p \in \E{\overline{\sL_1}}$ if and only if $\Stab{G}{\sG}$ contains a parabolic element. Here is how we achieve this. 

If all elements in $\Stab{G}{\sG}$ fix all tips of $\cG$, then it actually does not matter what we do inside the gap $\cG$. Hence, in this case, we add leaves $\{\{t_n, t_{-n}\}:n\in \NN\}$ in $\cG$ so that $\cG$ is decomposed into ideal polygons and there exists a rainbow at $p$ consisting of the new leaves in $\cG.$  

If there exists a parabolic element $h$ in $G$ which fixes the pivot of $\cG$ and preserves $\cG$, 
then one can just add leaves $\{\{t_n,p\}:n\in \ZZ \}$ in $\cG$ so that each tip of $\cG$ is connected to the pivot $p$ in $\cG$ by a leaf. 
Then such leaves are permuted by $h$. In any case, we have the extended $\cG$ is still $\Stab{G}{\sG}$-invariant, so by adding their $G$-orbit, one gets a lamination which contains $\Lambda_1$ and has one less orbit class of crowns. 

Note that there are at most countably many crowns in $\sL_1.$ By repeating this process inductively for all orbit classes of crowns, we get a very full lamination $\overline{\Lambda_1}$ which contains $\Lambda_1$. Now, we say that $\closure{\sL_1}$ is the lamination system associated with $\overline{\Lambda_1}.$ Then, $\sL_1\subset \overline{\sL_1}.$  Similarly, we can obtain an extended lamination system $\closure{\sL_2}$ for $\sL_2.$

Then, we observe the following claim.
\begin{claim}\label{Clm:noRainbow}
If $g$ is a parabolic element of $G$ with $\Fix{g}=\{p\}$, then for each $i\in \{1,2\}$, $p$ is either an end point of $\sL_i$ or the pivot of a crown in $\sL_i$ preserved by $g.$     
\end{claim}
\begin{proof}
%By \reflem{trichotomy} and \refprop{fixingCrown}, we only need to show that $p$ is not a rainbow point in both $\sL_1$ and $\sL_2.$ Assume that $p$ is a rainbow point in $\sL_1.$ Choose a point $q$ in $\opi{p}{p}.$ Since $p$ is a rainbow point, we can take an element $
%I$ in $\sL_1$ containing $\{q, g(q)\}$ with $p\in I^*.$ Then, $v(\ell(I))$ and $v(\ell(g(I)))$ are linked since $g$ acts on $\opi{p}{p}$ as a translation. This is a contradiction. Similarly, we can show that $p$ is not a rainbow point in $\sL_2.$ 
Observe that $p$ can not be a rainbow point in each $\sL_i$. Thus, the claim follows from \reflem{trichotomy}.
\end{proof}

Then, there are two possible case on a cusp point $p$ of $G$ by \refclm{noRainbow} and the fact that $\sL_1$ and $\sL_2$ are strongly transverse. One is that $p$ is a pivot point in both $\sL_1$ and $\sL_2.$ The other case is that there is $i\in\{1,2\}$ such that $p$ is a pivot in $\sL_i$ and an end point of $\sL_{i+1}$ (cyclically indexed). In both cases, $p$ is a common end point of $\E{\closure{\sL_1}}$ and $\E{\closure{\sL_2}}$ from the construction of $\closure{\sL_i}.$

From the construction of $\overline{\sL_i}$, $\E{\overline{\sL_i}} \setminus \E{\sL_i}$ consist of the pivots of crowns in $\sL_i$ which are fixed by some parabolic elements of $G$, i.e., they are cusp points of $G$. A point in $\E{\closure{\sL_1}}\cap \E{\closure{\sL_2}}$ is such a cusp point or a point already in $\E{\sL_1}\cap \E{\sL_2}$ but $\E{\sL_1}\cap \E{\sL_2}=\emptyset$ by assumption. Hence, we get that $\E{\overline{\sL_1}} \cap \E{\overline{\sL_2}}$ is precisely the set of all cusp points of $G$ and the pair $\{ \overline{\sL_1}, \overline{\sL_2}\}$ is a desired one which makes $G$ a pants-like $\COL_2$ group.

\subsection{Classification of Elementary Groups}

Let $\sC=\{\sL_1, \sL_2\}$ be a pseudo-fibered pair and $G$  a subgroup of $\Aut(\sC).$ Let $\sG$ be a gap of $\sC$, and we now study the structure of $\Stab{G}{\sG}$. In the case where $\sG$ is a crown, an element of $\Stab{G}{\sG}$ is called a \emph{guardian of the crown} $\sG$, or more shortly just a GOC element. Let $\closure{\sC}=\{\overline{\sL_1}, \overline{\sL_2} \} $ be a pants-like pair for $G$ constructed in  \refsec{promotiontocol2}. The collection of ideal polygons in $\overline{\sL_i}$ which is just a decomposition of a crown in $\sL_i$ is called a \emph{fractured crown}. More generally, we call the process of obtaining $\overline{\sL_i}$ from $\sL_i$ described in \refsec{promotiontocol2} \emph{fracturing $\sL_i$}. Note that any element of $G$ which preserves a fractured crown in $\overline{\sL_i}$ is a GOC element in the perspective of $\sL_i$. 

Recall that a descending sequence of elements $\{I_n\}_{n\in \NN}$ in a lamination system is called a rainbow at $p$ if $\bigcap_{n\in\NN} I_n=\{p\}$ for some $p\in S^1$. If a descending sequence of elements $\{I_n\}_{n=1}^\infty$ in a lamination system satisfies that $\bigcap_{n\in \NN} \overline{I_n}=\{p\}$ for some $p\in S^1$, then we call it a \emph{quasi-rainbow} at $p$. Note that in our case, all quasi-rainbows in $\sL_i$ are actually rainbows by \refprop{threeLeaves}, but in the fractured lamination systems $\overline{\sL_i}$, there are new quasi-rainbows at pivot points of fractured crowns.

For a lamination system $\sL$ and for a subgroup $G$ of $\Aut(\sL)$, suppose that there exists a quasi-rainbow  $\{I_n\}_{n=1}^{\infty}$ at $x \in S^1$. A sequence $\{(g_n, I_n)\}_{n=1}^{\infty}$  of  elements of $G \times \sL$ is called a \emph{pre-approximation sequence} at $x$ if 
there is a point $y$ in $S^1$ such that $g_n(y)\in I_n$ for all $n\in \NN.$

Since $\{\overline{\sL_1}, \overline{\sL_2}\}$ is a pants-like pair for $G$ in the sense of \cite{Baik15}, the following lemma from \cite{BaikKim20} holds. 

\begin{lem}[\cite{BaikKim21}] \label{Lem:presequence}
Suppose that we have a sequence $\{x_n\}_{n=1}^{\infty}$ of elements of $S^1$ converging to $x\in S^1$ and a sequence $\{g_n\}_{n=1}^{\infty}$ of distinct elements of $G$ such that $\{g_n(x_n)\}_{n=1}^\infty$ converges to $x'\in S^1$. Then we can have one of the following cases:
\begin{enumerate}
\item there is a subsequence $\{g_{n_k}\}_{k=1}^{\infty}$ such that $g_{n_k}(x)=x'$ for all $k\in\NN$; 
\item there is a subsequence $\{g_{n_k}\}_{k=1}^{\infty}$ and a quasi-rainbow $\{I_k\}_{k=1}^\infty$ at $x'$ in $\overline{\sL_i}$ for some $i\in \{1,2\}$ such that the sequence $\{(g_{n_k},I_k)\}_{k=1}^{\infty}$ is a pre-approximation sequence at $x'$;
\item  there is a subsequence $\{g_{n_k}\}_{k=1}^{\infty}$ and  a quasi-rainbow $\{I_k\}_{k=1}^\infty$ at $x$ in $\overline{\sL_i}$ for some $i\in \{1,2\}$  such that the sequence $\{(g_{n_k}^{-1},I_k)\}_{k=1}^{\infty}$ is a pre-approximation  sequence at $x$. 
\end{enumerate} 
\end{lem} 

Here is the key lemma of the current section. 

\begin{lem}\label{Lem:discreteAction}
Let $\sC=\{\sL_1, \sL_2\}$ be a pseudo-fibered pair and $G$ a subgroup of $\Aut(\sC)$. Assume that $G$ is not trivial. If $G$ preserves an open interval $I=\opi{u}{v}$ and faithfully and freely acts on $I$, then there is an automorphism $g$ generating $G$ so that $G$ is isomorphic to $\ZZ$ and acts on $I$ as a translation without fixed point. 
\end{lem}
\begin{proof}
Choose a point $p$ in $I.$ The set $G\cdot p$ is the orbit of $p$ in $I$ under the $G$-action. Suppose that there is a sequence $\{g_n\}_{n=1}^\infty$ of distinct elements of $G$ such that the sequence $\{g_n(p)\}_{n=1}^\infty$ converges to a point $q$ in $I.$  By \reflem{presequence}, there is a subsequence $\{g_{n_k}\}_{k=1}^\infty$ 
 of $\{g_n\}_{n=1}^\infty$ and a quasi-rainbow $\{I_k\}_{k=1}^\infty$ in some $\closure{\sL_i}$ such that either $\{(g_{n_k}, I_k)\}_{k=1}^\infty$ is a pre-approximation sequence at $p$ or  $\{(g_{n_k}^{-1}, I_k)\}_{k=1}^\infty$ is a pre-approximation sequence at $q.$ Otherwise, $g_n(p)=g_m(p)=q$ for some $n\neq m.$ This implies that the non-trivial element $g_n^{-1}\circ g_m$ in $G$ has a fixed point $p$ in $I.$ This is a contradiction by assumption. 

 If $\{(g_{n_k}, I_k)\}_{k=1}^\infty$ is a pre-approximation sequence at $p,$ then there is a point $x_p$ in $S^1$ such that $g_{n_k}(x_p) \in I_k$ for all $k\in \NN.$ Then, since $ \bigcap_{k\in\NN}\overline{I}_k=\{p\}$ and $p\in I,$ there is a natural number $N$ such that $\overline{I}_k \subset I$ for all $k>N.$ Fix $k_0>N.$ Then 
 $$x_p \in g_{n_{k_0}}^{-1}(I_{n_{k_0}})\cap g_{n_{k_0+1}}^{-1}(I_{n_{k_0+1}})
 \subseteq g_{n_{k_0}}^{-1}(I_{n_{k_0}})\cap g_{n_{k_0+1}}^{-1}(I_{n_{k_0}}).$$
 Therefore, $$ g_{n_{k_0+1}} g_{n_{k_0}}^{-1}(I_{n_{k_0}})\cap I_{n_{k_0}} \neq \emptyset.$$
 Note that $g_{n_{k_0+1}} g_{n_{k_0}}^{-1}$ is not the identity element since $\{g_n\}_{n=1}^\infty$ is a sequence of distinct elements of $G.$
 Then, by unlinkedness, there are three cases: 
\begin{enumerate}
\item $g_{n_{k_0+1}} g_{n_{k_0}}^{-1}(I_{n_{k_0}})\subseteq I_{n_{k_0}};$
\item $I_{n_{k_0}} \subseteq g_{n_{k_0+1}} g_{n_{k_0}}^{-1}(I_{n_{k_0}});$
\item $g_{n_{k_0+1}} g_{n_{k_0}}^{-1}(I_{n_{k_0}}^*)\subseteq I_{n_{k_0}}.$
\end{enumerate}
The case where $g_{n_{k_0+1}} g_{n_{k_0}}^{-1}(I_{n_{k_0}}^*)\subseteq I_{n_{k_0}}$ does not occurx since every element in $G$ preserves $I$ and fixes $u$ and $v.$ In the remaining cases, there is a fixed point of $g_{n_{k_0+1}} g_{n_{k_0}}^{-1}$ in $\overline{I}_{n_{k_0}}.$ Hence,  $g_{n_{k_0+1}} g_{n_{k_0}}^{-1}$ has a fixed point in $I.$ This is a contradiction since $I\cap \Fix{g_{n_{k_0+1}} g_{n_{k_0}}^{-1}} =\emptyset$ by assumption. 
In the case where $\{(g_{n_k}^{-1}, I_k)\}_{k=1}^\infty$ is a pre-approximation sequence at $q,$ we can do in a similar way.
Therefore, for any $p$ in $I,$ $G\cdot p$ has no limit point in $I.$
Hence, $G\cdot p$ is a countable closed subset of $I$ since the set $G\cdot p$ is a discrete closed subset of $I.$

Fix $p$ in $I.$ Then, there is a connected component of $I \setminus G\cdot p$ which is $(g(p), p)_{S^1}$ for some $g$ in $G.$ Now, we show that $G$ is generated by $g.$ Let $\langle g \rangle$ be the subgroup of $G$ generated by $g.$ Suppose that $h$ is an element in $G\setminus\langle g \rangle.$ Then there is an integer $m$ in $\ZZ$ such that $h(p)\in \cldi{g^{m+1}(p)}{g^{m}(p)}$ since 
$$I=\displaystyle \bigcup_{n\in \ZZ} [g^{n+1}(p),g^{n}(p)]_{S^1}.$$
Therefore, $g^{-m}h(p)\in [g(p), p]_{S^1} \cap G\cdot p =\{g(p), p\}$ and since $G$ acts freely on $I,$ $g^{-m}h=g$ or $g^{-m}h=\id_G.$ Hence, $h=g^{m+1}$ or $h=g^m.$ This is a contradiction since $h\notin \langle g \rangle.$ Thus, $G=\langle g \rangle.$
\end{proof}

\begin{thm}\label{Thm:elementary}
Let $\sC=\{\sL_1, \sL_2\}$ be a pseudo-fibered pair and $G$ a subgroup of $\Aut(\sC).$ If $G$ has an infinite cyclic normal subgroup, $G$ is isomorphic to $\ZZ,$ the infinite dihedral group, $\ZZ\times \ZZ_n$ for some $n\in \NN$ with $n>1,$ or $\ZZ\times \ZZ.$
Furthermore, one of the following cases holds.
\begin{enumerate}
    \item When $G\cong \ZZ,$
    \begin{enumerate}
        \item $G$ is generated by a parabolic automorphism, 
        \item $G$ is generated by a hyperbolic automorphism, or
        \item $G$ is generated by a pA-like automorphism.
    \end{enumerate} 
    \item When $G$ is isomorphic to the infinite dihedral group, $G$ is generated by an hyperbolic automorphism $g$ and an elliptic automorphism $e$ of order two, and so $G=\langle  g, e \ | \ e^2=1,\ ege=g^{-1} \rangle.$  
    \item When $G\cong \ZZ \times \ZZ_n$ for some $n\in\NN$ with $n>1,$ $G$ preserves a unique interleaving pair of  ideal polygons of $\sC$ and  there is a pA-like automorphism $g$ and an elliptic automorphism $e$ of order $n$ such that $G=\langle g \rangle \times \langle e \rangle.$
    \item When $G\cong \ZZ \times \ZZ,$ $G$ preserves a unique asterisk of crowns of $\sC$ and there is a properly pseudo-Anosov $g$ and a parabolic automorphism $h$ such that $G=\langle g \rangle \times \langle h \rangle.$ 
\end{enumerate}
\end{thm}
\begin{proof} 
Let $\langle h\rangle$ be a normal subgroup of $G$ for some infinite order element $h$.
\begin{claim}\label{Clm:preservePer}
Every element in $G$ preserves $\Per{h}.$ Therefore, $\Per{h}\subset \Per{g}$ for all $g\in G.$
\end{claim}
\begin{proof}
Choose $g$ in $G.$ Since $H$ is normal in $G,$ $ghg^{-1}=h^n$ for some non-zero $n\in \ZZ.$ Then, by \refcor{constantPer}, $g(\Per{h})=\Per{ghg^{-1}}=\Per{h^n}=\Per{h}$.
\end{proof}

Given $g\in G$, we have $ghg^{-1} = h^n$ and $g^{-1} h g = h^m$ for some $m,n\in \ZZ$. We compute $h =  g^{-1} h^n g = (h^m)^n= h^{mn}$. Since $h$ is not of finite order, we know that $mn=1$. This shows that for any $g\in G$, we have either $g h g^{-1} =h$ or $ghg^{-1} =h^{-1}$. 

We split the cases according to the type of $h$. Note that the possible types of $h$ are parabolic, hyperbolic and pA-like.

Case 1: $h$ is a  parabolic automorphism preserving a weakly interleaving pair of crowns $(\sG_1,\sG_2)$. 

If all non-trivial elements of $G$ are parabolic, then they share the same fixed point as $h$. We apply \reflem{discreteAction} to conclude that $G$ is cyclic. Suppose that there is a non-trivial $g\in G$ that has more than one fixed point. Observe that $g$ must be properly pseudo-Anosov with $\Fix{g} = v(\sG_1)\cup v(\sG_2)$; otherwise $g$ maps the $h$-invariant crowns to another crowns, which is not allowed. Hence, $(\sG_1,\sG_2)$ is an interleaving pair. By \reflem{discreteAction}, we know that the set $G_{pA}$ of properly pseudo-Anosovs together with the identity is a cyclic group generated by a properly pseudo-Anosov $g_0 \in G$.

Let $h_0\in G$ be a parabolic element that has the smallest translation on the set of vertices of $\sG$. We claim that $g_0h_0 g_0 ^{-1} = h_0$. We prove this by mean of contradiction. Hence assume that $g_0 h_0 g_0 ^{-1} = h_0^{-1}$. Let $v$ be a fixed point of $g_0$ that is a tip of $\sG_1$. Then, we have that $h_0 (v) = g_0 h_0 g_0 ^{-1}(v) = h_0 ^{-1}(v)$, which is absurd. Thus, $g_0 h_0 g_0^{-1} = h_0$. 

It is clear that $h_0$ and $g_0 $ generate $G$; for any $g\in G$, since $h_0$ has the smallest translation on $v(\sG_1)$, there is $m\in \ZZ$ such that  $h_0^m g$ fixes $v(\sG_1)\cup v(\sG_2)$ pointwise, i.e., $h_0 ^m g$ is in $G_{pA}$. 
Therefore, $G$ is abelian as $[g_0,h_0]=1$.  Thus, it follows that $G$ is either  $\langle h_0 \rangle \cong \ZZ$ for some parabolic $h_0$ or $G=\langle h_0\rangle \times \langle g_0 \rangle \cong \ZZ\times \ZZ$ for some parabolic $h_0$ and properly pseudo-Anosov $g_0$.

Case 2: $h$ is hyperbolic. 

In this case, any $g\in G$ permutes two fixed points of $h$ and this action gives rise to a homomorphism $\psi: G\to \ZZ_2$. Note that each non-trivial element in $\ker \psi$ can not be pA-like because if it were, two fixed points of $\Fix{h}$ would be tips of a crown or polygon, violating \reflem{hyperbolic} and \reflem{quadrachotomy}. Hence, each non-trivial element of $\ker \psi$ is hyperbolic and $\ker \psi$ acts faithfully and freely on each component of $S^1 \setminus \Fix{h}$.  By \reflem{discreteAction}, $\ker \psi$ is an infinite cyclic group generated by a hyperbolic  $h_0$.  If $\ker \psi = G$, we are done.

If not, we take any $g \in G \setminus \ker \psi$. Since $g^2\in \ker \psi$, $g$ is either elliptic of order two or pA-like. As the kernel has no pA-like automorphism, $g$ is an elliptic element of order two. We claim now that $gh_0 g^{-1} = h_0 ^{-1}$.  Suppose, on the contrary, that $gh_0 g^{-1} = h_0$. Choose a neighborhood $U$ of the attracting fixed point of $h_0$ so that  $g (U) \cap U = \emptyset$. Let $x\in U\setminus \Fix{h_0}$. Choose a large enough $n$ such that $h_0 ^n g^{-1}(x)\in U$ and $h_0 ^n(x) \in U$.  Because $h_0 ^n g^{-1} (x) \in U$ and because $g(U) \cap U = \emptyset$, we have that $gh_0 ^n g^{-1} (x) \notin U$. On the other hand, $gh_0 ^n g^{-1}(x) = h_0 ^n (x) \in U$, a contradiction. Therefore, $gh_0 g^{-1} = h_0 ^{-1}$. We therefore have shown that $G$ is either  $\langle h_0\rangle\cong\ZZ$ for some hyperbolic $h_0$ or the infinite dihedral group $\langle h_0,g \,|\, g^2 = 1,\, gh_0 g^{-1} = h_0^{-1}\rangle$ for some hyperbolic $h_0$ and elliptic $g$ of order two. 

Case 3: $h$ is pA-like with an interleaving invariant pair $(\sG_1,\sG_2)$ of crowns or polygons. 

First assume that $\Per{h}$ is finite, equivalently, $\sG_i$ are polygons. Since any $g\in G$ cyclically permutes $\Per{h}$, we have a homomorphism $\psi: G\to \ZZ_n$ for some $n\in \ZZ$. As $\ker \psi$ faithfully and freely acts on each component of $S^1\setminus \Per{h}$, by \reflem{discreteAction}, we know that $\ker \psi$ is cyclic generated by some properly pseudo-Anosov $h_0$. 

If $G\setminus \ker \psi$ is not empty, we choose $g$ in $G\setminus \ker \psi$ so that $\psi(g)$ generates the cyclic group $G/\ker \psi\le \ZZ_n$. As $g^n\in\ker \psi$, we know that $g$ is either pA-like  or elliptic with the same invariant interleaving pair $(\sG_1,\sG_2)$. Either cases, we claim that $gh_0 g^{-1} = h_0$. By the method of contradiction, we assume that $gh_0 g^{-1} = h_0 ^{-1}$. Pick a neighborhood $U$ near an attracting fixed  point  of $h_0$  so that $U$ avoids small neighborhood of the set of repelling fixed points of $h_0$. Let $x\in U\setminus \Fix{h_0}$. There is large $m\in \ZZ$ such that $h_0 ^m g^{-1}(x) \in g^{-1}(U)$ and $h_0^{-m}(x)\notin U$. We then have that $gh_0 g^{-1} (x) \in U$. But $gh_0 ^m g^{-1} (x) = h_0 ^{-m} (x) \notin U$, a contradiction. This shows that $G$ is an abelian group generated by $h_0$ and $g$. By elementary algebra, we can see that $G$ is either an infinite cyclic group generated by a pA-like automorphism or $\langle f \rangle \times \langle e \rangle \cong \ZZ \times \ZZ_k$ for some pA-like $f$ and some elliptic $e$ of order $k$.

Finally, suppose that $\Per{h}$ is not finite, namely, $\sG_i$ are crowns. Then any nontrivial $g\in G$ is either parabolic or properly pseudo-Anosov. If all non-trivial elements of $G$ are properly pseudo-Anosov, we know that $G= \langle h_0\rangle\cong \ZZ$ for some $h_0\in G$. If not, consider the subgroup $G_{pA}=\{g\in G\,:\, g\text{ is properly pseudo-Anosov or  trivial }\}$ and choose an element $p\in G$ that has the smallest translation on $v(\sG_1)$. $G_{pA}$ is cyclic generated by, say,  $f$. Then by the same argument as in Case 2, we see that $pfp^{-1} = f$. By the same argument as in Case 1, $p$ and $f$ generate $G$. Therefore, $G\cong \langle f \rangle \times\langle p \rangle \cong \ZZ\times \ZZ$. 
\end{proof}

\begin{rmk}
\refthm{elementary} gives the strong Tits alternative questioned in \cite{AlonsoBaikSamperton}. In \cite{AlonsoBaikSamperton}, they showed the Tits alternative for the automorphism group of a pseudo-fibered pair of very full laminations with the countable fixed point condition that every element has at most countable fixed points.
\end{rmk}

\subsection{Classification of Gap Stabilizers}

\begin{thm}\label{Thm:gapStab}
Let $\sC$ be a pseudo-fibered pair and $G$ be a subgroup of $\Aut(\sC).$ Then, the non-trivial stabilizer $\Stab{G}{\sG}$ of a gap $\sG$ of $\sC$ falls into one of the following cases.
\begin{enumerate}
    \item When $\sG$ is an ideal polygon,
    \begin{enumerate}
        \item $\Stab{G}{\sG}$ is generated by an elliptic automorphism $e$ of order $n$ in $\Stab{G}{\sG}$ and so $\Stab{G}{\sG}\cong \ZZ_n.$ 
        \item $\Stab{G}{\sG}$ is generated by a pA-like element in $\Stab{G}{\sG}$ without fixed point and $\Stab{G}{\sG}\cong \ZZ.$
        \item $\Stab{G}{\sG}$ is generated by a properly pseudo-Anosov in $\Aut(\sC)$  and $\Stab{G}{\sG}\cong \ZZ.$
        \item $\Stab{G}{\sG}$ is generated by a pA-like automorphism $g$ and an elliptic automorphism $e$ of order $n$ in $\Aut(\sG)$ and $$\Stab{G}{\sG}=\langle g, e \ | \ [g,e]=1, \ e^n=1 \rangle\cong \ZZ \times \ZZ_n.$$
    \end{enumerate}
    \item When $\sG$ is a crown, 
    \begin{enumerate}
        \item $\Stab{G}{\sG}$ is generated by a parabolic element in $\Aut(\sG)$ and $\Stab{G}{\sG}\cong \ZZ.$
        \item $\Stab{G}{\sG}$ is generated by a properly pseudo-Anosov in $\Aut(\sG)$ and $\Stab{G}{\sG}\cong \ZZ.$
        \item $\Stab{G}{\sG}$ is generated by a properly pseudo-Anosov $g$ and a parabolic element $h$  in $\Aut(\sG)$ and $$\Stab{G}{\sG}=\langle g,h \ | \ [g,h]=1 \rangle \cong \ZZ\times \ZZ.$$
    \end{enumerate}
\end{enumerate}
\end{thm}

\begin{proof}
Let $\sG$ be a gap of $\sC.$ First, we consider the case where $\sG$ is an ideal polygon.  Then, by \refthm{classification}, each non-trivial element in $\Stab{G}{\sG}$ is elliptic or pA-like. 

If there is no pA-like automorphism, then every element in $\Stab{G}{\sG}$ is elliptic. Choose a tip $t$ of $\sG.$ Observe that the orbit $\Stab{G}{\sG} \cdot t$ has at least two points. Then, we may take an automorphism $g$ in $\Stab{G}{\sG}$  such that the good interval $\opi{t}{g(t)}$ is a connected component of $S^1 \setminus \Stab{G}{\sG} \cdot t.$ Let $n$ be the order of $g.$ Then, $$\bigcup_{k=1}^n \opi{g^{k-1}(t)}{g^k(t)}=S^1 \setminus \Stab{G}{\sG} \cdot t$$ and so $\langle g \rangle \cdot t= \Stab{G}{\sG} \cdot t.$ Hence, for any non-trivial $h$ in $\Stab{G}{\sG},$ $h(t)=g^{k_0}(t)$ for some $k_0$ in $\ZZ.$ Then, $g^{-k_0}h$ fixes $t$ and $g^{-k_0}h$ is the identity. Therefore, $h=g^{k_0}$ and $g$ generates $\Stab{G}{\sG}.$ This is the first case where $\sG$ is an ideal polygon.

Assume that there is a pA-like automorphism in $\Stab{G}{\sG}.$ This implies that there is a properly pseudo-Anosov in $\Stab{G}{\sG}$ by powering the pA-like automorphism. Also, by \refthm{classification},  there is a interleaving gap $\sH$ of $\sG.$ Without loss of generality, we may assume that $(\sG,\sH)$ is an interleaving pair.  

Let $H$ be the set of all properly pseudo-Anosov automorphisms in  $\Stab{G}{\sG}$ with the identity. Then, $H$ is a normal subgroup of $\Stab{G}{\sG}.$ Note that by \refthm{classification}, 
$\Fix{g}=v(\sG)\cup v(\sH)$ for all $g\in \Stab{G}{\sG}.$ Then, $H$ faithfully and freely on each connected component of $S^1\setminus (v(\sG)\cup v(\sH)) $ and by \reflem{discreteAction}, there is a properly pseudo-Anosov $h$ generating $H.$ Therefore, $H$ is an infinite cyclic normal subgroup of $\Stab{G}{\sG}.$ Thus, the result follows from \refthm{elementary} and the fact that every non-trivial element in $\Stab{G}{\sG}$ is either elliptic or pA-like. 

Now, we assume that $\sG$ is a crown. Then, by \refthm{classification}, each non-trivial element in $\Stab{G}{\sG}$ is a parabolic element or a properly pseudo-Anosov. 

If there is no properly pseudo-Anosov in $\Stab{G}{\sG}$, then every element in $\Stab{G}{\sG}$ is parabolic. Choose a tip $t$ of $\sG.$ Then, we may take an automorphism $g$ in $\Stab{G}{\sG}$  such that $\opi{t}{g(t)}$ is a connected component of $S^1\setminus \Stab{G}{\sG}\cdot t.$ Note that $\langle g \rangle \cdot t = \Stab{G}{\sG}\cdot t.$ For any $h$ in $\Stab{G}{\sG},$ $h(t)=g^n(t)$ for some $n\in \ZZ.$ As $g^{-n}h$ fixes $t$ and there is no properly pseudo-Anosov in $\Stab{G}{\sG},$ $g^{-n}h=1$ and so $h=g^n.$ Therefore, $\Stab{G}{\sG}$ is generated by the parabolic automorphism $g.$ 

Otherwise, there is a properly pseudo-Anosov in $\Stab{G}{\sG}$ and so there is an interleaving crown $\sH$ of $\sG.$ Without loss of generality, we may assume that $(\sG,\sH)$ is an interleaving pair. Let $H$ be the set of all properly pseudo-Anosov elements in $\Stab{G}{\sG}.$ We can see that $H$ is a normal subgroup of $\Stab{G}{\sG}.$ Note that $\Fix{g}=v(\sG)\cup v(\sH)$ for all non-trivial $g$ in $H.$ Hence, the subgroup $H$ faithfully and freely acts on each connected components of $S^1\setminus (v(\sG)\cup v(\sH)).$ By \reflem{discreteAction}, there is a properly Anosov $g$ generating $H.$ Therefore, $H$ is an infinite cyclic normal subgroup of $\Stab{G}{\sG}.$ Thus, the result follows from \refthm{elementary} and the fact that  every non-trivial element in $\Stab{G}{\sG}$ is a parabolic element or a properly pseudo-Anosov. 
\end{proof}

Recall that for each non-zero integer $m$ and $n$, the \emph{Baumslag-Solitar group} $\operatorname{BS}(m,n)$ is given by the group presentation
$$\langle a, b \ | \ ba^mb^{-1}=a^n \rangle.$$

\begin{cor}\label{Cor:noBSgroup}
Let $\sC=\{\sL_1, \sL_2\}$ be a pseudo-fibered pair. If a Baumslag-Solitar group $\operatorname{BS}(m,n)$ is contained in $\Aut(\sC),$ $m=n=1.$
\end{cor}
\begin{proof}We split the cases. Recall that $\operatorname{BS}(m,n)$ is torsion-free. Hence, $a$ is parabolic, hyperbolic, or pA-like. 

Case 1: $a$ is parabolic. Let $(\sG_1,\sG_2)$ be the $a$-invariant asterisk of crowns. We order the set of tips of $\sG_1$ counter-clockwise so that $a$ maps a tip $t_i$ to $t_{i+d}$ for some $d\in \NN$. Since $\Per{a^n}=\Per{a^m}=\Per{a}$ by \refcor{constantPer}, $b$ permutes $\Per{a}$. Hence, by \refprop{fixingCrown} and \refthm{classification}, $b$ is either a properly pA-like or a parabolic automorphism  with the same invariant asterisk $(\sG_1,\sG_2)$ of crowns. In either cases, choose any tip $t_i$ of $\sG_1$. Then $t_{i+md}=b a^m b^{-1}(t_i) = a^n(t_i) =t_{i+nd}$. This shows that $m=n$ and that $\langle a^m\rangle\cong \ZZ$ is an infinite cyclic normal subgroup of $\operatorname{BS}(m,m)$. Therefore, by \refthm{elementary}, $\operatorname{BS}(m,m)$ is $\ZZ$, $\ZZ\times \ZZ_n$, $\ZZ\times \ZZ$ or an infinite dihedral group. Among them, only $\ZZ\times \ZZ\cong \operatorname{BS}(1,1)$ can be a Baumslag-Solitar group.

Case 2: $a$ is hyperbolic. We know that $b$ permutes two fixed points of $a$. If $b$ fixes $\Fix{a}$ pointwise, we know that $\langle a,b\rangle$ is in fact cyclic, which is not a Baumslag-Solitar group. See the proof of \refthm{elementary}. Hence, $b$ must exchange two fixed points of $a$. However, in this case, $\langle a,b \rangle$ contains an elliptic element of order two, which is a contradiction.

Case 3: $a$ is a pA-like automorphism preserving an asterisk $(\sG_1,\sG_2)$ of $\sC$. Then, by \refprop{fixingCrown} and \refthm{classification}, $b$ is either a parabolic or pA-like automorphism  preserving $(\sG_1,\sG_2)$. Then, by \refthm{gapStab}, $\operatorname{BS}(m,n)\leq \Stab{\operatorname{BS}(m,n)}{\sG_i}$ and so $\operatorname{BS}(m,n)$ is abelian. Hence, we get that $a^m=ba^mb^{-1}=a^n$ and $m=n.$ Therefore, $\langle a^m \rangle$ is an infinite cyclic normal subgroup of $\operatorname{BS}(m,m).$ Therefore, we have that $\langle a,b\rangle \cong \ZZ\times \ZZ\cong \operatorname{BS}(1,1)$ as in Case 1.
\end{proof}
This corollary is a supportive evidence for that $\Aut(\sC)$ is hyperbolic or relatively hyperbolic. See \cite[Q1.1]{Bestvina}.

\subsection{2-Torsions and the Canonical Marking}
We show that given a veering pair $\sV=\{\sL_1, \sL_2\},$ each laminar group $G\leq \Aut(\sV)$ canonically gives rise to a $G$-invariant marking $\Mrk$.

The following lemma shows that no two fixed points of elliptic elements of order two can sit on a same leaf. 

\begin{lem}\label{Lem:etaisinjective}
Let $\sV=\{\sL_1, \sL_2\}$ be a veering pair. Suppose that there are elliptic elements $e_1$ and $e_2$ in $\Aut(\sV).$ If $e_1$ and $e_2$ preserve a leaf $\ell$ of $\sV,$ then both $e_1$ and $e_2$ preserve a common stitch in the thread $\oslash(\ell)$.
\end{lem}
\begin{proof}
Since $e_1$ and $e_2$ preserve a leaf of $\sV$, we know that both $e_1$ and $e_2$ are of order 2. Without loss of generality, we may assume that $\ell$ is a leaf of $\sL_1$. Let $\ell_1$ and $\ell_2$ be leaves of $\sL_2$ that are linked to $\ell$ and preserved by $e_1$ and $e_2$ respectively. See \refthm{classOfLaminarAut}. 

If $\ell_1\ne \ell_2$, the composition $g:=e_1e_2$ is not the identity and has at least two fixed points the vertices of $\ell$. Hence $g$ is either hyperbolic or properly pseudo-Anosov by \refthm{classOfLaminarAut}. Note that $g$ cannot be hyperbolic because no leaf can have hyperbolic fixed points as its vertices. Hence, $g$ is properly pseudo-Anosov like preserving $\ell.$ 

Then, there is a stitch $s=(\ell, m)$ preserved by $g$ such that $\Fix{g}=v(\ell)\cup v(m).$ 

Observe that 
$$g^{-1}=(e_1e_2)^{-1}=e_2^{-1}e_1^{-1}=e_2e_1$$ and 
for each $i\in \{1,2\}$,
$$e_ige_i^{-1}=e_ige_i=e_2e_1=g^{-1}$$ as $e_i$ are of order two.
Then, for each $i \in \{1,2\}$ 
$$e_i\Fix{g}=\Fix{e_ige_i^{-1}}=\Fix{g^{-1}}=\Fix{g}.$$ This implies that $e_i$ preserve $m$ as $\Fix{g}=v(\ell)\cup v(m).$ Therefore, $m=\ell_1=\ell_2$ and it is a contradiction. Thus, both $e_1$ and $e_2$ preserve the same stitch.
\end{proof}

\begin{lem}\label{Lem:veeringpair}
Let $\sV=\{\sL_1, \sL_2\}$ be a veering pair. 
\begin{enumerate}
    \item If an order $2$ elliptic element of $\Aut(\sV)$  preserves a stitch $(\ell_1,\ell_2),$ then each $\ell_i$ is a real leaf of $\sL_i$. Therefore, $(\ell_1,\ell_2)$ is a genuine stitch of $\sV$.
    \item The set of stitches preserved by elliptic elements of $\Aut(\sV)$ is closed and discrete in the stitch space $\fS(\sV)$.
\end{enumerate}    
\end{lem}
\begin{proof}

(1) We show that a leaf $\ell=\{I,I^*\}$ of $\sL_1$ preserved by an order 2 elliptic element $g$ is a real leaf. Since $\ell$ is not isolated by \refprop{realGap}, at least one of $I$ and $I^*$ is not isolated. Without loss of generality, we may say that $I$ is not isolated. Then, there is a $I$-side sequence $\{\ell_i\}_{i\in\NN}$ in $\sL_1.$ Since $g$ swaps $I$ and $I^*$, the sequence $\{g(\ell_i)\}_{i\in \NN}$ is an $I^*$-side sequence. This shows that $\ell$ is a real leaf.

(2) Let $\Mrk$ be the set of stitches preserved by elliptic elements of order 2. Suppose that, in $\fS(\sV)$, $\Mrk$ has an accumulation point $(\ell_\infty ^1, \ell_\infty ^2)$. Let $\{(\ell_i ^1, \ell^2 _i)\}_{i\in \NN}$ be a sequence of distinct stitches in $\Mrk$ with  $(\ell_\infty ^1, \ell_\infty ^2)\ne (\ell_i ^1, \ell^2 _i)$ for all $i\in \NN$ that converges to a stitch $(\ell_\infty ^1, \ell_\infty ^2)$.  By \reflem{etaisinjective}, we know that the sets $\{\ell_i^ 1\}_{i\in \NN}$ and $\{\ell_i ^2\}_{i\in \NN}$ are infinite. Passing to a subsequence, we assume that we have sequences of strictly increasing intervals $K_1^j\subset K_2^j\subset \cdots \subset K_\infty^j $  such that  $\ell_i ^j = \ell(K_i^j)$ for $i\in \{1,2,\cdots,\infty\}$, $j\in\{1,2\}$. Note that, as $\ell_i ^j$ is real, $\closure{K_i ^j}\subset K_{i+1} ^j$ for $i\in \NN$ and $j\in\{1,2\}$. Let $K_i ^j=\opi{x_i ^j}{y_i ^j}$ for $i\in \{1,2,\cdots, \infty\}$ and $j\in\{1,2\}$.  Find disjoint good intervals $I_1,J_1 \in \sL_2$ and $I_2,J_2\in \sL_1$  such that the closures of the good intervals are pairwise disjoint and  $x_\infty ^1\in I_1$, $y_\infty ^1\in J_1$, $x_\infty ^2\in I_2$ and $y_\infty ^2\in J_2$. 

Choose a large enough $N$ such that $x_k ^1\in I_1$, $y_k ^1\in J_1$, $x_k ^2\in I_2$, and  $y_k ^2\in  J_2$ for all $k>N$. We claim the following fact:
\begin{claim}\label{disjoint}
For any $k>N$, $e_k(I_1) \cap I_1=\emptyset$. The same is true for $I_2$, $J_1$ and $J_2$.
\end{claim}
\begin{proof}
Suppose, on the contrary, that $e_k(I_1) \cap I_1\ne \emptyset$ for some $k$. Because $e_{k}(I_1)$ is connected and meets $J_1$ and $I_1$, it contains either $J_2$ or $I_2$ as well. Let us first assume that $e_k(I_1)$ contains $I_2$. Then, since $e_{k}(x_k ^2)=y_k ^2\in J_2$, we have $J_2\cap I_1 \ne \emptyset$, a contradiction. The same argument reveals that $e_k(I_1)$ cannot contain $J_2$ either. Hence, $e_k(I_1)\cap I_1 = \emptyset$ for all $k>N$. 
\end{proof}

We now claim that there is a subsequence $\{e_{n_i}\}_{i\in\NN}$ such that $e_{n_i}(I_1) \cap e_{n_j}(I_1) \ne \emptyset$ for all $i,j$. For this, we consider the set  $\{e_k(x_\infty ^1)\}_{k>N}$. We may assume there is a large enough $M$ with $M\geq N$ that  $e_k(x_\infty ^1)$ is in $J_1$ for all $k>M$. If not, we have $e_k(I_1)\supset J_1$ for infinitely many $k$ and the claim immediately follows. We may also assume that $\{e_k(x_\infty ^1)\}_{k>M}$ is an infinite set. If not, the claim  also follows. Therefore, the set $\{e_k(x_\infty ^1)\}_{k>M}$ accumulates in $\overline{J_1}$. Give a linear order on $\overline{J_1}$ so that $\{y_k^1\}_{k>M}$ is a strictly increasing sequence. Note that  $y_k^1=e_k(x_k ^1) > e_k(x_\infty ^1)$ for all $k>M$. 

Let $z$ be an accumulation point of $\{e_k(x_\infty ^1)\}_{k>M}$. We have two cases. Either $e_k(x_\infty^1)$ has a subsequence $e_{n_i}(x_\infty^1)$  that converges strictly monotonically to $z$ from the below or from the above.  If $e_{n_i}(x_\infty^1)>z$ for all $i$, then $e_{n_i}(I_1)$ contains all $y_{n_j} ^1=e_{n_j}(x_{n_j}^1)\in e_{n_j}(I_1)$ with $j<i$. Therefore, $e_{n_i}(I_1) \cap e_{n_j}(I_1) \ne \emptyset$ for all $i,j$. 

On the other hand, suppose that $e_{n_i}(x_\infty^1)<z$ for all $i$. We first rule out one possible case:

\begin{claim}
If $z= y_\infty ^1$, then we can take a  further subsequence $n_i$ so that $e_{n_i}(I_1) \cap e_{n_j}(I_1) \ne \emptyset$ for all $i,j$.
\end{claim}
\begin{proof}
In fact, we will show that  $y_\infty ^1 \in e_{n_i}(I_1)$ for all sufficiently large $i$. By mean of contradiction, suppose that there is another subsequence such that $y_\infty ^1 \notin e_{n_i}(I_1)$ for all $i$. By Claim \ref{disjoint}, we have $e_{n_i}(J_1)\cap J_1 = \emptyset$ for all $n_i$. From this, we conclude that $e_{n_i}(y_\infty ^1)$ lies in the complement of $I_1\cup J_1$. Therefore, the sequence of leaves $(e_{n_i}(\ell_\infty ^1))_{i\in \NN}$ converges to a leaf different from $\ell_\infty ^1$ but with one of its endpoint $y_\infty ^1=z=\lim_i e_{n_i}(x_\infty ^1)$. This violates the fact that $\ell_\infty ^1 = \ell(K_\infty ^1)$ has the $K_\infty^1$-side sequence $(\ell_k ^1)_{k>N}$. Hence,  $e_{n_i}(I_1)$ contains $y_\infty ^1$ for all large enough $n_i$. In particular, by taking another subsequence, $e_{n_i}(I_1)\cap e_{n_j}(I_1)$ contains $y_\infty ^1$ for all $i,j$ concluding the claim. 
\end{proof} 

Therefore, we only need to handle the case when $z\ne y_\infty ^1$. Since $\{y_{n_i}^1\}_{i\in \NN}$ converges to $y_\infty ^1>z$, we know that there is a large $L$ such that $y_{n_i}^1>z$ for all $i>L$. Therefore, by taking a further subsequence, we know that $e_{n_i}(I_1) \cap e_{n_j}(I_1)$ contains $z$, and therefore is nonempty for each $i,j$.

Now by repeating the above argument for $I_2$, $J_1$ and $J_2$, we know that there are large enough $m>n$ such that  $e_{m}(I_1) \cap e_{n}(I_1)$,  $e_{m}(I_2) \cap e_{n}(I_2)$,  $e_{m}(J_1) \cap e_{n}(J_1)$, and $e_{m}(J_2) \cap e_{n}(J_2)$ are not empty. We now  show that $e_me_n$ is a properly pseudo-Anosov by claiming that $e_me_n$ has at least four fixed points, one for the closure of each $I_1$, $I_2$, $J_1$, and $J_2$ (\refthm{classOfLaminarAut}). In fact, we only show that $e_me_n$ has a fixed point in $\closure{I_1}$. Then the rest of cases follow by symmetry. For this, observe, by Claim \ref{disjoint} and the fact that $e_{m}(I_1) \cap e_{n}(I_1)\ne \emptyset$, that we have either $e_{m}(I_1) \subset e_{n}(I_1)$ or $e_{m}(I_1) \supset e_{n}(I_1)$. But in any case, $e_m e_n$ attains a fixed point in 
$\closure{I_1}$. 

To conclude the proof, we claim that $g:=e_m e_n$ cannot be  a properly pseudo-Anosov. Observe that $g^{-1}=(e_n e_m)^{-1}=e_m^{-1}e_n^{-1}=e_m e_n$ and that $e_m g e_m ^{-1}=e_n g e_n^{-1} =g^{-1}.$ Hence, $$e_m \Fix{g}=\Fix{e_m g e_m^{-1}}=\Fix{g^{-1}}=\Fix{g}$$ and, similarly, $e_n \Fix{g}=\Fix{g}$.
Let $(\sG_1, \sG_2)$ be the interleaving pair preserved by $g.$ Note that by \refthm{classOfLaminarAut}, $\Fix{g}=v(\sG_1)\cup v(\sG_2).$ Observe that, for each $i=1,2$,
$$e_m(\sG_i)=e_m g^{-1}(\sG_i)=e_m(e_m g e_m^{-1})(\sG_i)=g(e_m^{-1}(\sG_i))=g(e_m(\sG_i)).$$
Thus, $\sG_i=e_m(\sG_i)$ by the uniqueness of the interleaving pair preserved by $g.$  Likewise, $e_n$ preserves $\sG_1$ and $\sG_2.$
But  by \refthm{classOfLaminarAut} we have, $(\sG_1, \sG_2)=(\ell_m ^1, \ell_m ^2)=(\ell_n ^1, \ell_n ^2)$. This contradicts the assumption that $(\ell_m^1, \ell_m^2)\neq (\ell_n^1, \ell_n^2).$ 
\end{proof}

Now we can prove the main theorem of this section:
\begin{thm}\label{Thm:canonicalmarking}
Given a veering pair $\sV=\{\sL_1,\sL_2\}$, 
\[
\Mrk:=\{s\in \fS(\sV)\,:\,g(s)=s\text{ for some order 2 elliptic element }g\in \Aut(\sV)\}
\]
is a marking. 
\end{thm}
\begin{proof}
By \reflem{veeringpair}, we know that all elements of $\Mrk$ are genuine and  $\Mrk$ is a  and closed subspace of $\fS(\sV)$. By \reflem{etaisinjective}, we know that $\eta_i|_{\Mrk}$ is injective for each $i\in \{1,2\}$. 
\end{proof}
\begin{cor}\label{Cor:canonicalmarking}
Let $\sV$ be a veering pair. Given a subgroup $G$ of $\Aut(\sV),$
$$\Mrk(G):=\{s\in \fS(\sV): g(s)=s \text{ for some order $2$ elliptic element $g\in G$}\}$$
is a marking. Furthermore, $\Mrk(G)$ is $G$-invariant.
\end{cor}
\begin{proof}
It follows from \refthm{canonicalmarking} since any subset of a marking is also a marking.
\end{proof}

\section{Group Actions on Weavings}\label{Sec:frameactionsec}

In this section we study the laminar group actions preserving veering pairs. We will see that such an action induces the action on the associated loom space and weaving. 

Let $\sV=\{\sL_1,\sL_2\}$ be a veering pair with a marking $\Mrk$.
It is clear that $\Aut(\sV)$ acts on $\fS(\sV)$ as homeomorphisms by the diagonal action. 
An \emph{automorphism} of $(\sV,\Mrk)$ is an automorphism $g$ in $\Aut(\sV)$ preserving $\Mrk$. We denote by $\Aut(\sV,\Mrk)$ the group of automorphisms of $(\sV,\Mrk)$. Note that $\Aut(\sV,\Mrk)< \Aut(\sV)$ and that for any marking $\Mrk'$ with $\Mrk'\subseteq \Mrk$, $\Aut(\sV,\Mrk)< \Aut(\sV,\Mrk')$.

An \emph{automorphism} of $\fW^\circ(\sV,\Mrk)$ is a homeomorphism on $\fW^\circ(\sV,\Mrk)$ that preserves  the orientation of $\fW^\circ(\sV,\Mrk)$ and maps each leaf of $\cF_i^\circ(\sV,\Mrk)$ to a leaf of  $\cF_i^\circ(\sV, \Mrk)$ for each $i\in\{1,2\}$. We denote by $\Aut(\fW^\circ (\sV,\Mrk))$ the group of automorphisms of $\fW^\circ(\sV,\Mrk)$.

% Let $\sV=\{\sL_1, \sL_2\}$ be a veering pair with a marking $\Mrk$. It is clear that $\Aut(\sV)$ acts on $\fS(\sV)$ as homeomorphisms by the diagonal action. Hence, it acts on the marking $\Mrk\subset \fS(\sV)$. An \emph{automorphism} of $(\sV,\Mrk)$ is an element $g$ in $\Aut(\sV)$ preserving $\Mrk$. We denote by $\Aut(\sV,\Mrk)$ the group of automorphisms of  $(\sV,\Mrk)$.  Note that $\Aut(\sV,\Mrk)< \Aut(\sV)$ and that if $\Mrk'$ is a marking with $\Mrk'\subseteq \Mrk,$ $\Aut(\sV,\Mrk)<\Aut(\sV,\Mrk')$. 

% An automorphism of $\fW^\circ(\sV,\Mrk)$ is a homeomorphism on $\fW^\circ(\sV,\Mrk)$ that preserves  the orientation of $\fW^\circ(\sV,\Mrk)$ and maps each leaf of $\cF_i^\circ(\sV,\Mrk)$ to a leaf of  $\cF_i^\circ(\sV, \Mrk)$ for each $i\in\{1,2\}$. We denote by $\Aut(\fW^\circ (\sV,\Mrk))$ the group of automorphisms of $\fW^\circ(\sV,\Mrk)$.

\begin{prop}
Let $\sV=\{\sL_1, \sL_2\}$ be a veering pair with a marking $\Mrk$. The group  $\Aut(\sV,\Mrk)$ faithfully acts on $\fW^\circ (\sV, \Mrk)$ as automorphisms of $\fW^\circ (\sV, \Mrk)$. Moreover, This action gives rise to an isomorphism  $\Aut(\sV,\Mrk)\to \Aut(\fW^\circ(\sV,\Mrk))$.
\end{prop}
\begin{proof}
Let $g$ be an automorphism of $(\sV,\Mrk)$. As for any pair of stitches $s_1,s_2$ with $s_1\frown s_2$, $g(s_1)\frown g(s_2)$. Hence, $g$ induces an orientation preserving homeomorphism of $\fW(\sV)$. Furthermore, the induced homeomorphism preserves $\Mrk$ in $\fW(\sV)$ as $g$ preserves $\Mrk$ in $\fS(\sV)$. Then, the  restriction of  the induced homeomorphism to $\fW^\circ(\sV,\Mrk)$ is an automorphism of $\fW^\circ(\sV,\Mrk)$.

Conversely, let $f$ be an automorphism of $\fW^\circ(\sV,\Mrk)$. We first claim that $f$ can be uniquely  extended to a homeomorphism on $\closure{\fW}(\sV)$ whose restriction to $\fW(\sV)$ preserves the orientation.
We define a map $\hat{f}$ on $\closure{\fW}(\sV)$ as follows. For $x$ in $\closure{\fW}(\sV)$, $\hat{f}(x)=f(x)$ if $x\in \fW^\circ(\sV,\Mrk)$ and, otherwise,  $\hat{f}(x)$ is $\cusp(f(R))$ for some cusp rectangle $R$ with $x=\cusp(R)$. See \refrmk{cuspCorner}. Like in loom spaces, we say that two cusp rectangles $P$ and $Q$ in $\fW^\circ(\sV,\Mrk)$ are \emph{equivalent} if and only if there is a finite sequence of cusp rectangles $P=R_1,R_2,\cdots,R_n=Q$ such that for each pair $(R_i, R_{i+1})$, some cusp side of one is contained in a cusp side of the other. Observe that for any pair  $\{P, Q\}$ of cusp rectangles in $\fW^\circ(\sV,\Mrk)$, $P$ and $Q$ are equivalent if and only if $\cusp(P)=\cusp(Q)$. Therefore, $\hat{f}$ is a well-defined homeomorphism on $\closure{\fW}(\sV)$ preserving $\closure{\cF_1}(\sV)$ and $\closure{\cF_2}(\sV)$.

%%\jung{For each $x\in \fE_i$ pick any $\ell$ with $x\in v(\ell)$ and any $m$ crossing over $x$. Let $s=(\ell,m)$. $\#\oslash(\ell,s,x)$ is a half-leaf. $f(\#\oslash(\ell,s,x))$ is another half-leaf and represented by $\#\oslash(\ell',s',x')$.  $x'$ does not depend of the choices of $\ell$, $s$, $\ell'$ and $s'$. $x\mapsto x'$ is therefore well defined. This was my original argument.}

Now, let $\hat{g}$ be the homeomorphism on $\im{\hat{\omega}}$ defined as   $\hat{g}(x)=\hat{\omega}\circ \hat{f} \circ \hat{\omega}^{-1}(x)$, and for each $i\in \{1,2\}$, let $\fE_i$ be the union of $\E{\sL_i}$ and the pivot points. We then claim that $\hat{g}$ induces a bijection $g^\circ$ on $\fE_1\cup \fE_2$ that preserves the circular order, that is, for any counter-clockwise triple $(x_1,x_2,x_3)$ of elements of $\fE_1\cup \fE_2$, $(g^\circ(x_1),g^\circ(x_2),g^\circ(x_3))$ is counter-clockwise

% Before constructing $g^\circ$, we observe the following claim.
% \begin{claim}
% Let $\mu$ be an element of $\closure{\cF_1}$ and $\sG$ be the real gap of $\sL_1$ such that  $\#^{-1}(\mu)=\bigcup_{I\in \sG } \overt(\ell(I))$. If $\sH$ is the real gap such that $\#^{-1}(\hat{f}(\mu))=\bigcup_{I\in \sH} \overt(\ell(I))$, then there is a homeomorphism $g_\mu$ from $v(\sG)$ to $v(\sH)$ preserving the circular order. Symmetrically, the similar statement holds for the elements of $\closure{\cF_2}$. 
% \end{claim}
% \begin{proof}
% Note that $\closure{\hat{\omega}(\mu)} \cap \partial \cD(C(\sV))=v(\sG)$ and $\closure{\hat{\omega}(\hat{f}(\mu))} \cap \partial \cD(C(\sV))=v(\sH)$. The restriction map $f$ 
 
% \end{proof}

First,  fix $i\in \{1,2\}$ and choose an element $\mu$ in $\closure{\cF_i}(\sV)$. Then, there is a unique real gap $\sG$ in $\sL_i$ such that  $\#^{-1}(\mu)=\bigcup_{I\in \sG } \oslash(\ell(I))$, and $\closure{\hat{\omega}(\mu)}\cap \partial \cD(C(\sV))=v(\sG)$ and $\closure{\hat{\omega}(\mu)}\cap \cD(C(\sV))=\hat{\omega}(\mu)$. Likewise, as $\hat{f}(\mu)\in \closure{\cF_1}(\sV)$,  $\hat{f}(\mu)$ has a unique real gap $\sH$ such that  $\#^{-1}(\hat{f}(\mu))=\bigcup_{I\in \sH } \oslash(\ell(I))$, and $\closure{\hat{\omega}(\hat{f}(\mu))}\cap \partial \cD(C(\sV))=v(\sH)$ and $\closure{\hat{\omega}(\hat{f}(\mu))}\cap \cD(C(\sV))=\hat{\omega}(\hat{f}(\mu))$ . Therefore, the homeomorphism $\hat{g}|\hat{\omega}(\mu)$ from $\hat{\omega}(\mu)$ to $\hat{g}(\hat{\omega}(\mu))=\hat{\omega}( \hat{f}(\mu))$ can be uniquely extended to the homeomorphism $\hat{g}_\mu$ from $\closure{\hat{\omega}(\mu)}$ to $\closure{\hat{g}(\hat{\omega}(\mu))}$. Thus, the restriction $\hat{g}_\mu|v(\sG)$ is the circular order preserving homeomorphism from $v(\sG)$ to $v(\sH)$.

Now, we define $g^\circ$ on $\fE_1$ as follows. For $x$ in $\fE_1$, there is a unique element $\mu_x$ in $\closure{\cF_1}(\sV)$ such that $\closure{\hat{\omega}(\mu_x)}$ contains $x$, and so we define $g^\circ(x)=\hat{g}_{\mu_x}(x)$. In a similar way, we define $g^\circ$ on $\fE_2$. Since $\hat{f}$ is an orientation preserving homeomorphism, by construction, $g^\circ$ preserves the circular order.

Finally, we extend $g^\circ$ to an automorphism $g$ in $\Aut(\sV)$ as follows. For each point $x$ in $S^1\setminus (\fE_1\cup \fE_2)$, $x$ is a rainbow point in both $\sL_1$ and $\sL_2$ and if   $\{I_n\}_{n\in \NN}$ is a rainbow at $x$ in some $\sL_i$, $\{x\}=\bigcap_{n\in\NN} \closure{I_n\cap \fE_i}$. For $x$ in $S^1$,  $g(x)=g^\circ(x)$ if $x\in \fE_1 \cup \fE_2$ and, otherwise, we define $g(x)$ to be the point in $\bigcap_{n\in\NN} \closure{\opi{g^\circ(u_n)}{g^\circ(v_n)}\cap \fE_i}$ for some rainbow $\{\opi{u_n}{v_n}\}_{n\in \NN}$ at $x$ in some $\sL_i$. The map $g$ is well-defined since $\{\opi{g^\circ(u_n)}{g^\circ(v_n)}\}_{n\in\NN}$ is also a rainbow in $\sL_i$. If not, $\{\ell(\opi{g^\circ(u_n)}{g^\circ(v_n)})\}_{n\in\NN}$ converges to a leaf $\ell$ in $\sL_i$. This implies that the sequence of lines $f(\#(\oslash(\ell(\opi{u_n}{v_n}))))$ converges to the line $\#(\oslash(\ell))$. This contradicts to the fact that $\{\opi{u_n}{v_n}\}_{n\in \NN}$ is a rainbow at $x$ as $f$ is an orientation preserving homeomorphism. By the fact that every bijection on $S^1$ preserving the circular order is in $\Homeop(S^1)$, $g$ is in $\Homeop(S^1)$. Therefore, as $g^\circ(\cC(\sL_i))=\cC(\sL_i)$ for all $i\in \{1,2\}$, $g$ also preserves each $\sL_i$ and so $g\in \Aut(\sV)$. Thus, the result follows
\end{proof}

We now show that no nontrivial element of $G$ preserves a rectangle of the regular weaving. At the end this result related to the induced $G$-action on some 3-manifold is properly discontinuous and free. 

\begin{prop}\label{Prop:tetrastabilizer}
Let $\sV=\{\sL_1, \sL_2\}$ be a veering pair and $G$ be a subgroup of $\Aut(\sV).$ The stabilizer of a  rectangle in $\fW^\circ (\sV,\Mrk(G))$ under the $G$-action is trivial where $\Mrk(G)$ is the marking obtained in \refcor{canonicalmarking}.
\end{prop}
\begin{proof}
Let $R$ be a rectangle on $\fW^\circ(\sV,\Mrk(G)).$ Let $g$ be an element in $G$ such that $g(R)=R.$ By \reflem{minimalRep}, there is a minimal frame representative $(I_1, I_2, I_3, I_4)$ of $R.$ Then, $(g(I_1),g(I_2),g(I_3),g(I_4))$ is a minimal representative of $g(R)$ and so $(g(I_1),g(I_2),g(I_3),g(I_4))$ is also a minimal representative of $R.$ By \reflem{minimalRep}, there are two possible cases:
\begin{itemize}
    \item $(I_1, I_2, I_3, I_4)=(g(I_1),g(I_2),g(I_3),g(I_4));$
    \item $(I_3, I_4, I_1, I_2)=(g(I_1),g(I_2),g(I_3),g(I_4)).$
\end{itemize} 

If $(I_1, I_2, I_3, I_4)=(g(I_1),g(I_2),g(I_3),g(I_4)),$ then $g$ fixes leaves $\ell(I_1)$ and $\ell(I_3).$ By the minimality of $(I_1, I_2, I_3, I_4),$ there is no gap of $\sL_1$ that has  $\ell(I_1)$ and $\ell(I_3)$ as boundary leaves. Therefore, by \refthm{classification}, $g$ is the identity. 

When $(I_3, I_4, I_1, I_2)=(g(I_1),g(I_2),g(I_3),g(I_4)),$ $g^2$ is the identity by the first case. Hence, $g$ is an elliptic element of order $2.$ By \refthm{classification}, there is a unique interleaving pair $(\sG_1, \sG_2)$ preserved by $g.$
Note that $\sG_i$ are ideal polygons. 

Now, we claim that $(\sG_1,\sG_2)$ is a genuine stitch lying on $(I_1, I_2, I_3, I_4).$ 
Observe that there are distinct elements $J_1$ and $J_3$ in $\sG_1$ such that $g(J_1)=J_3$ and $I_i\subset J_i$ for all $i\in \{1,3\}.$
Then, $\ell(I_2)$ and $\ell(I_4)$ are linked with $\sG_1.$

Assume that $\sG_1$ is not a leaf. By \refprop{twoGap}, $\ell(I_2)$ and $\ell(I_4)$ cross over tips $t_2$ and $t_4$ of $\sG_1,$ respectively. If $t_2\neq t_4,$ then $v(\ell(J_1))=v(\ell(J_3))=\{t_2,t_4\}$ and it is a contradiction. Hence, $t_2=t_4$ and $t_2\in v(\ell(J_1))\cap v(\ell(J_3))$ 

Now, we write $J_1=\opi{a}{b}$ and $J_3=\opi{u}{v}.$ Note that $g(J_1)=J_3$ and so $g(a)=u$ and $g(b)=v.$ If $t_2=b,$ then $t_2=u$ and $b=u.$ Therefore, 
$$a=g(g(a))=g(u)=g(b)=v$$
and so $J_1=J_3^*.$ This implies that $\sG_1$ is a leaf and it is a contradiction. Likewise, if $t_2=a,$ then $t_2=v$ and $a=v.$ Also,
$$b=g(g(b))=g(v)=g(a)=u.$$
This implies that $\sG_1$ is a leaf and it is a contradiction. Therefore, $\sG_1$ is a leaf with $\sG_1=\{J_1, J_3\}.$

Moreover, since $I_1^*\neq I_3$ and $g(I_1)=I_3,$ $J_i\neq I_i$ for all $i\in \{1,3\}.$ The leaf $\sG_1$ lies between $\ell(I_1)$ and $\ell(I_3).$ Also, by the minimality of $(I_1, I_2, I_3, I_4)$, $\sG_1$ and $\ell(I_i)$ are ultraparallel for all $i\in \{1,3\}.$ Therefore,  $\sG_1$ properly lies between $\ell(I_1)$ and $\ell(I_3).$

Similarly, we can show that $\sG_2$ is a leaf which properly lies between $\ell(I_2)$ and $\ell(I_4).$ By \reflem{veeringpair}, $(\sG_1, \sG_2)$ is a genuine stitch lying on $(I_1, I_2, I_3, I_4).$ 
Then, $\#((\sG_1,\sG_2))\in R.$ 
However, it is a contradiction since $(\sG_1,\sG_2)\in \Mrk(G).$
Therefore, $g$ can not be an elliptic element of order two. Thus, $g$ is the identity.
\end{proof}

\section{Group Actions on Triangulated 3-Manifolds}\label{Sec:triangulationactionsec}

In this section, we study the action of laminar groups on 3-manifolds with veering triangulation.

\subsection{From Loom Spaces to Veering Triangulations}

We recall definitions of veering triangulations and functors between the category of loom spaces and the category of veering triangulations on $\RR^3$. We refer readers to \cite{SchleimerSegerman19}, \cite{SchleimerSegerman20}, and \cite{SchleimerSegerman21} for details. 

A model veering tetrahedron is a tetrahedron with the following extra data: 
\begin{itemize}
    \item (Co-orientations) two faces are oriented outward and the other two are oriented inward. 
    \item (Dihedral angles) two faces with the same orientation meet at angle $\pi$ and faces with opposite orientations meet at angle 0. 
    \item (Colors on edges) View the tetrahedron from the top. The equatorial edges are colored blue or red such that on each visible top face, we see $\pi$-edge, red-edge, and blue-edge in the counter-clockwise order. Colors of $\pi$-edges are indefinite.
\end{itemize}
Figure~\ref{model} depicts a model veering tetrahedron. 
\begin{figure}[!ht]
\begin{tikzpicture}[scale=2]
\draw[thick, dashed] (-1,0) -- (0.5,0) node[below]{$\pi$}-- (1,0);
\draw[thick] (0,1) -- (0,0.5) node[left]{$\pi$} -- (0,-1);
\draw[thick, red] (0,1) -- (-0.5, 0.5) node[left]{0}--(-1,0);
\draw[thick, blue] (-1,0)-- (-.5,-.5) node[left]{0} -- (0,-1);
\draw[thick,red] (0,-1) -- (.5, -.5) node[right]{0}--(1,0);
\draw[thick, blue] (1,0) --(.5, .5) node[right]{0} --(0,1);
\end{tikzpicture}
\caption{A model veering tetrahedron, top view. }\label{model}
\end{figure}
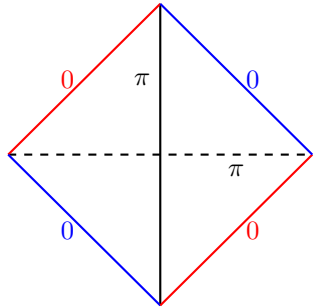

\begin{defn}
Let $M$ be a 3-manifold with boundary. A (transverse) veering triangulation $\cV$ on $M$ is an ideal triangulation with extra data:
\begin{itemize}
    \item (Transversality) each face is given an orientation such that for each ideal tetrahedron, two of its faces are oriented outward and other two are oriented inward.  
    \item (Taut angle structure) each interference of two faces of $\cV$ is given an angle $\pi$ or $0$ in such a way that the dihedral angle sum around each edge is $2\pi$.
    \item (Edge coloring) each edge is given a color either red or blue such that for each ideal triangle $t\in \cV$, there is a cellular map from a model veering tetrahedron onto $t$ that preserves co-orientations, dihedral angles and colors. 
\end{itemize}
\end{defn}

Note that some authors omit the transversality condition in veering triangulation. 

We summarize some notations used in \cite{SchleimerSegerman19}.
\begin{itemize}
    \item The veering functor $\fV$ is the functor from the category of loom spaces to the category of veering triangulations on $\RR^3$.
    \item  The loom functor $\fL$ is the functor from the category of veering triangulations on $\RR^3$ to the category of loom spaces. 
    \item For a tetrahedron rectangle $R$ in a loom space $\cL$, $\mathsf{c}(R)$ denotes the corresponding tetrahedron in $\fV(\cL)$. We define $\mathsf{c}$ for cusp, edge and face rectangles in the same manner. 
    \item For a tetrahedron $t$ in a veering triangulation $\cV$ of $\RR^3$, $\fR(t)$ denotes the corresponding tetrahedron rectangle in the loom space $\fL(\cV)$. We define $\fR$ for ideal vertices, edges and faces in the same manner. 
\end{itemize}

The definition of $\mathsf{c}(c)$ for a cusp $c$ in a loom space deserves a comment. Recall that a cusp is defined as an equivalence class of cusp rectangles. Hence, $\mathsf{c}(c)$ should be defined first for a cusp rectangle $R$ in the equivalence class $c$ by the common ideal vertex of the set of edges $\mathsf{c}(R_1)\cap \mathsf{c}(R_2)$ where $R_1$ and $R_2$ are edge rectangles containing $R$. It is clear that this assignment descends to the equivalence class $c=[R]$. 

Conversely, $\fR(v)$ for an ideal vertex $v$ should be also defined as a equivalence class that contains the cups rectangle $\fR(e_1)\cap \fR(e_2)$ where $e_1$ and $e_2$ are edges sharing the ideal vertex $v$.

In \cite{SchleimerSegerman19}, it is shown that $\fL$ and $\fV$ are equivalences between two categories and the explicit natural transform between $\fL\circ \fV$ and the identity functor is constructed.

Let $\sV=\{\sL_1,\sL_2\}$ be a veering pair and $G$ a subgroup of $\Aut(\sV).$  Then, there is a canonical marking $\Mrk,$ from \refcor{canonicalmarking}, namely, $\Mrk=\Mrk(G).$ Let $D$ be the group of deck transformations of the universal covering $\widetilde{\fW^\circ}(\sV,\Mrk)\to \fW^\circ(\sV,\Mrk)$. By \cite{SchleimerSegerman21}, we can associate the canonical veering triangulation $\widetilde{\fV}(\sV, \Mrk):=\fV(\widetilde{\fW^\circ}(\sV,\Mrk))$ with the \emph{realisation} $|\widetilde{\fV}(\sV,\Mrk)|$ homeomorphic to $\RR^3$. 
This association is functorial in the sense that loom isomorphisms induce taut isomorphisms of the taut triangulation. 
In particular, a deck transformation of the universal covering $\widetilde{\fW^\circ}(\sV,\Mrk)\to \fW^\circ(\sV,\Mrk)$ acts as a taut isomorphism on $\widetilde{\fV}(\sV,\Mrk)$.  Since each element of $D$ preserves the orientations of $\widetilde{\fW^\circ}(\sV,\Mrk),$ the associated taut isomorphism preserves the co-orientation of the horizontal branched surface and the veering coloring. See \reflem{V(V)noncompact}. Hence, $D$ preserves the orientation and  transverse veering structure of $\RR^3$.

In the subsequent sections, if not mentioned otherwise,
$\sV=\{\sL_1, \sL_2\}$ is a veering pair, $G\leq \Aut(\sV)$ is a fixed laminar group, and $\Mrk$ is the marking $\Mrk(G)$ from \refcor{canonicalmarking}. In fact, most of the following results also hold under the assumption that $\Mrk$ is a $G$-invariant marking. Nonetheless, for simplicity, we assume that $\Mrk$ is $\Mrk(G)$. We write $p:\widetilde{\fW^\circ}(\sV,\Mrk)\to \fW^\circ(\sV,\Mrk)$ for the universal covering and $D$ its group of deck transformations. 

\begin{lem}\label{Lem:V(V)noncompact}
Let $D$ be the group of deck transformations for the universal covering $\widetilde{\fW^\circ}(\sV,\Mrk)\to \fW^\circ(\sV,\Mrk)$. 
\begin{enumerate}
    \item The induced action $\fV(D)$ of $D$ on $|\widetilde{\fV}(\sV, \Mrk)|$ is free and properly discontinuous.
    \item The $\fV(D)$-action on $\widetilde{\fV}(\sV, \Mrk)$ preserves the transverse taut structure. 
    \item The quotient space $|\widetilde{\fV}(\sV, \Mrk)|/\fV(D)$ is an orientable 3-manifold with the induced transverse veering triangulation, $$\fV(\sV, \Mrk)=\widetilde{\fV}(\sV, \Mrk)/\fV(D).$$
\end{enumerate}
\end{lem}
\begin{proof}
(1) By \reflem{correspondence}, we know that no elements in $D\setminus\{1\}$ stabilize a rectangle in $\widetilde{\fW^\circ}(\sV)$. Hence, the stabilizers of the skeletal rectangles are trivial. Therefore, the $\fV(D)$-action is properly discontinuous and free. 

(2) We only need to prove that the $\fV(D)$-action respects the co-orientations of the faces. Recall that each face $f=\mathsf{c}(R)$ of $\widetilde{\fV}(\sV, \Mrk)$ is given the co-orientation according to the counter-clockwise ordering on the three cusps of $R$ inherited from the orientation of $\widetilde{\fW^\circ}(\sV,\Mrk)$ by the right-hand grip rule. See \refsec{rectangle} and \cite[The proof of Lemma~5.12]{SchleimerSegerman21}.  Since the $D$-action on $\widetilde{\fW^\circ}(\sV,\Mrk)$ is orientation preserving, the circular order of the cusps is preserved. See \refsec{orientation}.
 
(3) By (1), $|\widetilde{\fV}(\sV,\Mrk)|/\fV(D)$ is a 3-manifold. By (2) and the fact that $\fV(D)$ preserves the orientation of $|\widetilde{\fV}(\sV,\Mrk)|$, $|\widetilde{\fV}(\sV,\Mrk)|/\fV(D)$ is orientable. Since each element of $D$ preserves the orientation of $\widetilde{\fW^\circ}(\sV,\Mrk),$ each deck transformation maps red and blue edge rectangles to red and blue edge rectangles, respectively. Therefore, by (2) and \refprop{fullcorrespondence}, $\widetilde{\fV}(\sV,\Mrk)/\fV(D)$ is a transverse veering triangulation of $|\widetilde{\fV}(\sV,\Mrk
)|/\fV(D).$ \end{proof}

\begin{lem}\label{Lem:subgroupnatural}
Let $c$ be a cusp of $\widetilde{\fW^\circ}(\sV,\Mrk)$. The isomorphism $\fV: D\to \fV(D)$ induces the isomorphism $\Stab{D}{c}\to \Stab{\fV(D)}{\mathsf{c}(c)}.$
\end{lem}
\begin{proof}
Since $\fV$ is already known to be an isomorphism, it suffices to show that $\fV(\Stab{D}{c})=\Stab{\fV(D)}{\mathsf{c}(c)}$ as sets. But this claim clearly follows from the construction of the functors $\fL$ and $\fV$ and the fact that they are equivalences of categories. 
\end{proof}

\subsection{Actions on Veering Triangulations}
This subsection is devoted to proving our main theorem. 

\begin{lem}\label{Lem:pdaction}
The group $\Aut(\sV,\Mrk)$  acts on $|\fV(\sV,\Mrk)|$ properly discontinuously and freely. Moreover, this action respects the transverse taut structure $\fV(\sV,\Mrk)$ and the veering coloring.\end{lem}
\begin{proof}
Let $g$ be an automorphism in $\Aut(\fW^\circ (\sV,\Mrk))=\Aut(\sV,\Mrk)$. Any such homeomorphism can be lifted to a loom isomorphism $\widetilde{g}$ on $\widetilde{\fW^\circ}(\sV,\Mrk)$. Note that $\widetilde{g}$ commutes with the group $D$ of deck transformations of the universal covering $\widetilde{\fW^\circ}(\sV,\Mrk)\to \fW^\circ(\sV,\Mrk)$.   The homeomorphism $\fV(\widetilde{g})$ on $\RR^3$ associated to $\widetilde{g}$  preserves the veering triangulation $\fV(\widetilde{\fW^\circ}(\sV,\Mrk))$. Since $\widetilde{g}$ commutes with $D$, $\fV(\widetilde{g})$ commutes with $\fV(D)$. Therefore, $\fV(\widetilde{g})$ gives rise to a homeomorphism on $|\fV(\widetilde{\fW^\circ}(\sV,\Mrk))|/\fV(D)$ and this homeomorphism is independent of the choice of the lift $\widetilde{g}$. Hence we have the well-defined action of $\Aut(\sV,\Mrk)$ on $|\fV(\sV,\Mrk)|$ and this action preserves the induced taut ideal triangulation $\fV(\sV,\Mrk)=\widetilde{\fV}(\sV,\Mrk)/\fV(D)$.

The fact that the action defined above is properly discontinuous and free follows from \refprop{tetrastabilizer} and \refprop{fullcorrespondence}.  
\end{proof}

Given an ideal vertex $v$ of $\fV(\sV,\Mrk)$, we define its \emph{link} $L(v)$ as follows. Choose a model tetrahedron $t$ that has $v$ as its vertex. Consider $L_t(v)$ the convex hull in the tetrahedron $t$ of three points chosen from each of three edges adjacent to $v$. We assume that the points are chosen close enough to $v$ so that $L_t(v)$ is completely situated in upper and lower cusp neighborhoods of $v$. We do this in the $\Stab{G}{v}$ equivariant manner. Then we define 
\[
L(v)=\bigcup_{t} L_t(v)
\]
where the union is taken over all tetrahedra $t$ that contain $v$ as their ideal vertex. 

We say that $v$ is a \emph{singular} vertex if $L(v)$ is homeomorphic to the infinite cylinder. If $L(v)$ is homeomorphic to the plane, $v$ is called a \emph{cusp} vertex. 

For a cusp vertex $v$ of $\widetilde{\fV}(\sV,\Mrk)$, we define its vertex link $L(v)$ similarly. We assume that $L(v)$ is chosen in the $\Stab{\fV(D)}{v}$-equivariant manner. 
\begin{lem}\label{Lem:stabilizerInD}
Let $v$ be an ideal vertex of $\fV(\sV,\Mrk)$. Let $q: \widetilde{\fV}(\sV,\Mrk) \to \fV(\sV,\Mrk)$ be the universal covering. Then, 
\begin{enumerate}
    \item $v$ is either a cusp or a singular vertex.
    \item $v$ is a cusp if and only if the stabilizer of each component of $q^{-1}(L(v))$ in $\fV(D)$ is trivial. 
    \item $v$ is a  singular point if and only if the stabilizer of each component of  $q^{-1}(L(v))$ in $\fV(D)$ is infinite cyclic.  
\end{enumerate}
\end{lem}
\begin{proof}
In the universal cover $\widetilde{\fV}(\sV,\Mrk)=\fV(\widetilde{\fW^\circ}(\sV,\Mrk))$,
each connected component of $q^{-1}(L(v))$ gives rise to a vertex link $L(c)$ of some ideal vertex in $\widetilde{\fV}(\sV,\Mrk)$. It follows that  $q^{-1}(L(v))$ consists of copies of planes. Since the ladderpole curves and  the transverse taut structure are preserved under the action of $\fV(D)$, $L(v)$ is either an infinite cylinder, a plane or a torus. The torus case can be ruled out because $D$ has no free abelian subgroups of rank 2. \end{proof}

By \cite{FuterGueritaud13}, we know that the ladderpole slopes decompose $L(c)$ into disjoint strips so-called \emph{ladders}. This ladder decomposition induces the \emph{equivalence relation} on the set of tetrahedra that have $c$ as their ideal vertex as follows: $t_1$ and $t_2$ are equivalent if and only if $L_{t_1}(c)$ and $L_{t_2}(c)$ are contained in the same ladder. The set of ladders is linearly ordered from right to left seen from the cusp. 

On the other hand, each cusp $c$ (see \refsec{loom} for the definition of cusps) of a loom space induces a partition on the set of tetrahedron rectangles containing the cusp $c$. To illustrate this, let us first define a \emph{west division} rectangle as a rectangle $R$ with an orientation preserving homeomorphism $f_R:(0,1)^2\to R$ such that there is a homeomorphism extension $\overline{f_R}:[0,1]^2 \setminus \{ 0 \times a\} \to \overline{R}$ for some $0< a < 1$. 

We define a \emph{east}, \emph{north} and \emph{south division} rectangles likewise. We say that tetrahedron rectangles $R_1$ and $R_2$  are equivalent if $R_1\cap R_2$ contains a division rectangle. A \emph{division} at $c$ is an equivalence class of this equivalence relation. The set of divisions is linearly ordered by the adjacency relation. 

\begin{rmk}
In \cite{SchleimerSegerman21}, a sector at a cusp $x$ in a loom space is defined as a connected component of the complement of all leaves passing through $x$. In fact, a division corresponds to two consecutive sectors in the sense of \cite[Definition 6.9]{SchleimerSegerman21}. Hence, the following proposition holds for sectors. 
\end{rmk}

\begin{prop}\label{Prop:divisioncorresp}
Let $c$ be a cusp vertex of a veering triangulation $\widetilde{\fV}(\sV,\Mrk)$. There is an order preserving one-to-one correspondence between the set of ladders in the ladder decomposition of $L(c)$ and the set of divisions at the corresponding cusp $\fR(c)$ in the loom space.
\end{prop}
\begin{proof}
For any tetrahedron rectangle $R$ that contains $\fR(c)$ as a cusp, let $\mathfrak{l}(R)$ be the ladder that contains $L_{\mathsf{c}(R)}(c)$.  We show that if tetrahedron rectangles $R_1$ and $R_2$ are in the same division $S$ then $\mathfrak{l}(R_1)=\mathfrak{l}(R_2)$. We prove this only for a west division. Remaining cases are similar. 

Choose a west division rectangle $R_0$ for the division  $S$.  We can split $R_0$ into two cusp rectangles, one with south-west ideal corner and the other with north-west ideal corner. These cusp rectangles generate staircases $U$ and $V$ respectively with the common lower axis ray $m$. Denote by $\Delta$ the union of all exterior cusps of these two staircases. Then there is a (injective) projection $\pi_m$ from $\Delta$ into $m$. For a tetrahedron rectangle $R$ in $S$,  the image of the east cusp of $R$ under the map $\pi_m$ is also denoted by  $\pi_m(R)\in m$. Note that a tetrahedron rectangle $R$ is in the division $S$ if and only if we have  $R \subset U\cup V$, $R\cap U\ne \emptyset$, and $R\cap V \ne\emptyset$.

Let $R_1$ and $R_2$ be tetrahedron rectangles in the same division $S$. Without loss of generality, we assume $R_1<R_2$, i.e., $R_1$ west-east spans $R_2$. We use the induction on the cardinality $n$ of image of $\pi_m(\Delta)$ between $\pi_m(R_1)$ and $\pi_m(R_2)$, which is finite by the Astroid lemma \cite[Lemma 4.10]{SchleimerSegerman21}. As the base case, assume that $n=1$. Let $x\in \pi_m(\Delta)$ be between $\pi_m(R_1)$ and $\pi_m(R_2)$. The preimage of $x$, say $z$, is either north or south cusp of $R_2$. In either cases, we can find a tetrahedron rectangle $Q$ in the same division that has $z$ as its east cusp such that $L_{\mathsf{c}(Q)}(c)$ shares rungs with $L_{\mathsf{c}(R_1)}(c)$ and with $L_{\mathsf{c}(R_2)}(c)$ since $\mathsf{c}(Q)$ shares a face with each $\mathsf{c}(R_i)$ which is corresponded to a rung. See Figure~\ref{config1}. Therefore, $R_1$ and $R_2$ are in the same ladder. Suppose now that the assertion holds for $n<k$ and assume $n=k$. We choose the closest point $x$ to $\pi_m(R_2)$ among the points in $\pi_m(\Delta)$ between $\pi_m(R_1)$ and $\pi_m(R_2)$. Let $z$ be the preimage of $x$. Then, $z$ must be either north or south cusp of $R_2$. We then find a tetrahedron rectangle $Q$ as in the $n=1$ case. We now use the induction hypothesis to conclude that $\mathsf{c}(Q)$ and $\mathsf{c}(R_2)$ are in the same ladder. 
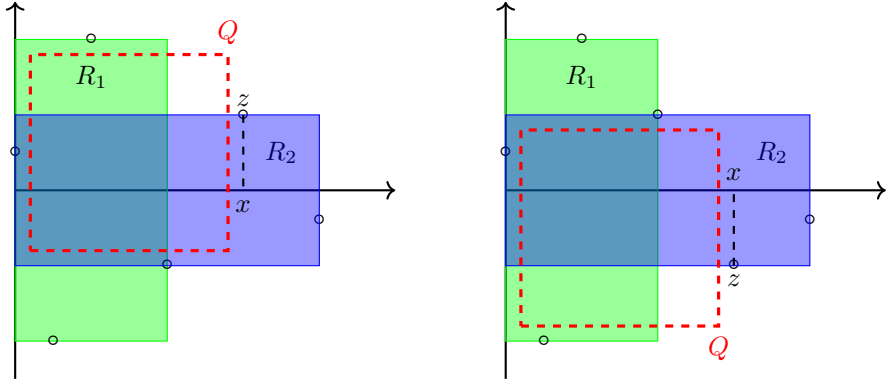
\begin{figure}[hbt]
    \centering
    \begin{tikzpicture}
    \draw[->, thick] (0,-2.5) -- (0,2.5);
    \draw[->, thick] (0,0) -- (5,0);
    \filldraw[fill=green, draw=green, fill opacity = 0.4] (0,-2) rectangle (2,2);
    \draw (1,1.5) node{$R_1$} (3.5,0.5) node{$R_2$};
    \draw (0,0.5)  node{$\circ$} (1,2) node{$\circ$}  (3,1) node{$\circ$} node[above]{$z$} (4,-0.4) node{$\circ$} (2,-1) node{$\circ$} (0.5,-2)node{$\circ$};
    \filldraw[fill=blue, draw=blue, fill opacity = 0.4] (0,-1) rectangle (4,1);
    \draw[dashed, thick] (3,1) -- (3,0) node[below]{$x$}; 
    \draw[very thick, red, dashed] (.2,-0.8) rectangle (2.8, 1.8) node[above]{$Q$};
    \end{tikzpicture}  \hspace{1cm} 
    \begin{tikzpicture}
    \draw[->, thick] (0,-2.5) -- (0,2.5);
    \draw[->, thick] (0,0) -- (5,0);
    \filldraw[fill=green, draw=green, fill opacity = 0.4] (0,-2) rectangle (2,2);
    \draw (1,1.5) node{$R_1$} (3.5,0.5) node{$R_2$};
    \draw (0,0.5)  node{$\circ$} (1,2) node{$\circ$}  (3,-1) node{$\circ$} node[below]{$z$} (4,-0.4) node{$\circ$} (2,1) node{$\circ$} (0.5,-2)node{$\circ$};
    \filldraw[fill=blue, draw=blue, fill opacity = 0.4] (0,-1) rectangle (4,1);
    \draw[dashed, thick] (3,-1) -- (3,0) node[above]{$x$}; 
    \draw[very thick, red, dashed] (.2,0.8) rectangle (2.8, -1.8) node[below]{$Q$};
    \end{tikzpicture}
    \caption{Two possible configurations for the proof of the proposition}
    \label{config1}
\end{figure}

For the inverse, let $t$ be a tetrahedron with a vertex $c$. Define $\mathfrak{s}(t)$ to be the division that contains $\fR(t)$. Again we prove that if $t_1$ and $t_2$ are in the same ladder, then $\mathfrak{s}(t_1)=\mathfrak{s}(t_2)$. Without loss of generality, we may assume that $t_1<t_2$. Then there is an increasing sequence of tetrahedra $t_1=u_1<\cdots<u_n = t_2$ such that $L_{u_i}(c)$ and $L_{u_{i+1}}(c)$ share a rung. We show that $\fR(u_i)$ and $\fR(u_{i+1})$ are in the same division. For this let $Q$ be the face rectangle  $\fR(u_i)\cap \fR(u_{i+1})$. Because $u_{i}$ and $u_{i+1}$ share a cusp $c$, $Q$ also contains the corresponding cusp $\fR(c)$ in the loom space. Since  $u_{i}$ and $u_{i+1}$ share a rung, the cusp $\fR(c)$ cannot be a corner of $Q$. Say $\fR(c)$ is a west cusp of $Q$. Then, it follows that $\fR(u_i)\cap \fR(u_{i+1})$ contains a west division rectangle, completing the well-definedness of $\mathfrak{s}$. 

From the construction, it is clear that $\mathfrak{s}$ and $\mathfrak{l}$ are inverse to each other and preserve orders. \end{proof}

Let $s$ be a cone class in $\closure{\fW}(\sV)\setminus \fW^\circ(\sV,\Mrk)$ of order $n\ge 2$. Let $c$ be a cusp in $\widetilde{\fW^\circ}(\sV,\Mrk)$ such that $\cusp(R)=s$ for all $R\in \mathfrak{p}(c)$ (see \refprop{fullcorrespondence} for the definition of $\mathfrak{p}$ and also see \refrmk{cuspCorner} for  the notation $\cusp(R)$). By \reflem{stabilizerInD}, the group $\Stab{D}{c}$ is an infinite cyclic group with a generator $\gamma$. The action of $\gamma$ on the linearly ordered set of divisions $\{\cdots, S_{-1},S_0,S_1,\cdots\}$ can be understood as follow: Assume that $\gamma$ preserves a division $S_k$ for some $k\in \ZZ.$ Then, there is a tetrahedron rectangle $R$ in $S_k$ such that $\gamma(R)$ is also in $S_k.$ By \reflem{correspondence}, this is a contradiction as $R$ and $\gamma(R)$ share  a division rectangle. Therefore, there is no division preserved by $\gamma.$ Thus, $\gamma$ acts on the set $\{\cdots, S_{-1},S_0,S_1,\cdots\}$ of divisions at $c$ as $S_i \mapsto S_{i+2n}$. This shows the following lemma. 

\begin{lem}\label{Lem:ladderdescend}
Let $c=q^{-1}(s)$ be a lift of a singular vertex $s$ of $\fV(\sV,\Mrk)$ where $q:\widetilde{\fV}(\sV,\Mrk) \to \fV(\sV,\Mrk)$ is the universal covering. Let $g\in \Stab{\fV(D)}{c}$ be a nontrivial element. For each ladder $L_i$, we have $g(L_i)\cap L_i=\emptyset$. In particular, the ladderpole lines in $L(c)$ descend to parallel lines in $L(s)$.
\end{lem}

%Consequently, we have the ladderpole decomposition on links $L(c)$ in $\fV(\sV,\Mrk)$. Each connected component of the complement of parallel ladderpole lines in $L(c)$ is called a \emph{ladder}. 

%Now we have the dictionary between the action of $G\le \Aut(\sV)$ on $\fW^\circ(\sV, \Mrk)$ and the induced action on $\fV(\sV, \Mrk)$.
%\begin{prop}\label{Prop:actioncorrespondence}
%We have $G$-equivariant one-to-one correspondences 
%\begin{align*}
%\{\text{Asterisks of crowns}\} &\longleftrightarrow \{\text{Cusp vertices in }\fV(\sV,\Mrk)\}\\
%\left\{\substack{\text{Asterisks of non-leaf polygons,}\\\text{Marked stitches}}\right\} &\longleftrightarrow \{\text{Singular vertices in }\fV(\sV,\Mrk)\}\\
%\{\text{tips of  $\sG_1$ and $\sG_2$ } \} &\longleftrightarrow \left\{\substack{\text{Ladders in the link of an ideal vertex}\\  \fV(\sV,\Mrk)\text{of associated with } w(\sG_1,\sG_2)}\right\}\\
%\{\text{Tetrahedron rectangles in  } \fW^\circ(\sV,\Mrk)\} &\longleftrightarrow \{\text{Tetrahedra in } \fV(\sV,\Mrk)\}
%\end{align*}
%\end{prop}

The following lemma is used to extend an action on an infinite cylinder to a solid cylinder. Let $\RR/\ZZ\times \RR/\ZZ$ and $\RR/\ZZ\times \RR$ be the standard torus and cylinder respectively, obtained as the quotients of $\RR\times \RR$ by the actions $(x,y)\mapsto (x+1,y)$, $(x,y)\mapsto (x,y+1)$.

\begin{lem}\label{Lem:extensionIntoCyl} Let $L$ be an infinite cylinder. Let $G=\ZZ+ \ZZ_n$. Suppose that $G$ acts freely and  properly discontinuously on $L$. Assume that the $G$-action fixes the ends of $L$.  Then there is a topological conjugacy  $\RR/\ZZ\times \RR \to L$ such that the $G$-action on $\RR/\ZZ \times \RR$ is standard: $(x,y)\mapsto (x+\frac{1}{n},y)$, $(x,y) \mapsto (x,y+1)$. Similar results hold when $G=\ZZ$ or $\ZZ_n$.
\end{lem}
\begin{proof}
We prove the lemma for the case when $G=\ZZ+\ZZ_n$. Other cases can be shown along the same line. 

We know that $L/ G$ must be a torus, otherwise, the $\ZZ$ factor of $G$ would not act properly discontinuously and freely on the infinite cylinder $L$. We regard $\pi_1(L)\cong \ZZ$ as a subgroup of  $\pi_1(L/G)$. Let $\phi : \ZZ +\ZZ \to \pi_1(L/G)$ be an isomorphism such that $\phi$ maps the first factor to $\pi_1(L)$. Then there is a homeomorphism $\psi: \RR/\ZZ \times \RR/\ZZ\to L/G$  such that $\psi_* = \phi$. The lift  $\widetilde{\psi} : \RR/\ZZ \times \RR \to L$ of $\psi$ provides the desired topological conjugacy.   
\end{proof}

For a singular vertex $s$ of $\fV(\sV,\Mrk)$, any essential simple closed curve in $L(s)$ is called a \emph{meridian}. 
\begin{lem}\label{Lem:pi1generators}
Let $\sV$ be a veering pair. Let $\mathfrak{S}$ be the set of singular vertices of $\fV(\sV,\Mrk)$.  Fix a base point $x_0$ in $|\fV(\sV,\Mrk)|$. Choose a meridian of $L(s)$ for each $s\in\mathfrak{S}$ in $|\fV(\sV,\Mrk)|$ and by choosing a path from $x_0$ to $L(s)$, regard this meridian as a loop $\gamma_s$ based at $x_0$. Then $\pi_1(|\fV(\sV,\Mrk)|,x_0)$ is normally generated by $\{[\gamma_s]\}_{s\in \mathfrak{S}}$, where $[\gamma_s]$ denotes an element in $\pi_1(|\fV(\sV,\Mrk)|,x_0)$ represented by $\gamma_s$.
\end{lem}
\begin{proof}
We have isomorphisms $\pi_1(|\fV(\sV,\Mrk)|)\cong \fV(D)\cong \pi_1(\fW^\circ (\sV,\Mrk))\cong D$. Note that $D$ is  generated by $\bigcup \Stab{D}{c}$, where the union is taken over all preimages of all singular and marked classes in $\fW^\circ(\sV, \Mrk)$. By \reflem{subgroupnatural}, $\Stab{D}{c}\cong \Stab{\fV(D)}{\mathsf{c}(c)}$. Thus, $\pi_1(|\fV(\sV,\Mrk)|,x_0) \cong \fV(D)$ is generated by $\bigcup \Stab{\fV(D)}{\mathsf{c}(c)}$.  For each generator $g$ of  $\Stab{\fV(D)}{\mathsf{c}(c)}$ there is $h\in \fV(D)$  such that $g=h [\gamma_s]^{\pm 1}h^{-1}$. Therefore, $\pi_1(|\fV(\sV,\Mrk)|)$ is normally generated by meridians.
\end{proof}

\begin{thm}\label{Thm:3orbifoldgroup}
Let $\sV=\{\sL_1, \sL_2\}$ be a veering pair and $G$ be a subgroup of $\Aut(\sV).$
Then $G$ is the fundamental group of an irreducible 3-orbifold. Moreover, this orbifold is obtained by orbifold Dehn fillings of a 3-manifold with a taut veering  triangulation.
\end{thm}

\begin{proof}
Let $X^\circ$ be the space obtained from $|\fV(\sV,\Mrk)|$ by removing small neighborhoods of ideal vertices. Boundary components of $X^\circ$ are cusp links. $X^\circ$ is a 3-manifold and $G$ acts on $|\fV(\sV,\Mrk)|$ properly discontinuously and freely (\reflem{V(V)noncompact} and \reflem{pdaction}). 

Let $\mathfrak{S}$ be the set of singular vertices of $\fV(\sV,\Mrk)$. By \reflem{extensionIntoCyl}, we can find a topological conjugacy $\phi_{s}:\RR/\ZZ\times \RR \to L(s)$ for each $s\in \mathfrak{S}$. Let $X$ be the space obtained by gluing the standard solid cylinders $C_s=\{z\in\CC\,:\,\|z\|\le 1\}\times \RR$ to each $\{L(s)\}_{s\in \mathfrak{S}}$ via the gluing maps $\{\phi_s\}_{s\in \mathfrak{S}}$. By \reflem{pi1generators}, $X$ is simply-connected. 

Now we extend the $G$-action on $X^\circ$ to $X$. First note that the standard action on $\RR/\ZZ\times \RR$ can be coned off to give the action on the standard solid cylinder. More precisely, introduce the cylindrical coordinates $(r,\theta, t)$ on the standard solid cylinder $\{z\in\CC\,:\,\|z\|\le 1\}\times \RR$. Given a homeomorphism $g$ on the boundary $\{(1,\theta,t)\}$, we define $\widehat{g}$ by $\widehat{g} (r,\theta,t) = (r, g(\theta,t))$. Then, given an element $g\in G$, we define $\widetilde{g}:X\to X$ as follow: For $x\in X^\circ$, we let $\widetilde{g}(x) = g(x)$. In the solid cylinder $C_s$ attached to $s\in \mathfrak{S}$, $\widetilde{g}:C_s \to C_{g\cdot s}$ is given by $\widehat{\phi_{g\cdot s}^{-1} \circ g \circ\phi_{s}}$. This gives a well-defined homeomorphism on $X$:  For if $x\in \partial C_s$ and $\phi_s(x)\in X^\circ$ are equivalent in $X$, we have $\phi_{g\cdot s} (\widehat{g}(x)) = \phi_{g\cdot s}\circ \phi_{g\cdot s}^{-1} \circ g \circ \phi_s (x)=g\circ \phi_s(x)$. Moreover, we have $\widetilde{gh} = \widetilde{g} \circ \widetilde {h}$ for all $g,h\in G$. The $G$-action on $X$ is then given by $g\cdot x = \widetilde{g} ( x)$. 

For the standard $\ZZ$, $\ZZ_n$, and $\ZZ_n+\ZZ$ actions on the standard cylinder, their coned-off actions on the solid cylinder are also properly discontinuous and their quotients are (possibly singular) solid cylinders or  solid tori. Since the $\Stab{G}{s}$-action on $\partial C_s$ is standard, the extended $G$-action on $X$ is still properly discontinuous (but may not be free). Therefore, $X/G$ is a 3-orbifold with $\pi_1 ^{\operatorname{orb}}(X) = G$. From the construction, we know that $X/G$ is a Dehn filling of $X^\circ/G$ by (possibly singular) solid cylinders or solid tori $C_s/\Stab{G}{s}$.

Finally, we show that $X/G$ is irreducible. Due to Proposition 3.23 followed by the remark in \cite{Porti2003}, a 3-orbifold is irreducible if and only if its universal cover is irreducible. Therefore, it is enough to show that $X$ is irreducible.

Consider the Mayer-Vietoris sequence: 
\[
H_2(\bigcup_{s\in\mathfrak{S}} \partial C_s) \to H_2(\bigcup_{s\in\mathfrak{S}} C_s)\oplus H_2(X^\circ) \to H_2(X) \to H_1(\bigcup_{s\in\mathfrak{S}} \partial C_s)\to H_1(\bigcup_{s\in\mathfrak{S}} C_s) \oplus H_1(X^\circ).
\]
We know that $H_2(\bigcup_{s\in\mathfrak{S}} \partial C_s) =  H_2(\bigcup_{s\in\mathfrak{S}} C_s)=H_1(\bigcup_{s\in\mathfrak{S}} C_s)  =0$, reducing the above sequence to
\[
0 \to  H_2(X^\circ) \to H_2(X) \to H_1(\bigcup_{s\in\mathfrak{S}} \partial C_s)\to H_1(X^\circ).
\]
Note that $H_1(\bigcup_{s\in\mathfrak{S}} \partial C_s)$ is generated by meridians of each $\partial C_s$ and by \reflem{pi1generators}, so is $H_1(X^\circ)$. Therefore, $H_1(\bigcup_{s\in\mathfrak{S}} \partial C_s)\to H_1(X^\circ)$ is an isomorphism. As a result, we get $H_2(X;\ZZ) \cong H_2(X^\circ;\ZZ)$.

Observe that $X^\circ$ is a $K(\pi_1(X^\circ),1)$ space. Hence, we know that $H_*(X^\circ;\ZZ)$ is isomorphic to the group homology $H_*(\pi_1(X^\circ);\ZZ)$. Since $\pi_1(X^\circ)$ is a free group (with at most countably many generators), we have that $H_2(\pi_1(X^\circ);\ZZ)=0$. Thus, $H_2(X^\circ;\ZZ)=0$ as well. Combined with the previous computation, we have that $H_2(X;\ZZ) = H_2(X^\circ;\ZZ)=0$. 

Since $X$ is simply-connected, it follows that $\pi_2(X)\cong H_2(X;\ZZ)=0$. By the sphere theorem, $X$ is irreducible. 
\end{proof}

\begin{rmk}
Here is another proof for the irreducibility of $X/G$ when the laminations are very full and $G$ does not have 2-torsions. Being cofinite means that $X/G$ is compact. By construction and \reflem{ladderdescend}, each torus boundary of $X^\circ /G$ has at least two parallel ladderpole curves and $X/G$ is the Dehn filling of $X^\circ/G$ along slopes $s$ that intersect every ladderpole curve transversely. Hence, $|\langle s,l\rangle |\ge 2$ where $l$ is the collection of ladderpole curves. By \cite{AgolTsang22}, $X/G$ admits a transitive pseudo-Anosov flow without perfect fits. Since $G$ does not contain $\ZZ+\ZZ$, $X/G$ is atoroidal. Therefore, the universal cover $X$ of $X/G$ is homeomorphic to $\RR^3$ and $X/G$ is irreducible \cite{CalegariDunfield03, Gabai89}.
\end{rmk}

\section{Veering Pairs and Kleinian Groups}\label{Sec:kleiniansec}
In this section, we improve our main \refthm{3orbifoldgroup} to yield more geometric conclusion. To do this, we need ``cocompactness'' of the action, although we believe that this assumption is not essential. 
\subsection{Cofinite Actions}

Let $\sV=\{\sL_1, \sL_2\}$ be a veering pair with a marking $\Mrk$ and $G$ a subgroup of $\Aut(\sV, \Mrk).$
Now, we say that the $G$ action  on the veering triangulation $\fV(\sV,\Mrk)$ is \emph{cofinite} if there are only finitely many orbit classes of ideal tetrahedra. In this section, we give a necessary and sufficient condition to make the $G$-action be cofinite.

Let $\sG$ be a gap of $\sV$ that is either a marked leaf or a non-leaf gap. The gap $\sG$ has \emph{the maximal rank stabilizer} if it is one of the following cases.
\begin{itemize}
    \item If $\sG$ is an ideal polygon, then $\Stab{G}{\sG}$ contains an infinite cyclic group. 
    \item If $\sG$ is a crown, then $\Stab{S}{\sG}$ contains a rank two free abelian group. 
\end{itemize}
We say that $G$ has the \emph{maximal stabilizer property} if every gaps that is either a marked leaf or a non-leaf gap of $\sV$ has the maximal rank stabilizer.

Assume that $G$ has the maximal stabilizer property. Let $\sG$ be a gap of $\sL_1$ that is either a marked leaf or a non-leaf gap. When $\sG$ is a marked leaf, then, by \refthm{gapStab}, there is a pA-like element $g$ in $\Stab{G}{\sG}.$ Furthermore, there is a leaf $\sH$ of $\sL_2$ such that $(\sG,\sH)$ is the stitch preserved by $g.$ Observe that $\sH$ is also marked. Hence, $(\sG, \sH)\in \Mrk.$ Also, observe that $\Stab{G}{\sG}$ contains a properly pseudo-Anosov as at least one of $g$ or $g^2$ is properly pseudo-Anosov. Similarly, when $\sG$ is an ideal polygon, $\Stab{G}{\sG}$ contains a properly pseudo-Anosov by \refthm{gapStab}. When $\sG$ is a crown, by \refthm{gapStab}, $\Stab{G}{\sG}$ contains a properly pseudo-Anosov and also contains a parabolic automorphism.

Let $\fF=(I_1,I_2, I_3, I_4)$ be a tetrahedron frame under $\Mrk.$ Then, by \reflem{extension}, $\cI(c_4(\fF),c_1(\fF))$ contains either a unique marked stitch or exactly two singular stitches. Let $s$ be such a stitch. There is a unique tip $t$ in $v(\eta_2(s))$ over which $I_1^*$  cross. Observe that $t$ does not depend on the choice of $s$ by \refrmk{tipBtwStitches}. Therefore, we say that the tetrahedron frame $\fF=(I_1,I_2,I_3,I_4)$ under $\Mrk$ is \emph{penetrated} by the tip $t.$ Note that $t\in \E{\sL_2}$ and $I_3$ crosses over $t.$ 

Conversely, if $p$ in $\E{\sL_2}$ and the tip gap $\tipg{p}$ is a marked leaf or a non-leaf gap, then we can take a tetrahedron frame under $\Mrk$ penetrated by $p$ in the similar way of the proof of  \reflem{propertiesofweaving}-(1). Now, we denote the set of tetrahedron frames under $\Mrk$ by $\fT(\sV,\Mrk).$ Also, for any $t$ in $\E{\sL_2}$ such that $\tipg{t}$ is either a marked leaf or a non-leaf gap, we denote the set of tetrahedron frames under $\Mrk$ penetrated by a tip $t$ by $\fT(\sV,\Mrk,t).$ Then, we have that  the set $$\{\fT(\sV,\Mrk,t):\text{$t\in \E{\sL_2}$ and $\tipg{t}$ is a marked leaf or a non-leaf gap}\}$$
is a partition for $\fT(\sV,\Mrk).$

Then, observe that if two distinct tetrahedron frames $\fF_1=(I_1,I_2, I_3, I_4)$ and $\fF_2=(J_1, J_2, J_3, J_4)$ are penetrated by a  same tip $t,$ then $I_1=J_1.$ To see this, we first assume that $\tipg{t}$ is a marked leaf. Then, both $(\ell(I_1),\tipg{t})$ and $(\ell(J_1),\tipg{t})$ are marked leaves and so $I_1=J_1$ since $I_1^*$ and $J_1^*$ cross over $t$ and $\eta_1|\Mrk$ is injective by the definition of markings. When $\tipg{t}$ is a non-leaf gap, both $I_1^*$ and $J_1^*$ are the element of the interleaving gap of  $\tipg{t}$ crossing over $t$ and so $I_1=J_1.$ Hence, for a tetrahedron frame $\fF=(I_1, I_2, I_3, I_4)$ under $\Mrk$ penetrated by a tip $t,$ we denote $I_1$ by $I_t$ and $I_3$ by $\pi_t(\fF).$

Let $t$ be an end point of $\sL_2$ whose tip gap is either a marked leaf or a non-leaf gap. 
We may think of $\pi_t$ as an injective map from $\fT(\sV,\Mrk,t)$ to the stem $\stem{t}{I_t^*}$ in $\sL_1.$ Hence, we can define a linear order  $\leq_t$ on $\fT(\sV,\Mrk,t)$ as follows. For any $\fF_1$ and $\fF_2$ in $\fT(\sV,\Mrk,t),$ $\fF_1\leq_t \fF_2$ if and only if $\pi_t(\fF_1)\subseteq \pi_t(\fF_2).$ Hence, $\fT(\sV,\Mrk,t)$ is  linearly ordered and $\pi_t$ is order preserving.

\begin{rmk}\label{Rmk:tipFrame}
Let $t$ be an end point of $\sL_2$ whose tip gap is either a marked leaf or non-leaf gap.
For any automorphism $g$ in $G,$ $$g(\fT(\sV,\Mrk,t))=\fT(\sV,\Mrk,g(t))$$ and $g$ is order preserving under $\leq_t$ and $\leq_{g(t)}.$
\end{rmk}

\begin{lem}\label{Lem:cofinite}
Let $\sV=\{\sL_1,\sL_2\}$ be a veering pair and $G$ a subgroup of $\Aut(\sV).$ Set $\Mrk$ as the marking $\Mrk(G).$ Assume that $G$ has the maximal stabilizer property. If the sets of non-leaf gaps of $\sL_i$ consist of finitely many orbit classes under the $G$-action and $\Mrk$ also consists of finitely many orbit classes under the $G$-action, then the $G$-action on $\fV(\sV,\Mrk)$ is cofinite.  
\end{lem}
\begin{proof}
Let $t$ be an end point of $\sL_2$ whose tip gap is either  a marked leaf or a non-leaf gap. Since $G$ has the maximal stabilizer property, there is a properly pseudo-Anosov automorphism $g$ in $\Stab{G}{\tipg{t}}.$ Note that $I_t$ is the element of the interleaving gap of $\tipg{t}$ as explained. Then, $g$ preserves $\fT(\sV,\Mrk,t).$ This implies that  $\fT(\sV,\Mrk,t)$ consists of finitely many orbit classes under the $\langle g \rangle$-action. 

Now, we consider the set  $\fE$ of all end points of $\sL_2$ whose tip gap are marked leaves or non-leaf gaps. Note that the stabilizer of each crown in $\sL_2$ has a parabolic automorphism. Hence, by assumption, we can see that $\fE$ has finitely many orbit classes under the $G$-action. Since $\{\fT(\sV,\Mrk,t):t\in \fE\}$ is a partition of $\fT(\sV,\Mrk),$ by \refrmk{tipFrame}, $\fT(\sV,\Mrk)$ consists of finitely many orbit classes under the $G$-action. Thus, by \refrmk{tetFrameToRect}, we can conclude that the $G$-action on $\fV(\sV,\Mrk)$ is cofinite. 
\end{proof}

The converse implication is also true. See the following proposition.
\begin{prop}\label{Prop:converseCofinite}
Let $\sV=\{\sL_1,\sL_2\}$ be a veering pair and $G$ a subgroup of $\Aut(\sV).$ Set $\Mrk$ as the marking $\Mrk(G).$ If the $G$-action on $\fV(\sV,\Mrk)$ is cofinite, then $G$ satisfies the following.
\begin{enumerate}
    \item $G$ has the maximal stabilizer property,
    \item the sets of non-leaf gaps of $\sL_i$ consist of finitely many orbit classes under the $G$-action, and 
    \item $\Mrk$ consists of finitely many orbit classes under the $G$-action.
\end{enumerate}
\end{prop}
\begin{proof}
Since the $G$-action on $\sV(\sV,\Mrk)$ is cofinite, equivalently, $\fT(\sV,\Mrk)$ consists of the finitely many orbit classes under the $G$-action. Hence, the (2) and (3) follow immediately. It is enough to show (1). 

First, let $(\sG_1,\sG_2)$ be an interleaving pair 
in $\sV$ such that $(\sG_1,\sG_2)$ is either a marked stitch or an asterisk. Choose a tip $t$ of $\sG_2.$ Then, observe that $\fT(\sV,\Mrk,t)$ is countable since $\Mrk$ is at most countable and in each $\sL_i,$ there are countably many non-leaf gaps by totally disconnectedness. Hence, as $\fT(\sV,\Mrk)$ consists of finitely many orbit classes under the $G$-action, by the Pigeonhole principle, there are distinct frames  $\fF_1$ and $\fF_2$ in $\fT(\sV,\Mrk,t)$ such that $g(\fF_1)=\fF_2$ for some non-trivial automorphism $g$ in $G.$ Then, as $$g(\fT(\sV,\Mrk,t))=\fT(\sV,\Mrk,g(t))=\fT(\sV,\Mrk,t),$$ $g(t)=t$ and $g(I_t)=I_t.$ 
Since $I_t\in \sG_1,$ by \refthm{classOfLaminarAut}, $g(\sG_1)=\sG_1.$ Hence, $g$ is a properly pseudo-Anosov in $\Stab{G}{\sG_i}.$ Therefore, each $\Stab{G}{\sG_i}$ contains an infinite cyclic subgroup.

Now, assume that $(\sG_1,\sG_2)$ is an asterisk of crowns of $\sV$. Then, for each tip $s$ of $\sG_2,$ $\fT(\sV,\Mrk, s)$ is countable and the set 
$$\bigcup_{\text{a tip }s\in v(\sG_2)} \fT(\sV,\Mrk, s)$$
is also countable.  Since $\fT(\sV,\Mrk)$ consists of finitely many orbit classes under the $G$-action, there are distinct tips $t_1$ and $t_2$ in $v(\sG_2)$ penetrating frames $\fF_1$ and $\fF_2,$  respectively, such that $h(\fF_1)=\fF_2$ for some non-trivial automorphism $h\in G.$ Then, $h(I_{t_1})=I_{t_2}$ and, by \refthm{classOfLaminarAut}, $h(\sG_2)=\sG_2$. Then, $\sG_i$ are preserved by $\langle h\rangle$ and as $h$ does not fix the tips, $h$ is parabolic. Therefore, as  $\langle g, h\rangle \leq \Stab{G}{\sG_i},$ by \refthm{gapStab}, each $\Stab{G}{\sG_i}$ contains a rank two free abelian group. Thus, $G$ has the maximal stabilizer property. 
\end{proof}

\subsection{Veering Pairs and Kleinian Groups}
We have some corollaries related to  Conjecture 8.8 of \cite{Baik15}. Note that every veering pair is also a pants-like $\COL_2$ pair as shown in \refsec{promotiontocol2}. 

\begin{cor}\label{Cor:geometric}
Let $\sV=\{\sL_1,\sL_2\}$ be a veering pair and $G$ a subgroup of $\Aut(\sV).$ Set $\Mrk$ as the marking $\Mrk(G).$ If the $G$-action on $\fV(\sV,\Mrk)$ is cofinite, then $G$ is the fundamental group of a hyperbolic 3-orbifold.  
\end{cor}
\begin{proof}
We show this by appealing to the geometrization theorem.

We have already showed some properties of $X/G$. In \refthm{3orbifoldgroup}, we proved that $X/G$ is irreducible. Moreover, thanks to  \refthm{elementary} and \refprop{converseCofinite}, we know that $G$ does not contain an infinite cyclic normal subgroup nor a non-peripheral rank 2 free abelian subgroup. Hence, $X/G$ is homotopically atoroidal.

If $\sV$ does not contain any non-leaf ideal polygon gaps and $\Mrk$ is empty, $G$ is torsion-free  by \refthm{classOfLaminarAut}, $X/G$ becomes a compact irreducible atoroidal 3-manifold with infinite $\pi_1(X/G)=G$. By the Perelman-Thurston hyperbolization theorem, $X/G$ is hyperbolic. 

If $\sV$ contains some (and therefore infinitely many) non-leaf ideal polygon gaps or marked leaves, $X$ becomes irreducible homotopically atoroidal 3-orbifold with non-empty singular locus. By the orbifold geometrization theorem \cite[Theorem 9.1]{Porti2003}, $X$ is geometric. Since $X$ is homotopically atoroidal and $G=\pi_1(X)$ is infinite, the only possible geometry that $X/G$ can support is hyperbolic. \end{proof}

\begin{ex}
There are many examples where \refcor{geometric} applies. The most common situation is when $G$ is the fundamental group of the mapping torus of a hyperbolic surface by a pseudo-Anosov mapping class. In this case, $G$ preserves a veering pair induced from the invariant laminations of the pseudo-Anosov mapping class. 
\end{ex}

\begin{cor}
Let $\sV$ be a veering pair and $G$ a subgroup of $\Aut(\sV)$. If the $G$-action on $\fV(\sV,\Mrk(G))$ is cofinite, then $G$ is relatively hyperbolic with respect to $\{\Stab{G}{\sG}\,|\,\sG\text{ is a crown}\}$.
\end{cor}

\begin{cor} \label{Cor:tautFoli}
Let $\sV=\{\sL_1,\sL_2\}$ be a veering pair without non-leaf polygons and $G$ a torsion free subgroup of $\Aut(\sV)$. If the $G$-action on $\fV(\sV)$ is cofinite, then $|\fV(\sV)|/G$ is a tautly foliated hyperbolic $3$-manifold whose fundamental group is $G$. 
\end{cor}
\begin{proof}
By \refthm{3orbifoldgroup} and  \refcor{geometric}, $|\fV(\sV)|/G$ is a hyperbolic $3$-manifold with a taut veering triangulation $\fV(\sV)/G$ whose fundamental group is $G$. Moreover, as remarked in \cite{Lackenby00}, the horizontal branched surface of $\fV(\sV)/G$  carries at least one taut foliation. Thus, the result follows. 
\end{proof}

\section{Next step}
We suspect that $\Aut(\sV)=\Aut(\sV,\Mrk(\Aut(\sV)))$ is still geometric even without the cofinite condition. More weakly, we also expect that every subgroup $G$ of $\Aut(\sV)$ is relatively hyperbolic with respect to $\{\Stab{G}{\sG}:\sG \text{ is a crown}\}.$ One of supportive evidences is  \refcor{noBSgroup} in the sense of the Gersten conjecture \cite[Q1.1]{Bestvina}.  Unfortunately, our argument does rely on the cocompactness of the action and does not seem to work well in the noncompact setting. One of directions to tackle this issue is to see the action of $\Aut(\sV)$ on the ``boundary'' of the space $X$ in the proof of \refcor{geometric},  although no natural compactification of this space is known yet. With the characterization of the Bowditch boundary \cite{Yaman}, this leads us to the following question. 

\begin{ques}
Let $\sV$ be a veering pair. Does $\Aut(\sV)$ admit a convergence action on $S^2$?
\end{ques}

More precisely, given a veering pair $\sV=\{\sL_1,\sL_2\},$ the partition $Q=Q_1\cup J(Q_2)$ is a cellular decomposition of $\hat{\CC}$ where $Q_1,$ $Q_2,
$ and $J$ are defined in \refsec{decompositions}. Hence, by \refthm{Moore}, $\cD(Q)$ is homeomorphic to the sphere. Furthermore, $\Aut(\sV)$ faithfully acts on $\cD(Q).$ By \refthm{classification}, each element of $\Aut(\sV)$ acts on the sphere like an element of $\PSL{\CC}.$ Thereofore, the precise statement is of the following form.
\begin{ques}
Does any subgroup $G$ of $\Aut(\sV)$ act on $\cD(Q)$ as a convergence action? Furthermore, if the $G$-action on $\fV(\sV,\Mrk(G))$ is cofinite, does $\cD(Q)$ consist only of conical limit points and bounded parabolic points?
\end{ques}
A related result can be found in \cite{AlonsoBaikSamperton} where they proved that a subgroup of $\Homeop(S^1)$ preserving a pseudo-fibered pair of very full lamination systems such that all elements are hyperbolic admits a convergence action on $S^2$. 

As of now, we do not know any ``non-trivial'' example of veering pair. Namely, all known veering pairs come from some extra structures defined in $3$-manifolds, e.g. pseudo-Anosov flow, essential lamination, or veering triangulation.  
Because our setting is quite general and seemingly independent of 3-dimensional topology, we expect that there are interesting constructions and examples of veering pairs. Especially, we hope that the following question holds.

\begin{ques}
Does every hyperbolic group with sphere boundary act on the circle preserving a veering pair?
\end{ques}

 If so, our theorem gives an answer of  Cannon's conjecture.

\end{document}